\documentclass[11pt,reqno]{amsart}

\usepackage{amsfonts,latexsym,amsthm,amssymb,amsmath,amscd,euscript}
\usepackage{framed}
\usepackage{fullpage}
\usepackage[margin=0.75in]{geometry}
\usepackage[bookmarks,bookmarksopen,bookmarksdepth=2]{hyperref}
    \hypersetup{colorlinks=true,citecolor=blue,urlcolor =gray,linkcolor=gray,linkbordercolor={1 0 0}}
\usepackage{tikz-cd}
\usepackage{comment}
\usepackage{mathtools}
\usepackage{mathdots}
\usepackage{enumerate}
\usepackage[new]{old-arrows}
\usepackage{mathrsfs}
\usepackage{pgfplots}
\usepackage{cancel}
\usepackage{rotating}
\usepackage{bookmark}
\usepackage{schemata}
\usepackage{environ}
\usepackage{graphics}
\usepackage{graphicx}
\usepackage{xfrac}
\usepackage{nicefrac}
\usepackage{tensor}
\usepackage{expdlist}
\usepackage{stmaryrd}
\usetikzlibrary{positioning}
\DeclareFontFamily{U}{mathx}{\hyphenchar\font45}
\DeclareFontShape{U}{mathx}{m}{n}{
      <5> <6> <7> <8> <9> <10>
      <10.95> <12> <14.4> <17.28> <20.74> <24.88>
      mathx10
      }{}
\DeclareSymbolFont{mathx}{U}{mathx}{m}{n}
\DeclareFontSubstitution{U}{mathx}{m}{n}
\DeclareMathAccent{\widecheck}{0}{mathx}{"71}
\DeclareMathAccent{\wideparen}{0}{mathx}{"75}

\makeatletter
\DeclareFontFamily{U}{tipa}{}
\DeclareFontShape{U}{tipa}{m}{n}{<->tipa10}{}
\newcommand{\arc@char}{{\usefont{U}{tipa}{m}{n}\symbol{62}}}%

\newcommand{\arc}[1]{\mathpalette\arc@arc{#1}}

\newcommand{\arc@arc}[2]{%
  \sbox0{$\m@th#1#2$}%
  \vbox{
    \hbox{\resizebox{\wd0}{\height}{\arc@char}}
    \nointerlineskip
    \box0
  }%
}
\makeatother

\makeatletter
\def\widebreve{\mathpalette\wide@breve}
\def\wide@breve#1#2{\sbox\z@{$#1#2$}%
     \mathop{\vbox{\m@th\ialign{##\crcr
\kern0.08em\brevefill#1{0.8\wd\z@}\crcr\noalign{\nointerlineskip}%
                    $\hss#1#2\hss$\crcr}}}\limits}
\def\brevefill#1#2{$\m@th\sbox\tw@{$#1($}%
  \hss\resizebox{#2}{\wd\tw@}{\rotatebox[origin=c]{90}{\upshape(}}\hss$}
\makeatletter

\usepackage{color}
\definecolor{blaze}{rgb}{0.00,0.420,0.00}

\DeclareRobustCommand\longtwoheadrightarrow{\relbar\joinrel\twoheadrightarrow}
\allowdisplaybreaks[1]

\NewEnviron{Thm*}[1][]{
	\vspace{10pt}\newline\noindent\ignorespaces\hspace*{-12.5pt}\DoBrackets\schema{}{\vspace{-7.5pt}\begin{theorem*}[#1]\BODY\newline\noindent$\square$\end{theorem*}\vspace{-7.5pt} }\ignorespacesafterend\leavevmode\vspace{12.5pt}\newline
}
\NewEnviron{Thm}[1][]{
	\vspace{10pt}\newline\noindent\ignorespaces\hspace*{-12.5pt}\DoBrackets\schema{}{\vspace{-7.5pt}\begin{theorem}[#1]\BODY\newline\noindent$\square$\end{theorem}\vspace{-7.5pt} }\ignorespacesafterend\leavevmode\vspace{12.5pt}\newline
}
\NewEnviron{Prop*}[1][]{
	\vspace{10pt}\newline\noindent\ignorespaces\hspace*{-12.5pt}\DoBrackets\schema{}{\vspace{-7.5pt}\begin{proposition*}[#1]\BODY\newline\noindent$\square$\end{proposition*}\vspace{-7.5pt} }\ignorespacesafterend\leavevmode\vspace{12.5pt}\newline
}
\NewEnviron{Prop}[1][]{
	\vspace{10pt}\newline\noindent\ignorespaces\hspace*{-12.5pt}\DoBrackets\schema{}{\vspace{-7.5pt}\begin{proposition}[#1]\BODY\newline\noindent$\square$\end{proposition}\vspace{-7.5pt} }\ignorespacesafterend\leavevmode\vspace{12.5pt}\newline
}
\NewEnviron{Def*}[1][]{
	\vspace{10pt}\newline\noindent\ignorespaces\hspace*{-7.5pt}\DoParens\schema{}{\vspace{-7.5pt}\begin{definition*}[#1]\BODY\end{definition*}\vspace{-7.5pt} }\ignorespacesafterend\leavevmode\vspace{12.5pt}\newline
}
\NewEnviron{Def}[1][]{
	\vspace{5pt}\newline\noindent\ignorespaces\hspace*{-7.5pt}\DoParens\schema{}{\vspace{-7.5pt}\begin{definition}[#1]\BODY\end{definition}\vspace{-7.5pt} }\ignorespacesafterend\leavevmode\vspace{7.5pt}\newline
}
\NewEnviron{Lem*}[1][]{
	\vspace{10pt}\newline\noindent\ignorespaces\hspace*{-12.5pt}\DoBrackets\schema{}{\vspace{-7.5pt}\begin{lemma*}[#1]\BODY\newline\noindent$\square$\end{lemma*}\vspace{-7.5pt} }\ignorespacesafterend\leavevmode\vspace{12.5pt}\newline
}
\NewEnviron{Lem}[1][]{
	\vspace{10pt}\newline\noindent\ignorespaces\hspace*{-12.5pt}\DoBrackets\schema{}{\vspace{-7.5pt}\begin{lemma}[#1]\BODY\newline\noindent$\square$\end{lemma}\vspace{-7.5pt} }\ignorespacesafterend\leavevmode\vspace{12.5pt}\newline
}
\NewEnviron{Fact}[1][]{
	\vspace{10pt}\newline\noindent\ignorespaces\hspace*{-12.5pt}\DoBrackets\schema{}{\vspace{-7.5pt}\begin{fact}[#1]\BODY\newline\noindent$\square$\end{fact}\vspace{-7.5pt} }\ignorespacesafterend\leavevmode\vspace{12.5pt}\newline
}
\NewEnviron{Fact*}[1][]{
	\vspace{10pt}\newline\noindent\ignorespaces\hspace*{-12.5pt}\DoBrackets\schema{}{\vspace{-7.5pt}\begin{fact*}[#1]\BODY\newline\noindent$\square$\end{fact*}\vspace{-7.5pt} }\ignorespacesafterend\leavevmode\vspace{12.5pt}\newline
}
\NewEnviron{THM*}[1][]{
	\vspace{10pt}\newline\noindent\ignorespaces\hspace*{-12.5pt}\DoBrackets\schema{}{\vspace{-7.5pt}\begin{theorem*}[#1]\BODY\end{theorem*}\vspace{-6pt} }\ignorespacesafterend\leavevmode\vspace{12.5pt}\newline
}
\NewEnviron{FACT*}[1][]{
	\vspace{10pt}\newline\noindent\ignorespaces\hspace*{-12.5pt}\DoBrackets\schema{}{\vspace{-7.5pt}\begin{fact*}[#1]\BODY\end{fact*}\vspace{-6pt} }\ignorespacesafterend\leavevmode\vspace{12.5pt}\newline
}
\NewEnviron{LEM*}[1][]{
	\vspace{10pt}\newline\noindent\ignorespaces\hspace*{-12.5pt}\DoBrackets\schema{}{\vspace{-7.5pt}\begin{lemma*}[#1]\BODY\end{lemma*}\vspace{-6pt} }\ignorespacesafterend\leavevmode\vspace{12.5pt}\newline
}
\NewEnviron{PROP*}[1][]{
	\vspace{10pt}\newline\noindent\ignorespaces\hspace*{-12.5pt}\DoBrackets\schema{}{\vspace{-7.5pt}\begin{proposition*}[#1]\BODY\end{proposition*}\vspace{-6pt} }\ignorespacesafterend\leavevmode\vspace{12.5pt}\newline
}
\NewEnviron{COR*}[1][]{
	\vspace{10pt}\newline\noindent\ignorespaces\hspace*{-12.5pt}\DoBrackets\schema{}{\vspace{-7.5pt}\begin{corollary*}[#1]\BODY\end{corollary*}\vspace{-6pt} }\ignorespacesafterend\leavevmode\vspace{12.5pt}\newline
}
\NewEnviron{THM}[1][]{
	\vspace{10pt}\newline\noindent\ignorespaces\hspace*{-12.5pt}\DoBrackets\addtocounter{theorem}{-1}\schema{}{\vspace{-7.5pt}\begin{theorem}[#1]\BODY\end{theorem}\vspace{-6pt} }\ignorespacesafterend\leavevmode\vspace{12.5pt}\newline
}
\NewEnviron{FACT}[1][]{
	\vspace{10pt}\newline\noindent\ignorespaces\hspace*{-12.5pt}\DoBrackets\addtocounter{theorem}{-1}\schema{}{\vspace{-7.5pt}\begin{fact}[#1]\BODY\end{fact}\vspace{-6pt} }\ignorespacesafterend\leavevmode\vspace{12.5pt}\newline
}
\NewEnviron{LEM}[1][]{
	\vspace{10pt}\newline\noindent\ignorespaces\hspace*{-12.5pt}\DoBrackets\addtocounter{theorem}{-1}\schema{}{\vspace{-7.5pt}\begin{lemma}[#1]\BODY\end{lemma}\vspace{-6pt} }\ignorespacesafterend\leavevmode\vspace{12.5pt}\newline
}
\NewEnviron{PROP}[1][]{
	\vspace{10pt}\newline\noindent\ignorespaces\hspace*{-12.5pt}\DoBrackets\addtocounter{theorem}{-1}\schema{}{\vspace{-7.5pt}\begin{proposition}[#1]\BODY\end{proposition}\vspace{-6pt} }\ignorespacesafterend\leavevmode\vspace{12.5pt}\newline
}
\NewEnviron{COR}[1][]{
	\vspace{10pt}\newline\noindent\ignorespaces\hspace*{-12.5pt}\DoBrackets\addtocounter{theorem}{-1}\schema{}{\vspace{-7.5pt}\begin{corollary}[#1]\BODY\end{corollary}\vspace{-6pt} }\ignorespacesafterend\leavevmode\vspace{12.5pt}\newline
}
\NewEnviron{EX*}[1][]{
    \vspace{-1em}\begin{adjustwidth}{2em}{0pt}
    \begin{example*}
    \BODY
    \end{example*}
    \end{adjustwidth}\vspace{0.5em}
}
\NewEnviron{PRF}[1][]{
    \begin{adjustwidth}{2em}{0pt}
    \begin{proof}
    \BODY
    \end{proof}
    \end{adjustwidth}\vspace{0.5em}
}
\NewEnviron{DEF}[1][]{
	\vspace{0.5em}\newline\noindent\ignorespaces\hspace*{-7.5pt}\DoParens\addtocounter{theorem}{-1}\schema{}{\vspace{-7.5pt}\begin{definition}[#1]\BODY\end{definition}\vspace{-7.5pt} }\ignorespacesafterend\leavevmode\vspace{0.6em}\newline
}
\NewEnviron{PRFNAME}[1][]{
    \vspace{0.5em}\begin{adjustwidth}{2em}{0pt}
    \begin{proof}[#1]
    \BODY
    \end{proof}
    \end{adjustwidth}\vspace{0.5em}
}
\NewEnviron{CLAIM}[1][]{
	\vspace{10pt}\newline\noindent\ignorespaces\hspace*{-12.5pt}\DoBrackets\addtocounter{theorem}{-1}\schema{}{\vspace{-7.5pt}\begin{claim}[#1]\BODY\end{claim}\vspace{-6pt} }\ignorespacesafterend\leavevmode\vspace{12.5pt}\newline
}
\NewEnviron{CONJ}[1][]{
	\vspace{10pt}\newline\noindent\ignorespaces\hspace*{-12.5pt}\DoBrackets\addtocounter{theorem}{-1}\schema{}{\vspace{-7.5pt}\begin{conjecture}[#1]\BODY\end{conjecture}\vspace{-6pt} }\ignorespacesafterend\leavevmode\vspace{12.5pt}\newline
}

\theoremstyle{definition}
\newtheorem{theorem}{Theorem}[section]
\newtheorem*{theorem*}{Theorem}%
\newtheorem*{proposition*}{Proposition}
\newtheorem{proposition}[theorem]{Proposition}
\newtheorem{lemma}[theorem]{Lemma}
\newtheorem*{lemma*}{Lemma}
\newtheorem{fact}[theorem]{Fact}
\newtheorem*{fact*}{Fact}
\newtheorem{corollary}[theorem]{Corollary}
\newtheorem{conjecture}[theorem]{Conjecture}

\theoremstyle{definition}
\newtheorem{definition}[theorem]{Definition}
\newtheorem*{definition*}{Definition}

\newtheorem*{claim*}{Claim}
\newtheorem{claim}[theorem]{Claim}
\newtheorem*{corollary*}{Corollary}
\newtheorem*{example*}{Example}

\theoremstyle{remark}
\newtheorem*{remark}{Remark}

\newcommand{\BC}{\mathbb C}

\newcommand{\BZ}{\mathbb Z}
\newcommand{\BN}{\mathbb N}

\newcommand{\id}{\textnormal{id}}%used to be id; 09/24
\renewcommand{\d}{\textnormal{d}}
\DeclareMathOperator{\img}{Img}

\newcommand{\nc}{\newcommand}

\nc{\on}{\operatorname}
\nc{\spec}{\on{Spec}}

\DeclareMathOperator{\Ker}{Ker}
\renewcommand{\ker}{\Ker}
\DeclareMathOperator{\Hom}{Hom}
\renewcommand{\hom}{\Hom}

\DeclarePairedDelimiter\pr{(}{)}

\DeclareMathOperator{\sgn}{sgn}
\DeclareMathOperator{\Spec}{Spec}
\renewcommand{\spec}{\Spec}

\newcommand{\ol}[1]{\overline{#1}}

\DeclareMathOperator{\End}{End}

\DeclareMathOperator{\rad}{Rad}

\nc{\wt}[1]{\widetilde{#1}}
\nc{\wh}[1]{\widehat{#1}}
\nc{\pd}{\partial}
\nc{\cal}[1]{\mathcal{#1}}
\renewcommand{\frak}[1]{\mathfrak{#1}}
\nc{\scr}[1]{\mathscr{#1}}
\nc{\tsubseteq}[1]{\overset{\textnormal{#1}}{\subseteq}}
\nc{\tsupseteq}[1]{\overset{\textnormal{#1}}{\supseteq}}
\nc{\tsubset}[1]{\overset{\textnormal{#1}}{\subset}}
\nc{\tsupset}[1]{\overset{\textnormal{#1}}{\supset}}
\nc{\vphi}{\varphi}
\nc{\beau}{\displaystyle}
\nc{\existss}{\exists\;}
\nc{\foralls}{\forall\;}
\DeclarePairedDelimiter\abs{\lvert}{\rvert}
\DeclarePairedDelimiter\set{\{}{\}}
\nc{\ov}[1]{\overrightarrow{#1}}
\nc{\cdotsc}{\cdots\hspace{-0.1pt}}

\nc{\D}{\textnormal{D}}

\nc{\tn}[1]{\textnormal{#1}}
\nc{\lto}{\longrightarrow}
\nc{\lmto}{\longmapsto}

\DeclareMathOperator{\gal}{Gal}
\DeclareMathOperator{\res}{res}

\nc{\sm}{\setminus}

\nc{\scomp}{\textsf{c}}
\DeclareMathOperator{\cnt}{cnt}

\nc{\ncnt}{\lnot\cnt}
\nc{\ucong}{\overset{!}{\cong}}
\nc{\tcong}[1]{\overset{#1}{\cong}}
\nc{\coker}{\on{Coker}}

\nc{\eps}{\varepsilon}
\nc{\BS}{\mathbb{S}}
\nc{\BT}{\mathbb{T}}
\nc{\SC}{\mathscr{C}}

\DeclareMathOperator{\tor}{Tor}
\nc{\wa}[1]{\wideparen{#1}}
\nc{\qidl}{\frak{q}}
\nc{\nidl}{\frak n}

\nc{\te}[1]{\text{#1}}
\nc{\imps}{\quad\quad\ \:\,}
\nc{\comp}{\complement}
\nc{\van}{\cal V}
\nc{\ide}{\cal I}
\nc{\nfrac}[2]{\nicefrac{#1}{#2}}
\nc{\snfrac}[2]{\scalebox{1.45}{$\nfrac{#1}{#2}$}}
\nc{\inj}{\hookrightarrow}
\nc{\surj}{\twoheadrightarrow}
\nc{\linj}{\longhookrightarrow}
\nc{\lsurj}{\longtwoheadrightarrow}
\nc{\four}{\cal F}
\nc{\urot}[1]{\mathbin{\rotatebox[origin=c]{90}{$#1$}}}
\nc{\drot}[1]{\mathbin{\rotatebox[origin=c]{-90}{$#1$}}}
\nc{\vrot}[2]{\mathbin{\rotatebox[origin=c]{#2}{$#1$}}}

\nc{\wb}[1]{\widebreve{#1}}
\DeclareMathOperator{\Id}{Id}
\nc{\BH}{\mathbb{H}}

\nc{\T}{\textnormal{T}}
\nc{\N}{\vrot{\T}{180}\!}
\nc{\blt}{\bullet}
\nc{\lbd}{\lambda}

\renewcommand{\tt}[1]{\texttt{#1}}
\renewcommand{\sf}[1]{\mathsf{#1}}
\nc{\sq}{{\mathop{\square}}}
\nc{\onl}[1]{\operatorname*{#1}}
\nc{\bigger}[1]{\scalebox{1.5}{$#1$}}
\nc{\tbt}[4]{\begin{pmatrix}#1&#2\\#3&#4\end{pmatrix}}
\nc{\thbth}[9]{\begin{pmatrix}#1&#2&#3\\#4&#5&#6\\#7&#8&#9\end{pmatrix}}
\nc{\tbo}[2]{\begin{pmatrix}#1\\#2\end{pmatrix}}
\nc{\thbo}[3]{\begin{pmatrix}#1\\#2\\#3\end{pmatrix}}
\nc{\tsq}[3]{\tensor{{#1}}{^{#2}_{#3}}}
\nc{\rnk}{\on{rnk}}
\nc{\Ad}{\on{Ad}}
\nc{\ad}{\on{ad}}
\nc{\glie}{\frak{g}}
\nc{\SO}{\on{SO}}
\nc{\sk}{\on{Sk}}
\nc{\snakeanchor}{\ar[draw=none]{d}[name=X, anchor=center]{}}
\nc{\snakearrow}[1]{\ar[rounded corners,
	to path={ -- ([xshift=2ex]\tikztostart.east)
		|- (X.center) \tikztonodes
		-| ([xshift=-2ex]\tikztotarget.west)
		-- (\tikztotarget)}]{dll}{#1}}
\nc{\cosnakearrow}[1]{\ar[rounded corners,
	to path={ -- ([xshift=-2ex]\tikztostart.west)
		|- (X.center) \tikztonodes
		-| ([xshift=2ex]\tikztotarget.east)
		-- (\tikztotarget)}]{urr}{#1}}
\nc{\colim}{\on*{colim}}
\nc{\sing}{\on{Sing}}
\makeatletter
\newcommand{\dircolim@}[2]{%
  \vtop{\m@th\ialign{##\cr
    \hfil$#1\operator@font colim$\hfil\cr
    \noalign{\nointerlineskip\kern1.5\ex@}#2\cr
    \noalign{\nointerlineskip\kern-\ex@}\cr}}%
}
\newcommand{\dircolim}{%
  \mathop{\mathpalette\dircolim@{\rightarrowfill@\textstyle}}\nmlimits@
}
\makeatother
\makeatletter
\newcommand{\xdashrightarrow}[2][]{\ext@arrow 0359\rightarrowfill@@{#1}{#2}}
\newcommand{\xdashleftarrow}[2][]{\ext@arrow 3095\leftarrowfill@@{#1}{#2}}
\newcommand{\xdashleftrightarrow}[2][]{\ext@arrow 3359\leftrightarrowfill@@{#1}{#2}}
\def\rightarrowfill@@{\arrowfill@@\relax\relbar\rightarrow}
\def\leftarrowfill@@{\arrowfill@@\leftarrow\relbar\relax}
\def\leftrightarrowfill@@{\arrowfill@@\leftarrow\relbar\rightarrow}
\def\arrowfill@@#1#2#3#4{%
  $\m@th\thickmuskip0mu\medmuskip\thickmuskip\thinmuskip\thickmuskip
   \relax#4#1
   \xleaders\hbox{$#4#2$}\hfill
   #3$%
}
\makeatother
\nc{\lcm}{\on{lcm}}
\nc{\Stab}{\on{Stab}}
\nc{\Nbhd}{\on{Nbhd}}
\nc{\Lie}{\on{Lie}}
\nc{\gl}{\on{\mathfrak{gl}}}
\nc{\Gr}{\on{Gr}}
\nc{\act}{\on{act}}
\nc{\ev}{\on{ev}}
\nc{\Sq}{\on{Sq}}
\nc{\ext}{\on{Ext}}
\nc{\Ext}{\on{Ext}}
\nc{\Dist}{\on{Dist}}
\nc{\dlie}{\frak{d}}
\nc{\hlie}{\frak{h}}
\nc{\pt}{\te{pt}}
\nc{\qedcheck}{\vspace{-1em}\begin{flushright}$\checkmark\quad$\end{flushright}\vspace{-0.5em}}
\nc{\lemproof}[1]{\smallskip\noindent\underline{#1:}}
\nc{\F}{\te{F}}
\nc{\jota}{\jmath}
\nc{\Tor}{\tor}
\nc{\Fl}{\on{Fl}}
\nc{\prob}{\on{prob}}
\nc{\didl}{\frak{d}}
\nc{\Rep}{\on{\sf{Rep}}}
\nc{\Irrep}{\on{\sf{irRep}}}
\nc{\Res}{\on{Res}}
\nc{\Ind}{\on{Ind}}
\nc{\Der}{\on{Der}}
\renewcommand{\sl}{\on{\frak{sl}}}
\nc{\so}{\on{\frak{so}}}

\nc{\Img}{\img}
\nc{\Sidl}{\frak{S}}
\nc{\Slie}{\frak{S}}
\nc{\iita}{\imath}
\nc{\hidl}{\hlie}
\nc{\blie}{\frak{b}}
\nc{\bidl}{\frak{b}}
\nc{\nlie}{\nidl}
\nc{\larr}{\leftarrow}
\nc{\darr}{\downarrow}
\nc{\nwarr}{\nwarrow} 
\nc{\searr}{\searrow} 
\nc{\Tot}{\on{Tot}}
\nc{\Wt}{\on{Wt}}
\nc{\sfcong}[1]{\overset{\sf{#1}}{\cong}}
\nc{\vtheta}{\vartheta}
\nc{\JH}{\on{JH}}
\nc{\Typ}{\on{Typ}}
\nc{\alie}{\frak{a}}

\nc{\llra}{\longleftrightarrow}
\nc{\Llra}{\Longleftrightarrow} 
\nc{\lla}{\longleftarrow} 
\nc{\lra}{\longrightarrow} 
\nc{\Lra}{\Longrightarrow}  
\newcommand\sqplus{\mathbin{\ooalign{$\sqcup$\cr%
   \hfil\raise0.42ex\hbox{$\scriptscriptstyle+$}\hfil\cr}}}
\newcommand\sqminus{\mathbin{\ooalign{$\sqcup$\cr%
   \hfil\raise0.42ex\hbox{$\scriptscriptstyle-$}\hfil\cr}}}
\newcommand\uminus{\mathbin{\ooalign{$\cup$\cr%
   \hfil\raise0.42ex\hbox{$\scriptscriptstyle-$}\hfil\cr}}}
   \DeclarePairedDelimiter\bigpr{\big(}{\big)}

\nc{\CO}{\cal O}
\nc{\nor}{\on{nor}}
\nc{\CP}{\cal P}
\nc{\CF}{\cal F}
\nc{\CG}{\cal G}
\nc{\sh}{\te{sh}}
\nc{\Supp}{\on{Supp}}
\nc{\Sh}{\on{\sf{Sh}}}
\nc{\rv}{\rvert}
\nc{\bigrv}{\big\rvert}
\nc{\Bigrv}{\Big\rvert}
\nc{\biggrv}{\bigg\rvert}
\nc{\Biggrv}{\Bigg\rvert}
\nc{\Cl}{\on{Cl}}
\nc{\existsu}{\exists!\;}
\nc{\disc}{\on{disc}}
\nc{\Proj}{\on{Proj}}
\nc{\op}{\te{\,op}}
\nc{\actson}{\mathrel{\vrot{\circlearrowright}{270}}}
\nc{\Mod}{\sf{Mod}}
\nc{\Grp}{\sf{Grp}}
\nc{\Set}{\sf{Set}}
\nc{\Alg}{\sf{Alg}}

\nc{\Ring}{\sf{Ring}}
\nc{\Ann}{\on{Ann}}
\DeclarePairedDelimiter\wan{\langle}{\rangle}
\nc{\osubseteq}{\overset{\te{o}}{\subseteq}}
\nc{\osubset}{\overset{\te{o}}{\subset}}
\nc{\osupseteq}{\overset{\te{o}}{\supseteq}}
\nc{\osupset}{\overset{\te{o}}{\supset}}
\nc{\CV}{\cal V}
\nc{\CI}{\cal I}
\nc{\glu}{\on{glu}}
\nc{\jidl}{\frak j}
\nc{\Span}{\on{Span}}

\nc{\oscong}[1]{\overset{#1}{\cong}}
\nc{\ossubseteq}[1]{\overset{#1}{\subseteq}}
\nc{\ossupseteq}[1]{\overset{#1}{\supseteq}}
\nc{\ossubset}[1]{\overset{#1}{\subset}}
\nc{\ossupset}[1]{\overset{#1}{\supset}}
\nc{\Gal}{\gal}
\nc{\uad}{\:\;}
\nc{\CQ}{\cal Q}
\nc{\stspace}{\ \ \,}
\nc{\ractson}{\mathrel{\vrot{\circlearrowleft}{90}}}
\nc{\len}{\on{len}}
\nc{\CC}{\cal C}
\nc{\bsfrac}[2]{\scalebox{1.4}[1.5]{$\sfrac{#1}{#2}$}}
\newcommand\quot[2]{
        \mathchoice
            {% \displaystyle
                \text{\raise1ex\hbox{$#1$}\Big/\lower1ex\hbox{$#2$}}%
            }
            {% \textstyle
                #1\,/\,#2
            }
            {% \scriptstyle
                #1\,/\,#2
            }
            {% \scriptscriptstyle  
                #1\,/\,#2
            }
    }
\nc{\acts}{\actson}
\nc{\racts}{\ractson}
\nc{\calD}{\cal D}
\nc{\Fib}{\on{Fib}}
\nc{\cartsq}{\arrow[dr, phantom, "\ulcorner", very near start]}
\nc{\Glu}{\on*{Glu}}
\nc{\bnfrac}[2]{\raisebox{-0.175em}{\scalebox{1.4}[1.5]{$\nfrac{#1}{#2}$}}}
\nc{\mfrac}[2]{\raisebox{-0.15em}{\scalebox{1.32}[1.4]{$\nfrac{#1}{#2}$}}} %%%% USE THIS FOR 11PT
\nc{\Sch}{\sf{Sch}}
\nc{\tlto}[1]{\overset{\te{#1}}{\lto}}
\nc{\tlmto}[1]{\overset{\te{#1}}{\lmto}}
\nc{\tlsurj}[1]{\overset{\te{#1}}{\lsurj}}
\nc{\tlinj}[1]{\overset{\te{#1}}{\linj}}
\nc{\oslto}[1]{\overset{#1}{\lto}}
\nc{\oslmto}[1]{\overset{#1}{\lmto}}
\nc{\oslsurj}[1]{\overset{#1}{\lsurj}}
\nc{\oslinj}[1]{\overset{#1}{\linj}}

\nc{\cornk}{\on{cornk}}
\nc{\ute}[1]{\uad\te{#1}}
\nc{\CL}{\cal L}
\nc{\val}{\on{val}}
\nc{\warn}[1]{\textcolor{red}{#1}}
\nc{\tsp}{\ \ }
\nc{\tte}[1]{\tsp\te{#1}}
\nc{\QCoh}{\sf{QCoh}}
\nc{\qcoh}{\sf{QCoh}}
\nc{\onqcoh}{\on{\qcoh}}
\renewcommand{\sh}{\on{Sh}}
\nc{\Weil}{\on{Weil}}
\nc{\vrho}{\varrho}
\nc{\cD}{\cal D}
\nc{\shom}{\on{\cal H \te{om}}}
\nc{\Rho}{\te{P}}
\nc{\CM}{\cal M}
\nc{\Pic}{\on{Pic}}
\nc{\sep}{\te{sep}}
\nc{\BG}{\mathbb{G}}
\nc{\BGm}{{\BG_{\te m}}}

\nc{\Lbd}{\Lambda}
\nc{\Cone}{\on{Cone}}
\nc{\cone}{\Cone}
\nc{\Cd}{\cal D}
\nc{\CH}{\cal H}
\nc{\CX}{\cal X}
\nc{\heart}{\heartsuit}
\nc{\CA}{\cal A}
\nc{\BGa}{\BG_\te a}
\nc{\Lla}{\Longleftarrow}
\nc{\suptrieq}{\mathbin{\unrhd}}
\nc{\R}{\te{R}}

\nc{\rhom}{\on{RHom}}
\nc{\lotimes}{\mathbin{\overset{\te{L}}{\otimes}}}
\nc{\bdd}{\te{bdd}}
\nc{\twt}{\on{Wt}}
\nc{\weit}{\on{Wt}}
\nc{\typ}{\on{Typ}}
\nc{\fun}{\on{\sf{Fun}}}
\nc{\fib}{\on{Fib}}
\nc{\Cofib}{\on{Cofib}}
\nc{\cofib}{\on{Cofib}}
\nc{\calid}{\cal{I}\te{d}}
\nc{\ladj}{{\vrot{\vdash}{90}}}
\nc{\radj}{{\vrot{\dashv}{90}}}
\nc{\Ch}{\on{\sf{Ch}}}
\nc{\ohlie}{\ol{\hlie^*}}
\nc{\Lan}{\on{Lan}}
\nc{\Ran}{\on{Ran}}
\nc{\CB}{\cal B}
\makeatletter
\newcommand{\subalign}[1]{%
  \vcenter{%
    \Let@ \restore@math@cr \default@tag
    \baselineskip\fontdimen10 \scriptfont\tw@
    \advance\baselineskip\fontdimen12 \scriptfont\tw@
    \lineskip\thr@@\fontdimen8 \scriptfont\thr@@
    \lineskiplimit\lineskip
    \ialign{\hfil$\m@th\scriptstyle##$&$\m@th\scriptstyle{}##$\hfil\crcr
      #1\crcr
    }%
  }%
}
\makeatother
\nc{\veps}{\epsilon}
\nc{\aut}{\on{Aut}}
\nc{\saut}{\on{\cal{A}\te{ut}}}
\nc{\Fgt}{\on{Fgt}}
\nc{\fgt}{\on{Fgt}}
\nc{\ul}[1]{\underline{#1}}
\nc{\SL}{\on{SL}}
\nc{\radu}{\rad_\te{u}}
\nc{\tlie}{\frak{t}}
\nc{\Rad}{\on{Rad}}
\nc{\CE}{\cal E}

\nc{\CU}{\cal U}
\nc{\sproj}{\on{\cal P\te{roj}}}
\nc{\chom}{\shom}
\nc{\cproj}{\sproj}
\nc{\ao}{\te{ao}}
\nc{\CProj}{\cproj}
\nc{\gari}{{g_\te{ari}}}
\nc{\Coh}{\on{\sf{Coh}}}
\nc{\red}{\te{red}}
\nc{\Bl}{\on{Bl}}
\nc{\Hilb}{\on{Hilb}}
\nc{\FX}{\frak{X}}
\usepackage{CJKutf8}
\nc{\AlgGrp}{\sf{AlgGrp}}
\nc{\subnoreq}{\trianglelefteq}
\nc{\subnor}{\triangleleft}
\nc{\supnoreq}{\trianglerighteq}
\nc{\supnor}{\triangleright}
\nc{\Var}{\sf{Var}}
\renewcommand{\D}{\sf D}
\nc{\longlongline}{\noindent\makebox[\linewidth]{\rule{\linewidth}{0.4pt}}}%https://tex.stackexchange.com/questions/19579/horizontal-line-spanning-the-entire-document-in-latex
\nc{\vareps}{\epsilon}
\usepackage{array}

\makeatletter
\newcommand{\dirlim@}[2]{%
  \vtop{\m@th\ialign{##\cr
    \hfil$#1\operator@font lim$\hfil\cr
    \noalign{\nointerlineskip\kern1.5\ex@}#2\cr
    \noalign{\nointerlineskip\kern-\ex@}\cr}}%
}
\newcommand{\dirlim}{%
  \mathop{\mathpalette\dirlim@{\leftarrowfill@\textstyle}}\nmlimits@
}
\makeatother
\makeatletter

\renewcommand{\R}{\sf{R}}
\renewcommand{\rhom}{\on{\sf{R}\te{Hom}}}
\renewcommand{\lotimes}{\mathbin{\overset{\sf{L}}{\otimes}}}
\nc{\Ass}{\on{Ass}}

\usepackage{changepage}
\nc{\midwedge}{{\textstyle\bigwedge}}
\renewcommand{\calid}{\on{\cal{I}\te{d}}}

\usepackage{accents}

\renewcommand{\shom}{\on{\mathcal H\hspace*{-0.5pt}\textsl{om}}}

\usepackage{bbm}
\newcommand{\bk}{\mathbbm{k}}
\makeatletter %%\def\l@subsection{\@tocline{2}{0pt}{1pc}{5pc}{}}
\def\l@subsection{\@tocline{2}{0pt}{2pc}{6pc}{}} 
\makeatother

\newcommand{\scaleeq}[2]{\scalebox{#1}{$#2$}}

\newcommand{\from}{\leftarrow}
\renewcommand{\acts}{{\actson}}
\renewcommand{\racts}{{\ractson}}

\usetikzlibrary{decorations.markings}%%https://pbelmans.ncag.info/blog/2016/11/30/improved-open-and-closed-immersions-in-tikz-cd/

\makeatletter
\tikzcdset{
  open/.code     = {\tikzcdset{hook, circled};},
  closed/.code   = {\tikzcdset{hook, slashed};},
  open'/.code    = {\tikzcdset{hook', circled};},
  closed'/.code  = {\tikzcdset{hook', slashed};},
  circled/.code  = {\tikzcdset{markwith = {\draw (0,0) circle (.375ex);}};},
  slashed/.code  = {\tikzcdset{markwith = {\draw[-] (-.4ex,-.4ex) -- (.4ex,.4ex);}};},
  markwith/.code ={
    \pgfutil@ifundefined%
    {tikz@library@decorations.markings@loaded}%
    {\pgfutil@packageerror{tikz-cd}{You need to say %
      \string\usetikzlibrary{decorations.markings} to use arrows with markings}{}}{}%
    \pgfkeysalso{/tikz/postaction = {
      /tikz/decorate,
      /tikz/decoration={markings, mark = at position 0.5 with {#1}}}
    }
  },
}
\makeatother

\newcommand{\CK}{\cal K}

\DeclareFontFamily{OT1}{pzc}{}
\DeclareFontShape{OT1}{pzc}{m}{it}%
{<-> s * [1.15] pzcmi7t}{}
\DeclareMathAlphabet{\mathpzc}{OT1}{pzc}{m}{it} %https://tex.stackexchange.com/questions/165621/how-to-insert-a-special-math-alphabet

\usepackage{ytableau}
\ytableausetup{smalltableaux}

\newcommand{\ydia}[1]{{\hackcenter{\scaleeq{0.5}{\ydiagram{#1}}}}}
\newcommand{\yd}[1]{{\scaleeq{0.35}{\ydiagram{#1}}}}
\renewcommand{\yd}[1]{{\mathchoice
  {{\hackcenter{\scaleeq{0.4}{\ydiagram{#1}}}}}% \displaystyle
  {{\hackcenter{\scaleeq{0.35}{\ydiagram{#1}}}}}% \textstyle
  {\hspace{0.5pt}\hackcenter{\scaleeq{0.35}{\ydiagram{#1}}}}% \scriptstyle
  {{\hspace{0.25pt}\hackcenter{\scaleeq{0.25}{\ydiagram{#1}}}}}% \scriptscriptstyle
}}

\newcommand{\simlto}{\oslto\sim}

\nc{\lmfrac}[2]{\raisebox{-0.15em}{\scalebox{1.32}[1.4]{$\reflectbox{\nfrac{\reflectbox{\ensuremath{#1}}}{\reflectbox{\ensuremath{#2}}}}$}}} %see here https://tex.stackexchange.com/questions/390405/how-do-you-typeset-a-horizontally-flipped-nicefrac
\nc{\lrfrac}[3]{\raisebox{0em}{{$\lmfrac{#2}{#1}\hspace{-0.4em}\mfrac{\ }{#3}$}}}
\newcommand{\CS}{\cal S}

\DeclareMathAlphabet{\mathdutchcal}{U}{dutchcal}{m}{n}
\SetMathAlphabet{\mathdutchcal}{bold}{U}{dutchcal}{b}{n}
\DeclareMathAlphabet{\mathdutchbcal}{U}{dutchcal}{b}{n}

\renewcommand{\midwedge}{\raisebox{0.15em}{${\textstyle\bigwedge}$}}

\newcommand{\Db}{\sf D^{\te{b}}}
\nc{\BK}{\mathbb K}

\makeatletter %%\def\l@subsection{\@tocline{2}{0pt}{1pc}{5pc}{}}
\def\l@subsection{\@tocline{2}{0pt}{2pc}{6pc}{}} 
\makeatother 

\nc{\Nabla}{\nabla}
\nc{\sic}{\sf{St}\infty\sf{Cat}}
\nc{\gr}{\on{gr}}
\nc{\mr}[1]{\mathring{#1}}
\nc{\tsl}[1]{\textsl{#1}}
\nc{\sym}{\te{sym}}
\renewcommand{\ext}{\on{ext}}
\nc{\CZ}{\mathcal Z}
\nc{\nat}{\natural}
\nc{\specseqimplies}{\uad\implies\uad}

\nc{\tX}{\te X}\nc{\tH}{\te H}\nc{\tY}{\te Y}
\nc{\NB}{{\te{N}\CB}}
\nc{\paperprogressmarker}{\warn{PAPER\ WRITEUP\ PROGRESS\ MARKER\ HERE}}
\nc{\RHom}{\rhom}
\nc{\kidl}{\frak{k}}
\nc{\klie}{\frak{k}}
\newcommand{\longsquigglyrightarrow}{\xymatrix{{}\ar@{~>}[r]&{}}} % https://tex.stackexchange.com/questions/99017/long-squiggly-arrows-in-latex
\usepackage[T2A]{fontenc}
\makeatletter
\def\easycyrsymbol#1{\mathord{\mathchoice
  {\mbox{\fontsize\tf@size\z@\usefont{T2A}{\rmdefault}{m}{n}#1}}
  {\mbox{\fontsize\tf@size\z@\usefont{T2A}{\rmdefault}{m}{n}#1}}
  {\mbox{\fontsize\sf@size\z@\usefont{T2A}{\rmdefault}{m}{n}#1}}
  {\mbox{\fontsize\ssf@size\z@\usefont{T2A}{\rmdefault}{m}{n}#1}}
}}
\makeatother

\renewcommand{\d}{\textnormal{d}}

\newcommand{\altfrac}[2]{\ifmmode\def\tmp{$}\else\def\tmp{}\fi\mbox{%
    {\raisebox{.24\ht\strutbox}{\tmp#1\tmp}}%
    \kern-2pt\scalebox{1.6}[1.5]{/}\kern-3pt%
    \mbox{%
    {\raisebox{-.24\ht\strutbox}{\tmp#2\tmp}}%
    }}}
\renewcommand{\mfrac}[2]{\altfrac{#1}{#2}}

\nc{\NH}{{\te{N}\CH}}
\nc{\refl}{\on{refl}}
\renewcommand{\sh}{\on{sh}}
\nc{\CR}{\cal R}
\nc{\BY}{\mathbb{Y}}
\nc{\by}{\mathbbm{y}}
\nc{\cont}{\on{cont}}
\nc{\slie}{\frak{s}}
\nc{\tK}{\te{K}}
\nc{\Std}{\on{Std}}
\nc{\conf}{\on{conf}}
\nc{\dm}[1]{\overset{\diamond}{#1}}
\nc{\diam}{\diamond}
\nc{\doml}{\vartriangleleft}
\nc{\domg}{\vartriangleright}
\nc{\domle}{\trianglelefteq}
\nc{\domge}{\trianglerighteq}
\nc{\shape}{\on{shape}}
\nc{\ulie}{\frak{u}}
\usepackage{tabularray}
\nc{\cyan}[1]{\textcolor{cyan}{#1}}
\usepackage{adjustbox}
\usepackage{stackengine}  %https://tex.stackexchange.com/questions/413333/overlay-symbols
\newcommand\stackenginehack[2][c]{%
  \bgroup%
  \setstackEOL{ }% 
  \setstackgap{L}{0pt}%
  \Longstack[#1]{#2}%
  \egroup%
}

\nc{\JT}{{\mr{\CR}}}
\nc{\DCR}{{\mr{\CR}}}
\nc{\Zhe}{{\easycyrsymbol{\CYRZH}}}
\nc{\zhe}{{\easycyrsymbol{\cyrzh}}}
\nc{\isp}{\vrot{\psi}{180}}
\nc{\Isp}{\vrot{\Psi}{180}}

\def\lollipop{\vrot{\kern-3pt\mbox{$\bullet$}\kern-1.5pt\mbox{$\relbar$}\kern+1pt}{270}}
\def\lollipopsub{\vrot{\mbox{$\bullet$}\kern-1.5pt\mbox{$\relbar$}\kern+5pt}{270}}
\nc{\lol}{{\mathchoice{{\scalebox{0.7}{$\lollipop$}}}{{\scalebox{0.8}{$\lollipop$}}}{{\scalebox{0.6}{$\lollipopsub$}}}  % 0.7
{{\scalebox{0.45}{$\lollipopsub$}}}}} %0.5
\nc{\LL}{\on{LL}}
\nc{\tU}{\te{U}}
\nc{\tD}{\te{D}}

\NewEnviron{diagram}%
  {\mathord{\hackcenter{\begin{tikzpicture}[scale=0.375]
    \BODY
\end{tikzpicture}}
  }}
\tikzset{every picture/.style={line width=0.9pt}} %https://tex.stackexchange.com/questions/206734/pgfplots-how-can-i-set-all-the-default-line-widths-thickness-values-to-a-cer

\renewcommand{\tbo}[2]{\begin{pmatrix*}[l]#1\\#2\end{pmatrix*}}
\renewcommand{\thbo}[3]{\begin{pmatrix*}[l]#1\\#2\\#3\end{pmatrix*}}

\def\fliplollipop{\vrot{\kern+2pt\mbox{$\bullet$}\kern-1.5pt\mbox{$\relbar$}\kern+1pt}{90}}
\def\fliplollipopsub{\kern+1pt\vrot{\mbox{$\bullet$}\kern-1.5pt\mbox{$\relbar$}\kern-8pt}{90}}
\nc{\usdlol}{{\mathchoice{{\scalebox{0.7}{$\fliplollipop$}}}{{\scalebox{0.8}{$\fliplollipop$}}}{{\scalebox{0.6}{$\fliplollipopsub$}}}  % 0.7
{{\scalebox{0.45}{$\fliplollipopsub$}}}}}
\nc{\eqqcolon}{\vrot{\coloneqq}{180}}
\definecolor{darkgreen}{rgb}{0.1,0.7,0.1}
\nc{\bff}{{\bf{f}}}
\nc{\bfe}{\mathbf{e}}
\renewcommand{\bfe}{\alpha}

\nc{\ulx}{{\ul x}}
\nc{\uly}{{\ul y}}
\nc{\ulw}{{\ul w}}
\nc{\ulu}{{\ul u}}
\nc{\ulv}{{\ul v}}
\nc{\uli}{{\ul i}}
\nc{\ulj}{{\ul j}}
\nc{\ulk}{{\ul k}}

\NewEnviron{raisediagram}[1][]{
	\raisebox{#1}{$\mathord{\hackcenter{\begin{tikzpicture}[scale=0.375]
    \BODY
\end{tikzpicture}}
  }$}
}
\nc{\Plie}{{\frak{P}}}
\nc{\Qlie}{{\frak{Q}}}
\renewcommand{\Plie}{{\frak{p}}}
\renewcommand{\Qlie}{{\frak{q}}}
\nc{\olLL}{\ol\LL}
\renewcommand{\olLL}{\on{\Gamma\Gamma}}
\renewcommand{\op}{\te{op}}
\nc{\fil}{\on{fil}}

\NewDocumentEnvironment{mytheorem}{m}%
  {\renewcommand{\thetheorem}{$\mathbf{#1}$}%
   \begin{theorem}
  }%
  {\end{theorem}}
\nc{\pad}{\on{pad}}
\nc{\mult}{\on{mult}}
\renewcommand{\coker}{\on{coKer}}
\renewcommand{\cofib}{\on{coFib}}

\DeclareMathSymbol{\tempshrek}{\mathclose}{operators}{'074}
\nc{\coD}{\sf{coD}}
\nc{\coMod}{\sf{coMod}}
\nc{\jmth}{\jmath}
\nc{\imth}{\imath}
\nc{\osactson}[1]{\mathrel{\overset{#1}{\vrot{\circlearrowright}{270}}}}
\newcommand{\shrek}{{\mathchoice{{\scalebox{1}{$\vrot{!}{180}$}}}{{\scalebox{1}{$\vrot{!}{180}$}}}{{\scalebox{0.725}{$\vrot{!}{180}$}}}  % 0.7
{{\scalebox{0.525}{$\vrot{!}{180}$}}}}}
\newcommand{\antishriek}{\text{\raisebox{\depth}{\textexclamdown}}}
\renewcommand{\shrek}{\antishriek}

\renewcommand{\Mod}{\on{\sf{Mod}}}
\nc{\coactson}{\osactson{\te{co}}}
\nc{\emp}{\emptyset}
\nc{\nexistss}{\nexists\;}
\usepackage[dvipsnames]{xcolor}
\input xy
\usepackage[all]{xy}
\xyoption{line}
\xyoption{arrow}
\xyoption{color}
\SelectTips{cm}{}
\newcommand{\hackcenter}[1]{
 \xy (0,0)*{#1}; \endxy}

\title{BGG resolutions, Koszulity, and stratifications, part II:\\ the Jacobi-Trudi algebra} % IMPORTANT: Change the problemset number as needed.
\pgfplotsset{compat=1.14}
\begin{document}

\maketitle

\vspace*{-0.25in}
\centerline{Fan Zhou}
% Just so that your CA's can come knocking on your door when you don't hand in that problemset on time...
%\centerline{Thayer 504}
% \centerline{\href{mailto:fanzhou@college.harvard.edu}{{\tt{\textcolor{black}{fanzhou@college.harvard.edu}}}}}
\centerline{\href{mailto:fz2326@columbia.edu}{{\tt{\textcolor{black}{fz2326@columbia.edu}}}}}
% \centerline{\small{February 9, 2024}}
\vspace*{0.15in}

\begin{abstract}
    % We categorify the Jacobi-Trudi determinant formula for Schur functions as a shadow of a highest-weight phenomenon by considering certain quasi-hereditary quotients of certain cyclotomic KLR algebras, which we call ``Jacobi-Trudi algebras''. We prove this by showing such Jacobi-Trudi algebras are Morita-equivalent to certain Soergel calculi which are ``nil-Koszul'', i.e. which have ``lower half subalgebras'' which are Koszul. Hence this paper gives another example of a nil-Koszul algebra appearing naturally in categorification and gives another demonstration of the intimate connection between nil-Koszulity and BGG resolutions.
    We categorify the Jacobi-Trudi determinant formula for Schur functions as a shadow of a highest-weight phenomenon by considering certain quasi-hereditary quotients of certain cyclotomic KLR algebras, which we call ``Jacobi-Trudi algebras''. These algebras come equipped with a map from $\BC S_n$, and we show that the dominant simple modules for these algebras admit BGG resolutions which, when restricted to $\BC S_n$, become resolutions of Specht modules by permutation modules. We establish these BGG resolutions by showing that these Jacobi-Trudi algebras, as well as the Soergel calculi to which they are Morita equivalent, are ``nil-Koszul'', meaning that they have ``lower half subalgebras'' which are Koszul. We also show that Koszul duality with respect to this half subalgebra can be used to recover the differentials of the BGG resolutions. Hence this paper gives another example of a nil-Koszul algebra appearing naturally in categorification and gives another demonstration of the intricate connection between nil-Koszulity and BGG resolutions.
\end{abstract}

\textcolor{white}{$\Db$}\vspace{-1em}

\tableofcontents

% \newpage

\section{Introduction}\label{sect:intro}
The Jacobi-Trudi determinant formula is a classical fact from the study of symmetric polynomials, stating that the Schur functions $\tsl s_\lbd$ can be written as determinants of complete homogeneous functions $\tsl h$'s. More precisely, it states that
\[\tsl s_\lbd=\det\pr*{\tsl h_{\lbd_i-i+j}}_{i,j}=\det\begin{pmatrix}
    \tsl h_{\lbd_1} & \tsl h_{\lbd_1+1} & \cdots & \tsl h_{\lbd_1+\ell-1}\\
    \tsl h_{\lbd_2-1} &\tsl h_{\lbd_2}&\cdots &\vdots\\
    &&\ddots&\\%\tsl h_{\lbd_{\ell-1}+1}\\
    \tsl h_{\lbd_\ell-\ell+1}&\cdots &\tsl h_{\lbd_\ell-1}&\tsl h_{\lbd_\ell}
\end{pmatrix},\]
where $\ell=\ell(\lbd)$ is the number of rows in the partition $\lbd$. 

It is well-known that the Schur polynomials $\tsl s_\lbd$ correspond to simple Specht modules $\Sigma_\lbd$ under the isomorphism between the Grothendieck ring $\bigoplus_{n\ge0}K_0(\Rep S_n)$ and the ring of symmetric functions. Under this isomorphism, the complete homogeneous polynomials $\tsl h_\alpha$ correspond to permutation modules\footnote{The $E$ here stands for the French \textit{ensemble}, meaning `set'. This notation is coming from the theory of combinatorial species.}
\[E_\alpha\coloneqq \Ind_{S_\alpha}^{S_n}\on{triv}.\] 
A natural question is whether the Jacobi-Trudi formula can be categorified into a resolution of a simple Specht module by permutation modules. Moreover, there are several phenomena in the representation theory of $S_n$ which suggest highest-weight explanations, foremost of which is the decomposition
\[E_\lbd\cong \Sigma_\lbd\oplus \bigoplus_{\mu\domg\lbd} \Sigma_\mu^{\oplus m_{\mu,\lbd}}.\] 
This is to be compared with the central block $\CO_0$ of category $\CO$, where the Verma module\footnote{Our labeling convention is that $\Delta_1$ is the dominant (projective) Verma, while $\Delta_{w_0}$ is the antidominant (simple) Verma.} $\Delta_w$ has a Jordan-Holder filtration where $L_w$ appears once and every other $L_u$ appearing has $u>w$. In other words, in the Grothendieck group $K_0(\CO)$, 
\[[\Delta_w]=[L_w]+\sum_{u>w}m_{u,w}[L_u].\] 

For there to be a satisfying answer to these questions, we would like to exhibit $E_\lbd$ as some sort of a `Verma module', with $\Sigma_\lbd$ appearing as a quotient but not a submodule. However, the representation theory of $S_n$ is semisimple, so there is no nontrivial highest-weight theory there. So we cannot simply look at $S_n$; we should look at something more complicated, something capable of witnessing non-semisimple phenomena. A motivating example which, among other things, tells us that an answer should exist can be found at \cite{zhou2025categorifying}.

The goal is to categorify Jacobi-Trudi by considering modules over some (quasi-hereditary) algebra $A$ admitting a map $\BC S_n\lto A$, so that restricting a BGG resolution to $\BC S_n$ recovers Jacobi-Trudi as a character formula. In particular, we expect a simple module of $A$ to restrict to a simple Specht module over $S_n$ (or at least the dominant simple), and a Verma module of $A$ to restrict to a permutation module over $S_n$. Under the philosophy that only the dominant simple representation should have a BGG resolution, we might expect $A$ to depend on $\lbd$.

Fix $\lbd\vdash n$, which indicates which $s_\lbd$ we wish to categorify. Consider the set (with multiplicity) of contents\footnote{Recall the ``content'' of a box $(r,c)$ (i.e. a box in the $r$-th row and $c$-th column) is given by $\delta+c-r$.} of this partition, $\cont\lbd$, where we place the partition in such a way (using English notation) that the northwest box has content $\delta$. Here the parameter $\delta$ indicates that the underlying structure derives from $\gl_\delta$ (cf. work of Arakawa-Suzuki \cite{arakawa1998duality} and Orellana-Ram \cite{orellana2004affine}); for simplicity we could set $\delta=n$. Then let us consider the following block of the cyclotomic degenerate affine Hecke:
\[A\coloneqq A_\lbd\coloneqq \mfrac{\mfrac{\wh\CH_n}{(x_1-(\delta-\ell(\lbd)+1))\cdots (x_1-\delta)}}{\wan{\tsl e_i(x_1,\cdotsc,x_n)=\tsl e_i(\cont\lbd)}}.\]
Our convention for $\wh\CH_n$ is that right minus left is straights, i.e. $s_1x_2-x_1s_1=1$. 
\begin{EX*}
    In the extreme case of $\lbd$ being the column, its contents is precisely the set $\{1,\cdotsc,n\}$, so the cyclotomic quotient is unnecessary. On the other hand, if $\lbd$ is the row, the set of contents is still the same, but we wish for $A_\lbd$ to only have one simple (namely the row), so in this other extreme case the cyclotomic quotient is indeed necessary.
\end{EX*}

By the bridge of Brundan-Kleshchev \cite{brundan2009blocks}, this block of the cyclotomic degenerate affine Hecke is isomorphic to an appropriate block of a cyclotomic KLR algebra, defined in \cite{khovanov2009KL1}. The KLR cyclotomic parameter is the same as that of the Hecke (namely $y_1^1=0$ on any string labeled by $\{\delta-\ell(\lbd)+1,\cdotsc,\delta\}$), and the block is indicating that the (multi)set of colors for the KLR strands is exactly the (multi)set $\cont\lbd$. In other words, the cyclotomic parameter is
\[\omega=\varpi_\delta+\varpi_{\delta-1}+\cdots+\varpi_{\delta-\ell(\lbd)+1},\] 
and the block is determined by
\[\alpha=\sum_{i\in\cont\lbd} \alpha_i.\]
Let us denote this corresponding cyclotomic KLR block as %$\CR_\lbd$, or simply $\CR$. 
\[\CR_\lbd\coloneqq \CR^\omega_\alpha,\] 
or simply just $\CR=\CR_\lbd$. 
We will defer to the standard notation of $\psi_i$ and $y_i$ for the KLR crossings and dots, but will change the notation $e(\vec i)$ for idempotents to $e_\beta$, where $\beta$ is the sequence of colors of the strings.
% , where $c$ stands for the `content' vector.

\subsection{Main results}
We will construct in Proposition \ref{prop:JTalg} a cellular and quasi-hereditary algebra $\DCR_\lbd$, which we call the ``Jacobi-Trudi algebra'', by quotienting out the cyclotomic KLR $\CR_\lbd=\CR_\alpha^\omega$ above by a particular ideal. This Jacobi-Trudi algebra $\DCR_\lbd$ has weights labeled by some (lower-)ideal in the poset of the Bruhat order of $S_{\ell(\lbd)}$, such that the `dominant simple' $L_1$ has a BGG resolution, i.e. a resolution by standard modules, or `Verma modules', $\Delta_w$. 

Moreover, we can consider the algebra homomorphisms
\[\BC S_n\linj \wh\CH_n\lsurj \wh\CH^\omega_\alpha\simlto \CR^\omega_\alpha\lsurj \DCR_\lbd,\]
where the first map is the inclusion of the symmetric group into the degenerate affine Hecke algebra, the second map is the projection onto a cyclotomic block of the degenerate affine Hecke, the third map is the Brundan-Kleshchev isomorphism \cite{brundan2009blocks} between cyclotomic blocks of degenerate affine Hecke and KLR, and the last map is a quotient map we define. When the BGG resolution of $L_1$ above is restricted to $\BC S_n$ along these maps, we obtain a resolution of the Specht module $\Sigma_\lbd$ by various permutation modules $E_{w\circ\lbd}$. This categorifies the Jacobi-Trudi formula.

In summary, our main result is:
\vspace{10pt}\newline\noindent\ignorespaces\hspace*{-12.5pt}\DoBrackets%\addtocounter{theorem}{-1}
\schema{}{\vspace{-7.5pt}\begin{mytheorem}{A}\label{thm:mainresult}
    Let $\lbd\vdash n$ be a partition with $\ell(\lbd)$ rows. There is a quasi-hereditary nil-Koszul quotient $\DCR_\lbd$ of a type $A_\infty$ cyclotomic KLR with a cellular structure labeled by a poset $W_\lbd$ which is an interval in $S_{\ell(\lbd)}$. Moreover, $\DCR_\lbd$ has a BGG resolution of the dominant simple by cell modules:
    \[0\lto \Delta_{w_\lbd}\lto  \cdots\lto \bigoplus_{\substack{\ell(w)=k\\w\in W_\lbd}}\Delta_w\lto\cdots\lto \Delta_1\lto L_1\lto 0,\]
    where the differentials can be described explicitly as right multiplication with certain KLR diagrams.
    
    There is a homomorphism $\vphi_\lbd^\te{BK}\colon \BC S_n\lto \DCR_\lbd$ such that when the above BGG resolution is restricted along $\vphi_\lbd^\te{BK}$, we obtain a resolution of the Specht module by permutation modules,
    \[0\lto E_{w_\lbd\circ\lbd}\lto\cdots\lto \bigoplus_{\substack{\ell(w)=k\\w\in W_\lbd}} E_{w\circ\lbd} \lto\cdots\lto E_{\lbd}\lto \Sigma_\lbd\lto 0.\] 
    By taking the alternating sum of Frobenius characters, we recover the Jacobi-Trudi identity. 

    Also, any cell ideal $\CS$ of a cyclotomic Soergel calculus (of any type and any rank) corresponding to the ideal $W_\CS\subseteq W$ has a BGG resolution 
    \[\cdots\lto \bigoplus_{\substack{\ell(w)=k\\ w\in W_\CS}} \Delta_w\lto\cdots\lto \Delta_1\lto L_1\lto 0;\]
    the differentials can also be described explicitly as right multiplication with lollipops. 
\end{mytheorem}\vspace{-6pt} }\ignorespacesafterend\leavevmode\vspace{12.5pt}\newline
The BGG resolutions on the Jacobi-Trudi and Soergel sides are respectively Theorems \ref{thm:JTBGG} and \ref{thm:SoergelBGG}, and the maps are described respectively in Theorems \ref{thm:JTBGGmaps} and \ref{thm:SoergelBGGmaps}. 

To prove Theorem \ref{thm:mainresult}, we use our second main result:
\vspace{10pt}\newline\noindent\ignorespaces\hspace*{-12.5pt}\DoBrackets%\addtocounter{theorem}{-1}
\schema{}{\vspace{-7.5pt}\begin{mytheorem}{B}\label{thm:mainresulttwo}
    Let $\lbd\vdash n$ be a partition with $\ell(\lbd)$ rows. Then the Jacobi-Trudi algebra $\DCR_\lbd$ is nil-Koszul. 

    Moreover, any cell ideal $\CS$ of a cyclotomic Soergel calculus (of any type and any rank) is also nil-Koszul, with the subalgebra $\CS^-$ generated by rightside-up lollipops as the nilalgebra. In particular, the cell ideal $\CS_\lbd$ (corresponding to the ideal $W_\lbd\subseteq S_{\ell(\lbd)}$) is nil-Koszul.

    Koszul duality with respect to the nilalgebras of these nil-Koszul algebras followed by a `standardization functor', namely $\Delta\otimes\CK_{A^+}(L_1)$, recovers the BGG resolutions above.
\end{mytheorem}\vspace{-6pt} }\ignorespacesafterend\leavevmode\vspace{12.5pt}\newline
The nil-Koszulity statements for Jacobi-Trudi and Soergel are respectively Theorems \ref{thm:JTnilKoszul} and \ref{thm:SoergelnilKoszul}, and the Koszul duality statements are Theorems \ref{thm:koszulBGG} and \ref{thm:JTkoszulBGG}. For the definition of a nil-Koszul algebra, see Subsection \ref{subsect:nilkoszul}; essentially it means to have a `nilalgebra' which is Koszul. This nil-Koszulity result gives us good control on Ext groups between standard modules and simple modules, which allows us to produce BGG resolutions.

\subsection{History}
The question of categorifying Jacobi-Trudi has been studied before. In works by Zelevinsky \cite{zelevinskii1987resolvents}, Akin \cite{akin1988complexes,akin1992complexes}, Arakawa-Suzuki \cite{arakawa1998duality}, and Orellana-Ram \cite{orellana2004affine}, an exact functor (the so-called Arakawa-Suzuki functor) was considered from category $\CO$ to the category of modules over the symmetric group (or the associated Hecke algebra). This functor sends Vermas to permutation modules and the dominant simple to a Specht module, so that the BGG resolution of category $\CO$ directly produces a `BGG resolution' in e.g. $\Mod\BC S_n$ which categorifies Jacobi-Trudi. 

Even though a BGG resolution, which in some sense is a phenomenon born out of non-semisimplicity, is produced in $\Mod\BC S_n$, this category is famously semisimple. Our approach seeks to witness the non-semisimplicity giving rise to the BGG resolution natively in some appropriate $\Mod A$, for example to say in what sense the permutation modules should be Verma modules. 

Though there are some superficial differences in the construction (for instance we have favored the traditional Hu-Mathas basis here), the Jacobi-Trudi algebras $\DCR_\lbd$ constructed in this paper are a special case of a more general family of algebras studied by Bowman et. al. in \cite{bowman2023klrvssoergel} and \cite{bowman2022lightleaves}. Resolutions over such algebras were studied in \cite{bowman2018characteristicfree} and \cite{bowman2024unitary} (which appealed to the resolution over Soergel calculus from \cite{bowman2022weylkac}), but the approaches in these papers were combinatorial -- they constructed such resolutions by hand and checked $\Ker=\Img$ directly. The approach in this paper uses instead the `reconstruction-from-stratification' machinery from \cite{ayala2022stratified} and \cite{dhillon2019bernsteingelfandgelfand}, powered by Koszul theory. This philosophy, that a sheaf on a stratified space can be reconstructed from what it is on the strata, is well-known to geometers. Moreover, in Section \ref{subsect:koszulperspective} we show how Koszul duality can be used to directly produce BGG resolutions, and we show that the BGG differentials are given by the comultiplication structure of the appropriate Koszul dual coalgebra. The upshot is that we are able to avoid $\ker=\img$ calculations entirely.

\subsection{Methodology}
That the construction for $\DCR_\lbd$ is well-defined and that $\DCR_\lbd$ is quasi-hereditary requires arguments specific to KLR, for instance the cellular basis.

To construct the BGG resolution for $\DCR_\lbd$, we utilize the reconstruction machinery of \cite{ayala2022stratified} and \cite{dhillon2019bernsteingelfandgelfand}. This was the strategy in \cite{zhou2024bgg} as well. In order to power this machine, we need to compute certain Ext groups, which we do by using Koszul theory, similarly to \cite{zhou2024bgg}. In the present case, we compute these Ext groups by proving that a certain cyclotomic Soergel calculus $\CS_\lbd$, which is Morita-equivalent to $\DCR_\lbd$, is ``nil-Koszul'', which is roughly saying that ``half'' of $\CS_\lbd$ is Koszul. In fact moreover $\DCR_\lbd$ is nil-Koszul as well, and Koszul duality with respect to the positive nilalgebra will give the desired BGG resolution.

Such `nil-Koszul' algebras appear to be quite ubiquitous in categorification. For instance, in \cite{zhou2024bgg} we prove that the nil-Brauer algebra of \cite{brundan2023nil-brauer1},\cite{brundan2023nil-brauer2} is nil-Koszul. The algebras categorifying Hermite (forthcoming work with Mikhail Khovanov and Radmila Sazdanovi\'c) and Chebyshev (\cite{khovanov4171760}) polynomials are also nil-Koszul. The present paper adds one more algebra to this list. One might wonder if there is some deeper explanation for this ubiquity, but we have no answer. We plan to investigate other algebras in future papers.

\subsection{Outline}
% After recalling some known facts about KLR, we will construct the ``Jacobi-Trudi algebra'' as a quasi-hereditary quotient of a certain cyclotomic KLR algebra. This algebra comes with a natural stratification, and we will use a Morita equivalence between it and a Soergel calculus in order to compute some Ext groups between standard modules and simple modules. In other words, we will show that the Jacobi-Trudi algebra is Morita-equivalent to a diagrammatic algebra which is nil-Koszul. 

In Section \ref{sect:klr} we recall various things about the KLR. In Section \ref{sect:JT} we construct the ``Jacobi-Trudi algbra'' by using a different ordering than the usual dominance order. In Section \ref{sect:reconnilkoszul} we recall the categorical machinery of reconstruction from stratification, as well as the notion of ``nil-Koszulity''. In Section \ref{sect:JTstratandrecon} we apply this to the Jacobi-Trudi algebra. In Section \ref{sect:soergel} we make a detour through Soergel calculus in order to complete the homological calculations necessary to carry out reconstruction; here we prove cyclotomic Soergel calculus is nil-Koszul. In Section \ref{sect:BGG} we finish the reconstruction argument and thereby prove BGG, and we also give a Koszul theoretic perspective on BGG. In Section \ref{sect:JTnilkoszul} we show that the Jacobi-Trudi algebra itself is directly nil-Koszul. In Section \ref{sect:monoidal} we construct a monoidal product on the Jacobi-Trudi algebras. 

Some sections are mostly recollections. The novel content of this paper is mostly in Sections \ref{sect:JT}, \ref{sect:JTstratandrecon}, \ref{sect:soergel}, \ref{sect:BGG}, and \ref{sect:monoidal}.  

\subsection{Conventions}
Throughout this paper, an ``upper-ideal of a poset'' is defined as a subset $I$ such that
\[x\in I,\ y\ge x\implies y\in I,\]
while a ``lower-ideal'' is defined as
\[x\in I,\ y\le x\implies y\in I.\]

Our convention on cellular algebras is such that the cell labeled by the maximal (rather than minimal) element of the poset is an ideal. 

When we speak of category $\CO$, unless otherwise specified, implicitly we mean only the principal central block $\CO_0$.

Our convention on graded vector spaces is that $q^n$ shifts the grading \textit{up} by $n$, so that 
\[(q^nV)_d=V_{d-n},\] 
and our definition of the graded dimension is
\[\dim_q V=\sum_n q^n \dim V_n.\]

\subsection{Acknowledgements}

We thank Chris Bowman, Jon Brundan, Charles Fu, Dennis Gaitsgory, Matthew Hase-Liu, Aaron Lauda, Cailan Li, Yixuan Li, Alvaro Martinez, Amal Mattoo, Rob Muth, and Arnav Tripathy for helpful discussions for this project. We are also deeply grateful to our advisors Mikhail Khovanov and Joshua Sussan for their help and insight. The author was partially supported by NSF grant DMS-1807425 and the Simons Collaboration on New Structures in Low-dimensional Topology.
% A more thorough acknowledgment will be available in the author's thesis.

% We thank Chris Bowman, Jon Brundan, Charles Fu, Dennis Gaitsgory, Matthew Hase-Liu, Aaron Lauda, Cailan Li, Yixuan Li, Alvaro Martinez, Amal Mattoo, Rob Muth, and Arnav Tripathy for helpful discussions for this project, as well as our advisors Mikhail Khovanov and Joshua Sussan for their help and insight.
% A more thorough acknowledgment will be available in the author's thesis.

%below is an alternate more full version
% We thank: Chris Bowman, Jon Brundan, Maud de Visscher, Ben Elias, Charles Fu, Dennis Gaitsgory, Matthew Hase-Liu, Aaron Lauda, Cailan Li, Yixuan Li, Ivan Losev, Alvaro Martinez, Amal Mattoo, Rob Muth, You Qi, Arun Ram, Alistair Savage, Andrew Snowden, Linliang Song, Arnav Tripathy, Daniel Tubbenhauer, Monica Vazirani, and Ben Webster. 

\section{Recollections on KLR}\label{sect:klr}
% I talked to many people at the conferences in California, and most of them told me I should look at KLR instead of the Hecke. They were so right. 

% KLR is the morally correct lens through which to view the minutiae of this problem. Indeed, perhaps this is impossible without KLR. To be precise:

In this section let us recall some classical facts about KLR. Our exposition will be somewhat rushed; for more details, see e.g. \cite{brundan2009blocks},\cite{mathas2015cyclotomic},\cite{kleshchev2010representation},\cite{khovanov2009KL1}.

\subsection{Defining KLR}
We will adhere to the conventions of \cite{brundan2009blocks}. The KLR algebra $\CR$ is defined as follows: the monoidal generators are 
% \[\hackcenter{\begin{tikzpicture}[scale=0.375]
%     \draw (0,0)--(0,2);
%         % \node at (3,2) { $\ddots$};
% \end{tikzpicture}}\ ,\ \hackcenter{\begin{tikzpicture}[scale=0.375]
%     \draw (0,0)--(0,2);
%     \fill (0,1) circle (5pt);
%         % \node at (3,2) { $\ddots$};
% \end{tikzpicture}}\ ,\ \hackcenter{\begin{tikzpicture}[scale=0.375]
%     \draw (0,0)--(2,2);
%     \draw (2,0)--(0,2);
%         % \node at (3,2) { $\ddots$};
% \end{tikzpicture}},\]
% where the strings carry arbitrary colors. 
\begin{center}
\begin{tabular}{ c c c c c c }
 & $\begin{diagram}
    \draw (0,0)--(0,2);
    \node at (0,-0.5) {\tiny $i$};
\end{diagram}$ & $\begin{diagram}
    \draw (0,0)--(0,2);
    \fill (0,1) circle (5pt);
    \node at (0,-0.5) {\tiny $i$};
\end{diagram}$ & 
$\begin{diagram}
    \draw (0,0)--(2,2);
    \draw (2,0)--(0,2);
    \node at (0,-0.5) {\tiny $i$};
    \node at (2,-0.5) {\tiny $i$};
\end{diagram}$ & 
$\begin{diagram}
    \draw (0,0)--(2,2);
    \draw (2,0)--(0,2);
    \node at (0,-0.5) {\tiny $i$};
    \node at (2,-0.5) {\tiny $i\pm 1$};
\end{diagram}$ & 
$\begin{diagram}
    \draw (0,0)--(2,2);
    \draw (2,0)--(0,2);
    \node at (0,-0.5) {\tiny $i$};
    \node at (2,-0.5) {\tiny $j$};
\end{diagram}$
\\[1em]
    degree & 0 & $2$ & $-2$ & \hspace*{-6.5pt}$1$ & $0$ 
\end{tabular}
\end{center}
    where $|j-i|>1$.
The important\footnote{i.e. parts that people differ on} relations are
\begin{align*}
    {}\raisebox{-0.5em}{$\hackcenter{\begin{tikzpicture}[scale=0.375]
    \draw (0,0)--(2,2);
    \draw (2,0)--(0,2);
    \node at (0,-0.5) {\tiny $i$};
    \node at (2,-0.5) {\tiny $i$};
    \fill (1.5,1.5) circle (5pt);
\end{tikzpicture}}$}
-
\raisebox{-0.5em}{$\hackcenter{\begin{tikzpicture}[scale=0.375]
    \draw (0,0)--(2,2);
    \draw (2,0)--(0,2);
    \node at (0,-0.5) {\tiny $i$};
    \node at (2,-0.5) {\tiny $i$};
    \fill (0.5,0.5) circle (5pt);
\end{tikzpicture}}$}
&=
\raisebox{-0.5em}{$\hackcenter{\begin{tikzpicture}[scale=0.375]
    \draw (0,0)--(2,2);
    \draw (2,0)--(0,2);
    \node at (0,-0.5) {\tiny $i$};
    \node at (2,-0.5) {\tiny $i$};
    \fill (1.5,0.5) circle (5pt);
\end{tikzpicture}}$}
-
\raisebox{-0.5em}{$\hackcenter{\begin{tikzpicture}[scale=0.375]
    \draw (0,0)--(2,2);
    \draw (2,0)--(0,2);
    \node at (0,-0.5) {\tiny $i$};
    \node at (2,-0.5) {\tiny $i$};
    \fill (0.5,1.5) circle (5pt);
\end{tikzpicture}}$}
=
\raisebox{-0.5em}{$\hackcenter{\begin{tikzpicture}[scale=0.375]
    \draw (0,0)--(0,2);
    \draw (2,0)--(2,2);
    \node at (0,-0.5) {\tiny $i$};
    \node at (2,-0.5) {\tiny $i$};
\end{tikzpicture}}$},\\ 
{}\raisebox{-0.5em}{$\hackcenter{\begin{tikzpicture}[scale=0.375]
    % \draw (0,0)--(2,2)--(0,4);
    % \draw (2,0)--(0,2)--(2,4);
    % \draw (0,0) to[out=0,in=-0] (0,4);
    % \draw (2,0) to[out=+180,in=-180] (2,4);
    \draw (0,0)..controls(2.5,2)..(0,4);
    \draw (2,0)..controls(-0.5,2)..(2,4);
    \node at (0,-0.5) {\tiny $i$};
    \node at (2,-0.5) {\tiny $i$};
    % \fill (1.5,1.5) circle (5pt);
\end{tikzpicture}}$}
&=0,\\ 
{}\raisebox{-0.5em}{$\hackcenter{\begin{tikzpicture}[scale=0.375]
    % \draw (0,0)--(2,2)--(0,4);
    % \draw (2,0)--(0,2)--(2,4);
    \draw (0,0)..controls(2.5,2)..(0,4);
    \draw (2,0)..controls(-0.5,2)..(2,4);
    \node at (0,-0.5) {\tiny $i$};
    \node at (2,-0.5) {\tiny $i\pm 1$};
    % \fill (1.5,1.5) circle (5pt);
\end{tikzpicture}}$}
&=\pm 
\raisebox{-0.5em}{$\hackcenter{\begin{tikzpicture}[scale=0.375]
    \draw (0,0)--(0,4);
    \draw (2,0)--(2,4);
    \node at (0,-0.5) {\tiny $i$};
    \node at (2,-0.5) {\tiny $i\pm 1$};
    \fill (2,2) circle (5pt);
\end{tikzpicture}}$}
\mp 
\raisebox{-0.5em}{$\hackcenter{\begin{tikzpicture}[scale=0.375]
    \draw (0,0)--(0,4);
    \draw (2,0)--(2,4);
    \node at (0,-0.5) {\tiny $i$};
    \node at (2,-0.5) {\tiny $i\pm 1$};
    \fill (0,2) circle (5pt);
\end{tikzpicture}}$},\\ 
{}\raisebox{-0.5em}{$\hackcenter{\begin{tikzpicture}[scale=0.375]
    \draw (0,0)--(4,4);
    \draw (4,0)--(0,4);
    \draw (2,0) ..controls(-0.5,2).. (2,4);
    \node at (0,-0.5) {\tiny $i$};
    \node at (2,-0.5) {\tiny $i\pm 1$};
    \node at (4,-0.5) {\tiny $i$};
    % \fill (1.5,1.5) circle (5pt);
\end{tikzpicture}}$}
-
\raisebox{-0.5em}{$\hackcenter{\begin{tikzpicture}[scale=0.375]
    \draw (0,0)--(4,4);
    \draw (4,0)--(0,4);
    \draw (2,0) ..controls(4.5,2).. (2,4);
    \node at (0,-0.5) {\tiny $i$};
    \node at (2,-0.5) {\tiny $i\pm 1$};
    \node at (4,-0.5) {\tiny $i$};
    % \fill (1.5,1.5) circle (5pt);
\end{tikzpicture}}$}
&=\pm
\raisebox{-0.5em}{$\hackcenter{\begin{tikzpicture}[scale=0.375]
    \draw (0,0)--(0,4);
    \draw (4,0)--(4,4);
    \draw (2,0)--(2,4);
    \node at (0,-0.5) {\tiny $i$};
    \node at (2,-0.5) {\tiny $i\pm 1$};
    \node at (4,-0.5) {\tiny $i$};
    % \fill (1.5,1.5) circle (5pt);
\end{tikzpicture}}$};
\end{align*}
the `unimportant' relations are 
\begin{align*}
    \raisebox{-0.5em}{$\begin{diagram}
        \draw (0,0)..controls(2.5,2)..(0,4);
    \draw (2,0)..controls(-0.5,2)..(2,4);
    \node at (0,-0.5) {\tiny $i$};
    \node at (2,-0.5) {\tiny $j$};
    \end{diagram}$}
    &=
    \raisebox{-0.5em}{$\begin{diagram}
        \draw (0,0)--(0,4);
        \draw (2,0)--(2,4);
         \node at (0,-0.5) {\tiny $i$};
    \node at (2,-0.5) {\tiny $j$};
    \end{diagram}$}
    \quad\te{ for }|i-j|>1,\\ 
    \raisebox{-0.5em}{$\begin{diagram}
        \draw (0,0)--(2,2);
    \draw (2,0)--(0,2);
    \node at (0,-0.5) {\tiny $i$};
    \node at (2,-0.5) {\tiny $j$};
    \fill (1.5,1.5) circle (5pt);
    \end{diagram}$}
    &=
    \raisebox{-0.5em}{$\begin{diagram}
        \draw (0,0)--(2,2);
    \draw (2,0)--(0,2);
    \node at (0,-0.5) {\tiny $i$};
    \node at (2,-0.5) {\tiny $j$};
    \fill (0.5,0.5) circle (5pt);
    \end{diagram}$}
    \quad\te{ for }i\neq j,\\ 
    \raisebox{-0.5em}{$\begin{diagram}
        \draw (0,0)--(2,2);
    \draw (2,0)--(0,2);
    \node at (0,-0.5) {\tiny $i$};
    \node at (2,-0.5) {\tiny $j$};
    \fill (0.5,1.5) circle (5pt);
    \end{diagram}$}
    &=
    \raisebox{-0.5em}{$\begin{diagram}
        \draw (0,0)--(2,2);
    \draw (2,0)--(0,2);
    \node at (0,-0.5) {\tiny $i$};
    \node at (2,-0.5) {\tiny $j$};
    \fill (1.5,0.5) circle (5pt);
    \end{diagram}$}
    \quad\te{ for }i\neq j,\\ 
    \raisebox{-0.5em}{$\begin{diagram}
        \draw (0,0)--(4,4);
    \draw (4,0)--(0,4);
    \draw (2,0) ..controls(-0.5,2).. (2,4);
    \node at (0,-0.5) {\tiny $i$};
    \node at (2,-0.5) {\tiny $j$};
    \node at (4,-0.5) {\tiny $k$};
    \end{diagram}$}
&=
    \raisebox{-0.5em}{$\begin{diagram}
        \draw (0,0)--(4,4);
    \draw (4,0)--(0,4);
    \draw (2,0) ..controls(4.5,2).. (2,4);
    \node at (0,-0.5) {\tiny $i$};
    \node at (2,-0.5) {\tiny $j$};
    \node at (4,-0.5) {\tiny $k$};
    \end{diagram}$}
    \quad\te{ for }(j,k)\neq (i\pm 1,i).
\end{align*}

Given a nonnegative sum of simple roots $\alpha=\sum_i c_i\alpha_i$, $\CR_\alpha$ is the subalgebra such that there are $c_i$ strings colored by $i$ for each $i$.  

For an ordered sequence of colors $\beta=(\beta_1,\cdotsc,\beta_n)$, we let $e_\beta$ denote the diagram of propagating strings with colors indicated by $\beta$:
\[e_\beta=\raisebox{-0.5em}{$\begin{diagram}
    \draw (0,0)--(0,2);
    % \fill (0,1) circle (5pt);
    % \node at (-1.2,1.2) {\tiny $\alpha_i^*(\omega)$};
    \node at (0,-0.5) {\tiny $\beta_1$};
    \draw (1,0)--(1,2);
    \node at (1,-0.5) {\tiny $\beta_2$};
    \node at (2.1,1) {$\cdots$};
    \draw (3,0)--(3,2);
    \node at (3,-0.5) {\tiny $\beta_n$};
\end{diagram}$}.\]

\subsection{Cyclotomic KLR}
Given a cyclotomic parameter $\omega\in\Lbd^+$, the cyclotomic KLR algebra $\CR^\omega_\alpha$ is a quotient of $\CR_\alpha$ defined via
\[\CR^\omega_\alpha\coloneqq\mfrac{\CR_\alpha}{\wan{y_1^{\alpha_{\beta_1}^*(\omega)}e_\beta=0}}.\]
Diagrammatically this is saying that the cyclotomic KLR has the additional relation
\[\raisebox{-0.5em}{$\begin{diagram}
    \draw (0,0)--(0,2);
    \fill (0,1) circle (5pt);
    \node at (-1.2,1.3) {\tiny $\alpha_i^*(\omega)$};
    \node at (0,-0.5) {\tiny $i$};
    \draw (1,0)--(1,2);
    \node at (2,1) {$\cdots$};
    \draw (3,0)--(3,2);
\end{diagram}$}=0.\]

Given a partition $\lbd$, we will consider the cyclotomic KLR (for the $A_\infty$ quiver)
\[\CR_\lbd\coloneqq\CR_\alpha^\omega\] 
for
\[\alpha=\sum_{i\in\cont\lbd}\alpha_i\]
and 
\[\omega=\varpi_\delta+\varpi_{\delta-1}+\cdots+\varpi_{\delta-\ell(\lbd)+1},\]
where $\varpi_i$ is the $i$-th fundamental weight and $\cont\lbd$ is the (multi-)set of contents of the boxes of $\lbd$, where the box $(r,c)$ has content
\[\cont(r,c)=\delta+c-r.\]
For simplicity one could set $\delta=|\lbd|$. We would like to remark that the cyclotomic quotient we consider is quite simple -- in particular, $(\omega,\alpha_i)\le 1$, i.e., $\omega$ is multiplicity-free. Another point to mention is that we have chosen to write $\omega$ in terms of the fundamental weights in such a way that the subscripts labeling the fundamental weights is decreasing. For simplicity, we will often use the notation
\[\kappa_i=\delta-i+1;\] 
% \warn{decide on $\omega_k$ versus $\kappa_k$} 
in general $\kappa_i$ is supposed to be the label of the $i$-th fundamental weight in the preferred expression for $\omega$. In the literature this sequence $(\kappa_1,\cdotsc,\kappa_{\ell(\lbd)})$ is often called the ``multi-charge''. In other words our cyclotomic quotient $\CR_\lbd$ will satisfy
% \[\raisebox{-0.5em}{$\begin{diagram}
%     \draw (0,0)--(0,2);
%     \fill (0,1) circle (5pt);
%     % \node at (-1.2,1.2) {\tiny $\alpha_i^*(\omega)$};
%     \node at (0,-0.5) {\tiny $i$};
%     \draw (1,0)--(1,2);
%     \node at (2.1,1) {$\cdots$};
%     \draw (3,0)--(3,2);
% \end{diagram}$}=0\qquad \te{for }i\in\{\kappa_1,\cdotsc,\kappa_{\ell(\lbd)}\},\]
% as well as
% \[\raisebox{-0.5em}{$\begin{diagram}
%     \draw (0,0)--(0,2);
%     % \fill (0,1) circle (5pt);
%     % \node at (-1.2,1.2) {\tiny $\alpha_i^*(\omega)$};
%     \node at (0,-0.5) {\tiny $i$};
%     \draw (1,0)--(1,2);
%     \node at (2.1,1) {$\cdots$};
%     \draw (3,0)--(3,2);
% \end{diagram}$}=0\qquad\te{for }i\not\in\{\kappa_1,\cdotsc,\kappa_{\ell(\lbd)}\}.\]

\begin{align*}
    \raisebox{-0.5em}{$\begin{diagram}
    \draw (0,0)--(0,2);
    \fill (0,1) circle (5pt);
    % \node at (-1.2,1.2) {\tiny $\alpha_i^*(\omega)$};
    \node at (0,-0.5) {\tiny $i$};
    \draw (1,0)--(1,2);
    \node at (2.1,1) {$\cdots$};
    \draw (3,0)--(3,2);
\end{diagram}$}&=0\qquad \te{for }i\in\{\kappa_1,\cdotsc,\kappa_{\ell(\lbd)}\},\\
\raisebox{-0.5em}{$\begin{diagram}
    \draw (0,0)--(0,2);
    % \fill (0,1) circle (5pt);
    % \node at (-1.2,1.2) {\tiny $\alpha_i^*(\omega)$};
    \node at (0,-0.5) {\tiny $i$};
    \draw (1,0)--(1,2);
    \node at (2.1,1) {$\cdots$};
    \draw (3,0)--(3,2);
\end{diagram}$}&=0\qquad\te{for }i\not\in\{\kappa_1,\cdotsc,\kappa_{\ell(\lbd)}\}.
\end{align*}

%below is a different version of the above block of code, keep exactly one
% \begin{align*}
%     \raisebox{-0.5em}{$\begin{diagram}
%     \draw (0,0)--(0,2);
%     \fill (0,1) circle (5pt);
%     % \node at (-1.2,1.2) {\tiny $\alpha_i^*(\omega)$};
%     \node at (0,-0.5) {\tiny $i$};
%     \draw (1,0)--(1,2);
%     \node at (2.1,1) {$\cdots$};
%     \draw (3,0)--(3,2);
% \end{diagram}$}&=0\qquad \te{for }i\in[\delta-\ell(\lbd)+1,\delta],\\
% \raisebox{-0.5em}{$\begin{diagram}
%     \draw (0,0)--(0,2);
%     % \fill (0,1) circle (5pt);
%     % \node at (-1.2,1.2) {\tiny $\alpha_i^*(\omega)$};
%     \node at (0,-0.5) {\tiny $i$};
%     \draw (1,0)--(1,2);
%     \node at (2.1,1) {$\cdots$};
%     \draw (3,0)--(3,2);
% \end{diagram}$}&=0\qquad\te{for }i\not\in[\delta-\ell(\lbd)+1,\delta].
% \end{align*}

\subsection{The map from $\BC S_n$}
The map from $\BC S_n$ passes through the degenerate affine Hecke algebra $\wh \CH_n$, which we define presently. This is the algebra generated monoidally by
\[\begin{diagram}
    \draw(0,0)--(0,2);
    \fill (0,1) circle (5pt);
\end{diagram}\ ,\qquad \begin{diagram}
    \draw(0,0)--(2,2);
    \draw(2,0)--(0,2);
\end{diagram}\ ,\]
subject to relations 
\begin{align*}
    \begin{diagram}
    \draw (0,0)--(2,2);
    \draw (2,0)--(0,2);
    \fill (1.5,1.5) circle (5pt);
\end{diagram}
-
\begin{diagram}
    \draw (0,0)--(2,2);
    \draw (2,0)--(0,2);
    \fill (0.5,0.5) circle (5pt);
\end{diagram}
\ &=\ 
\begin{diagram}
    \draw (0,0)--(2,2);
    \draw (2,0)--(0,2);
    \fill (1.5,0.5) circle (5pt);
\end{diagram}
-
\begin{diagram}
    \draw (0,0)--(2,2);
    \draw (2,0)--(0,2);
    \fill (0.5,1.5) circle (5pt);
\end{diagram}
\ =\ 
\begin{diagram}
    \draw (0,0)--(0,2);
    \draw (2,0)--(2,2);
\end{diagram}\ ,\\ 
\begin{diagram}
    % \draw (0,0)--(2,2)--(0,4);
    % \draw (2,0)--(0,2)--(2,4);
    % \draw (0,0) to[out=0,in=-0] (0,4);
    % \draw (2,0) to[out=+180,in=-180] (2,4);
    \draw (0,0)..controls(2.5,2)..(0,4);
    \draw (2,0)..controls(-0.5,2)..(2,4);
    % \fill (1.5,1.5) circle (5pt);
\end{diagram}
&=\ \begin{diagram}
    \draw(0,0)--(0,4);
    \draw(2,0)--(2,4);
\end{diagram}\ ,\\ 
\begin{diagram}
    \draw (0,0)--(4,4);
    \draw (4,0)--(0,4);
    \draw (2,0) ..controls(-0.5,2).. (2,4);
    % \fill (1.5,1.5) circle (5pt);
\end{diagram}
&=
\begin{diagram}
    \draw (0,0)--(4,4);
    \draw (4,0)--(0,4);
    \draw (2,0) ..controls(4.5,2).. (2,4);
    % \fill (1.5,1.5) circle (5pt);
\end{diagram}.
\end{align*}
Algebraically, the crossing of the $i$-th and $(i+1)$-th strings is labeled $s_i$, and the dot on the $i$-th string is labeled $x_i$. We choose this symbol for the crossing because there is an obvious copy of $S_n$ sitting inside of $\wh \CH_n$, given by the crossings. 

One may further quotient out this $\wh \CH_n$ to obtain the ``cyclotomic degenerate affine Hecke algebra'' by
\[\wh\CH_n^\omega\coloneqq \mfrac{\wh\CH_n}{\wan*{\prod_i (x_1-i)^{\alpha_i^*(\omega)}}}.\]
% \[(x_1-(\delta-\ell(\lbd)+1))\cdots (x_1-\delta)=0.\] 
Consider a block of this, labeled by $\alpha$ a sum of $n$ roots, defined by further requiring
\[\tsl e_i(x_1,\cdotsc,x_n)=\tsl e_i(\{j:\alpha_j\te{ appears in }\alpha\}),\]
where the set $\{j:\alpha_j\te{ appears in }\alpha\}$ is counted with multiplicity, i.e. the number of times $j$ appears is $\alpha^*(\varpi_j)$. This algebra is called $\wh\CH_\alpha^\omega$.

An important result of Brundan-Kleshchev \cite{brundan2009blocks} is an isomorphism
\begin{align*}
    \vphi_\te{BK}\colon \wh\CH^\omega_\alpha&\simlto \CR^\omega_\alpha\\
s_i&\lmto \sum_{\beta\in I^\alpha} (\psi_iq_i-p_i) e_\beta,
\end{align*}
% \[\wh\CH^\omega_\alpha\simlto \CR^\omega_\alpha.\] 
% This map is constructed by letting 
% \[s_i\lmto \sum_{c\in I^\alpha} (\psi_iq_i-p_i) e_c,\] 
where
\begin{align*}
    p_ie_\beta&=\begin{cases}
    1 & \beta_{i+1}=\beta_i\\ 
    \frac{1}{(\beta_i-\beta_{i+1})-(y_{i+1}-y_i)} &\te{else}
\end{cases},\\
    q_i e_\beta&=\begin{cases}
        1+y_{i+1}-y_i & \beta_{i+1}=\beta_i\\
        \frac{1-p_i^2}{y_{i+1}-y_i} 
 & \beta_{i+1}=\beta_i+1\\
        1 & \beta_{i+1}=\beta_i-1\\
        1-p_i & \te{else}
    \end{cases}.
\end{align*}
% \begin{align*}
%     p_ie_\gamma&=\begin{cases}
%     1 & \gamma_{i+1}=\gamma_i\\ 
%     \frac{1}{(\gamma_i-\gamma_{i+1})-(y_{i+1}-y_i)} &\te{else}
% \end{cases},\\
%     q_i e_\gamma&=\begin{cases}
%         1+y_{i+1}-y_i & \gamma_{i+1}=\gamma_i\\
%         \frac{1-p_i^2}{y_{i+1}-y_i} 
%  & \gamma_{i+1}=\gamma_i+1\\
%         1 & \gamma_{i+1}=\gamma_i-1\\
%         1-p_i & \te{else}
%     \end{cases}.
% \end{align*}
We do not describe where the dots go because it is not necessary for our purposes, as we only need a map from $S_n$.

% \[\mfrac{\mfrac{\wh\CH_n}{(x_1-(\delta-\ell(\lbd)+1))\cdots (x_1-\delta)}}{\wan{\tsl e_i(x_1,\cdotsc,x_n)=\tsl e_i(\cont\lbd)}}\]

% \warn{explain brundan-kleshchev}

The map we will construct from $\BC S_n$ to the Jacobi-Trudi algebra $\DCR_\lbd$ is as follows:
\[\vphi_\lbd^\te{BK}\colon \BC S_n\linj \wh\CH_n\lsurj \wh\CH^\omega_\alpha\simlto \CR^\omega_\alpha\lsurj \DCR_\lbd,\]
where only the last quotient map remains to be described. This will be done in Section \ref{sect:JT}.

Here is an example to illustrate the map $\BC S_n\lto \CR^\omega_\alpha$.
\begin{EX*}
    Let $\lbd=\ydia{1,1}$, $n=\delta=2$, so that $\alpha=\alpha_1+\alpha_2$, $\omega=\varpi_2+\varpi_1$, and $\kappa=(2,1)$. We track the image of $s=s_1$. There are two color vectors in $I^\alpha$, namely $\beta=(1,2)$ and $\beta=(2,1)$. Note that $y_1e_{(1,2)}=y_1e_{(2,1)}=0$ in $\CR^\omega_\alpha$, and also $y_2^2e_{(1,2)}=y_2^2e_{(2,1)}=0$, so that
    \begin{align*}
        p_1e_{(2,1)}&=\frac{1}{-1-y_2+y_1}e_{(2,1)}=-\frac{1}{1+(y_2-y_1)}e_{(2,1)}=-\sum_i (-1)^i(y_2-y_1)^ie_{(2,1)}=(-1+y_2)e_{(2,1)},\\
        p_1e_{(1,2)}&=\frac{1}{-1-y_2+y_1}e_{(1,2)}=-\frac{1}{1+(y_2-y_1)}e_{(1,2)}=-\sum_i (-1)^i(y_2-y_1)^ie_{(1,2)}=(-1+y_2)e_{(1,2)},
    \end{align*}
    while %\warn{this is wrong}
    % \begin{align*}
    %     q_1e_{(2,1)}&=1,\\
    %     q_1e_{(1,2)}&=\frac{1-(-1+y_2)^2}{y_{2}-y_1}e_{(1,2)}\\
    %     &=\frac{-y_2^2+2y_2}{y_{2}}e_{(1,2)}\\
    %     &=(2-y_2)e_{(1,2)}.
    % \end{align*}
    \begin{align*}
        q_1e_{(2,1)}&=1,\\
        q_1e_{(1,2)}&=\frac{1-p_1^2}{y_{2}-y_1}e_{(1,2)}=\frac{1-\sum_{i=1}^\infty i(-1)^{i-1}(y_2-y_1)^{i-1}}{y_2-y_1}=\sum_{i=2}^\infty i(-1)^i (y_2-y_1)^{i-2}e_{(1,2)}=(2-3y_2)e_{(1,2)}.
    \end{align*}
    Hence we see that $s_1$ is mapped to
    \[\psi_1 e_{(2,1)}-(-1+y_2)e_{(2,1)}+\psi_1(2-3y_2)e_{(1,2)}-(-1+y_2)e_{(1.2)},\] 
    which diagrammatically can be written as (letting 1 be black and 2 be red)
    \[s_1\lmto 
    \begin{diagram}
        \draw[red](0,0)--(2,2);
        \draw (2,0)--(0,2);
    \end{diagram}
    \ -\ 
    \begin{diagram}
        \draw[red](0,0)--(0,2);
        \draw(2,0)--(2,2);
    \end{diagram}
    \ -\ 
    \begin{diagram}
        \draw[red](0,0)--(0,2);
        \draw(2,0)--(2,2);
        \fill (2,1) circle (5pt);
    \end{diagram}
    \ +
    2\begin{diagram}
        \draw(0,0)--(2,2);
        \draw[red](2,0)--(0,2);
    \end{diagram}
    \ +\ 
    \begin{diagram}
        \draw(0,0)--(0,2);
        \draw[red](2,0)--(2,2);
    \end{diagram}
    \ -\ 
    \begin{diagram}
        \draw(0,0)--(0,2);
        \draw[red](2,0)--(2,2);
        \fill[red] (2,1) circle (5pt);
    \end{diagram}\ ,
    \]
    where we have noted that $\psi_1y_2e_{(1,2)}=y_1\psi_1 e_{(1,2)}=0$. 

    We have yet to describe the last map to $\DCR_\lbd$ in general, but in this case it acts by killing $y_2 e_{(2,1)}$, so that
    \begin{align*}
        \BC S_2&\lto \DCR_\yd{1,1}\\
        s_1&\lmto \begin{diagram}
        \draw[red](0,0)--(2,2);
        \draw (2,0)--(0,2);
    \end{diagram}
    \ -\ 
    \begin{diagram}
        \draw[red](0,0)--(0,2);
        \draw(2,0)--(2,2);
    \end{diagram}
    \ +
    2\begin{diagram}
        \draw(0,0)--(2,2);
        \draw[red](2,0)--(0,2);
    \end{diagram}
    \ +\ 
    \begin{diagram}
        \draw(0,0)--(0,2);
        \draw[red](2,0)--(2,2);
    \end{diagram}
    \ -\ 
    \begin{diagram}
        \draw(0,0)--(0,2);
        \draw[red](2,0)--(2,2);
        \fill[red] (2,1) circle (5pt);
    \end{diagram}\ .
    \end{align*}
    Note that there are two idempotent diagrams here, given by the diagrams with only vertical strings with no dots. They are $e_{(2,1)}$, which has a negative sign in front of it, and $e_{(1,2)}$, which has a positive sign. This is indicative of $e_{(2,1)}$ corresponding to the sign representation and $e_{(1,2)}$ to the trivial representation -- this will be elucidated by the coming sections. 
\end{EX*}

\subsection{The cellular structure of $\CR_\lbd$}
As per Hu-Mathas \cite{hu2010graded}, $\CR\coloneqq \CR_\lbd$ is cellular. We will refer the reader to \cite{hu2010graded} and \cite{mathas2015cyclotomic} for more details, for instance on the definition of a cellular algebra, which we take here for granted. Let us only remark that our convention for a cellular algebra is such that the cell labeled by a maximal (rather than a minimal) element of the labeling poset is an ideal of the algebra. We will try to adhere to the conventions of Hu-Mathas wherever possible. 
% ; in fact, we will see the shape of the cellular structure kind of looks like a sword. We will correspondingly call its two parts the ``hilt'' and the ``blade''

Let us briefly set the stage (i.e. speedrun through definitions). Implicit in everything is a predestined choice of reduced word (or ``reduced expression'', or ``redex'') for each element of $S_n$. We will think of multipartitions as being written vertically, e.g. written like a vector
\[\thbo{\ydia{3,2}}{\ydia{4}}{\emptyset}.\]
For such a multipartition, we define the ``content'' of the box in the $r$-th row of the $c$-column of the $l$-th level as 
\[\cont(r,c,l)=c-r+\kappa_l.\]
We had declared from the onset that our quiver was $A_\infty$; if instead we were considering the affine type $\wt A_e$, we could further consider the ``residue'' in $\BZ/e\BZ$,
\[\res(r,c,l)=c-r+\kappa_l\mod e.\]
For $e=\infty$ we may conflate these notions. We denote by $\Lbd_\alpha^\omega$ the set of $\ell$-multipartitions such that the $\alpha=\sum_{i\in\res\lbd}\alpha_i$, where the dependence on $\omega$ is implicit in the definition of residue. $\Lbd_\alpha^\omega$ can be endowed with poset structures, the most famous of which is the dominance order, defined by $\lbd\domge \mu$ if $\sum_{j=1}^{m-1} |\lbd^{(j)}|+\sum_{i=1}^k \lbd^{(m)}_i\ge \sum_{j=1}^{m-1} |\mu^{(j)}|+\sum_{i=1}^k \mu^{(m)}_i \quad\foralls m,k$.

As per Hu-Mathas, let $\Std\mu$ denote the set of standard tableaux of shape $\mu$, so that on each level $\mu^{(i)}$ (the $i$-th entry in the vector picture above) of $\mu$ we have a standard tableau. Let $\tlie\rv_{\le n}$ denote the multitableau obtained by deleting all boxes labeled $>n$. Let $\domg$ be the partial order on $\Std\mu$ such that $\tlie\domg \slie$ if $\shape(\tlie\rv_{\le i})\domg\shape(\slie\rv_{\le i})$ (referring to the dominance order of multipartitions) for all $i$. Let $\tlie^\mu$ be the largest (in this order $\domg$) standard multitableau corresponding to the multipartition $\mu$ (namely fill in the boxes from left to right, top to bottom). For $\tlie\in\Std\mu$, let $w^\tlie$ be the permutation such that $w^\tlie\tlie^\mu =\tlie$; in an abuse of notation, let the KLR diagram $\psi_{w^\tlie} e$ be denoted $\psi_\tlie e$. If $A$ is a box of content $i$ in a multipartition $\mu$, then let
\[d_A(\mu)=\#\{\te{addable nodes of residue }i\te{ strictly below }A\}-\#\{\te{removable nodes of residue }i\te{ strictly below }A\}.\] 
Then, letting $\tlie^{-1}(i)$ denote the box containing $i$ in $\tlie$, one defines $\deg\tlie$ by requiring inductively
\[\deg\tlie\rv_{\le k}=\deg \tlie\rv_{\le k-1}+d_{\tlie^{-1}(k)}(\shape \tlie\rv_{\le k}).\]
Also let us define
\[d_i^\mu = d_{\tlie^{\mu,-1}(i)}(\shape\tlie^\mu\rv_{\le i}),\]
so that we may define the elements
\begin{align*}
    e^\mu &= e_{\res(\tlie^\mu)},\\
    y^\mu &= y_1^{d^\mu_1}\cdots y_n^{d^\mu_n}.
\end{align*}
Then it is the celebrated theorem of Hu-Mathas that
% \begin{THM}[Hu-Mathas]
%     A cyclotomic KLR algebra $\CR_\alpha^\omega$ has a graded cellular basis with respect to the dominance order $\domge$ on $\Lbd_\alpha^\omega$ such that the basis elements $\psi_{\slie\tlie}=\psi(\slie,\tlie)$ for $\slie,\tlie\in\Std\mu$ are defined by
%     \[\psi(\slie,\tlie)=\psi_{\slie} e^\mu y^\mu \psi_{\tlie}^\top,\] 
%     where $\psi^\top$ means flipping the KLR diagram upside-down. The cellular antiinvolution is given by $\psi(\slie,\tlie)^\dag=\psi(\tlie,\slie)$, and the grading is
%     \[\deg\psi(\slie,\tlie)=\deg\slie+\deg\tlie.\] 
%
%     Here the convention for cellular algebras is lowest weight and right-handed, namely that
%     \[\psi(\slie,\tlie) x=\sum_\ulie \psi(\slie,\ulie)\cdot\on{const}_{\tlie\ulie}(x)\mod \CR^{\domg\mu}.\]
%     I'm sure you can flip this handedness preference if you preferred, which I do, so that
%     \[x\psi(\slie,\tlie)=\sum_{\ulie} \on{const}_{\ulie\slie}\cdot\psi(\ulie,\tlie)\mod\CR^{\domg\mu}.\] 
% \end{THM}
\begin{THM}[Hu-Mathas]
    A cyclotomic KLR algebra $\CR_\alpha^\omega$ has a graded cellular basis with respect to the dominance order $\domge$ on $\Lbd_\alpha^\omega$ such that the basis elements $\psi_{\slie\tlie}=\psi(\slie,\tlie)$ for $\slie,\tlie\in\Std\mu$ are defined by
    \[\psi(\slie,\tlie)=\psi_{\slie} e^\mu y^\mu \psi_{\tlie}^\top,\] 
    where $\psi^\top$ means flipping the KLR diagram upside-down. The cellular anti-involution is given by $\psi(\slie,\tlie)^\dag=\psi(\tlie,\slie)$, and the grading is
    \[\deg\psi(\slie,\tlie)=\deg\slie+\deg\tlie.\] 

    % Here the convention for cellular algebras is lowest weight and right-handed, namely that
    % \[\psi(\slie,\tlie) x=\sum_\ulie \psi(\slie,\ulie)\cdot\on{const}_{\tlie\ulie}(x)\mod \CR^{\domg\mu}.\]
    % I'm sure you can flip this handedness preference if you preferred, which I do, so that
    Moreover, multiplication satisfies
    \[x\psi(\slie,\tlie)=\sum_{\ulie} \on{const}_{\ulie\slie}\cdot\psi(\ulie,\tlie)\mod\CR^{\domg\mu}.\] 
\end{THM}
We use the notation $\CR^{\ge\mu}$ to denote the ideal of cells $\domge\mu$; we will also frequently simply use $\ge$ to denote the dominance order $\domge$ of multipartitions when the context is clear. 

Since we are interested in just a particular block of a particular multiplicity-free cyclotomic quotient, things are relatively easy for us. The cells are labeled by multipartitions $\mu$ of $n$ such that $\cont\mu=\cont\lbd$. Note that in defining $\CR_\lbd=\CR^\omega_\alpha$ we have written the cyclotomic parameter $\omega$ in terms of the fundamental weights in reverse order, so that the first level of $\mu$ has $\delta$ added to the content, while the last level of $\mu$ has $\delta-\ell(\lbd)+1$ added to the content. %We produce below two basic examples.
\begin{EX*}
    For $n=\delta=2$, $\lbd=\ydia{1,1}$, the cellular structure of $\CR_\lbd$ looks like:
    
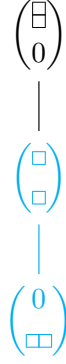
\begin{center}
        \begin{tikzcd}
            \beau\binom{\ydia{1,1}}{0} \arrow[dash]{d}\\
            \beau\textcolor{cyan}{\binom{\ydia{1}}{\ydia{1}}}\arrow[dash,color=cyan]{d}\\
            \beau\textcolor{cyan}{\binom{0}{\ydia{2}}}
        \end{tikzcd}
    \end{center}
    The cyclotomic parameter is $\omega=\varpi_2+\varpi_1$, so that the content of the first level is shifted by $+\kappa_1=+2$ and the content of the second level is shifted by $+\kappa_2=+1$. The condition of being a Kleshchev multipartition is $\mu^{(i)}_{k+\kappa_i-\kappa_{i+1}}\le \mu^{(i+1)}_k$, namely that $\mu^{(i)}_{k+1}\le\mu^{(i+1)}_k$. Only the bottom two cells, drawn in cyan, satisfy this. The cellular basis is as follows, where red refers to $\alpha_2$ and black refers to $\alpha_1$.
%     \begin{center}
%         \begin{math}\begin{tblr}{
%   vline{2} = {2}{},
%   hline{2} = {2}{},
% }
%  & \tbo{\begin{ytableau}
%      1\\2
%  \end{ytableau}}{0}  \\
% \tbo{\begin{ytableau}
%      1\\2
%  \end{ytableau}}{0} & y_2e_{(2,1)}
% \end{tblr}\end{math}
%     \end{center}
% \vspace{1em}  
%     \begin{center}
%     \begin{tblr}{
%       vline{2} = {2-3}{},
%       hline{2} = {2-3}{},
%     }
%       & $\tbo{\begin{ytableau}
%           1
%       \end{ytableau}}{\begin{ytableau}
%           2
%       \end{ytableau}}$ & $\tbo{\begin{ytableau}
%           2
%       \end{ytableau}}{\begin{ytableau}
%           1
%       \end{ytableau}}$ \\
%     $\tbo{\begin{ytableau}
%           1
%       \end{ytableau}}{\begin{ytableau}
%           2
%       \end{ytableau}}$ & $e_{(2,1)}$  & $e_{(1,2)}\psi_1 e_{(2,1)}$   \\
%     $\tbo{\begin{ytableau}
%           2
%       \end{ytableau}}{\begin{ytableau}
%           1
%       \end{ytableau}}$ & $e_{(2,1)}\psi_1 e_{(1,2)}$  &   $e_{(1,2)}\psi_1\psi_1 e_{(1,2)}=y_2 e_{(1,2)}$
%     \end{tblr}
%     \end{center}
% \vspace{1em}
%     \begin{center}
%         $\begin{tblr}{
%   vline{2} = {2}{},
%   hline{2} = {2}{},
% }
%  & \tbo{0}{\begin{ytableau}
%      1 & 2
%  \end{ytableau}}  \\
% \tbo{0}{\begin{ytableau}
%      1 & 2
%  \end{ytableau}} & e_{(1,2)}
% \end{tblr}$
% \end{center}
    % Diagrammatically this is
    \begin{center}
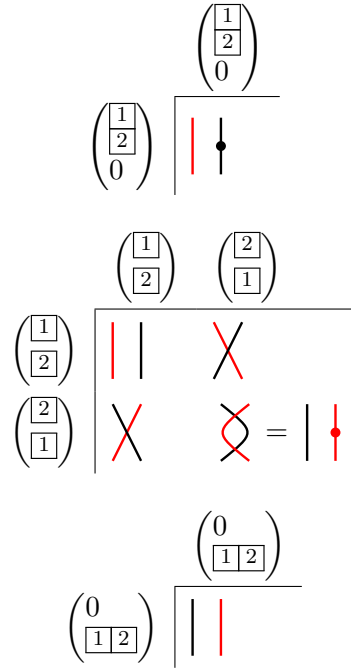
        \begin{math}\begin{tblr}{
  vline{2} = {2}{},
  hline{2} = {2}{},
}
 & \tbo{\begin{ytableau}
     1\\2
 \end{ytableau}}{0}  \\
\tbo{\begin{ytableau}
     1\\2
 \end{ytableau}}{0} & \begin{diagram}
     \draw[red](0,0)--(0,2);
     \draw(1,0)--(1,2);
     \fill (1,1) circle (5pt);
 \end{diagram}
\end{tblr}\end{math}
    \end{center}
\vspace{1em} 
    \begin{center}
    \begin{tblr}{
      vline{2} = {2-3}{},
      hline{2} = {2-3}{},
    }
      & $\tbo{\begin{ytableau}
          1
      \end{ytableau}}{\begin{ytableau}
          2
      \end{ytableau}}$ & $\tbo{\begin{ytableau}
          2
      \end{ytableau}}{\begin{ytableau}
          1
      \end{ytableau}}$ \\
    $\tbo{\begin{ytableau}
          1
      \end{ytableau}}{\begin{ytableau}
          2
      \end{ytableau}}$ & $\begin{diagram}
          \draw[red](0,0)--(0,2);
          \draw(1,0)--(1,2);
      \end{diagram}$  & $\begin{diagram}
          \draw[red](0,2)--(1,0);
          \draw(0,0)--(1,2);
      \end{diagram}$   \\
    $\tbo{\begin{ytableau}
          2
      \end{ytableau}}{\begin{ytableau}
          1
      \end{ytableau}}$ & $\begin{diagram}
          \draw[red](0,0)--(1,2);
          \draw(0,2)--(1,0);
      \end{diagram}$  &   $\begin{diagram}
          \draw(0,0)..controls(1.25,1)..(0,2);
          \draw[red](1,0)..controls(-0.25,1)..(1,2);
      \end{diagram}=\ \begin{diagram}
          \draw(0,0)--(0,2);
     \draw[red](1,0)--(1,2);
     \fill[red] (1,1) circle (5pt);
      \end{diagram}$
    \end{tblr}
    \end{center}
\vspace{1em}
    \begin{center}
        $\begin{tblr}{
  vline{2} = {2}{},
  hline{2} = {2}{},
}
 & \tbo{0}{\begin{ytableau}
     1 & 2
 \end{ytableau}}  \\
\tbo{0}{\begin{ytableau}
     1 & 2
 \end{ytableau}} & \begin{diagram}
     \draw(0,0)--(0,2);
     \draw[red](1,0)--(1,2);
 \end{diagram}
\end{tblr}$
    \end{center}
\end{EX*}
\begin{EX*}
    For $n=\delta=3$ and $\lbd=\ydia{1,1,1}$, the cellular structure of $\CR_\lbd$ looks like
    \begin{center}
    \adjustbox{scale=0.75,center}
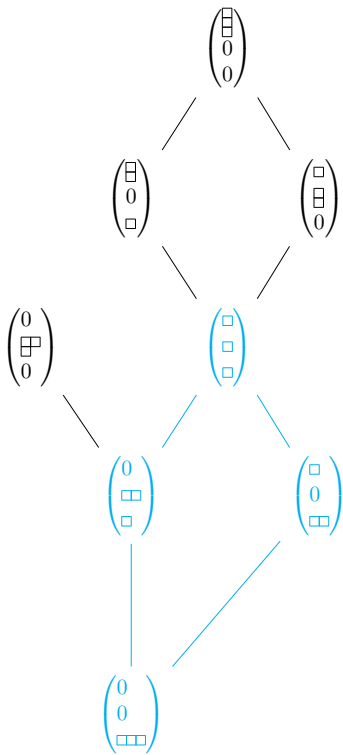
{
        \begin{tikzcd}[ampersand replacement=\&, column sep=0.9em,row sep=2.5em]
            \& \& \thbo{\ydia{1,1,1}}{0}{0}\arrow[dash]{dr}\arrow[dash]{dl}\&  \\
            \& \thbo{\ydia{1,1}}{0}{\ydia{1}}\arrow[dash]{dr}\& \& \thbo{\ydia{1}}{\ydia{1,1}}{0}\arrow[dash]{dl}\\
            \thbo{0}{\ydia{2,1}}{0}\arrow[dash]{dr}\& \& \cyan{\thbo{\ydia{1}}{\ydia{1}}{\ydia{1}}}\arrow[dash,color=cyan]{dl}\arrow[dash,color=cyan]{dr}\& \\
            \& \cyan{\thbo{0}{\ydia{2}}{\ydia{1}}}\arrow[dash,color=cyan]{dd}\& \& \cyan{\thbo{\ydia{1}}{0}{\ydia{2}}}\arrow[dash,color=cyan]{ddll}\\
            \& \ \& \ \& \ \\
            \& \cyan{\thbo{0}{0}{\ydia{3}}}\& \& 
        \end{tikzcd}
        }
    \end{center}
    The Kleshchev cells are shown in cyan. We omit the entire cellular basis.
\end{EX*}
From these examples it is easy to see that in general there are cells where every level of the labeling multipartition has at most one row, and cells where there exists a level with more than one row. We will call the former type of labeling multipartition ``1-row multipartitions'', and the latter type ``multi-row multipartitions''.

It is part of the general theory of cellular algebras that some cells correspond to simple modules while other cells contribute to some over-counting. More precisely, there is some bilinear form on cell modules such that each cell module quotiented out by the kernel of its bilinear form is either a simple module or zero; this process classifies simples. The cells for which this quotient is non-zero are often called ``Kleshchev cells'', or ``idempotent cells'', and those for which this quotient is zero are called ``non-Kleshchev cells'', or ``nilpotent cells''. The multipartitions labeling these cells are called ``idempotent'' or ``Kleshchev multipartitions'', or ``nilpotent'' or ``non-Kleshchev multipartitions'', respectively. 

There is a characterization for $e=0$ and the cyclotomic parameters listed in decreasing order (e.g. Example 3.2 of \cite{kleshchev2010representation}) that the Kleshchev multipartitions are those for which 
\[\mu^{(k)}_{i+\kappa_k-\kappa_{k+1}}\le \mu^{(k+1)}_{i},\] 
where for us $\kappa_k-\kappa_{k+1}=1$. It is then clear that 1-row multipartitions are idempotent.

By unwinding the definition one sees that the ``top-left'' element of each cell is
\[y^{\mu}=\prod_{\te{boxes }b\in\mu\te{ s.t. }b\te{ is the first box in a row numbered }>1}y_{\tlie^{\mu}(b)}=\prod_{(r,c,l)\in\mu:r>1,c=1}y_{\tlie^{\mu}(r,c,l)}.\]
Note in particular that $y^\mu=1$ for $\mu$ a 1-row multipartition.

\section{The Jacobi-Trudi algebra}\label{sect:JT}
To witness highest-weight phenomena, we would like to consider an algebra which is quasi-hereditary, or at least stratified. At the moment $\CR_\lbd$ still could have many nilpotent cells, which we would hence like to get rid of. Actually it turns out we will do something slightly more violent -- we will quotient out by all cells labeled by multi-row multipartitions. In other words, we would like to keep only the cells labeled by 1-row multipartitions. This requires some justification.

\subsection{The insufficiency of dominance}
First let us demonstrate that the dominance order is not enough to justify our actions. Indeed, in the dominance order, a priori it is possible to have $\nu\le\mu$ for a 1-row multipartition $\mu$ and a multi-row multipartition $\nu$. This would prevent us from justifying quotienting out by the ideal of cells formed by multi-row multipartitions.
\begin{EX*}
    % For $\lbd=\ydia{2,2,2}\vdash 6$, one has that 
    % \[\mu(s_2)=\begin{pmatrix*}[l]\ydia{2}\\ \ydia{1}\\ \ydia{3}\end{pmatrix*},\] 
    % whereas
    % \[\mu(s_2s_1)=\thbo{\ydia{1}}{\ydia{1}}{\ydia{4}}\] 
    % has an `associated' nilpotent cell
    % \[\nu(s_2s_1)=\thbo{\ydia{1,1}}{0}{\ydia{4}}\] 
    % with the property that
    % \[\nu(s_2s_1)<\mu(s_2)\] 
    % in the dominance order. 
    For $\lbd=\ydia{2,2,2}\vdash 6$, one has that 
    \[\begin{pmatrix*}[l]\ydia{2}\\ \ydia{1}\\ \ydia{3}\end{pmatrix*}\]
    is a multipartition with the correct contents having at most one row per level, whereas
    \[\thbo{\ydia{1,1}}{0}{\ydia{4}}\] 
    is also a multipartition with the correct contents with a level with two rows. These two satisfy that
    \[\thbo{\ydia{1,1}}{0}{\ydia{4}}<\begin{pmatrix*}[l]\ydia{2}\\ \ydia{1}\\ \ydia{3}\end{pmatrix*}\] 
    in the dominance order. 

    This example shows that the set of multi-row multipartitions does not form an upper-ideal in the dominance poset. 
\end{EX*}
% This example shows that in principle the two-sided ideal generated by a positive degree $y^\mu$ (which is then nilpotent) could intersect nontrivially with a two-sided ideal generated by an idempotent. By computing some examples one can see this seems to not be the case, and indeed the theorem of Bowman above shows this is not the case. Indeed, in the Lusztig-Webster ordering, it turns out that the set of multi-row multipartitions form an ideal.
However, computing some examples will suggest that, even though the set of multi-row multipartitions is not an upper-ideal, the cells they label nonetheless form an ideal in the cellular algebra. The dominance order is clearly insufficient to justify this; thankfully, it turns out there is another ordering in which the multi-row multipartitions \textit{do} form an upper-ideal.

\subsection{The coarsened order of Uglov}
The following ideas were suggested to us by Chris Bowman. 

To set the stage let us recall the parts we need from \cite{bowman2017many}. Given a ``charge'' $\sigma=(e;\sigma_1,\cdotsc,\sigma_\ell)\in\BZ_{>1}\times\BZ^\ell$ (we choose the symbol $\ell$ here because eventually we will see that for us $\ell=\ell(\lbd)$), Bowman defines the ``$e$-charge'' to be
%\[\omega(\sigma)=(e;\omega_1,\cdotsc,\omega_\ell)\in\BN_{>1}\times(\BZ/e\BZ)^\ell,\] 
\[\kappa(\sigma)=(e;\kappa_1,\cdotsc,\kappa_\ell)\in\BN_{>1}\times(\BZ/e\BZ)^\ell,\] 
i.e. the information of $e$ together with the reductions modulo $e$ of the $\sigma_i$. From this information we can then construct the appropriate cyclotomic Hecke or KLR algebra. Since there is now a quantum characteristic $e$ at play, we must distinguish between the content, which is\footnote{Here $(r,c,l)$ is a box, read as ``row, column, level''; note well the difference between $\ell$ and $l$).}
\[\cont(r,c,l)=\sigma_l+c-r,\] 
and the residue, which is the reduction of this mod $e$,
\[\res(r,c,l)=\sigma_l+c-r\mod e.\]

A charge is said to be ``asymptotic'' (with respect to $n$) if
\[\sigma_i-\sigma_{i+1}>n.\]

A key point in \cite{bowman2017many} (at least for us) is the consideration of a partial order on (multi)partitions differing from the usual dominance order. Evidently this order was defined by Lusztig and coarsened by Webster; we will hence call it the Lusztig-Webster ordering. 
% Given a charge $\sigma$, a box\footnote{Row, column, level.} $(r,c,l)$ of a partition is said to be ``$\sigma$-below'' another $(r',c',l')$, or rather $(r,c,l)\doml_\sigma (r',c',l')$, if 
% \begin{enumerate}
%     \item $\res(r,c,l)=\res(r',c',l')$ and
%     %the following thing used to say "either $\cont(r,c,l)<\cont(r',c',l')$, or $\cont(r,c,l)=\cont(r',c',l')$ and $l>l'$.". I decided at ICERM that this cannot be correct, because Chris says this order should coincide with reverse cylindric for non-asymptotic charges.
%     \item either $\cont(r,c,l)<\cont(r',c',l')$, or $\cont(r,c,l)=\cont(r',c',l')$ and $l>l'$. 
% \end{enumerate}
Given a charge $\sigma$, a box $(r_1,c_1,l_1)$ of a partition is said to be ``$\sigma$-below'' another $(r_2,c_2,l_2)$, or rather $(r_1,c_1,l_1)\doml_\sigma (r_2,c_2,l_2)$, if 
\begin{enumerate}
    \item $\res(r_1,c_1,l_1)=\res(r_2,c_2,l_2)$ and
    %the following thing used to say "either $\cont(r,c,l)<\cont(r',c',l')$, or $\cont(r,c,l)=\cont(r',c',l')$ and $l>l'$.". I decided at ICERM that this cannot be correct, because Chris says this order should coincide with reverse cylindric for non-asymptotic charges.
    \item either $\cont(r_1,c_1,l_1)>\cont(r_2,c_2,l_2)$, or $\cont(r_1,c_1,l_1)=\cont(r_2,c_2,l_2)$ and $l_1>l_2$. 
\end{enumerate}
Then, given two multipartitions of $n$, we say $\lbd\domge_\sigma\mu$ if there is a residue-preserving bijection between the boxes of $\mu$ and $\lbd$
\[f\colon \lbd\lto \mu\]
such that either $f(r,c,l)=(r,c,l)$ or $(r,c,l)\domg_\sigma f(r,c,l)$ for all boxes $(r,c,l)$ of $\lbd$. In other words, we require for each $(r,c,l)\in\lbd$ that either 
\begin{itemize}
    \item $\cont (r,c,l)=\cont f(r,c,l)$ and $l<f(l)$, or
    \item $\cont (r,c,l)<\cont f(r,c,l)$.
\end{itemize}

In \cite{bowman2017many}, Bowman proves the following:
\begin{PROP}[Bowman, Prop 7.3]
    For $\sigma$ an asymptotic charge, there is a cellular basis $A^\sigma_{\slie\tlie}$ of $\CR^{\kappa(\sigma)}_n$, where $\kappa(\sigma)$ is the associated $e$-charge, differing from the Hu-Mathas basis by a $\doml_\sigma$-unitriangular change-of-basis matrix. In particular the cell ideals are the same, and one can study the two-sided ideals generated by $y^\mu$ by looking at the Lusztig-Webster ordering instead.
\end{PROP}

The point is that morally speaking, for $e=\infty$, the asymptotic condition on $\sigma$ is not important. One can make this precise by considering a quiver for $e\gg n$; the (cyclotomic) KLR algebra on $n$ strands colored by a small subquiver is independent of $e$, and so we can pick an (asymptotic) lift $\sigma$ of the desired cyclotomic parameter by setting e.g. $\sigma_i=(\ell-i)e+\kappa_i$. It is then not hard to see that the orders $\doml_\sigma$ and $\doml_\kappa$ defined from $\sigma$ and $\kappa$ are the same:
\begin{LEM}
    In the setting above, i.e. $e\gg n$ and $\sigma_i=(\ell-i)e+\kappa_i$, we have that the partial orders $\doml_\sigma$ and $\doml_\kappa$ coincide.
\end{LEM}
\vspace{-0.5em}
\begin{PRF}
    $\mu\doml_\sigma\lbd$ is saying there exists some $f\colon \lbd\lto \mu$ such that $f(r,c,l)\doml_\sigma (r,c,l)$, and the same statement is true with $\sigma$ replaced by $\kappa$. Hence it suffices to show that $(r,c,l)\doml_\sigma (r',c',l')\iff (r,c,l)\doml_\kappa (r',c',l')$.

    Since $\sigma_i=\kappa_i\mod e$, and since $\kappa_i\le n$ in our setup, we know that
    \[\cont_\kappa(r,c,l)=\res_\kappa(r,c,l)=\res_\sigma(r,c,l).\] 

    First let us assume $(r,c,l)\doml_\sigma(r',c',l')$. Then we know $\res_\sigma(r,c,l)=\res_\sigma(r',c',l')$ and that one of two things is true: either (a) $\cont_\sigma(r,c,l)<\cont_\sigma(r',c',l')$ or (b) $\cont_\sigma(r,c,l)=\cont_\sigma(r',c',l')$ with $l>l'$. Note 
    \[\res_\sigma(r,c,l)=\res_\sigma(r',c',l')\implies \res_\kappa(r,c,l)=\res_\kappa(r',c',l'),\] 
    so one condition is satisfied already. In fact, we've already noted that $\res_\kappa=\cont_\kappa$, so actually we then know
    \[\cont_\kappa(r,c,l)=\cont_\kappa(r',c',l').\] 
    
    In case (a), 
    \begin{align*}
        (\ell-l)e+\kappa_l+c-r=\cont_\sigma(r,c,l)&<\cont_\sigma(r',c',l')=(\ell-l')e+\kappa_{l'}+c'-r'\\
        \implies \kappa_l+c-r&<(l-l')e+\kappa_{l'}+c'-r'\\
        \implies \cont_\kappa(r,c,l)&<\cont_\kappa(r',c',l')+(l-l')e\\
        \implies l&>l',
    \end{align*}
    where for the last implication we use that we already know $\cont_\kappa(r,c,l)=\res_\kappa(r,c,l)=\res_\kappa(r',c',l')=\cont_\kappa(r',c',l')$.

    In case (b), we get $l>l'$ for free. This concludes one direction.

    Secondly, for the other direction, if we know $(r,c,l)\doml_\kappa(r',c',l')$, then we know $\res_\kappa(r,c,l)=\res_\kappa(r',c',l')$ and $l>l'$ (since $\cont_\kappa=\res_\kappa$). We've already noted that $\res_\kappa=\res_\sigma$, so we get $\res_\sigma(r,c,l)=\res_\sigma(r',c',l')$ for free, and $l>l'$ implies
    \[\cont_\sigma(r,c,l)=(\ell-l)e+\cont_\kappa(r,c,l)<(\ell-l')e+\cont_\kappa(r',c',l')=\cont_\sigma(r',c',l').\]
    This concludes.
\end{PRF}
As a corollary, we then have
\begin{COR}
    For $e=\infty$, there is a cellular basis of $\CR^\omega_\alpha$ differing from the Hu-Mathas basis by a $\doml_{\kappa(\omega)}$-unitriangular change-of-basis matrix, where $\kappa(\omega)$ is the multi-charge associated to $\omega$. 
    % the cyclotomic parameter $\kappa_i=\wan{\alpha_i;\omega}$ obtained from $\omega$.
\end{COR}

% The reason this is important is that the cellular basis of Hu-Mathas is built on the dominance order, where it is possible for to have $\nu_i(u)\le \mu(w)$ for $u\neq w$ and some nilpotent cell $\nu_i(u)$. To make things concrete:
We often write $\doml_\omega$ for $\doml_{\kappa(\omega)}$ as well. 
% The reason this is important is that the cellular basis of Hu-Mathas is built on the dominance order, where a priori it is possible to have $\nu\le\mu$ for a one-row multipartition $\mu$ and a multi-row multipartition $\nu$. This would prevent us from justifying quotienting out by the ideal of cells formed by multi-row multipartitions.
% \begin{EX*}
%     For $\lbd=\ydia{2,2,2}\vdash 6$, one has that 
%     \[\mu(s_2)=\begin{pmatrix*}[l]\ydia{2}\\ \ydia{1}\\ \ydia{3}\end{pmatrix*},\] 
%     whereas
%     \[\mu(s_2s_1)=\thbo{\ydia{1}}{\ydia{1}}{\ydia{4}}\] 
%     has an associated nilpotent cell
%     \[\nu(s_2s_1)=\thbo{\ydia{1,1}}{0}{\ydia{4}}\] 
%     with the property that
%     \[\nu(s_2s_1)<\mu(s_2)\] 
%     in the dominance order. 
% \end{EX*}
% This example shows that in principle the two-sided ideal generated by a positive degree $y^\mu$ (which is then nilpotent) could intersect nontrivially with a two-sided ideal generated by an idempotent. By computing some examples one can see this seems to not be the case, and indeed the theorem of Bowman above shows this is not the case. Indeed, in the Lusztig-Webster ordering, it turns out that the set of multi-row multipartitions form an ideal.

\subsection{Taking a quotient}
The previous discussion is all in order to justify killing all cells of $\CR_\lbd$ labeled by multi-row multipartitions. In so doing all nilpotent cells are killed as well, so we obtain a cellular quotient which is moreover quasi-hereditary. 
\begin{PROP}\label{prop:JTalg}
    The set of multi-row multipartitions in $\Lbd_\lbd$ form an upper-ideal under the Lusztig-Webster ordering, or the ``Uglov ordering''. In other words, for a multi-row $\nu$, 
    \[\mu\domg_\omega\nu\implies \mu\te{ multi-row}.\] 
    Hence, the cells that they label form a two-sided ideal in $\CR=\CR_\lbd$. We can quotient out by this ideal to form an algebra $\DCR=\DCR_\lbd$:
    \[\DCR_\lbd\coloneqq\mfrac{\CR_\lbd}{\wan{\nu}_{\nu\te{ multi-row}}}.\] 
    This is what we call the ``Jacobi-Trudi algebra''. Since we have seen earlier that 
    \[y^{\nu}=\prod_{(r,c,l)\in\nu:r>1,c=1}y_{\tlie^{\nu}(r,c,l)},\]
    we can also say
    \[\DCR_\lbd=\mfrac{\CR_\lbd}{\wan{y^{\nu}}_{\nu\te{ multi-row}}}.\] 

    This quotient algebra is quasi-hereditary because all cells are now idempotent.
% 
    % Note that the Hu-Mathas cellular basis becomes, in this quotient algebra, a basis of elements of form
    % \[\psi(\slie,\tlie)=\psi_\slie e^{\mu(w)}\psi_\tlie^\top,\]
    % where $\slie,\tlie\in\Std\mu(w)$, cellular with respect to the dominance order on the set $\{\mu(w)\}$, which is opposite to the poset $W_\lbd$.
\end{PROP}
%the only idempotent cell $\mu(w)$ comparable to a $\nu_i(u)$ is the one for $w=u$. 
% \begin{PROP}
%     In our algebra $\CR=\CR_\lbd$, the space spanned by the basis elements in the nilpotent cells forms a two-sided ideal of $\CR$, $\wan{\nu_i(w)}_{i,w}$, and we can quotient out by it to form an algebra $\DCR$:
%     \[\DCR_\lbd\coloneqq\DCR\coloneqq\mfrac{\CR_\lbd}{\wan{\nu_i(w)}_{i,w}}.\] 
%     This is what we call the ``Jacobi-Trudi algebra''. Since we have seen earlier that $y_{\nu_i(w)}=\prod_{(r,c,l)\in\nu_i(w):r>1,\ c=1}y_{\tlie^{\nu_i(w)}(r,c,l)}$, we can also say
%     \[\DCR_\lbd=\mfrac{\CR_\lbd}{\wan{y_{\nu_i(w)}}}.\] 
%
%     This quotient algebra is quasi-hereditary because all cells are idempotent.
%
%     Note that the Hu-Mathas cellular basis becomes, in this quotient algebra, a basis of elements of form
%     \[\psi(\slie,\tlie)=\psi_\slie e^{\mu(w)}\psi_\tlie^\top,\]
%     where $\slie,\tlie\in\Std\mu(w)$, cellular with respect to the dominance order on the set $\{\mu(w)\}$, which is opposite to the poset $W$.
% \end{PROP}
\vspace{-0.5em}
\begin{PRF}
    The only thing to prove is that multi-row multipartitions form an ideal. This is more or less obvious. Indeed, assume for the sake of contradiction that there is a multi-row multipartition $\nu$ and a 1-row multipartition $\mu$ such that $\mu \domg_\omega\nu$. Suppose the $i$-th level of $\nu$ witnesses that $\nu$ is multi-row, i.e.
    \[\nu^{(i)}{}_2\neq 0.\] 
    Then the first box of first two rows of $\nu^{(i)}$, $b_1$ and $b_2$, have contents $\delta-i+1$ and $\delta-i$ respectively. In order for $\mu\domg_\omega\nu$, there must be a content-preserving (note well that the residue and the content coincide in our setup) bijection $f\colon \mu\lto \nu$ which either fixes a box or strictly moves it below where it used to be (in other words strictly increase the $l$ in $(r,c,l)$). As $\mu$ is 1-row, $f^{-1}(b_2)$ must have been a box in a higher level but still with content $\delta-i$. However, as $\mu$ is 1-row, all contents appearing in the first $i-1$ levels must be $\ge \delta-i+2$. This is a contradiction.
%
    % We have not yet fully explained the last paragraph of the statement, where $\slie,\tlie\in\mu(w)$; we will explain $\mu(w)$ in the next section.
\end{PRF}
Having constructed $\DCR_\lbd$ as a quotient of $\CR_\lbd=\CR_\alpha^\omega$, we have completed constructing the map
\[\vphi_\lbd^\te{BK}\colon\BC S_n\lto \DCR_\lbd.\]

Finally let us see that the previous example for the insufficiency of the dominance order is resolved by the Uglov ordering. 
\begin{EX*}
    We cannot have $\thbo{\ydia{1,1}}{0}{\ydia{4}}\doml_\omega \begin{pmatrix*}[l]\ydia{2}\\ \ydia{1}\\ \ydia{3}\end{pmatrix*}$ for the following reason. If this were so, then there is a content-preserving map
    \[f\colon \begin{pmatrix*}[l]\ydia{2}\\ \ydia{1}\\ \ydia{3}\end{pmatrix*}\lto \thbo{\ydia{1,1}}{0}{\ydia{4}},\] 
    such that each box is either fixed or moved to a higher level $l$ (i.e. moved downwards). This means the second box on the first level of the RHS, since it does not exist in the same position in the LHS, must have come from a lower level (i.e. higher upwards). However there is no lower level, hence such a map cannot exist.
\end{EX*}

\begin{remark}
    Note that the arguments above do not rely on $\lbd$ being a partition, so we could have carried the same construction for any composition $\lbd$. This is considered in Section \ref{sect:monoidal}. 
\end{remark}

% Since the idempotent cells are left untouched by this procedure, the idempotents above survive in the quotient unchanged, and so we use the same symbols to refer to their images under this quotient. It is on this quotient algebra $\mr\CR$ with this system of idempotents that we shall apply our stratification technique. 

% \begin{remark}
%     Earlier drafts of this paper attempted to directly stratify $\CR$, but I realized eventually (to my great dismay) that the obvious construction for $\jota_!$ is not exact. If we only had the dominance order to work with, the natural thing to try to force it to be exact is to have each $\CC^{\ge w}$ be defined by quotienting out by the nilpotent cells attached to $w$ in addition to the usual idempotents. But in so doing a great price is exacted, namely that the relevant Ext groups for reconstruction become very difficult to compute. Easier it is to do away altogether with the nilpotent cells from the beginning, and to do so we must appeal to the coarsened dominance order of Uglov.
% \end{remark}

% \subsection{Identification with previous work}
% In \cite{bowman2022lightleaves}, the authors constructed a quotient algebra from a cyclotomic KLR algebra, and they show that it is quasi-hereditary. With a slight shift in conventions and by choosing the correct parameters one can recover the Jacobi-Trudi algebras $\DCR_\lbd$ above. We briefly recall how this is done here.

% \warn{continue?}

% \section{Generalities of reconstruction from stratification}\label{genreconstrat}
\section{Reconstruction from stratification and nil-Koszulity}\label{sect:reconnilkoszul}
In this section we recall (e.g. from \cite{zhou2024bgg}) the technical setup for reconstructing the identity functor from a stratification, as well as the notion of nil-Koszulity. We will apply these techniques to our setting in order to obtain a BGG resolution. All parts of this section except `\nameref{subsect:nilkoszul}' can be skipped on a first read.
% The first parts of this section (before Section \ref{subsect:nilkoszul}) can be skipped on a first read.
% \warn{this is copy-pasted -- clean it up!}

\label{sect:reconstruction}
We use highest weight notation in this section on general machinery, and will return to lowest weight notation when we return to the setting of diagram algebras. Let $\sic$ denote the $\infty$-category of stable $\infty$-categories, and let $\Lbd$ be a poset with a unique final object. Our convention on arrows is that final means maximal, so $\mu\to\lbd$ means $\mu<\lbd$. 
\begin{DEF}
    A ``filtration'' on $\CC\in\sic$ is a functor
    \[\CF\colon\Lbd\lto\sic\]
    such that all arrows go to fully faithful embeddings and the final object goes to $\CC$. For $\lbd\in\Lbd$, we denote $\CC^{\le\lbd}=\CF(\lbd)$ and $\CC^{<\lbd}$ as the smallest stable subcategory containing all $\CC^{\le\mu}$ for all $\mu<\lbd$. Let $\CC^{=\lbd}=\CC^{\le\lbd}/\CC^{<\lbd}$ be the Verdier quotient.
\end{DEF}
% Our convention is to write $\mu\to\lbd$ for $\mu<\lbd$. 
Let us name
\begin{align*}
    \ol\imath_\lbd&\colon \CC^{\le\lbd}\lto\CC,\\
    \imath_\mu^{<\lbd}&\colon \CC^{\le\mu}\lto\CC^{<\lbd},\\
    \imath_\mu^\lbd&\colon \CC^{\le\mu}\lto\CC^{\le\lbd},\\
    \imath_\lbd&\colon \CC^{<\lbd}\lto\CC^{\le\lbd},\\
\jota^\lbd&\colon \CC^{\le\lbd}\lto\CC^{=\lbd};
\end{align*}
for our purposes we require both the inclusions of each $\CC^{\le\lbd}$ as well as the Verdier quotient functors to have both adjoints. This is for example the case with recollement situations. Let the adjoints, like in the recollement case, be named so that $\ol\imath_\lbd^*$ and $\jota^\lbd_!$ are left adjoints. %Also let $\imath_\mu^\lbd$ denote the embedding of $\CC^{\le\mu}\lto\CC^{\le\lbd}$.
% For our purposes let's require both the inclusions of each $\CC^{\le\lbd}$ as well as the Verdier quotient functors to have both adjoints. This is for example the case with recollement situations.

In another direction, one can filter an object of an $\infty$-category by a poset $\Lbd$:
\begin{DEF}
    A ``$\Lbd$-filtered object of $\CC$'' is a functor
    \[\CX\in\sf{Fun}(\Lbd,\CC).\] 
    Letting $\Lbd^0$ be the 0-skeleton of $\Lbd$, the ``associated graded'' of this filtered object is 
    \begin{align*}
        \gr\colon\sf{Fun}(\Lbd,\CC)&\lto\sf{Fun}(\Lbd^0,\CC)\\
        \CX&\lmto \gr\CX
    \end{align*}
    defined by
    \[(\gr\CX)(\lbd)=\gr^\lbd\CX=\Fib\pr*{\CX(\lbd)\lto\lim_{\mu\from\lbd}\CX(\mu)},\]
    assuming such limits exist. 

    An object $X\in\CC$ is said to have a $\Lbd$-filtration if after adding an initial/minimal element $\lbd_{-\infty}$ to $\Lbd$ there is a functor $\CX\colon\Lbd\cup\{\lbd_{-\infty}\}\lto\CC$ such that $\CX(\lbd_{-\infty})=X$.
\end{DEF}

% Then the general nonsense claim alluded to above is that \warn{the spectral sequence may need to be object-wise? or set up $\sf{End}\,\CC$ correctly?}
% In the previous subsection we alluded to a general fact that said that in appropriate stratification/recollement situations, the identity functor on the big category admits a filtration. More precisely, the statement is
The point is that, roughly speaking, in appropriate stratification/recollement situations (which ours will be), the identity functor on the parent category admits a filtration. More precisely\footnote{This can be found in \cite[Remarks 1.3.12, 1.3.13]{ayala2022stratified}.},
\begin{THM}[\cite{ayala2022stratified}, folklore]\label{thm:folklorefiltration}
    Let $\CC\in\sic$ admit a filtration by $\Lbd$ such that the inclusions $\imath_\lbd$ and quotients $\jota^\lbd$ have both adjoints. Then the identity functor $\Id_\CC\in\on{\sf{End}}\CC$ admits a $\Lbd^\op$-filtration with terms
    \[\calid(\lbd)=\ol\imath_\lbd\ol\imath_\lbd^*,\] 
    and the associated graded can be computed as
    \[\gr^\lbd\calid=\ol\imath_\lbd\jota^\lbd_!\jota^\lbd\ol\imath_\lbd^*,\] 
    provided that $\Lbd$ is locally-finite and ``lower-eventually disjoint-totally-ordered'' (which means that for any $\lbd$, there exists $\mu\le\lbd$ such that $\{\nu:\nu\le\mu\}$ is totally ordered).

    If there is a ``dimension/length function'' $\ell\colon\Lbd\lto\BZ^\op$ and $\CC$ moreover has a t-structure, then there is a spectral sequence
    \[E_1^{p,q}=\bigoplus_{\lbd\in\ell^{-1}(-p)}\pi^{p+q}(\gr^\lbd\calid)\specseqimplies E_\infty^{p,q}=\gr^{-p}\pi^{p+q}(\Id_\CC),\]
    where $\pi^\blt$ are the homotopy groups associated to the t-structure. 
\end{THM}
 
\begin{remark}
    If $\Lbd$ is lower-finite (namely that for any $\lbd$ the set $\{\mu:\mu\le\lbd\}$ is finite), then $\Lbd$ automatically satisfies the conditions of the theorem. This is for instance the case with (a block of) category $\CO$.
\end{remark}

In Theorem \ref{thm:folklorefiltration}, the maps on the first page $E_1$ come from the functoriality of the (co)cone. Indeed, one has a $\BZ$-filtration
\[\cdots\lto \bigoplus_{\ell(\lbd)=i}\calid(\lbd)\lto\bigoplus_{\ell(\lbd)=i+1}\calid(\lbd)\lto\cdots,\] 
where the morphisms $\ol\imath_\lbd \ol\imath_\lbd^*\to \ol\imath_\mu \ol\imath_\mu^*$ are given by the identifications
    \begin{align*}
        \id_{\ol\imath_\mu^* X}\in \hom_{\CC^{\le\mu}}(\ol\imath_\mu^* X,\ol\imath_\mu^* X)&\cong \hom_\CC(X,\ol\imath_\mu \ol\imath_\mu^* X)\\
        &\cong\hom_\CC(X,\ol\imath_\lbd \imath_\mu^\lbd\ol\imath_\mu^* X)\\
        &\cong \hom_{\CC^{\le\lbd}}(\ol\imath_\lbd^* X,\imath_\mu^ \lbd\ol\imath_\mu^* X)\\
        &\cong \hom_\CC(\ol\imath_\lbd \ol\imath_\lbd^* X, \ol\imath_\mu \ol\imath_\mu^* X),
    \end{align*}
so that applying the functoriality of the fiber to
\[\bigoplus_{\ell(\lbd)=i}\calid(\lbd)\lto\bigoplus_{\ell(\lbd)=i+1}\calid(\lbd)\lto\bigoplus_{\ell(\lbd)=i+2}\calid(\lbd)\]
gives the exact triangle
\[\bigoplus_{\ell(\lbd)=i}\gr^\lbd\calid \lto \Fib\pr*{\bigoplus_{\ell(\lbd)=i}\calid(\lbd)\to\bigoplus_{\ell(\lbd)=i+2}\calid(\lbd)} \lto \bigoplus_{\ell(\lbd)=i+1}\gr^\lbd\calid \lto[+1].\]
Rotating this triangle gives a map
\[\bigoplus_{\ell(\lbd)=i+1}\gr^\lbd\calid\lto \bigoplus_{\ell(\lbd)=i}\gr^\lbd\calid[1],\]
which under the action of taking $\pi$ gives the differential maps of $E_1$. 

\subsection{Recollections on (graded) triangular bases}
In this section we recall the graded triangular basis of \cite{brundan2023graded}. Here we also refer to an algebra with a graded triangular basis as a ``graded triangular-based algebra''. Let $A$ be any locally graded-finite-dimensional locally unital algebra. Let $I$ index the local units (orthogonal and homogeneous) of $A$, $\Phi\subseteq I$ label the ``distinguished idempotents'', $\Theta$ be a lower-finite\footnote{Here ``lower-finite'' means for any $\theta$ the set $\{\phi:\phi\le\theta\}$ is a finite set.} poset of ``weights'', and $\varpi\colon \Phi\lto\Theta$ be a map (with finite fibers) from distinguished idempotents to the weight poset:
\begin{center}
    \begin{tikzcd}
I                                                           &        \\
\Phi \arrow[u, hook] \arrow[r, "\varpi", two heads] & \Theta
\end{tikzcd}.
\end{center}
We will use symbols such as $i,j\in I$ and $\alpha,\beta\in\Phi$ and $\theta,\phi,\psi\in\Theta$. In the interest of appealing to weight notation later, we will denote the local unit labeled by $i$ as $1^i$. 

Let $A$ have a ``(graded) triangular basis'' in the sense of Brundan \cite{brundan2023graded}, which is to say that 
\begin{DEF}
    $A$ has a ``graded triangular basis'' if there are (homogeneous) sets $\te X(i,\alpha)\subseteq 1^i A 1^\alpha,\ \te H(\alpha,\beta)\subseteq 1^\alpha A 1^\beta,\ \te Y(\beta,j)\subseteq 1^\beta A 1^j$ such that 
\begin{enumerate}
    \item $$\set*{xhy:(x,h,y)\in\bigcup_{i,j,\alpha,\beta}\te X(i,\alpha)\times\te H(\alpha,\beta)\times \te Y(\beta,j)}\quad\te{forms a basis of $A$;}$$ 
    \item $\te X(\alpha,\alpha)=\tY(\alpha,\alpha)=\{1^\alpha\}$;
    \item for $\alpha\neq\beta$, 
    \begin{align*}
        \tX(\alpha,\beta)\neq\emptyset&\implies \varpi(\alpha)>\varpi(\beta),\\
        \tH(\alpha,\beta)\neq\emptyset&\implies \varpi(\alpha)=\varpi(\beta),\\
        \tY(\alpha,\beta)\neq\emptyset&\implies \varpi(\alpha)<\varpi(\beta);
    \end{align*}
    \item for each $i\in I-\Phi$, there are only finitely many $\alpha\in\Phi$ such that $\tX(i,\alpha)\cup \tY(\alpha,i)\neq\emptyset$.
\end{enumerate}
\end{DEF}
% I do not fully understand the last axiom, but as I understand it you can ignore it most of the time.
One colloquial description for this is that of the ``shirt-and-belt-and-pants'', or alternatively the ``sandwich'':
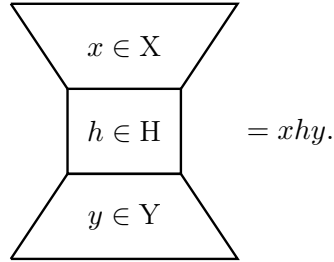
\[\hackcenter{
    \begin{tikzpicture}[scale=0.75]
        \draw (0,4.5)--(4,4.5)--(3,3)--(1,3)--(0,4.5);
        \draw (3,3)--(3,1.5)--(1,1.5)--(1,3);
        \draw (0,0)--(4,0)--(3,1.5)--(1,1.5)--(0,0);
        \node at (2,3.75) {$x\in\tX$};
        \node at (2,2.25) {$h\in\tH$};
        \node at (2,0.75) {$y\in\tY$};
    \end{tikzpicture}
}=xhy.\]
Note that for instance we have depicted the diagram for $y$ to be wider on the bottom than the top, because $y\in\tY$ should go from a `larger' idempotent to a 'smaller' one. 
For most diagram algebras, one can guess the sets $\tX,\tH,\tY$ simply by taking a typical element of the algebra, fitting it to the shirt-belt-pants picture, and declaring $\tX$ to be the set of shirts, $\tH$ to be the set of belts, and $\tY$ to be the set of pants.

Now we will put some additional constructs on this setup. Let 
\[e^\theta\coloneqq\sum_{\alpha\in\varpi^{-1}\theta}1^\alpha.\] 
Let $A^{\ge\theta}$ be defined as
\[A^{\ge\theta}\coloneqq \mfrac{A}{\wan{e^\phi\colon\phi\not\ge\theta}},\]
let e.g. $A^{\le\theta},A^{>\theta}$ be defined similarly, and let the ``Cartan algebras'' be
\[A^\theta\coloneqq e^\theta A^{\ge\theta}e^\theta.\] 
This plays the role of Cartan in the sense that modules over it are induced to form standard objects. Note that
\[A^{>\theta}=\mfrac{A^{\ge\theta}}{A^{\ge\theta}e^\theta A^{\ge\theta}};\] 
hence, by the general theory of Cline-Parshall-Scott \cite{Cline1988}, one expects a recollement $\D\Mod A^{>\theta}\lto\D\Mod A^{\ge\theta}\lto \D\Mod A^\theta$. For $M$ a module over $A$, we write 
\[M^\theta\coloneqq e^\theta M\]
in analogy with weight space notation. We often write simply $\alpha\in\theta$ as shorthand for $\alpha\in\varpi^{-1}(\theta)$, so that $e^\theta=\sum_{\alpha\in\theta} 1^\alpha$. 

Let $\Lbd_\theta$ label the simple modules $L_\lbd(\theta)$ of $A^\theta$ as well as their projective covers $P_\lbd(\theta)$ and injective hulls $Q_\lbd(\theta)$, and let $\Lbd=\bigsqcup_{\theta}\Lbd_\theta$ be the disjoint union of these label sets. Then one can show that $\Lbd$ labels all simples of $A$. Frequently it is the case (for example if the semisimplification of the Cartan is $A^\theta/\mathfrak{j} \cong\prod_{\alpha\in\varpi^{-1}(\theta)}$${\bk}$) that $\Lbd$ can be identified with $\Phi$; in this case we will instead use the symbols $\lbd,\mu\in\Lbd$ for the distinguished idempotents. Here is a picture of the situation:
\begin{center}
    \begin{tikzcd}
I                                                                                       & \Lambda \arrow[d, two heads] \\
\Phi \arrow[u, hook] \arrow[r, "\varpi", two heads] \arrow[ru, no head, dashed] & \Theta                      
\end{tikzcd}
\end{center}

\subsubsection{Recollement}
The following recollement diagrams are considered in \cite{brundan2023graded}:
\begin{center}
\begin{tikzcd}
&  & \perp &  &  &  & \perp &  &  \\[-1em]
\D^-\Mod A^{>\theta} \arrow[rrrr, "\imath_\theta"] &  &       &  & \D^-\Mod A^{\ge\theta} \arrow[rrrr, "\jota^\theta=e^\theta \sq "] \arrow[llll, swap,"\imath_\theta^*=A^{>\theta}\lotimes_{A^{\ge\theta}} \sq", bend right=60] \arrow[bend left=60]{llll}{\imath_\theta^!=\bigoplus_i\rhom_{A^{\ge\theta}}(A^{>\theta}1^i, \sq)} &  &       &  & \D^-\Mod A^\theta \arrow[llll, swap,"\jota^\theta_!=A^{\ge\theta}e^\theta\otimes_{A^\theta} \sq", bend right=60] \arrow[bend left=60]{llll}{\jota^\theta_*=\bigoplus_i\Hom_{A^\theta}(e^\theta A^{\ge\theta}1^i, \sq)} \\[-1em]
  &  & \perp &  & & &\perp &  &
\end{tikzcd}
\end{center}
and 
\begin{center}
\begin{tikzcd}
&  & \perp &  & \\[-1em]
\D^-\Mod A^{\ge\theta} \arrow[rrrr, "\ol\imath_\theta"] &  &       &  & \D^-\Mod A \arrow[bend left=60]{llll}{\ol\imath_\theta^!=\bigoplus_i\rhom_A(A^{\ge\theta}1^i,\sq)}\arrow[llll, "\ol\imath_\theta^*=A^{\ge\theta}\lotimes_A\sq", bend right=60,swap] \\[-1em]
&  & \perp &  & 
\end{tikzcd}
\end{center}
% \footnote{The big (co)Verma was originally dubbed the ``(co)standard module'', while the small (co)Verma was the ``proper (co)standard module''. We will use the terminology of `big/small'.}
Cf. the Cline-Parshall-Scott \cite{Cline1988} theory of algebraic recollements. Brundan proved in \cite{brundan2023graded} that for triangular-based algebras the functors $\jmath^\theta_!$ and $\jmath^\theta_*$ are exact. Hence one may define the ``(co)standard modules'' and ``proper (co)standard modules'' $\Delta$ and $\ol\Delta$ (resp. $\Nabla$ and $\ol\Nabla$) by:
\[\Delta_\lbd=\jota^\theta_! P_\lbd(\theta),\qquad \ol\Delta_\lbd=\jota^\theta_! L_\lbd(\theta),\] 
\[\Nabla_\lbd=\jota^\theta_* Q_\lbd(\theta),\qquad\ol\Nabla_\lbd=\jota^\theta_*L_\lbd(\theta).\]
Quite often colloquially these modules are referred to as big or small (co)Vermas. This is because these modules enjoy many homological properties in close analogy with highest weight categories such as category $\CO$, and for these many properties we refer the reader to \cite{brundan2023graded} for a detailed account. 

One may also define the module
\[\Delta(\theta)\coloneqq A^{\ge\theta} e^\theta=\bigoplus_{\lbd\in\varpi^{-1}\theta} \Delta_\lbd^{\ol{l_\lbd(\theta)}},\] 
where $l_\lbd(\theta)=\dim L_\lbd(\theta)$. The benefit of this module is that it has a right action $A^{\ge\theta} e^\theta\ractson e^\theta A^{\ge\theta}e^\theta$, so that
\[\Delta(\theta)\ractson A^\theta.\] 
Note that when $\Lbd=\Theta$ and $l_\lbd(\theta)=1$, the notions of $\Delta(\theta)$ and $\Delta_\lbd$ coincide.

The reason we recall this theory here is because we will think of both the Jacobi-Trudi algebra $\DCR_\lbd$ and, later, the Soergel calculus $\CS$ in this framework. 

\subsection{Application to triangular-based algebras}
Now let us return to the setting of the triangular-based algebra $A$, which has lowest weight notation.
% Like cellular algebras, we take for granted here the definition of a graded triangular basis for an algebra; the unfamiliar reader should consult \cite{brundan2023graded} for more details on this. 
Let $A$ have weight poset $\Theta$, with $\Lbd=\bigsqcup_\theta \Lbd_\theta$ labeling the simples of the Cartans $A^\theta$. Note that the recollements described in the previous subsection (cf. \cite[Section 2]{zhou2024bgg}) give rise to a $\Theta^\op$-filtration of the stable $\infty$-category $\D^-\Mod A$ where
\[(\D^-\Mod A)^{\le\theta}=\D^-\Mod A^{\ge\theta}.\] 
It is easy to see that
\[(\D^-\Mod A)^{<\theta}=\D^-\Mod A^{>\theta}=\D^-\Mod(\mfrac{A^{\ge\theta}}{A^{\ge\theta}e^\theta A^{\ge\theta}}),\] 
so that the Verdier quotient is $\D^-\Mod A^\theta$. Then, by Theorem \ref{thm:folklorefiltration}, 
% \begin{THM}\label{thm:triangularfiltration}
% Let $A$ be a triangular-based algebra with weight poset $\Theta$ locally-finite and eventually disjoint-totally-ordered. There exists a $\Theta$-filtration on the identity functor $\Id_{\D^-\Mod A}$ whose terms are $\calid(\theta)=\ol\imath_\theta\ol\imath_\theta^*$ and whose associated graded is
% \begin{align*}
%     \gr^\theta\calid=\ol\imath_\theta\jota^\theta_!\jota^\theta\ol\imath_\theta^*=\Delta(\theta)\otimes_{A^\theta}\rhom_A\pr*{\Delta(\theta),\sq^\dag}^*,
% \end{align*}
% where we recall the shorthand $\Delta(\theta)\coloneqq \bigoplus_{\lbd\in\theta}\Delta_\lbd^{\ol{l_\lbd(\theta)}}=A^{\ge\theta}e^\theta$ and $\ol{l_\lbd(\theta)}$ is the $q$-conjugate of the quantum dimension $\dim L_\lbd(\theta)$. Moreover, provided there is a length function (satisfying that $\ell(\Theta)$ has a minimal element) $\ell\colon \Theta\lto\BZ$, there is a spectral sequence (functorial in the input $\sq$)
% \[E_1^{p,q}=\bigoplus_{\ell(\theta)=-p}\Delta(\theta)\otimes_{A^\theta}\Ext^{-(p+q)}_A(\Delta(\theta),\sq^\dag)^* \specseqimplies E_\infty^{p,q}=\gr^{-p}H^{p+q}(\sq).\]
% \end{THM}
\begin{THM}\label{thm:triangularfiltration}
Let $A$ be a triangular-based algebra with a weight poset $\Theta$ which is upper-eventually disjoint-totally-ordered (meaning that for any $\theta$, there exists $\phi\ge\theta$ such that $\{\psi:\psi\ge\phi\}$ is totally ordered). Then there exists a $\Theta$-filtration on the identity functor $\Id_{\D^-\Mod A}$ whose terms are 
\[\calid(\theta)=\ol\imath_\theta\ol\imath_\theta^*=A^{\ge\theta}\lotimes_A\sq=\rhom_A(A^{\ge\theta},\sq^\dag)^\dag\] 
and whose associated graded is
\begin{align*}
    \gr^\theta\calid=\ol\imath_\theta\jota^\theta_!\jota^\theta\ol\imath_\theta^*=\Delta(\theta)\otimes_{A^\theta}\rhom_A\pr*{\Delta(\theta),\sq^\dag}^\dag,
\end{align*}
where $\Delta(\theta)= \bigoplus_{\lbd\in\theta}\Delta_\lbd^{\ol{l_\lbd(\theta)}}=A^{\ge\theta}e^\theta$ and $\ol{l_\lbd(\theta)}$ is the $q$-conjugate of the quantum dimension $\dim L_\lbd(\theta)$.

Moreover, provided there is a length function $\ell\colon \Theta\lto\BZ$ (satisfying that $\ell(\Theta)$ has a minimal element), there is a spectral sequence (functorial in the input $\sq$)
\[E_1^{p,q}=\bigoplus_{\ell(\theta)=-p}\Delta(\theta)\otimes_{A^\theta}\Ext^{-(p+q)}_A(\Delta(\theta),\sq^\dag)^\dag \specseqimplies E_\infty^{p,q}=\gr^{-p}H^{p+q}(\sq).\]
\end{THM}
Our strategy in general is to produce BGG resolutions from this theorem by plugging in the `dominant simple' for $\sq$ and hoping that the Ext groups are appropriately concentrated so that the first page $E_1$ lies on the horizontal axis. A notion of ``nil-Koszulity'' can be used to facilitate this. 

\subsection{Nil-Koszulity}\label{subsect:nilkoszul}
Our main tool for carrying out reconstruction is the notion of nil-Koszulity. First recall the definition of a Koszul algebra:
\begin{DEF}[\cite{polishchuk2005quadratic}]
    A graded (locally unital) algebra $A$ with $A_0=\BK$ is called ``Koszul'' if the following equivalent conditions hold:
    \begin{itemize}
        \item $\Ext_{A}^{i}(\BK,
        \BK)_j=0$ for $i\neq j$;
        \item $A$ is one-generated and the algebra $\Ext_A^\blt(\BK,\BK)$ is generated by $\Ext^1_A(\BK,\BK)$;
        \item $A$ is quadratic and $\Ext_A^\blt(\BK,\BK)\cong A^!$;
        \item the algebra $\Ext_A^\blt(\BK,\BK)$ is one-generated under the quadratic grading.
    \end{itemize}
\end{DEF}
For more details see \cite{polishchuk2005quadratic}. We warn however that in this setting we are working with locally unital algebras, so we must be careful distinguishing between left and right actions of the field $\BK$. Let the left dual be denoted as $V^*\coloneqq \hom_\BK(V,\BK)$, and the right dual as $V^\vee\coloneqq\hom_{-\BK}(V,\BK)$. 

We must also warn that in categorification when taking the dual of a graded space it is customary for the gradings to be flipped -- e.g., the linear map sending a degree 1 basis element to $1\in\bk$ and everything else to zero is degree $-1$. For this reason our convention for the Koszul/quadratic dual will be negative of the usual convention\footnote{That is to say, if $\BC[x]/x^2$ has $\deg x=1$, then $\deg x^!=-1$ for $x^!\in \BC[x^!]=(\BC[x]/x^2)^!$. }. 

The quadratic dual $A^!$ is 
\[A^!\coloneqq \mfrac{\bigotimes^\blt A_1^*}{\qidl^\perp},\] 
where $\qidl^\perp\subset \bigotimes{}^\blt A_1^*$ is the orthogonal complement to $\qidl$ under the pairing\footnote{Note well that this is \textit{not} defined as $\phi(v)\psi(w)$. To convince the reader this is what we want, we should have \newline$\begin{raisediagram}[0.3em]
    \draw[orange](0,0)--(1,0);
    \draw[orange,dashed](0,1)--(1,1);
    \draw[orange](0,2)--(1,2);
    \draw(0.25,2)--(0.25,1.5);
    \fill (0.25,1.5) circle (5pt);
    \draw[red](0.75,2)--(0.75,0.5);
    \fill[red] (0.75,0.5) circle (5pt);
    \node at (-0.5,1) {\tiny $\otimes$};
    \node at (1.2,2) {\tiny $!$};
\end{raisediagram}+\,\begin{raisediagram}[0.3em]
    \draw[orange](0,0)--(1,0);
    \draw[orange,dashed](0,1)--(1,1);
    \draw[orange](0,2)--(1,2);
    \draw(0.25,2)--(0.25,0.5);
    \fill (0.25,0.5) circle (5pt);
    \draw[red](0.75,2)--(0.75,1.5);
    \fill[red] (0.75,1.5) circle (5pt);
    \node at (-0.5,1) {\tiny $\otimes$};
    \node at (1.2,2) {\tiny $!$};
\end{raisediagram}=0$ as a relation in $\CS^{-,!}$, which should be orthogonally complementary to the relation $\begin{raisediagram}[0em]
    \draw[orange](0,0)--(1,0);
    \draw[orange,dashed](0,1)--(1,1);
    \draw[orange](0,2)--(1,2);
    \draw(0.25,0)--(0.25,0.5);
    \fill (0.25,0.5) circle (5pt);
    \draw[red](0.75,0)--(0.75,1.5);
    \fill[red] (0.75,1.5) circle (5pt);
    \node at (-0.5,1) {\tiny $\otimes$};
    % \node at (1.2,2) {\tiny $!$};
\end{raisediagram}\ -\begin{raisediagram}[0em]
    \draw[orange](0,0)--(1,0);
    \draw[orange,dashed](0,1)--(1,1);
    \draw[orange](0,2)--(1,2);
    \draw(0.25,0)--(0.25,1.5);
    \fill (0.25,1.5) circle (5pt);
    \draw[red](0.75,0)--(0.75,0.5);
    \fill[red] (0.75,0.5) circle (5pt);
    \node at (-0.5,1) {\tiny $\otimes$};
    % \node at (1.2,2) {\tiny $!$};
\end{raisediagram}\,=0$ in $\CS^-$.} $(\phi\otimes \psi)(v\otimes w)=\phi(w)\psi(v)$. The Koszul complex is then
\[K^\blt= A\otimes A^{!,\vee,\blt},\] 
where $\blt\le 0$ and extends to the left (in cohomological notation). The differential is given by
% \[A\otimes A^{!,\vee}_{n+1}\oslto{\Id\otimes \on{mult}^\vee}A\otimes A_1\otimes A_n^{!,\vee}\oslto{\on{mult}\otimes\Id} A\otimes A_n^{!,\vee},\] 
\begin{center}
    \begin{tikzcd}
        A\otimes A^{!,\vee}_{n+1}\arrow{r}{\Id\otimes \on{mult}^\vee} & A\otimes A_1\otimes A_n^{!,\vee} \arrow{r}{\on{mult}\otimes\Id}& A\otimes A_n^{!,\vee}
    \end{tikzcd}
\end{center}
where $\on{mult}^\vee\colon A_{n+1}^{!,\vee}\lto A_1\otimes A_n^{!,\vee}$ comes from the multiplication $A_n^!\otimes A_1^*\lto A_{n+1}^!$. Then it is well-known that
\begin{FACT}
    $A$ is Koszul if and only if $K^\blt\simeq \BK$. 
\end{FACT}

In terms of diagrammatics, taking duals ($\sq^*$ or $\sq^\vee$) will flip the right and left actions of $\BK$; since we would like to maintain that left actions correspond to stacking diagrams from the top, we will indicate this switch by drawing dual diagrams upside-down. 
\begin{EX*}
For example, if we consider the subalgebra of 1-color Soergel calculus with basis
\[\CS^-=\BC \set*{\begin{diagram}
    \draw[orange](0,0)--(1,0);
    \draw[orange](0,1)--(1,1);
\end{diagram}\,,\, 
\begin{diagram}
    \draw[orange](0,0)--(1,0);
    \draw[orange](0,1)--(1,1);
    \draw(0.5,0)--(0.5,0.5);
    \fill(0.5,0.5)circle(5pt);
\end{diagram}\,,\, 
\begin{diagram}
    \draw[orange](0,0)--(1,0);
    \draw[orange](0,1)--(1,1);
    \draw(0.5,0)--(0.5,1);
\end{diagram}
},\]
the Koszul dual will be drawn as
\[\CS^{-,!}=\BC \big\{\begin{raisediagram}[0.3em]
    \draw[orange](0,0)--(1,0);
    \draw[orange](0,1)--(1,1);
    \node at (1.2,1){\tiny$!$};
\end{raisediagram}\!,\, 
\begin{raisediagram}[0.3em]
    \draw[orange](0,0)--(1,0);
    \draw[orange](0,1)--(1,1);
    \draw(0.5,1)--(0.5,0.5);
    \fill(0.5,0.5)circle(5pt);
    \node at (1.2,1){\tiny$!$};
\end{raisediagram}\!,\, 
\begin{raisediagram}[0.3em]
    \draw[orange](0,0)--(1,0);
    \draw[orange](0,1)--(1,1);
    \draw(0.5,0)--(0.5,1);
    \node at (1.2,1){\tiny$!$};
\end{raisediagram}\!\!
\big\};\]
and if two duals are taken, then the diagrams are flipped upside-down twice, so that for instance 
\[\CS^{-,!,\vee}=\BC \big\{\begin{raisediagram}[0.3em]
    \draw[orange](0,0)--(1,0);
    \draw[orange](0,1)--(1,1);
    \node at (1.55,1){\tiny$!,\!\vee$};
\end{raisediagram}\!,\, 
\begin{raisediagram}[0.3em]
    \draw[orange](0,0)--(1,0);
    \draw[orange](0,1)--(1,1);
    \draw(0.5,0)--(0.5,0.5);
    \fill(0.5,0.5)circle(5pt);
    \node at (1.55,1){\tiny$!,\!\vee$};
\end{raisediagram}\!,\, 
\begin{raisediagram}[0.3em]
    \draw[orange](0,0)--(1,0);
    \draw[orange](0,1)--(1,1);
    \draw(0.5,0)--(0.5,1);
    \node at (1.55,1){\tiny$!,\!\vee$};
\end{raisediagram}\!\!
\big\}.\]
\end{EX*}

So Koszul algebras enjoy very strong homological concentration properties, which is what we want. Let us further define
\begin{DEF}
    Let $A$ be a (graded) triangular-based algebra in the sense of \cite{brundan2023graded}. Suppose $A$ has a subalgebra $A^-$, called the ``nilalgebra'', such that
    \[A\lotimes_{A^-} \bk e^\theta = \Delta(\theta)\coloneqq A^{\ge\theta}e^\theta;\] 
    in particular we require this derived tensor product to be concentrated in degree 0. Then one has by tensor-Hom
    \[\rhom_A(\Delta(\theta),\sq)=\rhom_{A^-}(\bk e^\theta,\sq).\] 

    If $A^-$ is Koszul, then $A$ is said to be ``nil-Koszul.''
\end{DEF}
The technical conditions such as having a triangular basis should not be taken too seriously -- they are here only to give a precise formulation of our driving philosophy, that the Koszulity of ``half of $A$'' is connected to BGG resolutions. The most important thing (aside from $A^-$ being Koszul) is that this `nilalgebra' $A^-$ allows one to do the tensor-Hom trick. 

Perhaps remarkably, it appears a great deal of algebras appearing in `nature' are nil-Koszul. We initiated our study of this phenomenon in \cite{zhou2024bgg}, where we show that the nil-Brauer algebra of \cite{brundan2023nil-brauer1},\cite{brundan2023nil-brauer2} is nil-Koszul. In forthcoming work with Mikhail Khovanov and Radmila Sazdanovi\'c we define a ``Hermite algebra'', categorifying Hermite polynomials, which is also nil-Koszul. Khovanov-Sazdanovi\'c's Chebyshev algebra is also nil-Koszul. We suspect a great deal of other algebras from the study of categorifications is nil-Koszul; we plan to study them one by one in due time.

% The reason we introduce this definition here is that it turns out we can add one more algebra to this list -- the (cyclotomic) Soergel calculus is nil-Koszul.
It turns out we can add one more algebra to this list -- the cyclotomic Soergel calculus and Jacobi-Trudi algebra are both nil-Koszul. We will study this in Sections \ref{sect:soergel} and \ref{sect:JTnilkoszul}.

\begin{remark}
    In \cite{zhou2024bgg}, ``nilcohomology'' was defined with respect to the nilalgebra as
    \[H^\blt(A^-:M)\coloneqq \Ext_{A^-}^\blt(\BK,M)=H^\blt( \R(M^{A^-})),\] 
    where $M^{A^-}\coloneqq\{v\in M:I\cdot v=0\}$ for $I$ the ideal of positive degree elements in $A^-$. This is in analogue to Lie algebra cohomology $H^\blt(\nlie^+:M)$. This nilcohomology has a left $\BK$-action coming from the right action $\BK\ractson\BK$, so that the $\theta$-weight space of this nilcohomology group is
    \[H^\blt(A^-:M)^\theta=\Ext_{A^-}^\blt(\bk e^\theta,M)=\Ext_A^\blt(\Delta(\theta),M),\] 
    which is the homological data we need to power the reconstruction spectral sequence. In the absence of better options, one can then compute this with the bar complex $\BK\simeq A^-\otimes I^{\otimes\blt}$; this is how nilcohomology was computed in the setting of nil-Brauer in \cite{zhou2024bgg}. However, in the present case of $A=\CS$, such tools will not be necessary because the structure of $\CS^-$ is much simpler. 
\end{remark}
\begin{remark}
    Perhaps a picture can best indicate our conventions for gradings in Koszul theory:
    
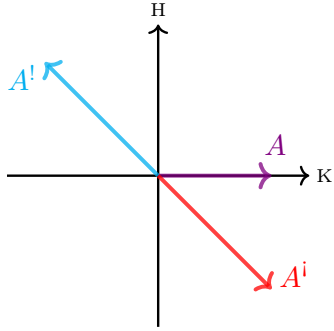
\begin{center}
        \begin{tikzpicture}
            \draw[->] (-2,0)--(2,0);
            \draw[->] (0,-2)--(0,2);
            \node at (2.2,0) {\tiny$\tK$};
            \node at (0,2.2) {\tiny$\tH$};
            \draw[->,violet,opacity=0.75,line width=1.5pt] (0,0)--(1.5,0);
            \node at (1.55,0.4) {\textcolor{violet}{$A$}};
            \draw[->,cyan,opacity=0.75,line width=1.5pt] (0,0)--(-1.5,1.5);
            \node at (-1.8,1.3) {\textcolor{cyan}{$A^!$}};
            \draw[->,red,opacity=0.75,line width=1.5pt] (0,0)--(1.5,-1.5);
            \node at (1.8,-1.3) {\textcolor{red}{$A^\shrek$}};
        \end{tikzpicture}
    \end{center}
    Here the ``Koszul dual coalgebra'' $A^\shrek$ will be defined in Section \ref{subsect:koszulperspective}. It bears warning again that $A^!$ has negative Koszul grading.
\end{remark}

\section{Stratifying and reconstructing Jacobi-Trudi}\label{sect:JTstratandrecon}
% Moving forwards there are two views we can take. We can either reduce to the `Morita-equivalent picture' in which the summands of $e^w$ are all identified with each other, or we can continue keeping the different $1^\tlie$ separate. This does not make a huge difference. Here we will stick to the latter.

\subsection{A system of idempotents for stratification}
% The following point of view on this system of idempotents was suggested to me by Chris Bowman. %We will refer to it as the ``mu-model'', or $\mu$-model'', where $\mu$ stands for `multipartition'. 

Let $\vrho=(\delta-1,\delta-2,\cdotsc,0)$ be the usual sum of fundamental weights for $\gl_\delta$. Consider the weight $w\circ\lbd$ obtained via the shifted Weyl action, and consider it as a shape of boxes in the plane by declaring that the leftmost box of the $i$-th row has content $\delta-i+1$. In other words, the top left box has content $\delta$, and the shape is left-justified. In accordance with existing notation in the literature, let $\tlie^{w\circ\lbd}$ be the standard tableau obtained by filling in with numbers first from left to right and then from top to bottom. We can read off the contents $\cont(\on{box}_i)$ of this tableau, where $\on{box}_i$ is the box labeled $i$ in this tableau; take the resulting content vector $\cont (w\circ \lbd)\coloneqq \cont (\tlie^{w\circ\lbd})$ and define the idempotent
\[1^w\coloneqq e_{\cont(w\circ\lbd)}.\] 

Another way to think about this shape $w\circ\lbd$ is as a multipartition $\mu(w)$ where the $i$-th level $\mu(w)^{(i)}$ consists of the $i$-th row of $w\circ\lbd$, 
\[\mu(w)^{(i)}\coloneqq (w\circ\lbd)_i.\] 
Here the contents of a multipartition are exactly as in the Hu-Mathas picture above, and we can let $\cont(\mu(w))\coloneqq\cont(\tlie^{\mu(w)})$. Then one sees that
% In order to make sense of the content, let us take the same cyclotomic parameter $\Lbd$ as before and order it by $\Lbd=\varpi_\delta+\varpi_{\delta-1}+\cdots+\varpi_{\delta-\ell(\lbd)+1}$, so that the first box of the first level has content $\delta$, while the first box of the last level has content $\delta-\ell(\lbd)+1$. This is exactly the same picture as earlier, and in particular
\[1^w=e_{\cont(\mu(w))}.\] 
More generally, we may define
\[1^\tlie\coloneqq e_{\cont(\tlie)}\] 
for $\tlie\in\Std\mu$ and $\mu\in\Lbd^\kappa_\alpha$, so that\footnote{We should perhaps remark on the unfortunate existence of the symbol $1^1$, which can be avoided by writing $1^\id$ instead.} $1^w=1^{\tlie^{\mu(w)}}$. In the future we may take the abbreviation $\tlie^w\coloneqq \tlie^{\mu(w)}$.

These units correspond to the general set $I$ of idempotents in the setup of Brundan-Stroppel \cite{brundan2021semiinfinite}. We can declare a subset of this to be ``distinguished idempotents'' $\Phi$, namely the idempotents $1^w$. The set $\Phi$ then already has a poset structure given by the Bruhat order, so let $\Theta=\Phi\cong W$ with $\varpi$ the identity map. Previously we had also used the symbol 
\[e^{\mu(w)}=e^w\coloneqq e_{\cont\tlie^{\mu(w)}}\] 
for $1^w$; this justifies this double notation since in the general theory one should have $e^\theta=\sum_{\alpha\in\varpi^{-1}\theta} 1^\alpha$.

\subsection{A combinatorial game}
Yet another way to think about the shape $w\circ\lbd$ is to write a reduced expression $w=s_{i_1}\cdots s_{i_k}$ and construct $w\circ\lbd$ from $\lbd$ by using the following move, which we call ``diagonal shifting'': if at any point we have shape $\mu$, applying $s_i$ to it amounts to, if possible, taking the last $\mu_i-\mu_{i+1}+1$ boxes of the $i$-th row and moving them diagonally to the southeast into the $(i+1)$-th row. To obtain $w\circ\lbd$, start with $\lbd$ and apply the moves $s_{i_k}$, then $s_{i_{k-1}}$, etc., then $s_{i_1}$. Note that this operation preserves content.

We could also play this game by thinking of $\lbd$ as a 1-row multipartition $\mu(1)$ and obtaining $\mu(w)$ from $\mu(1)$ by doing these diagonal shifting moves in the same manner: if we have a multipartition $\mu$ at any step, applying $s_i$ to it amounts to, if possible, taking the last $\mu^{(i)}-\mu^{(i+1)}+1$ boxes of the $i$-th level and moving them diagonally to the southeast into the $(i+1)$-th level. So $\mu(s_{i_1}\cdots s_{i_k})=(s_{i_1}\cdots s_{i_k})\cdot \mu(1)$. Again, this preserves content. 

More generally, given any transposition $\tau=(i,j)$, we can apply $\tau$ to a 1-row multipartition $\mu$ via diagonal shifting -- if possible, take the last $\mu^{(i)}-\mu^{(j)}+j-i$ boxes of the $i$-th level and move them diagonally to the southeast (i.e. in a content-preserving way) into the $j$-th level. If at any point a diagonal shifting move is not possible, we say $w\circ\lbd=0$, or $\mu(w)=0$. 

Note the combinatorial game at play here is equivalent to the problem of finding the terms in computing the determinant 
\[\det\begin{pmatrix}
    \tsl h_{\lbd_1} & \tsl h_{\lbd_1+1} & \cdots & \tsl h_{\lbd_1+\ell-1}\\
    \tsl h_{\lbd_2-1} &\tsl h_{\lbd_2}&\cdots &\vdots\\
    &&\ddots&\\
    \tsl h_{\lbd_\ell-\ell+1}&\cdots &\tsl h_{\lbd_\ell-1}&\tsl h_{\lbd_\ell}
\end{pmatrix},\]
where we think of the term $\prod \tsl h_{\lbd_i+w(i)-i}$ as corresponding to the multipartition $\mu^{(j)}=\lbd_{w^{-1}(j)}+j-w^{-1}(j)$. In other words, each term of the determinant corresponds to an arrangement of $\ell$ highlighted entries in this matrix, none of which share a row or column with another; the multipartition is the one whose $j$-th level is the subscript on the $j$-th highlighted entry, counted from left to right. 

In computing this determinant, when a subscript in a term is less than 0, that term dies; this corresponds to the fact that the (unique) row in each level of a multipartition must have a nonnegative number of boxes.

As long as $\mu(w)$ is a valid multipartition, any $u\le w$ will also have $\mu(u)$ a valid multipartition. Indeed, as
\[\mu(w)^{(i)}=\lbd_{w^{-1}(i)}+i-w^{-1}(i),\] 
suppose we know that $\mu(w)$ is valid, so that
\[\lbd_{w^{-1}(i)}+i-w^{-1}(i)\ge 0\] 
for each $i$. Let us write $w=u\tau$, or $u^{-1}=\tau w^{-1}$, where $\tau=(a,b)$ is a transposition swapping $a<b$ such that $u(a)<u(b)$ and there is no $a<c<b$ with $u(a)<u(c)<u(b)$. Then $w(a)=u\tau(a)=u(b)>u(a)=u\tau(b)=w(b)$, so that $w(a)\neq a,\ w(b)\neq b$. We would like to show that $\mu(u)^{(i)}=\lbd_{u^{-1}(i)}+i-u^{-1}(i)$ is nonnegative. Note that $u^{-1}(i)=\tau w^{-1}(i)$ is $w^{-1}(i)$ if $w^{-1}(i)\neq a,b$, in which case we are done immediately since then $\mu(w)^{(i)}=\mu(u)^{(i)}$. If $w^{-1}(i)=a$, then $u^{-1}(i)=b$, so that $\lbd_{u^{-1}(i)}+i-u^{-1}(i)=\lbd_b+i-b$. On the other hand, 
\[0\le \lbd_b+w(b)-b<\lbd_b+w(a)-b=\lbd_b+i-b,\] so that indeed
\[\lbd_{u^{-1}(i)}+i-u^{-1}(i)=\lbd_b+i-b\ge 0.\] 
If $w^{-1}(i)=b$, then $u^{-1}(i)=a$, so that $\lbd_{u^{-1}(i)}+i-u^{-1}(i)=\lbd_a+i-a$. On the other hand,
\[0\le \lbd_b+w(b)-b\le \lbd_a+w(b)-a=\lbd_a+i-a,\] 
so that indeed 
\[\lbd_{u^{-1}(i)}+i-u^{-1}(i)=\lbd_a+i-a\ge 0.\]

It is not hard to see that the largest possible $w$ for which $\mu(w)$ is a valid multipartition is the one corresponding to the summand in the Jacobi-Trudi determinant where we select as many $\tsl h_0$'s as possible and select the anti-diagonal on the remaining submatrix. 
\begin{LEM}
    Let $\lbd$ be such that the Jacobi-Trudi matrix for $\lbd$ has 0's (namely $\tsl h_i$ for $i<0$) appearing in the last $a$ rows. In other words, let $a$ be the maximal number such that $\lbd_k<k$ for all $k\ge a$. Then the minimal multipartition appearing in the weight poset of $\DCR_\lbd$ is $\mu(w_\lbd)$, where $w_\lbd$ can be obtained as the permutation corresponding to the term in the Jacobi-Trudi determinant where the 0's in the last $a$ rows are highlighted, and the highlighted entries in the first $\ell-a$ rows correspond to the maximal permutation on $\ell-a$ elements.

    Moreover, all $w\le w_\lbd$ also have $\mu(w)$ appearing in the weight poset for $\DCR_\lbd$. In other words, the weight poset for $\DCR_\lbd$ is 
    \[W_\lbd=\{w\in W:w\le w_\lbd\}.\] 
\end{LEM}

\subsection{A basis statement}
At this point we can make a basis statement which is more or less obvious. 
% \begin{PROP}
%     The Hu-Mathas cellular basis becomes, in the quasi-hereditary quotient algebra $\DCR_\lbd=\CR_\lbd/\wan{y^\nu}_{\nu\te{ multi-row}}$, a basis of elements of form
%     \[\psi(\slie,\tlie)=\psi_\slie e^{\mu(w)}\psi_\tlie^\top,\]
%     where $\slie,\tlie\in\Std\mu(w)$, cellular with respect to the dominance order on the set $\{\mu(w)\}$, which is opposite to the poset $W_\lbd$.
%
%     The algebra $\DCR_\lbd$ also has a triangular basis given by declaring
%     \begin{align*}
%         \tX&=\{\psi_{\slie} 1^w:\slie\in\Std \mu(w)\},\\ 
%         \tH&=\{1^w\},\\
%         \tY&=\{1^w \psi_\tlie^\top:\tlie\in\Std\mu(w)\}.
%     \end{align*}
% \end{PROP}
\begin{PROP}
    The Hu-Mathas cellular basis becomes, in the quasi-hereditary quotient algebra $\DCR_\lbd=\CR_\lbd/\wan{y^\nu}_{\nu\te{ multi-row}}$, a basis %of elements of form
    % \[\psi(\slie,\tlie)=\psi_\slie e^{w}\psi_\tlie^\top,\]
    \[\Big\{\psi(\slie,\tlie)=\psi_\slie e^{w}\psi_\tlie^\top:w\in W_\lbd,\ \slie,\tlie\in\Std\mu(w)\Big\},\]
    % where $\slie,\tlie\in\Std\mu(w)$, 
    cellular with respect to the dominance order on the set $\{\mu(w)\}$, which is opposite to the poset $W_\lbd$.

    The algebra $\DCR_\lbd$ also has a triangular basis given by declaring $\set*{e_\uli:\uli=\cont\tlie\te{ for some }\tlie\in\bigsqcup_{w\in W_\lbd} \Std\mu(w)}$ (whose elements one might simply call $1^\tlie$) to be the general set of idempotents and $\set{e^w:w\in W_\lbd}$ to be the distinguished idempotents (so that $\Phi=\Theta=\Lbd=W_\lbd$) and declaring
    \begin{align*}
        \tX&=\{\psi_{\slie} e^w:\slie\in\Std \mu(w)\},\\ 
        \tH&=\{e^w\},\\
        \tY&=\{e^w \psi_\tlie^\top:\tlie\in\Std\mu(w)\}.
    \end{align*}
\end{PROP}
\vspace{-0.5em}
\begin{PRF}
    Note that our choice of $\tX$ diagrams is exactly the diagrams appearing in the left-most columns of each cell in the cellular picture, i.e. $\psi(\slie,\tlie^w)$.

    The fact that this is a basis is due to the cellular basis on $\DCR$. The only thing remaining to check (up to reflecting) is that $\tX(1^u,1^w)\neq\emptyset\implies u\ge w$. Indeed, assume for the sake of contradiction that we had a diagram $x=1^u x 1^w$ with $u\not\ge w$. Then consider the product $\psi(\tlie^u,\tlie^u)\cdot x=1^u\cdot x=x$. By cellular theory this should lie in the ideal $\DCR^{\ge\mu(u)}\eqqcolon \DCR^{\le u}$. However $x$ lies in the cell corresponding to $w$, which is not $\le u$, contradiction. 
    % The fact that this is a basis is due to the cellular basis on $\DCR$. The only thing remaining to check (up to reflecting) is that $\tX(e^u,e^w)\neq\emptyset\implies u\ge w$. Indeed, assume for the sake of contradiction that we had a diagram $x=e^u x e^w$ with $u\not\ge w$. Then consider the product $\psi(\tlie^u,\tlie^u)\cdot x=e^u\cdot x=x$. By cellular theory this should lie in the ideal $\DCR^{\ge\mu(u)}\eqqcolon \DCR^{\le u}$. However $x$ lies in the cell corresponding to $w$, which is not $\le u$, contradiction. 
\end{PRF}
\begin{remark}
Although we have decided to define $e^{\mu(w)}$ using the standard tableau $\tlie^{\mu(w)}$, i.e. $e^w=e_{\cont\tlie^{\mu(w)}}$, we could have also defined it using the standard tableau $\Plie_{\mu(w)}=\Plie_w$ of \cite{bowman2023klrvssoergel}, which is defined by labeling the multipartition $\mu(w)$ by going down the first column of the first level, then the first column of the second level, and so on until the last level, and then going down the second column of the first level, etc.. As an example, 
\vspace{-1.5em}
\begin{EX*}
    For $\mu=\thbo{\ydia{2,1}}{\ydia{2}}{\ydia{3}}$, we have
$\Plie_{\mu}=\thbo{
\begin{ytableau}
    1 & 5\\
    2
\end{ytableau}\vspace*{4pt}}{
\begin{ytableau}
    3 & 6
\end{ytableau}}{
\begin{ytableau}
    4 & 7 & 8
\end{ytableau}}.$ 
\end{EX*}\vspace{-0.5em}
That is to say, another cellular basis can be given by defining (in this remark $\psi_\slie$ is constructed from the permutation sending $\Plie_{\mu(w)}$ to $\slie$)
\[\psi^\te{BCH}(\slie,\tlie)\coloneqq \psi_\slie e^{\Plie(w)}\psi_\tlie^\top,\] 
cellular with respect to the poset $W_\lbd^\op$, which can be thought of as a triangular basis by letting $\{e^{\Plie(w)}\}$ be the distinguished idempotents and declaring 
\begin{align*}
        \tX^\te{BCH}&=\{\psi_{\slie} e^{\Plie(w)}:\slie\in\Std \mu(w)\},\\ 
        \tH^\te{BCH}&=\{e^{\Plie(w)}\},\\
        \tY^\te{BCH}&=\{e^{\Plie(w)} \psi_\tlie^\top:\tlie\in\Std\mu(w)\}.
    \end{align*}
That one can define the cellular basis this way as well is shown in \cite[Theorem A, Theorem 1.24, Theorem 2.13]{bowman2022lightleaves}; indeed, just as how $\tlie^\mu$ is the maximal tableau with respect to the dominance ordering, $\Plie_\mu$ so defined is the maximal tableau with respect to the ``reverse cylindric ordering'' of \cite[Definition 1.3]{bowman2022lightleaves}. For future reference in Section \ref{sect:JTnilkoszul} we will fix now the notation of
\[e^{\Plie(w)}\coloneqq e^{\Plie_w}\coloneqq e_{\cont \Plie_{\mu(w)}}.\]
\end{remark}
\begin{remark}
    Judging from \cite{zhou2024bgg}, one might have guessed to try to have the $\tY$-diagrams generate a subalgebra of $\DCR_\lbd$ and to hope such a subalgebra is a Koszul nilalgebra. This turns out to be too much; more care is necessary. This will be taken up in Section \ref{sect:JTnilkoszul}.
\end{remark}

\subsection{Stratification}
We shall stratify $\mr\CR=\DCR_\lbd$ by using the system $\{e^w\}$ given above. In the below, note well the difference between $\mr\CR^{\ge w}$ and $\DCR^{\ge\mu(w)}$; the former is our stratification labeled by elements of $W_\lbd$, while the latter is the cellular ideal structure labeled by multipartitions. Let
\[\mr\CR^{\ge w}\coloneqq \mfrac{\mr\CR}{\wan{e^u:u\not\ge w}},\] 
similarly let $\mr\CR^{>w}\coloneqq \mr\CR/\wan{e^u:u\not >w}$, and let the ``Cartan algebras'' be
\[\mr\CR^w\coloneqq e^w \mr\CR^{\ge w} e^w.\] 
Since the only multipartition of shape $\mu(w)$ with the same content as $1^w$ is $\tlie^w$, we actually have
\[\DCR^w\cong \bk e^w\]
as algebras. In particular note that the Cartans are semisimple with exactly one simple module. 

Then we can consider the stratification of $\D^-\Mod\DCR$ by recollements such as
% \[\D^-\Mod \CR^{>w}\lto \D^-\Mod\CR^{\ge w}\lto \D^-\Mod \CR^w.\] 
\begin{center}
\begin{tikzcd}
&  & \perp &  &  &  & \perp &  &  \\[-1em]
\D^-\Mod\DCR^{>{w}} \arrow[rrrr, "\imath_{w}"] &  &       &  & \D^-\Mod\DCR^{\ge{w}} \arrow[rrrr, "\jota^{w}=e^{w} \sq "] \arrow[llll, swap,"\imath_{w}^*=\DCR^{>{w}}\lotimes_{\DCR^{\ge{w}}} \sq", bend right=60] \arrow[bend left=60]{llll}{\imath_{w}^!=\bigoplus_i\rhom_{\DCR^{\ge{w}}}(\DCR^{>{w}}1^i, \sq)} &  &       &  & \D^-\Mod\DCR^{w} \arrow[llll, swap,"\jota^{w}_!=\DCR^{\ge{w}}e^{w}\otimes_{\DCR^{w}} \sq", bend right=60] \arrow[bend left=60]{llll}{\jota^{w}_*=\bigoplus_i\Hom_{\DCR^{w}}(e^{w} \DCR^{\ge{w}}1^i, \sq)} \\[-1em]
  &  & \perp &  & & &\perp &  &
\end{tikzcd}
\end{center}
In this situation the `standardization' functor $\jota^w_!$ is well-behaved.
\begin{LEM}
    $\jota^w_!$ and $\jota^w_*$ are exact functors. %Moreover, letting $L_w(w)=\bk^{\oplus\dim L_w}$ be the unique graded simple of $\CR^w$, $\jota^w_!L_w(w)$ is isomorphic to a permutation module when restricted to $S_n$, corresponding to the composition given by $w\circ\lbd$. 
\end{LEM}
\vspace{-0.5em}
\begin{PRF}
    This is due to Brundan in \cite[Lemma 4.1]{brundan2023graded}.
    % See \href{https://math.stackexchange.com/questions/2275097/the-global-dimension-of-a-k-algebra-depends-only-on-the-simple-modules}{this Stackexchange link} for an argument on why the global dimension of $\CR^w$ depends only upon the projective dimension of its simple module. 
\end{PRF}

Following \cite{brundan2023graded}, we can then construct standard modules, or Verma modules:
\[\Delta_w\coloneqq \jmath^w_! \bk\cong \DCR^{\ge w}e^w.\] 
Note that since $\DCR^w\cong\bk$ is semisimple here, there is no distinction between big and small Vermas. 

It turns out that there are several incarnations of this Verma module.
% It turns out that there are several incarnations of $\Delta_w$:
\begin{PROP}
    The following modules are all isomorphic:
    \begin{itemize}
        \item $\Delta_w$, as constructed above;
        \item $C_w$, the cell module labeled by $\mu(w)$ (here we choose the cell module containing $1^w$; the rest are also isomorphic but up to shift);
        \item $\Ind_{\CR_{\cont(\mu(w)^{(1)})}\otimes\cdots\otimes \CR_{\cont(\mu(w)^{(\ell(\lbd))})}}^{\CR_\alpha}(\on{triv}\otimes\cdots\otimes\on{triv})$;
        \item $\CF^\te{AS}_\lbd(\Delta_{w^{-1}\circ 0})$, where $\Delta_{w^{-1}\circ 0}$ refers to the Verma module of category $\CO$.
    \end{itemize}
    Here $\CF^\te{AS}_\lbd$ is the Arakawa-Suzuki functor of \cite{arakawa1998duality}, which sends $\CF^\te{AS}_\lbd\colon \CO\lto \Mod \wh\CH_\alpha^\omega$; we consider the KLR algebra $\CR_\alpha^\omega$ to act on this via the Brundan-Kleshchev isomorphism, which we omit from the notation.

    In particular, we have that
    \[\Res^{\DCR_\lbd}_{S_n} \Delta_w\cong E_{w\circ\lbd},\] 
    where the restriction is along the map $\vphi_\lbd^\te{BK}\colon \BC S_n\lto \DCR_\lbd$ and $E_{w\circ\lbd}$ is the permutation module. 
\end{PROP}
\vspace{-0.5em}
\begin{PRF}
    That the second and third incarnations are equivalent is \cite[Corollary 3.6.3]{mathas2015cyclotomic}. Note that because each $\lbd^{(i)}$ is just a row in this case, the degree shifts appearing in Mathas's statement does not appear here. By the same cited corollary, the action factors through $\CR_\alpha^\omega$. 

    That the third and fourth incarnations are equivalent is due to e.g. Arakawa-Suzuki \cite[Theorem 3.3.1]{arakawa1998duality} and Orellana-Ram \cite[Proposition 6.27]{orellana2004affine}.
    %, specifically Proposition 6.27 of \cite{orellana2004affine} and Theorem 3.3.1 of \cite{arakawa1998duality}. 

    That the first and second incarnations are equivalent is self-evident: $C_w$ is the module obtained by taking the vector $1^w$, acting on it freely by $\DCR$ (cellular theory ensures that the resulting module stays within the cell ideal $\DCR^{\ge \mu(w)}$), and then modding out by $\DCR^{>\mu(w)}$. In comparison, $\Delta_w=\jota^w_! \bk$ is the module obtained by taking the vector $1^w$, acting on it by $\DCR$ (by cellular theory the result must be inside the cell ideal $\DCR^{\ge\mu(w)}$), and then modding out by $\{e^u:u\not\ge w\}$. Note $u\not\ge w$ is the same as $\mu(u)\not\le \mu(w)$; since $\DCR \cdot 1^w$ must be contained in $\DCR^{\ge \mu(w)}$, we know it suffices to mod out by $\{e^u: \mu(u)>\mu(w)\}$. Hence the resulting module is precisely the same as $C_w$. 

    One might note that the third and fourth incarnations so far only have actions of $\CR_\alpha^\omega$, but this action must factor through $\DCR_\lbd$ because they are both isomorphic to the second incarnation, namely the cell module, whose action clearly factors through $\DCR_\lbd$. 
\end{PRF}

% At this point we can already see that these Verma modules for $\DCR$ behave suspiciously like Verma modules for category $\CO$, and indeed in fact $\DCR$ and category $\CO$ are `basically the same':
Since Verma modules of category $\CO$ are related to Vermas modules over $\DCR$ via the Arakawa-Suzuki functor, one might guess already the following:
\begin{THM}\label{thm:JTvsO}
    $\Mod\DCR_\lbd$ is equivalent to a quotient of category $\CO$ for $\sl_{\ell(\lbd)}$, namely the quotient obtained by quotienting out by the subcategory generated by $\Delta_{w^{-1}\circ 0}$ for $w\not\le w_\lbd$. 
\end{THM}
\vspace{-0.5em}
\begin{PRF}
    This is essentially due to \cite{fujita2018tilting}. %cf also \href{https://gauss.math.yale.edu/~il282/hw_problems.pdf}{Losev}
    Let $\CO_\lbd$ be the quotient of category $\CO$ where we quotient out by the subcategory generated by the Vermas $\Delta_{u^{-1}\circ 0}$ for $u\not\le w_\lbd$ (note that the set of such $u$ forms an ideal, in the sense that any $v$ such that $v\ge u$ also has $v\not\le w_\lbd$). This is a highest weight category whose Vermas are the images of $\Delta_{w^{-1}\circ 0}$ for $w\le w_\lbd$; in other words, the weight poset for this highest weight category is isomorphic to $W_\lbd$. 
    
    The Arakawa-Suzuki functor $\CF_\lbd^\te{AS}$ factors through $\CO_\lbd$ because it kills the Vermas we quotiented out by, see \cite[Theorem 3.3.1]{orellana2004affine}; so we can think of it as
    \[\CF_\lbd^\te{AS}\colon \CO_\lbd\lto \Mod \DCR_\lbd.\] 
    It sends (the image under the quotient of) the Verma $\Delta_{w^{-1}\circ 0}$ to the Verma $\Delta_w$, so that this functor relabels Vermas by the bijection $w^{-1}\lmto w$, which is order-preserving. By  \cite[Theorem 3.9]{fujita2018tilting}, $\CF_\lbd^\te{AS}$ must be a equivalence of highest-weight categories.
    % 
    % \warn{cite Fujita, it sends Vermas to Vermas and preserves the poset}
\end{PRF}

\subsection{Reconstruction}
The goal is to invoke the general reconstruction theorem \ref{thm:triangularfiltration}. 
% \begin{THM}\label{thm:triangularfiltration}
% Let $\Theta$ be upper-eventually disjoint-totally-ordered (which means that for any $\theta$, there exists $\phi\ge\theta$ such that $\{\psi:\psi\ge\phi\}$ is totally ordered). There exists a $\Theta$-filtration on the identity functor $\Id_{\D^-\Mod A}$ whose terms are $\calid(\theta)=\ol\imath_\theta\ol\imath_\theta^*$ and whose associated graded is
% \begin{align*}
%     \gr^\theta\calid=\ol\imath_\theta\jota^\theta_!\jota^\theta\ol\imath_\theta^*=\Delta(\theta)\otimes_{A^\theta}\rhom_A\pr*{\Delta(\theta),\sq^\dag}^*,
% \end{align*}
% where we recall the shorthand $\Delta(\theta)\coloneqq \bigoplus_{\lbd\in\theta}\Delta_\lbd^{\ol{l_\lbd(\theta)}}$. Moreover, provided there is a length function $\ell\colon \Theta\lto\BZ$ (satisfying that $\ell(\Theta)$ has a minimal element), there is a spectral sequence (functorial in the input $\sq$)
% \[E_1^{p,q}=\bigoplus_{\ell(\theta)=-p}\Delta(\theta)\otimes_{A^\theta}\Ext^{-(p+q)}_A(\Delta(\theta),\sq^\dag)^* \specseqimplies E_\infty^{p,q}=\gr^{-p}H^{p+q}(\sq).\]
% \end{THM}
In the present case, $\Theta=W_\lbd$. The spectral sequence says
\[E_1^{p,q}=\bigoplus_{\ell(w)=-p}\Delta_w\otimes_{\bk e^w}\Ext^{-(p+q)}_{\DCR_\lbd}(\Delta_w,\sq^\dag)^\dag \specseqimplies E_\infty^{p,q}=\gr^{-p}H^{p+q}(\sq),\]
which we wish to apply to $L_1$. In order to calculate the terms of $E_1$, we wish to compute the Ext groups 
\[\Ext_{\DCR}^\blt(\Delta_w,L_1)=\Ext_{\DCR}^\blt(\DCR^{\ge w} e^w,L_1).\] 
By the equivalence of Theorem \ref{thm:JTvsO}, these Ext groups are implicitly the same as those in category $\CO$, so in principle we can just invoke Kostant's theorem on Lie algebra cohomology. However we wish to give a diagrammatic proof instead. For this we might invoke Koszul theory on some ``lower-half nilalgebra'' $\DCR^-$ which should satisfy something like $\hom_\DCR(\DCR^{\ge w}e^w,\sq)\overset{?}{=}\hom_{\DCR^-}(\bk e^w,\sq)$; rather than doing this directly (though eventually we do, in Section \ref{sect:JTnilkoszul}), we will instead pass through the world of diagrammatic Soergel calculus. 

\begin{remark}
    The reason we make the comparison to category $\CO$ here is so that we can pass to Soergel calculus, which is also equivalent to category $\CO$, in the next section. However, if one were to dislike arguments passing through category $\CO$, the technology developed in \cite{bowman2023klrvssoergel} allows us to skip category $\CO$ and directly give a Morita equivalence between $\DCR_\lbd$ and a Soergel calculus; this point of view is expanded upon in Section \ref{sect:JTnilkoszul}.
\end{remark}

% It turns out that for technical reasons the correct thing for $\CR^-$ to satisfy is
% \[\hom_\CR(\CR^{\ge w}e^w,\sq)=\hom_{\CR^-}(\bk 1^w,\sq).\] 

% % \section{The Koszul nilalgebra} 
% To power our reconstruction machine, we need again some homological input which will be provided by Koszul theory.

\section{(Cyclotomic) Soergel calculus}\label{sect:soergel}
In this section we give a diagrammatic proof of Kostant's theorem on Lie algebra cohomology, stating that (in the convention where $L_1$ is the finite-dimensional simple)
\[H^\blt(\nlie^+:L_1)=\bigoplus_w \BC_w[-\ell(w)],\] 
where the subscript of each $\BC$ indicates the $\hlie$-action. Equivalently this is saying that
\[\Ext^\blt_\CO(\Delta_w,L_1)=\BC[-\ell(w)],\] 
and yet another equivalent way (assuming the Kazhdan-Lusztig conjecture) to state this theorem is that the Kazhdan-Lusztig polynomial (under the conventions of Wikipedia) $P_{w,w_0}(q)$ is identically $1$ for all $w\in W$.

To give a diagrammatic proof of this, we will use the fact that
\[\CO\simeq \Mod\pr*{\BC\otimes_R\End\pr*{\bigoplus_{\ulx\in S} \on{BS}_{\ulx}}},\]
where the tensor with $\BC$ indicates we are sending barbells on the far left to zero, and $\ulx$ are some (not necessarily reduced) words in some set $S$. There are various choices for $S$ given Morita equivalent algebras -- for instance one can choose $S$ to be one fixed reduced word per $w\in W$ across all $w$; or all reduced words for all $w\in W$; or, in our case, we will choose $S$ to be the set of all words $\ulx$ (reduced or not) of length at most $\binom{\ell}{2}$ such that $\ulx$ is a subword of some reduced word for the longest element $w_0$, where $\ell-1$ is the rank of the type $A$ category $\CO$ we consider. For brevity we will denote 
\[\CS\coloneqq\BC\otimes_R \End\pr*{\bigoplus_{\ulx\in S} \on{BS}_{\ulx}},\]
and we will think of this algebra diagrammatically. The fact that we are tensoring with $\BC$ is indicated by the name of `cyclotomic' Soergel calculus, taken from \cite{bowman2023klrvssoergel}. 

We have seen earlier that the Jacobi-Trudi algebra $\DCR_\lbd$ we are interested in is equivalent to a quotient of category $\CO$; in order to find a Soergel calculus 
\[\CS_\lbd=\BC\otimes_R\End\pr*{ \bigoplus_{\ulx\in \ul W_\lbd} \on{BS}_{\ulx}}\] 
which is equivalent to $\DCR_\lbd$, we will just need to shrink our set $\ul W_\lbd$ slightly. In particular we will let
\[\ul W_\lbd=\{\ulx:\ell(\ulx)\le\ell(w_\lbd),\ \ulx\te{ is a subword of some (not fixed) reduced word for }w_\lbd\}.\]
Note that $\CS_\lbd$ is a subalgebra (in fact, an idempotent truncation) of the full $\CS$ above. Note also that for fixed $\ell=\ell(\lbd)$, a sufficiently large $\lbd_{\ell}$ will make $w_\lbd=w_0$, in which case we recover $\CS$. Also note that this algebra depends only on $\ul W_\lbd$ -- for example, $\CS_\yd{1,1}\cong\CS_\yd{2,1}$. 

For each $w\in S_n$ we will fix a preferred reduced expression, which we also denote by $w$; for instance we can pick the lexicographically minimal expression, which satisfies that all substrings $s_is_js_i$ for $|i-j|=1$ have $j=i+1$. 

In what follows, we will often simply abbreviate $\CS_\lbd$ to $\CS$, just as we have often abbreviated $\DCR_\lbd$ to $\DCR$, keeping in mind that there is an implicit dependence on $\lbd$. Another reason we choose to suppress the $\lbd$ here because these statements hold more generally for any cell ideal $\CS$ of any cyclotomic Soergel calculus, of any type and any rank; though we focus here on type $A$, the arguments work in any type. More precisely, given a lower-ideal $W_\CS\subseteq W$ of a Coxeter group of any type, we let
\[\ul W_\CS=\{\ulx:\ulx\te{ is a subword of some reduced word for some element in }W_\CS\},\]
and we let
\[\CS=\BC\otimes_R\End \pr*{ \bigoplus_{\ulx\in \ul W_\CS} \on{BS}_{\ulx}}.\]

\subsection{Defining the diagrammatic category (in type $A$)}
Let us briefly recall how the diagrammatic Soergel calculus of type $A$ is defined. We will state some relations (in our view, the most important ones) only as a reminder -- please refer to \cite[Section 3.1]{bowman2023klrvssoergel} for a full account in type $A$ and \cite{elias2016soergel} for a full account in general. This diagrammatic calculus is generated monoidally by
% \[\begin{diagram}
%     \draw (0,0)--(0,3);
% \end{diagram}
% \ ,\ 
% \begin{diagram}
%     \draw (0,0)--(0,1.5);
%     \fill (0,1.5) circle (5pt);
% \end{diagram}
% \ ,\ 
% \begin{diagram}
%     \draw (0,0)--(1,1.5)--(1,3);
%     \draw (2,0)--(1,1.5);
% \end{diagram}
% \ ,\ 
% \begin{diagram}
%     \draw(0,0)--(1,1.5)--(1,3);
%     \draw (2,0)--(1,1.5);
%     \draw[red] (1,0)--(1,1.5);
%     \draw[red] (0,3)--(1,1.5)--(2,3);
% \end{diagram}
% \]
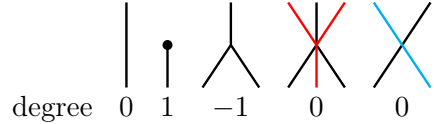
\begin{center}
\begin{tabular}{ c c c c c c }
 & $\begin{diagram}
    \draw (0,0)--(0,3);
\end{diagram}$ & $\begin{diagram}[0pt]
    \draw (0,0)--(0,1.5);
    \fill (0,1.5) circle (5pt);
    \draw[white] (0,3)--(0.1,3);
\end{diagram}$ & 
$\begin{diagram}
    \draw (0,0)--(1,1.5)--(1,3);
    \draw (2,0)--(1,1.5);
\end{diagram}$ & 
$\begin{diagram}
    \draw(0,0)--(1,1.5)--(1,3);
    \draw (2,0)--(1,1.5);
    \draw[red] (1,0)--(1,1.5);
    \draw[red] (0,3)--(1,1.5)--(2,3);
\end{diagram}$ & 
$\begin{diagram}
    \draw(0,0)--(2,3);
    \draw[cyan](2,0)--(0,3);
\end{diagram}$
\\[1em]
    degree & 0 & $1$ & $-1$ & $0$ & $0$
\end{tabular}
\end{center}
as well as the upside-down flip of these diagrams. Here red and black are adjacent colors while blue and black are distant colors\footnote{Our philosophy will be to order colors from least energetic to most energetic. So the first color is black, next color is red, etc.. In particular blue will be considered distant to black.}.

The following relations hold under rotations and reflections. Let us write down foremost the important 1-color relations.
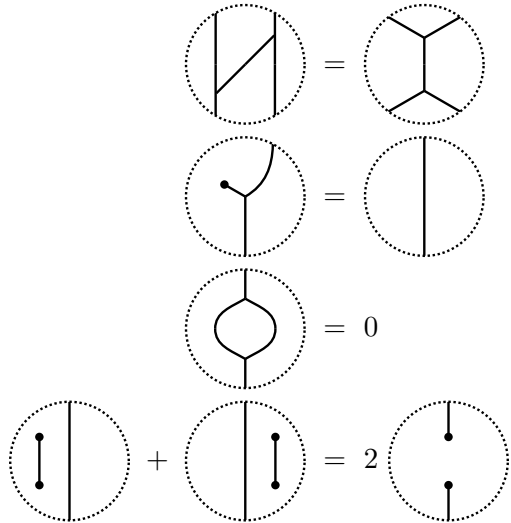
\begin{align*}
    \begin{diagram}
    \begin{scope}[xscale=0.85,yscale=0.85]
        \draw[densely dotted](0,0)  circle (70pt);
\clip(0,0) circle (70pt);
 \draw(0,0)--++(45:1.7);
 \draw(0,0)--++(-135:1.7);
 \path(0,0)--++(180:35pt) coordinate (L);
  \path(0,0)--++(0:35pt) coordinate (R);
   \draw(L)--++(-90:3) ;
      \draw(L)--++(90:3) ;
   \draw(R)--++(-90:3) ;
      \draw(R)--++(90:3) ;
      \end{scope}
    \end{diagram}
    \ &=\ 
    \begin{diagram}
    \begin{scope}[xscale=0.85,yscale=0.85]
        \draw[densely dotted](0,0)  circle (70pt);
        \clip(0,0) circle (70pt);
        \draw(0,0)--++(90:1.1) ;
        \path(0,0)--++(90:1.1) coordinate (U);
          \draw(0,0)--++(-90:1.1) ;
          \path(0,0)--++(-90:1.1) coordinate (D);
          \draw (U)--++(30:2);
          \draw (U)--++(150:2);
          \draw (D)--++(-30:2);
          \draw (D)--++(-150:2);
          \end{scope}
    \end{diagram}\\ 
    \begin{diagram}
    \begin{scope}[xscale=0.85,yscale=0.85]
        \draw[densely dotted](0,0)  circle (70pt);
        \clip(0,0) circle (70pt);
        % \draw (0,0)--++(30:1);
        % \path (0,0)--++(30:1) coordinate (A);
        % \draw (A)--++(90:3);
        \draw (0,0) to[out=30,in=-90] (1,5);
        \draw (0,0)--++(-90:3);
        \draw (0,0)--++(150:1);
        \path (0,0)--++(150:1) coordinate (B);
        \fill (B) circle (5pt);
        \end{scope}
    \end{diagram}
    \ &=\ 
    \begin{diagram}
    \begin{scope}[xscale=0.85,yscale=0.85]
        \draw[densely dotted](0,0)  circle (70pt);
        \clip(0,0) circle (70pt);
        \draw (0,70pt)--(0,-70pt);
        \end{scope}
    \end{diagram}
    \\
    \begin{diagram}
    \begin{scope}[xscale=0.85,yscale=0.85]
        \draw[densely dotted](0,0)  circle (70pt);
        \clip(0,0) circle (70pt);
        \draw (0,1.25)--(0,3);
        \draw (0,-1.25)--(0,-3);
        \draw (0,1.25) to[out=-30,in=90] (1.25,0);
        \draw (1.25,0) to[out=-90,in=30] (0,-1.25);
        \draw (0,1.25) to[out=-150,in=90] (-1.25,0);
        \draw (-1.25,0) to[out=-90,in=150] (0,-1.25);
        \end{scope}
    \end{diagram}
    \ &=\ 
    0
    \\
    \begin{diagram}
    \begin{scope}[xscale=0.85,yscale=0.85]
        \draw[densely dotted](0,0)  circle (70pt);
        \clip(0,0) circle (70pt);
        \draw (0,70pt)--(0,-70pt);
        \draw (-1.25,-1)--(-1.25,1);
        \fill (-1.25,-1) circle (5pt);
        \fill (-1.25,1) circle (5pt);
        \end{scope}
    \end{diagram}
    \ +\ 
    \begin{diagram}
    \begin{scope}[xscale=0.85,yscale=0.85]
        \draw[densely dotted](0,0)  circle (70pt);
        \clip(0,0) circle (70pt);
        \draw (0,70pt)--(0,-70pt);
        \draw (1.25,-1)--(1.25,1);
        \fill (1.25,-1) circle (5pt);
        \fill (1.25,1) circle (5pt);
        \end{scope}
    \end{diagram}
    \ &=\ 
    2\ 
    \begin{diagram}
    \begin{scope}[xscale=0.85,yscale=0.85]
        \draw[densely dotted](0,0)  circle (70pt);
        \clip(0,0) circle (70pt);
        \draw (0,70pt)--(0,1);
        \fill (0,1) circle (5pt);
        \draw (0,-70pt)--(0,-1);
        \fill (0,-1) circle (5pt);
        \end{scope}
    \end{diagram}
\end{align*}

\noindent The adjacent 2-color relations are
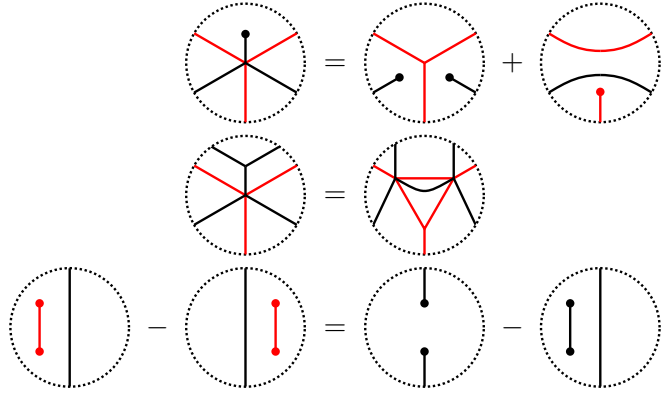
\begin{align*}
    \begin{diagram}
    \begin{scope}[xscale=0.85,yscale=0.85]
        \draw[densely dotted](0,0)  circle (70pt);
        \clip(0,0) circle (70pt);
        \draw[red] (0,0)--++(-90:3);
        \draw[red](0,0)--++(30:3);
        \draw[red](0,0)--++(150:3);
        \draw(0,0)--++(90:1.2);
        \fill (0,1.2) circle (5pt);
        \draw (0,0)--++(-30:3);
        \draw (0,0)--++(-150:3);
        \end{scope}
    \end{diagram}
    \ &=\ 
    \begin{diagram}
    \begin{scope}[xscale=0.85,yscale=0.85]
        \draw[densely dotted](0,0)  circle (70pt);
        \clip(0,0) circle (70pt);
        \draw[red] (0,0)--++(-90:3);
        \draw[red](0,0)--++(30:3);
        \draw[red](0,0)--++(150:3);
        \path(0,0)--++(-30:3) coordinate (B);
        \path(0,0)--++(-150:3) coordinate (A);
        \draw (A)--++(30:1.8);
        \path (A)--++(30:1.8) coordinate (C);
        \draw (B)--++(150:1.8);
        \path (B)--++(150:1.8) coordinate (D);
        \fill (C) circle (5pt);
        \fill (D) circle (5pt);
        \end{scope}
    \end{diagram}
    \ +\ 
    \begin{diagram}
    \begin{scope}[xscale=0.85,yscale=0.85]
        \draw[densely dotted](0,0)  circle (70pt);
        \clip(0,0) circle (70pt);
        \path(0,0)--++(30:3) coordinate (B);
        \path(0,0)--++(150:3) coordinate (A);
        \path(0,0)--++(-150:3) coordinate (C);
        \path(0,0)--++(-30:3) coordinate (D);
        \path(0,0)--++(-90:3) coordinate (E);
        \draw[red] (A) to[out=-30,in=180] (0,0.5);
        \draw[red] (0,0.5) to[out=0,in=-150] (B);
        \draw[red] (E)--++(90:1.8);
        \path(E)--++(90:1.8) coordinate (F);
        \fill[red] (F) circle (5pt);
        \draw (C) to[out=30,in=180] (0,-0.5);
        \draw (0,-0.5) to[out=0,in=150] (D);
        \end{scope}
    \end{diagram}
    \\
    \begin{diagram}
    \begin{scope}[xscale=0.85,yscale=0.85]
        \draw[densely dotted](0,0)  circle (70pt);
        \clip(0,0) circle (70pt);
        \draw[red] (0,0)--++(-90:3);
        \draw[red](0,0)--++(30:3);
        \draw[red](0,0)--++(150:3);
        \draw (0,0)--++(-30:3);
        \draw (0,0)--++(-150:3);
        \draw (0,0)--(0,1.2);
        \draw (0,1.2)--++(30:2);
        \draw (0,1.2)--++(150:2);
        \end{scope}
    \end{diagram}
    \ &=\ 
    \begin{diagram}
    \begin{scope}[xscale=0.85,yscale=0.85]
        \draw[densely dotted](0,0)  circle (70pt);
        \clip(0,0) circle (70pt);
        \path (0,0)--++(-90:3)coordinate (C);
        \path(0,0)--++(30:3)coordinate (B);
        \path(0,0)--++(150:3)coordinate (A);
        \path (0,0)--++(-90:1.4)coordinate (C');
        \path(0,0)--++(30:1.4)coordinate (B');
        \path(0,0)--++(150:1.4)coordinate (A');
        \draw[red] (A)--(A');
        \draw[red] (B)--(B');
        \draw[red] (C)--(C');
        \draw[red] (A')--(B')--(C')--(A');
        \path (0,0)--++(-150:2.45) coordinate (D);
        \path (0,0)--++(-30:2.45) coordinate (E);
        \draw(A')--++(90:2);
        \draw(B')--++(90:2);
        \draw (D)--(A');
        \draw (E)--(B');
        \draw (A') ..controls(0,0).. (B');
        \end{scope}
    \end{diagram}
    \\
    \begin{diagram}
    \begin{scope}[xscale=0.85,yscale=0.85]
        \draw[densely dotted](0,0)  circle (70pt);
        \clip(0,0) circle (70pt);
        \draw (0,70pt)--(0,-70pt);
        \draw[red] (-1.25,-1)--(-1.25,1);
        \fill[red] (-1.25,-1) circle (5pt);
        \fill[red] (-1.25,1) circle (5pt);
        \end{scope}
    \end{diagram}
    \ -\ 
    \begin{diagram}
    \begin{scope}[xscale=0.85,yscale=0.85]
        \draw[densely dotted](0,0)  circle (70pt);
        \clip(0,0) circle (70pt);
        \draw (0,70pt)--(0,-70pt);
        \draw[red] (1.25,-1)--(1.25,1);
        \fill[red] (1.25,-1) circle (5pt);
        \fill[red] (1.25,1) circle (5pt);
        \end{scope}
    \end{diagram}
    \ &=\ 
    \begin{diagram}
    \begin{scope}[xscale=0.85,yscale=0.85]
        \draw[densely dotted](0,0)  circle (70pt);
        \clip(0,0) circle (70pt);
        \draw (0,70pt)--(0,1);
        \fill (0,1) circle (5pt);
        \draw (0,-70pt)--(0,-1);
        \fill (0,-1) circle (5pt);
        \end{scope}
    \end{diagram}
    \ -\ 
    \begin{diagram}
    \begin{scope}[xscale=0.85,yscale=0.85]
        \draw[densely dotted](0,0)  circle (70pt);
        \clip(0,0) circle (70pt);
        \draw (0,70pt)--(0,-70pt);
        \draw (-1.25,-1)--(-1.25,1);
        \fill (-1.25,-1) circle (5pt);
        \fill (-1.25,1) circle (5pt);
        \end{scope}
    \end{diagram}
\end{align*}
One would of course be remiss not to write down the most famous consequence of the above relations:
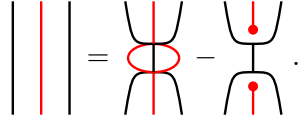
\[\begin{diagram}
    \draw(0,0)--(0,4);
    \draw[red](1,0)--(1,4);
    \draw(2,0)--(2,4);
\end{diagram} 
\;=\:
\begin{diagram}
    \draw (0,0)..controls(0.25,1.5)..(1,1.5);
    \draw (2,0)..controls(1.75,1.5)..(1,1.5);
    \draw[red] (1,2) circle [x radius=0.9, y radius=0.5];
    \draw (0,4)..controls(0.25,2.5)..(1,2.5);
    \draw (2,4)..controls(1.75,2.5)..(1,2.5);
    \draw[red](1,0)--(1,1.5);
    \draw[red](1,2.5)--(1,4);
    \draw (1,1.5)--(1,2.5);
\end{diagram}
\,-
\begin{diagram}
    \draw (0,0)..controls(0.25,1.5)..(1,1.5);
    \draw (2,0)..controls(1.75,1.5)..(1,1.5);
    \draw (0,4)..controls(0.25,2.5)..(1,2.5);
    \draw (2,4)..controls(1.75,2.5)..(1,2.5);
    \draw (1,1.5)--(1,2.5);
    \draw[red](1,0)--(1,1);
    \fill[red] (1,1) circle (5pt);
    \draw[red](1,4)--(1,3);
    \fill[red] (1,3) circle (5pt);
\end{diagram}\ .
\]

% For distant colors, the diagrams just pull apart (see \cite[Definition 3.6]{bowman2023klrvssoergel} for details). 
Strings of distant colors do not interact with each other at all, i.e., they can pass through each other with no error terms (see \cite[Definition 3.6]{bowman2023klrvssoergel} for details). 

% \warn{Write a little bit more about this? Should I do the massive Z. diagram?}

Lastly there is a tetrahedral relation (also known as the Zamolodchikov relation) for three colors \begin{tikzcd}[column sep=small]
  1 \arrow[dash]{r} & \textcolor{red}{2} \arrow[dash]{r} & \textcolor{cyan}{3}
\end{tikzcd} which we reproduce below; see \cite[Definition 3.6]{bowman2023klrvssoergel} for details.
\[\begin{diagram}
\begin{scope}[xscale=1.1,yscale=1.1]
    \draw[densely dotted](0,0)  circle (70pt);
    \clip(0,0) circle (70pt);
    \path(0,0)--++(45:1.4)coordinate (A');
    \path(0,0)--++(135:1.4)coordinate (B');
    \path(0,0)--++(225:1.4)coordinate (C');
    \path(0,0)--++(315:1.4)coordinate (D');
    \path(0,0)--++(45:3)coordinate (A);
    \path(0,0)--++(135:3)coordinate (B);
    \path(0,0)--++(225:3)coordinate (C);
    \path(0,0)--++(315:3)coordinate (D);
    \draw[red](A)--(A');
    \draw[red](B)--(B');
    \draw[red](C)--(C');
    \draw[red](D)--(D');
    \draw[red](A')--(B')--(C')--(D')--(A');
    \path(A')--++(0:1.3)coordinate (A'');
    \path(B')--++(180:1.3)coordinate (B'');
    \path(C')--++(180:1.3)coordinate (C'');
    \path(D')--++(0:1.3)coordinate (D'');
    \draw(C'')--(B')--(D')--(A'');
    \draw(D')--++(270:1.3);
    \draw(B')--++(90:1.3);
    \draw[cyan](B'')--(C')--(A')--(D'');
    \draw[cyan](A')--++(90:1.3);
    \draw[cyan](C')--++(270:1.3);
    \end{scope}
\end{diagram}
\ =\ 
\begin{diagram}
\begin{scope}[xscale=1.1,yscale=1.1]
    \draw[densely dotted](0,0)  circle (70pt);
    \clip(0,0) circle (70pt);
    \path(0,0)--++(45:1.4)coordinate (A');
    \path(0,0)--++(135:1.4)coordinate (B');
    \path(0,0)--++(225:1.4)coordinate (C');
    \path(0,0)--++(315:1.4)coordinate (D');
    \path(0,0)--++(45:3)coordinate (A);
    \path(0,0)--++(135:3)coordinate (B);
    \path(0,0)--++(225:3)coordinate (C);
    \path(0,0)--++(315:3)coordinate (D);
    \draw[red](A)--(A');
    \draw[red](B)--(B');
    \draw[red](C)--(C');
    \draw[red](D)--(D');
    \draw[red](A')--(B')--(C')--(D')--(A');
    \path(A')--++(90:1.3)coordinate (A'');
    \path(B')--++(90:1.3)coordinate (B'');
    \path(C')--++(270:1.3)coordinate (C'');
    \path(D')--++(270:1.3)coordinate (D'');
    \draw[cyan](C'')--(D')--(B')--(A'');
    \draw[cyan](D')--++(0:1.3);
    \draw[cyan](B')--++(180:1.3);
    \draw(B'')--(A')--(C')--(D'');
    \draw(A')--++(0:1.3);
    \draw(C')--++(180:1.3);
    \end{scope}
\end{diagram}
\]

\subsection{Cyclotomic Soergel}
There is an additional ``cyclotomic condition'' we may impose, corresponding to the $\BC\otimes_R\sq$ in the definition of $\CS_\lbd$ at the beginning of this section. Diagrammatically this is declaring that any diagram in which a barbell appears at the far left is zero, i.e.
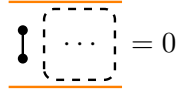
\[\begin{diagram}
        \draw[orange](0,3)--(4,3);
        \draw[orange](0,0)--(4,0);
        \draw (0.5,1)--(0.5,2);
        \fill (0.5,1) circle (5pt);
        \fill (0.5,2) circle (5pt);
        \node at (2.6,1.5) {$\cdots$};
        \draw[dashed, rounded corners] (1.25, 0.25) rectangle (3.75, 2.75);
    \end{diagram}=0\]

\subsection{Some structure}
% Let us briefly recall the double light leaves basis of Soergel calculus. This basis is very subtle and involves many choices, so our account here will be very rudimentary and will probably only serve as a reminder for someone who has already seen this before -- we refer the reader to \cite{elias2016soergel} for a full account of the details.
Let us briefly recall the double light leaves basis of Soergel calculus, first described by \cite{libedinsky2015light}. This basis is very subtle and involves many choices; see \cite{elias2016soergel} for a full account of the details.

After fixing some data, e.g. a preferred redex for each $w\in S_n$ as well as a rex move from any other redex to it, one can construct explicitly a ``light leaf morphism'' $\LL_{\ulx}^{\ulx^\alpha}$ per expression $\ulx$ (not necessarily reduced, of length $k$) and a sequence $\bfe$ of 1's and 0's such that $\ulx^\bfe=x_1^{\bfe_1}\cdots x_k^{\bfe_k}=\ulw$ is some redex for $w$. Such light leaves go from $\ulx$ to $\ulw$. The point is that, if $\uly^{\beta}=\ulx^{\bfe}$ and $\olLL^{\uly}_{\uly^{\beta}}$ is the upside-down light leaf, then 
\[\olLL^{\uly}_{\uly^{\beta}}\LL_{\ulx}^{\ulx^\bfe} \]
gives a basis of $\CS$. This is the `double light leaf basis'. 

Now let us briefly describe how to build a light leaf. From $\bfe$ and $\ulx$ we can construct via the following algorithm the ``Deodhar word'' $D(\bfe)=D(\bfe,\ulx)$, which is a word in the alphabet $\tU_1,\tU_0,\tD_1,\tD_0$:
\begin{enumerate}
    \item scan the word $\ulx$ from left to right, and define a string $(\ulx)_{\te{step }0}=\emp$;
    \item at step $i+1$, remember the word $(\ulx)_{\te{step }i}$ from the $i$-th step;
    \item if adding the $(i+1)$-th letter $(\ulx)_{i+1}$ in $\ulx$ to $(\ulx)_{\te{step }i}$ increases the length, then the new $(i+1)$-th letter to add to the Deodhar word is $\tU$; if the length goes down, then the new $(i+1)$-th letter to add is $\tD$;
    \item if $\bfe_{i+1}=1$, then the subscript of the $(i+1)$-th letter is 1, and we add the letter $(\ulx)_{i+1}$ to obtain $(\ulx)_{\te{step }i+1}$; if $\bfe_{i+1}=0$, then the subscript of the $(i+1)$-th letter is 0, and we keep $(\ulx)_{\te{step }i}$ unchanged to obtain $(\ulx)_{\te{step }i+1}$;
    \item continue scanning. 
\end{enumerate}
Then, from this Deodhar word, we recursively construct a light leaf as follows: we build the light leaf recursively by starting at the bottom left; depending on which symbol the $i$-th symbol in $D(\bfe)$ is, we add the following constructions:
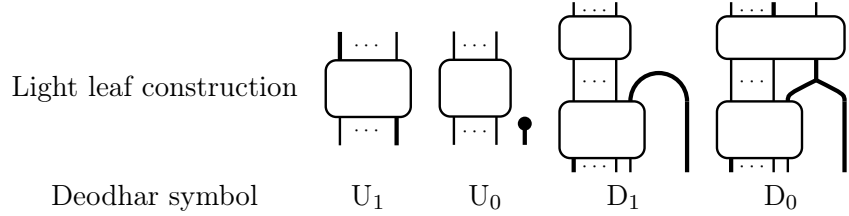
\begin{center}
\begin{tabular}{ c c c c c }
 Light leaf construction & $\begin{diagram}
    \draw (1,0)--(1,1);
    \node at (2,0.5) {\tiny$\cdots$}; 
    \draw[ultra thick] (3,0)--(3,1);
    \draw[rounded corners] (0.5,1) rectangle(3.5,3);
    \draw[ultra thick](1,3)--(1,4);
    \node at (2,3.5) {\tiny$\cdots$}; 
    \draw(3,3)--(3,4);
\end{diagram}$ & $\begin{diagram}
    \draw[rounded corners] (0.5, 0.5) rectangle (3, 2.5);
    \draw (1,-0.5)--(1,0.5); 
    \node at (1.8,-0.05) {\tiny$\cdots$};
    \draw(2.5,-0.5)--(2.5,0.5);
    \draw(1,2.5)--(1,3.5);
    \node at (1.8,3) {\tiny$\cdots$};
    \draw(2.5,2.5)--(2.5,3.5);
    \draw[ultra thick] (3.5,-0.5)--(3.5,0.25);
    \fill (3.5,0.25) circle (7pt);
\end{diagram}$ & 
$\begin{diagram}
    % \draw[orange](0,3)--(4,3);
    % \draw[orange](0,0)--(4,0);
    \draw[rounded corners] (0.5, 0.5) rectangle (3.5, 2.5);
    \draw[ultra thick] (5,2.5) arc (0:180:1);
    \draw[ultra thick] (5,2.5)--(5,0);
    \draw[ultra thick] (1,0)--(1,0.5); 
    \node at (1.8,0.2) {\tiny$\cdots$};
    \draw(2.5,0)--(2.5,0.5);
    \draw (3,0)--(3,0.5);
    \draw(1,2.5)--(1,4);
    \node at (1.8,3.2) {\tiny$\cdots$};
    \draw(2.5,2.5)--(2.5,4);
    \draw[rounded corners] (0.5, 4) rectangle (3, 5.5);
    \draw (1,5.5)--(1,6);
    \node at (1.8,5.7) {\tiny$\cdots$};
    \draw (2.5,5.5)--(2.5,6);
\end{diagram}$ & 
$\begin{diagram}
    \draw[rounded corners] (0.5, 0.5) rectangle (3.5, 2.5);
    \draw[ultra thick] (3,2.5) ..controls (3,2.75)..(4,3.25);
    \draw[ultra thick] (5,2.5) ..controls (5,2.75)..(4,3.25);
    \draw[ultra thick] (4,3.25)--(4,4);
    \draw[ultra thick] (5,2.5)--(5,0);
    \draw[ultra thick] (1,0)--(1,0.5); 
    \node at (1.8,0.2) {\tiny$\cdots$};
    \draw(2.5,0)--(2.5,0.5);
    \draw (3,0)--(3,0.5);
    \draw(1,2.5)--(1,4);
    \node at (1.8,3.2) {\tiny$\cdots$};
    \draw(2.5,2.5)--(2.5,4);
    \draw[rounded corners] (0.5, 4) rectangle (5, 5.5);
    \draw (1,5.5)--(1,6);
    \node at (1.8,5.7) {\tiny$\cdots$};
    \draw[ultra thick] (2.5,5.5)--(2.5,6);
    \draw (4,5.5)--(4,6);
\end{diagram}$
\\[1em]
    Deodhar symbol & $\tU_1$ & $\tU_0$ & $\tD_1$ & $\tD_0$ 
\end{tabular}
\end{center}
and then we move on to the next step. Here the boxes stand for rex moves required to move to preferred redexes and the bold elements are the things we add. 

% We may sometimes use the notation
% \[\phi_{\ulx,\alpha}\coloneqq \olLL^{\ulx}_{\ulx^\alpha},\] 
% where the $\phi$ is short for \textit{phos phyllos}. 
%(except that the convention on the weight poset must be flipped compared to Section \ref{sect:klr}, namely now the cell labeled by the minimal element, namely $1\in W$, is an ideal)

It is known that $\CS$ is cellular via double light leaves, and in fact it is quasi-hereditary. The weight poset is given by the opposite\footnote{This is because the cell labeled by $1\in W$ is an ideal.} poset of $W\coloneqq W_\CS$ (also denoted $W_\lbd$, if $\CS=\CS_\lbd$), and one can, as before, define $\Delta_w$ as the cell module labeled by $w$:
\[\Delta_w\coloneqq C_w.\]  
This is the module spanned by diagrams which start at some preferred redex $\ulw=w$, quotiented out by the condition that any diagram which passes through a subword of $\ulw$ is set to zero. 

Alternatively we can view $\CS$ as being a triangular-based algebra as before, by setting $\tH$ to be straight line diagrams (one for each $w\in W$, corresponding to a preferred redex; we will call this $1^{\ulw}$, or simply just
\[e^w\coloneqq\begin{diagram}
    \draw[orange](0,0)--(3,0);
    \draw[orange](0,2)--(3,2);
    \draw(0.5,0)--(0.5,2);
    \draw(1,0)--(1,2);
    \node at (1.8,1) {\tiny$\cdots$};
    \draw(2.5,0)--(2.5,2);
    \node at (1.5,-0.5) {\tiny$w=\ulw$};
    \node at (1.5,2.4) {\tiny$w=\ulw$};
\end{diagram},\]
keeping in mind that this idempotent and the cellular structure depends on a fixed choice to begin with of a preferred redex for each $w\in W$), $\tX$ to be the upside-down light leaves $\olLL$, and $\tY$ to be the right-side up light leaves $\LL$. Then, as before, let
\[\CS^{\ge w}=\mfrac{\CS}{\wan{e^u:u\not\ge w}}.\] 
The ``Cartan algebras'' are
\[\CS^{ w}=e^w\CS^{\ge w}e^w\cong \bk e^w,\] 
which are semisimple with one simple each, so we can define the Verma modules as
\[\Delta_w\coloneqq \jmath^{ w}_!\bk\cong \CS^{\ge w}e^w.\] 
As before, it is not hard to argue that these two notions for $\Delta_w$ coincide here.

What we wish to compute is
\[\Ext_\CS^\blt(\CS^{\ge w}e^w,L_1).\]

\subsection{(Cyclotomic) Soergel is nil-Koszul}
Define a subalgebra $\CS^-\subset\CS$ by
\[\CS^-\coloneqq\te{the subalgebra generated monoidally by lollipops (i.e. }\lollipop\te{) and vertical strings}.\] 
Just to be clear, we mean the subalgebra generated by diagrams consisting of one lollipop possibly with propagating strings to the left and right. This algebra is graded with the Soergel grading, namely letting lollipops be of degree 1.
% Note that we have avoided using the notation $\CS^-$, because under the triangular decomposition of $\CS$ induced by the double light leaves, it is not the case that $\tY$-diagrams form a subalgebra. Nonetheless, we are thinking of this $\CS^-$ as a `nilalgebra'; this name will be justified later.

First let us describe the structure of $\CS^-$. It turns out this subalgebra is essentially a commutative polynomial ring (modded out by $x^2=0$):
\begin{PROP}
    The lollipop generators of $\CS^-$ are algebraically independent. In particular, for any word $\ulw$ of length $n$, we have as vector spaces
    \[\CS^- 1^{\ulw}\cong \mfrac{\bk[x_1,\cdotsc,x_n]}{\wan{x_i^2=0}}.\] 
\end{PROP}
\vspace{-0.5em}
\begin{PRF}
    By idempotent truncation and degree considerations, it suffices to prove that the diagrams in $\CS^- 1^{\ulw}$ with exactly $k$ lollipops are linearly independent. To this end, let $\alpha$ be any word in $\{0,1\}$ of length $n$, and let $\rho^\alpha=1^{\ulw^\alpha}\rho^\alpha 1^{\ulw}$ be the diagram whose bottom idempotent is $1^{\ulw}$ and whose $i$-th line (counting from the left) is a lollipop if $\alpha_i=0$ and a vertical propagating line if $\alpha_i=1$. (This may be the opposite of the convention one might expect, but we are attempting to adhere to light leaf conventions, even though these diagrams are typically not light leaves.) We may also consider the word $\ol\alpha$, defined by switching 1's and 0's in $\alpha$. We will let $(\rho^\alpha)^\top$ denote the diagram which is obtained from $\rho^\alpha$ by flipping it upside down.

    The set of possible words $\alpha$ has a natural lexicographic total order extended from declaring $1>0$. We claim that, for $\beta<\alpha$, we have
    \[\rho^\beta\cdot(\rho^{\ol\alpha})^\top=0.\] 
    Indeed, if $\beta<\alpha$, then there is some $i$ such that 
    \begin{align*}
        \alpha_1&=\beta_1,\\
        &\vdots,\\
        \alpha_i&=\beta_i,\\
        1=\alpha_{i+1}&>\beta_{i+1}=0.
    \end{align*}
    This means that the first $i$ lines in the product $\rho^\beta\cdot(\rho^{\ol\alpha})^\top$ will each have exactly one lollipop, either right-side up or up-side down; this will pull apart into a double light leaf where every term in the Deodhar word for both leaves is $\tU_0$; and at position $i+1$ we have a lollipop from $\rho^\beta$ and an upside-down lollipop from $(\rho^{\ol\alpha})^\top$, which forms a barbell; since the lollipops in the first $i$ positions pull apart, this barbell travels to the far left and becomes zero in the cyclotomic quotient.

    As an example, we have $\beta=1001<1011=\alpha$, so that $\rho^\beta\cdot(\rho^{\ol\alpha})^\top$ is
    \[\begin{diagram}
        \draw[orange](0,0)--(4,0);
        \draw[orange](0,4)--(4,4);
        \draw[orange,dashed](0,2)--(4,2);
        \draw(0.5,4)--(0.5,1);
        \fill(0.5,1) circle (5pt);
        \draw(1.5,0)--(1.5,3);
        \fill(1.5,3) circle (5pt);
        \draw(2.5,1)--(2.5,3);
        \fill(2.5,1) circle (5pt);
        \fill(2.5,3) circle (5pt);
        \draw(3.5,4)--(3.5,1);
        \fill(3.5,1) circle (5pt);
    \end{diagram}
    \;=\;
    \begin{diagram}
        \draw[orange](0,0)--(4,0);
        \draw[orange](0,4)--(4,4);
        \draw[orange,dashed](0,2)--(4,2);
        \draw(2,4)--(2,3);
        \fill(2,3) circle (5pt);
        \draw(2,0)--(2,1);
        \fill(2,1) circle (5pt);
        \draw(3,4)--(3,3);
        \fill(3,3) circle (5pt);
        \draw(1,1)--(1,3);
        \fill(1,1) circle (5pt);
        \fill(1,3) circle (5pt);
    \end{diagram}
    \;=0.
    \]

    Finally, let us show that any linear dependence relation must be trivial. Suppose we had
    \[\sum_\alpha c_\alpha \rho^\alpha=0;\] 
    the summands of this are totally ordered, so let $\alpha_\te{max}$ be the largest nonzero term appearing. Then let us multiply both sides on the right by $(\rho^{\ol\alpha_\te{max}})^\top$; since $\rho^\alpha\cdot(\rho^{\ol\alpha_\te{max}})^\top=0$ for all $\alpha<\alpha_\te{max}$, the equation becomes
    \[c_{\alpha_\te{max}} \rho^{\alpha_\te{max}}(\rho^{\ol\alpha_\te{max}})^\top=0.\]
    However, it is clear from drawing diagrams that
    \[\rho^{\alpha}(\rho^{\ol\alpha})^\top=\olLL^{\ulw^{\alpha}}_\emptyset\cdot \LL^\emptyset_{\ulw^{\ol\alpha}}\] 
    is a double light leaf and hence a nonzero element; for example, at $\alpha=0110$ we have $\rho^\alpha\cdot(\rho^{\ol\alpha})^\top$ is
    \[\begin{diagram}
        \draw[orange](0,0)--(4,0);
        \draw[orange](0,4)--(4,4);
        \draw[orange,dashed](0,2)--(4,2);
        \draw(0.5,0)--(0.5,3);
        \fill(0.5,3) circle (5pt);
        \draw(1.5,4)--(1.5,1);
        \fill(1.5,1) circle (5pt);
        \draw(2.5,4)--(2.5,1);
        \fill(2.5,1) circle (5pt);
        \draw(3.5,0)--(3.5,3);
        \fill(3.5,3) circle (5pt);
    \end{diagram}
    \;=\;
    \begin{diagram}
        \draw[orange](0,0)--(4,0);
        \draw[orange](0,4)--(4,4);
        \draw[orange,dashed](0,2)--(4,2);
        \draw(1,0)--(1,1);
        \fill(1,1) circle (5pt);
        \draw(3,0)--(3,1);
        \fill(3,1)circle(5pt);
        \draw(1,4)--(1,3);
        \fill(1,3) circle (5pt);
        \draw(3,4)--(3,3);
        \fill(3,3)circle(5pt);
    \end{diagram}\;.\]
    Hence $c_{\alpha_\te{max}}=0$, contradicting maximality. This concludes.
\end{PRF}
% To more directly illustrate the point of this proof, let us do an example.
% \begin{EX*}
%     \warn{example}
% \end{EX*}

At this point it may already be intuitively obvious that $\CS^-$ is quadratic and moreover Koszul. Here we argue so by doing an idempotent reduction and appealing to distributive bases. 
% Let us first briefly recall what it means to be Koszul:
% \begin{DEF}
%     A graded (locally unital) algebra $A$ with $A_0=\BK$ is called ``Koszul'' if the following equivalent conditions hold:
%     \begin{itemize}
%         \item $\Ext_{A}^{i}{}_j=0$ for $i\neq j$;
%         \item $A$ is one-generated and the algebra $\Ext_A^\blt(\BK,\BK)$ is generated by $\Ext^1_A(\BK,\BK)$;
%         \item $A$ is quadratic and $\Ext_A^\blt(\BK,\BK)\cong A^!$;
%         \item the algebra $\Ext_A^\blt(\BK,\BK)$, equipped with the cap grading, is one-generated.
%     \end{itemize}
% \end{DEF}
% A related definition is 
\begin{DEF}[\cite{polishchuk2005quadratic} Chapter 1 Proposition 7.1]
    For $V$ a vector space, a collection of its subspaces $U_1,\cdotsc,U_n$ is said to be ``distributive'' (or ``generate a distributive lattice'') if there exists a basis of $V$ such that each $U_i$ is the span of a subset of this basis. Such a basis is called a ``distributing basis''.  
\end{DEF}
% It is the latter characterization we will use, that of the existence of a distributing basis.
% The significance of the second definition is that it can be used to characterize Koszulness of a quadratic algebra.
The significance of this definition is that it can be used to characterize Koszulness of a quadratic algebra.
\begin{PROP}[\cite{polishchuk2005quadratic} Chapter 2 Theorem 4.1]
    A quadratic algebra $A$ with quadratic kernel $\qidl$ is Koszul if and only if for every $n\ge 0$, the collection of subspaces
    \[U_i=A_1^{\otimes i-1}\otimes\qidl\otimes A_1^{\otimes n-i-1}\subset A_1^{\otimes n}\] 
    for $1\le i\le n-1$ is distributive. More precisely, the following are equivalent:
    \begin{itemize}
        \item $\Ext^i_A{}_j(\BK,\BK)=0$ for all $i<j\le n$;
        \item the Koszul complex is acyclic in positive internal degrees $\le n$;
        \item the collection $\{U_1,\cdotsc,U_{n-1}\}$ in $A_1^{\otimes n}$ is distributive.
    \end{itemize}
\end{PROP}
This point of view can be used to show that
\begin{PROP}\label{prop:lollipopkoszul}
    $\CS^-$ is a Koszul algebra under the grading where $\deg(\lol)=1$. 
\end{PROP}
\vspace*{-0.5em}
\begin{PRF}
    The above proposition gives us a basis of $\CS^-$, from which it is clear that $\CS^-$ is quadratic by declaring Soergel-degree 1 lollipops $\lol_i 1^{\ulw}$ to be our Koszul-degree 1 generators. The quadratic ideal $\qidl$ of relations is then generated by relations of the form $\lol_i\cdot\lol_j=\lol_j\cdot\lol_i$, or
    
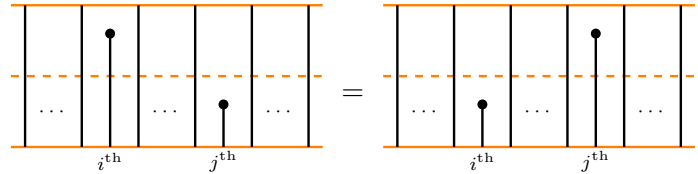
\[\begin{diagram}
        \draw[orange] (0,0)--(11,0);
        \draw[orange] (0,5)--(11,5);
        \draw[orange,dashed] (0,2.5)--(11,2.5);
        \draw (0.5,0)--(0.5,5);
        \node at (1.5,1.25) {\tiny $\cdots$};
        \draw (2.5,0)--(2.5,5);
        \draw (3.5,0)--(3.5,4);
        \fill (3.5,4) circle (5pt);
        \node at (3.5,-0.5) {\tiny $i^\te{th}$};
        \draw (4.5,0)--(4.5,5);
        \node at (5.5,1.25) {\tiny $\cdots$};
        \draw (6.5,0)--(6.5,5);
        \draw (7.5,0)--(7.5,1.5);
        \fill (7.5,1.5)circle(5pt);
        \node at (7.5,-0.5) {\tiny $j^\te{th}$};
        \draw (8.5,0)--(8.5,5);
        \node at (9.5,1.25) {\tiny $\cdots$};
        \draw (10.5,0)--(10.5,5);
    \end{diagram}
    \ =\ 
    \begin{diagram}
        \draw[orange] (0,0)--(11,0);
        \draw[orange] (0,5)--(11,5);
        \draw[orange,dashed] (0,2.5)--(11,2.5);
        \draw (0.5,0)--(0.5,5);
        \node at (1.5,1.25) {\tiny $\cdots$};
        \draw (2.5,0)--(2.5,5);
        \draw (3.5,0)--(3.5,1.5);
        \fill (3.5,1.5) circle (5pt);
        \node at (3.5,-0.5) {\tiny $i^\te{th}$};
        \draw (4.5,0)--(4.5,5);
        \node at (5.5,1.25) {\tiny $\cdots$};
        \draw (6.5,0)--(6.5,5);
        \draw (7.5,0)--(7.5,4);
        \fill (7.5,4)circle(5pt);
        \node at (7.5,-0.5) {\tiny $j^\te{th}$};
        \draw (8.5,0)--(8.5,5);
        \node at (9.5,1.25) {\tiny $\cdots$};
        \draw (10.5,0)--(10.5,5);
    \end{diagram}.\]

    To show $\CS^-$ is Koszul, it suffices to give a distributing basis of $\CS^{-,\otimes m}_1$ for each $m\ge 0$, distributing the collection of spaces 
    \[\{\CS_1^{-,\otimes i-1}\otimes \qidl\otimes \CS_1^{-,\otimes m-i-1}\}_i.\] 
    As diagrams with different beginning or ending idempotents or different degrees are linearly independent, it moreover suffices to give a distributing basis of $1^{\ulu} \CS^{-,\otimes m}_1 1^{\ulw}$, where $\ulu= \ulw^\alpha$ for some $\alpha$ which has $m$ zeros, or alternatively a distributing basis of $\CS^{-,\otimes m}_11^{\ulw}$. 

    However we have already seen that $\CS^- 1^{\ulw}\cong\bk[x_1,\cdotsc,x_n]/\wan{ x_i^2=0}$ as vector spaces. So let us consider the algebra $A(\ulw)=\bk[x_1,\cdotsc,x_n]/\wan{ x_i^2=0}$, which is generated by the subspace $A(\ulw)_1$ spanned by $x_i$, with a quadratic ideal of relations $\qidl(A(\ulw))$ generated by
    \begin{align*}
        x_i^2&=0,\\x_ix_j&=x_jx_i. 
    \end{align*}
    Note that this corresponds to the quadratic relations of $\CS^-$. If this $A(\ulw)$ were Koszul, then one has a distributing basis of $A(\ulw)_1^{\otimes m}$ distributing the collection $\{A(\ulw)_1^{\otimes i-1}\otimes \qidl(A(\ulw))\otimes A(\ulw)_1^{\otimes m-i-1}\}_i$; the same basis will be the distributing basis we desire in $\CS^{-,\otimes m}_1$ distributing $\{\CS_1^{-,\otimes i-1}\otimes \qidl\otimes \CS_1^{-,\otimes m-i-1}\}_i$, under the obvious vector space isomorphism $\bigoplus_{\ulw} A(\ulw)\cong \bigoplus_{\ulw} \CS^- 1^{\ulw}=\CS^-$.

    But that $A(\ulw)=\bk[x_1,\cdotsc,x_n]/\wan{x_i^2=0}$ is Koszul is a classically celebrated fact.
\end{PRF}

We call $\CS^-$ a `nilalgebra' because of the following fact. %(which is the main proof of this paper, in the sense that the proof here is the most non-trivial of this paper).
\begin{PROP}\label{prop:regseq}
    Although $\CS$ is not free (in fact it is not even flat) over $\CS^-$, it is still true that 
    \[\CS\lotimes_{\CS^-}\bk e^w=\Delta_w=\jmath^w_! \bk\]
    is concentrated in one degree homologically for $w=\ulw$ a reduced word.
\end{PROP}
\vspace*{-0.5em}
\begin{adjustwidth}{2em}{0pt}
\begin{proof}
    Let us first show that $\CS\otimes_{\CS^-}\bk e^w=\Delta_w$. This is more or less obvious -- $\Delta_w$ is spanned by the diagrams that start at $e^w$, where any diagram that passes through $e^u$ for $u\not\ge w$ is set to zero. By the double light leaves basis, if $u\not\ge w$, then any Soergel diagram from $w$ to $u$ must pass through some $u'$ such that $u'\le u$ and $u'<w$. But any diagram from $w$ to some $u'<w$ must involve lollipops $\lol$ (as both $u',w$ are reduced), which act by zero in $\CS^-\actson \bk e^w$, and hence such a diagram must be zero in $\CS\otimes_{\CS^-}\bk e^w$. 

    Now let us show that the higher derived tensor products all vanish when $\ulw=w$ is a reduced word. This is essentially done by showing that lollipops form a regular sequence for the right action of $\CS^- e^w$ on $\CS$, when $w$ is reduced.

    We have already seen that $\CS^-$ is Koszul, so we can consider the Koszul resolution of $\CS^-\actson \BK$ or $\bk e^w$:
    \[\CS^-\otimes_\BK \CS^{-,!,\vee,\blt} e^w\simeq \bk e^w.\] 
    As usual, for any Koszul algebra $A$, the maps in the Koszul resolution\footnote{Our convention for where the index goes may appear confusing at first -- when we write $A^{!,\vee,\blt}$, we are trying to indicate with the fact that the bullet appears in the superscript that the cohomological indices appearing are negative. Indeed, $A^!$ has $-\deg_{\te{Ksz}}=\deg_{\Ext}$, which are all $\ge 0$, so that the dual $A^{!,\vee}$ has degrees $-\deg_{\te{Ksz}}=\deg_{\Ext}\le 0$. In other words, $A\otimes A^{!,\vee,\blt}$ is a complex starting at cohomological degree 0 and extending to the left into negative cohomological degrees. Strictly speaking, one should be careful about the order in which the bullet appears; for instance,
    \[A^{!,\vee,-n}=A^!{}_n{}^\vee.\]
    However frequently we will be flippant about where the index appears and rely on context for whether the index comes before or after taking duals.} $A\otimes A^{!,\vee}_{n}\to A\otimes A^{!,\vee}_{n-1}$ are given by:
    \begin{align*}
        A^{!}_{n-1}\otimes A_1^*&\lto A^{!}_{n}\\
        \longsquigglyrightarrow \hspace{10em}\hspace{9pt}A^{!,\vee}_{n}&\lto A_1\otimes A^{!,\vee}_{n-1}\\
        \longsquigglyrightarrow A\otimes A^{!,\vee}_n\lto A\otimes A_1\otimes A^{!,\vee}_{n-1}&\lto A\otimes A^{!,\vee}_{n-1}.
    \end{align*}
    By unwinding definitions, one can find that the differentials will be given by taking 
    \begin{align*}
        \d\colon \CS^-\otimes_\BK\CS^{-,!,\vee}_{k} e^w&\lto \CS^-\otimes_\BK\CS^{-,!,\vee}_{k-1} e^w\\
        x\otimes P(\lol_1 e^w,\cdotsc, \lol_n e^w)&\lmto \sum_{\te{ways to write }P=\lol_i\cdot P'} x\cdot\lol_i\otimes P',
    \end{align*}
    where $P\in\CS^{-,!,\vee}_{k}$ is a degree $k$ anti-commutative polynomial in the lollipops (in other words, it behaves like wedge products). In words, this differential acts by summing over all ways to pull one lollipop from the multiplicity space $\CS^{-,!,\vee}_k$ across the tensor symbol into $\CS^-$. The signs are hidden in the fact that $\CS^{-,!,\vee}_k$ is anti-commutative. To make things simpler, implicitly we will write $P$ diagrammatically as having the left-most lollipop be at the bottom and right-most lollipop be at the top, or in other words each monomial of $P$ is written as $\lol_{i_k}\cdots\lol_{i_1}$ for $i_k>\cdots>i_1$; e.g., we will implicitly prefer to write
    \[\begin{diagram}
        \draw[orange](0,0)--(2.5,0);
        \draw[orange](0,2)--(2.5,2);
        \draw(0.5,0)--(0.5,0.5);
        \fill (0.5,0.5) circle(4pt);
        \draw(1,0)--(1,1);
        \fill (1,1)circle(4pt);
        \draw(1.5,0)--(1.5,2);
        \draw(2,0)--(2,1.5);
        \fill(2,1.5)circle(4pt);
    \end{diagram}^{!,\vee}.\]
    %We will not detail the unwinding required to obtain this explicit form, because we will not use it in this proof; it is included for completeness.
    
    This is of course a free resolution of $\CS^-\actson \bk e^w$, so we can tensor it with $\CS$ on the left to compute $\CS\lotimes_{\CS^-}\bk e^w$:
    \[\CS\otimes_\BK \CS^{-,!,\vee,\blt}e^w\simeq \CS\lotimes_{\CS^-}\bk e^w.\]
    We will prove this is concentrated in degree zero by emulating the proofs in the Stacks Project at \href{https://stacks.math.columbia.edu/tag/0621}{0621} and \href{https://stacks.math.columbia.edu/tag/062D}{062D}. Our situation is only slightly more complicated due to the noncommutative and locally unital nature of our algebras.

    In the classical literature one speaks of a ``Koszul complex'' $K_\blt(x_1,\cdotsc,x_n)$ over an algebra $A$. Let us state what this is for our setting. Consider a Koszul complex $K_\blt(\lol_1 e^{\ulu},\cdotsc, \lol_n e^{\ulu})$ over the algebra $\CS^-$, defined by considering the subcomplex of
    \[\CS^-\otimes_\BK \CS^{-,!,\vee,\blt}e^{\ulu}\simeq \bk e^{\ulu}\] 
    generated by the first $n$ lollipops (counting from the left) in $\CS^{-,!,\vee,\blt} e^{\ulu}$, where we assume $\ell(\ulu)\ge n$. To illustrate, the complex $K_\blt(\lol_1 e^{s_1s_2})$ is the subcomplex of $\CS^-\otimes_\BK \CS^{-,!,\vee,\blt}e^{s_1s_2}$ generated by the first lollipop, i.e.\footnote{The heart indicates homological degree zero.}
    \[\CS^-\otimes \begin{diagram}
        \draw[orange](0,0)--(1,0);
        \draw[orange](0,1)--(1,1);
        \draw(0.25,0)--(0.25,0.5);
        \fill (0.25,0.5) circle (4pt);
        \draw[red](0.75,0)--(0.75,1);
    \end{diagram}^{!,\vee}\lto \underbrace{\CS^-\otimes \begin{diagram}
        \draw[orange](0,0)--(1,0);
        \draw[orange](0,1)--(1,1);
        \draw(0.25,0)--(0.25,1);
        \draw[red](0.75,0)--(0.75,1);
    \end{diagram}^{!,\vee}}_{\heart}.\]

    % \warn{ignore this paragraph for now} For $w=\ulw$ a redex, let $w_{\le k}$ be the redex given by the first $k$ letters starting from the left. Our first claim is that, for $M$ a right $\CS^-$-module, we have a distinguished triangle
    % % \[M\otimes K_\blt(\lol_1 e^{w_{\le k}},\cdotsc, \lol_k e^{w_{\le k}})\lto M\otimes K_\blt(\lol_1 e^{w_{\le k+1}},\cdotsc, \lol_k e^{w_{\le k+1}})\lto M\otimes K_\blt(\lol_1 e^{w_{\le k+1}},\cdotsc, \lol_k e^{w_{\le k+1}},\lol_{k+1} e^{w_{\le k+1}})\oslto{+1}.\]
    % \[M\otimes K_\blt(\lol_1 e^{w_{\le k}},\cdotsc, \lol_k e^{w_{\le k}})\lto M\otimes K_\blt(\lol_1 e^{w_{\le k+1}},\cdotsc, \lol_k e^{w_{\le k+1}})\lto M\otimes K_\blt(\lol_1 e^{w_{\le k+1}},\cdotsc,\lol_{k+1} e^{w_{\le k+1}})\oslto{+1}.\]

    For $\ulu$ an expression, let $\ulu_{\le k}$ be the subexpression given by the first $k$ letters starting from the left. Our first claim is that
     \begin{adjustwidth}{-2em}{0pt}
     \ \vspace{-1em}
    \begin{LEM}
        % For $M$ a right $\CS^-$-module, we have a distinguished triangle
        Let $\ulu$ be an expression with $\ell(\ulu)\ge k$ such that $\ulu_{\le k}$ and $\ulu_{\le k}s_a$ are both redexes. Let $\ulu'$ be the word obtained by inserting $s_a$ immediately after the $k$-th letter in $\ulu$, denoted as $\ulu'=\on{ins}_{s_a,k+1}(\ulu)$. Then we have a distinguished triangle (here the tensors are over $\CS^-$)
        \[\CS\otimes K_\blt(\lol_1 e^{\ulu},\cdotsc, \lol_k e^{\ulu})\lto \CS\otimes K_\blt(\lol_1 e^{\ulu'},\cdotsc, \lol_k e^{\ulu'})\lto \CS\otimes K_\blt(\lol_1 e^{\ulu'},\cdotsc, \lol_{k+1} e^{\ulu'})\oslto{+1},\] 
        where the first map is given by sending $x\in\CS$ (or more precisely $x\otimes P\in\CS\otimes K_j$ for $P=e^{\ulv} P e^{\ulu}$ a lollipop diagram of an appropriate degree) to $x'\in\CS$ which is obtained by adding a right-side up lollipop of color $s_a$ after the $(k-\deg P)$-th position along the bottom edge of $x$ (more precisely this is $x'\otimes P'$, where $P'$ is the diagram obtained from $P$ by adding a propagating line of color $s_a$ after the $k$-th position along the bottom edge of $P$). Diagrammatically this first map, which we may call `lollipop insertion', can be described as
        \begin{align*}
            \on{ins}_{s_a,k+1}\colon \CS\otimes K_\blt(\lol_1 e^{\ulu},\cdotsc, \lol_k e^{\ulu})&\lto \CS\otimes K_\blt(\lol_1 e^{\ulu'},\cdotsc, \lol_k e^{\ulu'})\\
            \raisebox{-1.1em}{$\begin{diagram}
            \draw[orange](0,0)--(5,0);
            \draw[orange](0,2.5)--(5,2.5);
            \node at (-0.5,2.5) {\tiny$\otimes$};
            \draw[orange](0,5)--(5,5);
            \node at (1.25,1.25) {\tiny$\cdots$};
            \draw (2.5,0)--(2.5,2.5);
            \node at (3.5,1.25){\tiny$\cdots$};
            \draw (4.5,0)--(4.5,2.5);
            \draw[dashed, rounded corners] (0.25, 0.25) rectangle (4.75, 2.25);
            \node at (5.2,1.7) {\tiny$P$};
            \node at (2.5,4) {\tiny $x$};
            \node at (2.5,-1.5) {\tiny $(k+1)^\te{th}$};
            \draw[-Stealth] (2.5,-1.3)--(2.5,-0.2);
            \node at (5.3,-0.2) {\tiny $\ulu$};
            \node at (5.6,2.5) {\tiny$!,\vee$};
        \end{diagram}$}&\lmto \raisebox{-1.1em}{$\begin{diagram}
            \draw[orange](0,0)--(5,0);
            \draw[orange](0,2.5)--(5,2.5);
            \node at (-0.5,2.5) {\tiny$\otimes$};
            \draw[orange](0,5)--(5,5);
            \node at (1.25,1.25) {\tiny$\cdots$};
            \draw (2.5,0)--(2.5,2.5);
            \node at (3.5,1.25){\tiny$\cdots$};
            \draw (4.5,0)--(4.5,2.5);
            \draw[dashed, rounded corners] (0.25, 0.25) rectangle (4.75, 2.25);
            \node at (5.2,1.7) {\tiny$P$};
            \node at (2.5,4) {\tiny $x$};
            \node at (2.5,-1.5) {\tiny $(k+2)^\te{th}$};
            \draw[-Stealth] (2.5,-1.3)--(2.5,-0.2);
            \node at (5.3,-0.2) {\tiny $\ulu'$};
            \draw[red](2,0)--(2,3.25);
            \fill[red] (2,3.25) circle (4pt);
            \node at (5.6,2.5) {\tiny$!,\vee$};
        \end{diagram}$}
        \end{align*}
    \end{LEM}
    \end{adjustwidth}
    As an example, ignoring colors, this map might send\footnote{Here the middle orange line stands for the tensor symbol; the $!,\vee$ symbols indicate that the bottom of half of the diagram lives in $\CS^{-,!,\vee}$; and the $\ulu$ at the bottom right of the diagram indicates that the sequence of colors along the bottom-most boundary is given by $\ulu$.}
    \[\raisebox{-1.1em}{$\begin{diagram}
            \draw[orange](0,0)--(5,0);
            \draw[orange](0,2.5)--(5,2.5);
            \node at (-0.5,2.5) {\tiny$\otimes$};
            \draw[orange](0,5)--(5,5);
            \draw(0.25,0)--(0.25,1.25);
            \fill (0.25,1.25) circle (4pt);
            \draw(0.75,0)--(0.75,2.5);
            \draw(1.25,0)--(1.25,1.25);
            \fill (1.25,1.25)circle(4pt);
            \draw(1.75,0)--(1.75,2.5);
            \draw (2.5,0)--(2.5,2.5);
            \node at (3.5,1.25){\tiny$\cdots$};
            \draw (4.5,0)--(4.5,2.5);
            \node at (2.5,4) {\tiny $x$};
            \node at (2.5,-1.5) {\tiny $(k+1)^\te{th}$};
            \draw[-Stealth] (2.5,-1.3)--(2.5,-0.2);
            \node at (5.3,-0.2) {\tiny $\ulu$};
            \node at (5.6,2.5) {\tiny$!,\vee$};
        \end{diagram}$}\lmto \raisebox{-1.1em}{$\begin{diagram}
            \draw[orange](0,0)--(5,0);
            \draw[orange](0,2.5)--(5,2.5);
            \node at (-0.5,2.5) {\tiny$\otimes$};
            \draw[orange](0,5)--(5,5);
            \draw(0.25,0)--(0.25,1.25);
            \fill (0.25,1.25) circle (4pt);
            \draw(0.75,0)--(0.75,2.5);
            \draw(1.25,0)--(1.25,1.25);
            \fill (1.25,1.25)circle(4pt);
            \draw(1.75,0)--(1.75,2.5);
            \draw (2.5,0)--(2.5,2.5);
            \node at (3.5,1.25){\tiny$\cdots$};
            \draw (4.5,0)--(4.5,2.5);
            \node at (2.5,4) {\tiny $x$};
            \node at (2.5,-1.5) {\tiny $(k+2)^\te{th}$};
            \draw[-Stealth] (2.5,-1.3)--(2.5,-0.2);
            \node at (5.3,-0.2) {\tiny $\ulu'$};
            \draw[red](2.125,0)--(2.125,3.25);
            \fill[red](2.125,3.25) circle(4pt);
            \node at (5.6,2.5) {\tiny$!,\vee$};
        \end{diagram}$}.\]

        The proof is more or less the same as in \href{https://stacks.math.columbia.edu/tag/0628}{Stacks 0628}. We want to show that $\CS\otimes K_\blt(\lol_1 e^{\ulu'},\cdotsc, \lol_{k+1} e^{\ulu'})$ is isomorphic to the cone of $\on{ins}_{s_a,k+1}$. Let us briefly recall that the cone of $f\colon A\lto B$ is given by $B\oplus A[1]$, with differential given by
        \[\d_{\on{Cone}(f)}=\tbt{\d_B}{f[1]}{0}{\d_{A[1]}},\]
        where $\d_{A[1]}=-\d_{A}$. Firstly, the $n$-th term of the complex $\CS\otimes_{\CS^-}K_n(\lol_1 e^{\ulu'},\cdotsc,\lol_{k+1} e^{\ulu'})$ is spanned by tensors $x\otimes P$ where $\deg P=n$; either $P$ contains $\lol_{k+1}$ or it does not. The subspace of those $x\otimes P$ such that $P$ does not contain $\lol_{k+1}$ is isomorphic to $\CS\otimes_{\CS^-}K_n(\lol_1 e^{\ulu'},\cdotsc,\lol_{k} e^{\ulu'})$, while the subspace of those such that $P$ does contain $\lol_{k+1}$ is isomorphic to $\CS\otimes_{\CS^-}K_{n-1}(\lol_1 e^{\ulu},\cdotsc,\lol_{k} e^{\ulu})$. Hence
        \[\CS\otimes K_n(\lol_1 e^{\ulu'},\cdotsc,\lol_{k+1} e^{\ulu'})\cong \CS\otimes K_n(\lol_1 e^{\ulu'},\cdotsc,\lol_{k} e^{\ulu'})\ \oplus\  \CS\otimes K_{n-1}(\lol_1 e^{\ulu},\cdotsc,\lol_{k} e^{\ulu}).\] 
        It remains to check the differentials. The differential clearly agrees between the subspace of $x\otimes P$ such that $P$ does not contain $\lol_{k+1}$ and $\CS\otimes K_n(\lol_1 e^{\ulu'},\cdotsc,\lol_k e^{\ulu'})$. The differential restricted to the subspace of the LHS such that $P$ does contain $\lol_{k+1}$ either pulls $\lol_{k+1}$ across the tensor sign or it does not; if it does, then it corresponds to 
        \[\on{ins}_{s_a,k+1}\colon \CS\otimes K_{n-1}(\lol_1 e^{\ulu},\cdotsc,\lol_{k} e^{\ulu})\lto \CS\otimes K_{n-1}(\lol_1 e^{\ulu'},\cdotsc,\lol_{k} e^{\ulu'}),\]
        whereas if it does not, then it must pull another lollipop (from $\lol_1$ to $\lol_k$) across the tensor, corresponding to the differential 
        \[-\d_{\CS\otimes K_\blt}=\d_{\CS\otimes K_\blt[1]}\colon \CS\otimes K_{n-1}(\lol_1 e^{\ulu},\cdotsc,\lol_{k} e^{\ulu})\lto \CS\otimes K_{n-2}(\lol_1 e^{\ulu},\cdotsc,\lol_{k} e^{\ulu}),\] 
        where the extra minus sign from the cone construction accounts for the missing $\lol_{k+1}$ in $\CS\otimes K_{n-1}(\lol_1 e^{\ulu},\cdotsc,\lol_{k} e^{\ulu})$. 

        Perhaps the best illustration of the above proof is in computing the cone 
        \[\CS\otimes K_\blt(\begin{diagram}
        \draw[orange](0,0)--(1,0);
        \draw[orange](0,1)--(1,1);
        \draw(0.5,0)--(0.5,0.5);
        \fill (0.5,0.5) circle (4pt);
    \end{diagram}^{!,\vee}) \lto \CS\otimes K_\blt(\begin{diagram}
        \draw[orange](0,0)--(1,0);
        \draw[orange](0,1)--(1,1);
        \draw(0.25,0)--(0.25,0.5);
        \fill (0.25,0.5) circle (4pt);
        \draw[red](0.75,0)--(0.75,1);
    \end{diagram}^{!,\vee});\]
        indeed, the cone is given by
        \begin{center}
            \begin{tikzcd}[column sep=0.5em]
                            &\CS\otimes \begin{diagram}
                    \draw[orange](0,0)--(1,0);
                    \draw[orange](0,1)--(1,1);
                    \draw(0.25,0)--(0.25,1);
                    \draw[red](0.75,0)--(0.75,1);
                \end{diagram}^{!,\vee} & \\
                \CS\otimes \begin{diagram}
                    \draw[orange](0,0)--(1,0);
                    \draw[orange](0,1)--(1,1);
                    \draw(0.25,0)--(0.25,0.5);
                    \fill (0.25,0.5) circle (4pt);
                    \draw[red](0.75,0)--(0.75,1);
                \end{diagram}^{!,\vee}\arrow{ur}{x\cdot e^{s_2}\lol_1 e^{s_1s_2}} & \oplus & \CS\otimes \begin{diagram}
                    \draw[orange](0,0)--(1,0);
                    \draw[orange](0,1)--(1,1);
                    \draw(0.5,0)--(0.5,1);
                \end{diagram}^{!,\vee}\arrow[swap]{ul}{x\cdot e^{s_1} \lol_2 e^{s_1s_2}}\\
                &\CS\otimes \begin{diagram}
                    \draw[orange](0,0)--(1,0);
                    \draw[orange](0,1)--(1,1);
                    \draw(0.5,0)--(0.5,0.5);
                    \fill (0.5,0.5) circle (4pt);
                \end{diagram}^{!,\vee}\arrow{ul}{x\cdot e^{\emptyset}\lol_1 e^{s_2}}\arrow[swap]{ur}{-x\cdot e^{\emptyset}\lol_1 e^{s_1}} &
            \end{tikzcd}.
        \end{center}
        % \begin{diagram}
        %             \draw[orange](0,0)--(1,0);
        %             \draw[orange,dashed](0,1)--(1,1);
        %             \draw[orange](0,2)--(1,2);
        %         \end{diagram}
        Note the negative sign in the bottom right is due to homological shift in the cone construction, i.e. $-\d_A=\d_{A[1]}$. Compare this to $\CS\otimes K_\blt(\begin{diagram}
                    \draw[orange](0,0)--(1,0);
                    \draw[orange](0,1)--(1,1);
                    \draw(0.25,0)--(0.25,0.5);
                    \fill (0.25,0.5) circle (4pt);
                    \draw[red](0.75,0)--(0.75,1);
                \end{diagram}^{!,\vee},\begin{diagram}
                    \draw[orange](0,0)--(1,0);
                    \draw[orange](0,1)--(1,1);
                    \draw(0.25,0)--(0.25,1);
                    \fill[red] (0.75,0.5) circle (4pt);
                    \draw[red](0.75,0)--(0.75,0.5);
                \end{diagram}^{!,\vee})$, which is given by
        \begin{center}
            \begin{tikzcd}[column sep=0.5em]
                            &\CS\otimes \begin{diagram}
                    \draw[orange](0,0)--(1,0);
                    \draw[orange](0,1)--(1,1);
                    \draw(0.25,0)--(0.25,1);
                    \draw[red](0.75,0)--(0.75,1);
                \end{diagram}^{!,\vee} & \\
                \CS\otimes \begin{diagram}
                    \draw[orange](0,0)--(1,0);
                    \draw[orange](0,1)--(1,1);
                    \draw(0.25,0)--(0.25,0.5);
                    \fill (0.25,0.5) circle (4pt);
                    \draw[red](0.75,0)--(0.75,1);
                \end{diagram}^{!,\vee}\arrow{ur}{x\cdot e^{s_2}\lol_1 e^{s_1s_2}} & \oplus & \CS\otimes \begin{diagram}
                    \draw[orange](0,0)--(1,0);
                    \draw[orange](0,1)--(1,1);
                    \draw(0.25,0)--(0.25,1);
                    \draw[red](0.75,0)--(0.75,0.5);
                    \fill[red](0.75,0.5)circle(4pt);
                \end{diagram}^{!,\vee}\arrow[swap]{ul}{x\cdot e^{s_1} \lol_2 e^{s_1s_2}}\\
                &\CS\otimes \begin{diagram}
                    \draw[orange](0,0)--(1,0);
                    \draw[orange](0,1)--(1,1);
                    \draw(0.25,0)--(0.25,0.5);
                    \fill (0.25,0.5) circle (4pt);
                    \draw[red](0.75,0)--(0.75,0.75);
                    \fill[red] (0.75,0.75) circle (4pt);
                \end{diagram}^{!,\vee}\arrow{ul}{x\cdot e^{\emptyset}\lol_1 e^{s_2}}\arrow[swap]{ur}{-x\cdot e^{\emptyset}\lol_1 e^{s_1}} &
            \end{tikzcd}.
        \end{center}
        Note the negative sign in the bottom right is due to the fact that we must pull the black lollipop across the red lollipop, introducing a sign. So the sign from the homological shift of the cone construction corresponds precisely to the sign from pulling across one extra lollipop in $\CS^{-,!,\vee}$. 

        The lemma having been proven, we then obtain a long exact sequence of homology, for example looking like
    %     \begin{center}
    %     \begin{tikzcd}
    %         H_i(\CS\otimes K_\blt(\lol_1 e^{\ulu},\cdotsc, \lol_k e^{\ulu})) \arrow{r} & H_i(\CS\otimes K_\blt(\lol_1 e^{\ulu'},\cdotsc, \lol_k e^{\ulu'})) \snakeanchor\arrow{r}& H_i(\CS\otimes K_\blt(\lol_1 e^{\ulu'},\cdotsc, \lol_{k+1} e^{\ulu'}))\snakearrow{} & \\
    %         H_{i-1}(\CS\otimes K_\blt(\lol_1 e^{\ulu},\cdotsc, \lol_k e^{\ulu})) \arrow{r} &H_{i-1}(\CS\otimes K_\blt(\lol_1 e^{\ulu'},\cdotsc, \lol_k e^{\ulu'})) \arrow{r} &H_{i-1}(\CS\otimes K_\blt(\lol_1 e^{\ulu'},\cdotsc, \lol_{k+1} e^{\ulu'}))
    %     \end{tikzcd}
    % \end{center}
    \begin{adjustwidth}{-2em}{0pt}
     \begin{center}
        \begin{tikzcd}
            \cdots \arrow{r} & H_i(\CS\otimes K_\blt(\lol_1 e^{\ulu'},\cdotsc, \lol_k e^{\ulu'})) \snakeanchor\arrow{r}& H_i(\CS\otimes K_\blt(\lol_1 e^{\ulu'},\cdotsc, \lol_{k+1} e^{\ulu'}))\snakearrow{} & \\
            H_{i-1}(\CS\otimes K_\blt(\lol_1 e^{\ulu},\cdotsc, \lol_k e^{\ulu})) \arrow{r} &H_{i-1}(\CS\otimes K_\blt(\lol_1 e^{\ulu'},\cdotsc, \lol_k e^{\ulu'})) \arrow{r} &\cdots
        \end{tikzcd}.
    \end{center}
    \end{adjustwidth}
    
    We claim that 
    \[H_{>0}\big(\CS\otimes K_\blt(\lol_1 e^{\ulu},\cdotsc, \lol_{k} e^{\ulu})\big)=0\]
    for any word $\ulu$ such that $\ulu_{\le k}$ is reduced. We prove this by induction on the number of lollipops, i.e. $k$. Note that this claim proves what we want by setting $\ulu=w$ and $k=n=\ell(w)$. 

    The base case is when $k=1$. In this case the complex $K_\blt(\lol_1 e^{\ulu})$ has only two terms, so it suffices to show that $H_1=0$, namely that the map
    \[\CS\otimes \lol_1 e^{\ulu}\lto \CS\otimes e^{\ulu}\] 
    is injective. This is obvious, as this map sends $x\in\CS$ to $x'$ obtained by adding a right-side up lollipop at the bottom left of $x$; if $x$ is a double light leaf, this amounts to appending a $\tU_0$ to the beginning of the Deodhar word for the lower light leaf, which gives a new valid Deodhar word. Hence this map sends each double light leaf to a unique double light leaf, so that it must be injective. Note this argument does not depend on $\ell(\ulu)$. 

    The inductive step uses the long exact sequence above. Let $\ulu'=\on{ins}_{s_a,k+1}(\ulu)$, where $\ulu_{\le k}s_a$ is reduced. Then, for $i\ge 2$, the inductive hypothesis tells us the long exact sequence is 
    \begin{adjustwidth}{-2em}{0pt}
     \begin{center}
        \begin{tikzcd}
            \cdots \arrow{r} & 0 \snakeanchor\arrow{r}& H_i(\CS\otimes K_\blt(\lol_1 e^{\ulu'},\cdotsc, \lol_{k+1} e^{\ulu'}))\snakearrow{} & \\
           0 \arrow{r} &\cdots &
        \end{tikzcd},
    \end{center}
    \end{adjustwidth}
    which forces $H_i(\CS\otimes K_\blt(\lol_1 e^{\ulu'},\cdotsc, \lol_{k+1} e^{\ulu'}))=0$. At $i=1$ we instead have
    \begin{adjustwidth}{-2em}{0pt}
     \begin{center}
        \begin{tikzcd}
            \cdots \arrow{r} & 0\snakeanchor\arrow{r}& H_1(\CS\otimes K_\blt(\lol_1 e^{\ulu'},\cdotsc, \lol_{k+1} e^{\ulu'}))\snakearrow{} & \\
            H_{0}(\CS\otimes K_\blt(\lol_1 e^{\ulu},\cdotsc, \lol_k e^{\ulu})) \arrow{r} &H_{0}(\CS\otimes K_\blt(\lol_1 e^{\ulu'},\cdotsc, \lol_k e^{\ulu'})) \arrow{r} &\cdots
        \end{tikzcd},
    \end{center}
    \end{adjustwidth}
    so it suffices to show that the map 
    \[\on{ins}_{s_a,k+1}\colon H_{0}(\CS\otimes K_\blt(\lol_1 e^{\ulu},\cdotsc, \lol_k e^{\ulu})) \lto H_{0}(\CS\otimes K_\blt(\lol_1 e^{\ulu'},\cdotsc, \lol_k e^{\ulu'}))\] 
    is injective. Note that
    \[H_{0}(\CS\otimes K_\blt(\lol_1 e^{\ulu},\cdotsc, \lol_k e^{\ulu}))=\mfrac{\CS e^{\ulu}}{\CS\cdot\wan{\lol_1 e^{\ulu},\cdotsc, \lol_k e^{\ulu}}},\]
    so it suffices to show that the action of inserting a $s_a$-colored lollipop after the $k$-th position is injective on this space. This is done by utilizing the theory of light leaves.

    Let us first observe that it is easy to see by induction that $\mfrac{\CS e^{\ulu}}{\CS\cdot\wan{\lol_1 e^{\ulu},\cdotsc, \lol_k e^{\ulu}}}$ has a basis given by the image of the double light leaf basis, namely those such that the Deodhar word for the lower light leaf has no $\tU_0$'s in the first $k$ positions. Let $x\in\mfrac{\CS e^{\ulu}}{\CS\cdot\wan{\lol_1 e^{\ulu},\cdotsc, \lol_k e^{\ulu}}}$ be such a double light leaf, with no $\tU_0$'s in the first $k$ positions. However, since $\ulu_{\le k}$ is reduced, the first $k$ positions also cannot have $\tD_0$ or $\tD_1$. Hence the first $k$ letters of the Deodhar word of the lower leaf of $x$ must all be $\tU_1$. Since $\ulu_{\le k}s_a$ is still reduced, inserting a $s_a$-colored lollipop after the $k$-th position into the lower light leaf of $x$ amounts to inserting a $\tU_0$ after the $k$-th position of the Deodhar word of the lower light leaf. This yields another valid Deodhar word, as this $\tU_0$ does not affect the validity of the letters after it. Hence inserting this lollipop sends each double light leaf to a unique double light leaf, so that the map $\on{ins}_{s_a,k+1}$ is injective, as desired.
\end{proof}
\end{adjustwidth}\vspace{0.5em}

The condition that $w$ is reduced is crucial to the veracity of Proposition \ref{prop:regseq}, as one might have seen from the proof. Indeed, if $w$ is not reduced, we can give a counterexample:
\begin{EX*}
    Let $\ulx=sss$. Then $\CS\lotimes_{\CS^-}\bk e^{\ulx}$ is not concentrated in degree 0. 

    Indeed, to compute this derived tensor product, we may use the Koszul resolution of $\bk e^{\ulx}$:
    \[\CS^{-,!,\vee,\blt} e^{\ulx}\simeq\bk e^{\ulx}.\] 
    The Tor groups can then be computed as
    \[\Tor^{\CS^-}_\blt(\CS,\bk e^{\ulx})=H_\blt(\CS\otimes_\BK \CS^{-,!,\vee,\blt}).\] 

    In 1-color Soergel calculus, we have the relation
    \[\begin{diagram}
        % \draw[step=1,color=yellow] (0,0) grid (2,2);
        \draw(0,0)--(0,2);
        \draw(1,0)--(1,1);
        \fill (1,1) circle (5pt);
    \end{diagram}
    \ =\ 
    \begin{diagram}
        \draw(0,0)--(0,1);
        \fill (0,1) circle (5pt);
        \draw(1,0)--(1,2);
    \end{diagram}
    \ +\ 
    \begin{diagram}
        \draw(0,0) arc (180:0:0.5);
        \draw(0.5,2)--(0.5,1);
        \fill (0.5,1) circle (5pt);
    \end{diagram}
    \ -\ 
    \begin{diagram}
        \draw(0,0) arc (180:0:0.5 and 1);
        \draw(0.5,1)--(0.5,2);
        \draw(-0.5,0.65)--(-0.5,1.35);
        \fill (-0.5,0.65) circle (5pt);
        \fill (-0.5,1.35) circle (5pt);
    \end{diagram}
    \]
    In the cyclotomic quotient the last term dies. In particular this means that, in the cyclotomic quotient, we have
    \[\begin{diagram}
        \draw(0,0) arc (180:0:1 and 2);
        \draw(1,0)--(1,1);
        \fill (1,1) circle (5pt);
    \end{diagram}
    \ =\ 
    \begin{diagram}
        \draw(0,0)--(0,1);
        \fill (0,1)circle(5pt);
        \draw(1,0) arc (180:0:0.5 and 1);
    \end{diagram}
    \ +\ 
    \begin{diagram}
        \draw(0,0) arc (180:0:0.5 and 1);
        \draw(2,0)--(2,1);
        \fill (2,1)circle(5pt);
    \end{diagram}.\]

    Then consider the element
    \[x=\begin{diagram}
        \draw[orange](0,0)--(3,0);
        \draw[orange](0,3)--(3,3);
        \draw[orange,dashed](0,1.5)--(3,1.5);
        \draw(0.5,0)--(0.5,1.5);
        \draw(0.5,1.5) arc (180:0:1);
        \draw(2.5,0)--(2.5,1.5);
        \draw(1.5,0)--(1.5,1);
        \fill(1.5,1) circle (5pt);
    \end{diagram}
    \ - \ 
    \begin{diagram}
        \draw[orange](0,0)--(3,0);
        \draw[orange](0,3)--(3,3);
        \draw[orange,dashed](0,1.5)--(3,1.5);
        \draw(0.5,0)--(0.5,1);
        \fill(0.5,1) circle (5pt);
        \draw(1.5,0)--(1.5,1.5);
        \draw(1.5,1.5) arc (180:0:0.5 and 1);
        \draw(2.5,0)--(2.5,1.5);
    \end{diagram}
    \ - \ 
    \begin{diagram}
    \draw[orange](0,0)--(3,0);
        \draw[orange](0,3)--(3,3);
        \draw[orange,dashed](0,1.5)--(3,1.5);
    \begin{scope}[xscale=-1,shift={(-3,0)}]
        \draw(0.5,0)--(0.5,1);
        \fill(0.5,1) circle (5pt);
        \draw(1.5,0)--(1.5,1.5);
        \draw(1.5,1.5) arc (180:0:0.5 and 1);
        \draw(2.5,0)--(2.5,1.5);
    \end{scope}
    \end{diagram}
    \ \in \CS\otimes \CS^{-,!,\vee}_1,
    \]
    where the part of the diagram above the dotted orange line belongs to $\CS$, the part below it belongs to $\CS^{-,!,\vee}$, and the dotted orange line denotes the tensor symbol. Since the differential is given by the alternating sum over all ways of pulling a lollipop from the bottom into the top, we can see that 
    \[\d \pr*{\begin{diagram}
        \draw[orange](0,0)--(3,0);
        \draw[orange](0,3)--(3,3);
        \draw[orange,dashed](0,1.5)--(3,1.5);
        \draw(0.5,0)--(0.5,1.5);
        \draw(0.5,1.5) arc (180:0:1);
        \draw(2.5,0)--(2.5,1.5);
        \draw(1.5,0)--(1.5,1);
        \fill(1.5,1) circle (5pt);
    \end{diagram}
    \ - \ 
    \begin{diagram}
        \draw[orange](0,0)--(3,0);
        \draw[orange](0,3)--(3,3);
        \draw[orange,dashed](0,1.5)--(3,1.5);
        \draw(0.5,0)--(0.5,1);
        \fill(0.5,1) circle (5pt);
        \draw(1.5,0)--(1.5,1.5);
        \draw(1.5,1.5) arc (180:0:0.5 and 1);
        \draw(2.5,0)--(2.5,1.5);
    \end{diagram}
    \ - \ 
    \begin{diagram}
    \draw[orange](0,0)--(3,0);
        \draw[orange](0,3)--(3,3);
        \draw[orange,dashed](0,1.5)--(3,1.5);
    \begin{scope}[xscale=-1,shift={(-3,0)}]
        \draw(0.5,0)--(0.5,1);
        \fill(0.5,1) circle (5pt);
        \draw(1.5,0)--(1.5,1.5);
        \draw(1.5,1.5) arc (180:0:0.5 and 1);
        \draw(2.5,0)--(2.5,1.5);
    \end{scope}
    \end{diagram}} = \begin{diagram}
        \draw[orange](0,0)--(3,0);
        \draw[orange](0,3)--(3,3);
        \draw(0.5,0)--(0.5,1.5);
        \draw(0.5,1.5) arc (180:0:1);
        \draw(2.5,0)--(2.5,1.5);
        \draw(1.5,0)--(1.5,1);
        \fill(1.5,1) circle (5pt);
    \end{diagram}
    \ - \ 
    \begin{diagram}
        \draw[orange](0,0)--(3,0);
        \draw[orange](0,3)--(3,3);
        \draw(0.5,0)--(0.5,1);
        \fill(0.5,1) circle (5pt);
        \draw(1.5,0)--(1.5,1.5);
        \draw(1.5,1.5) arc (180:0:0.5 and 1);
        \draw(2.5,0)--(2.5,1.5);
    \end{diagram}
    \ - \ 
    \begin{diagram}
    \draw[orange](0,0)--(3,0);
        \draw[orange](0,3)--(3,3);
    \begin{scope}[xscale=-1,shift={(-3,0)}]
        \draw(0.5,0)--(0.5,1);
        \fill(0.5,1) circle (5pt);
        \draw(1.5,0)--(1.5,1.5);
        \draw(1.5,1.5) arc (180:0:0.5 and 1);
        \draw(2.5,0)--(2.5,1.5);
    \end{scope}
    \end{diagram}
    \ =\ 0,\] 
    so that $x\in\Ker\d$. On the other hand, $x$ cannot be a image of the differential, because the top of the diagram has no lollipops at all. Hence this element gives a nonzero element in $\Tor_1^{\CS^-}(\CS,\bk e^{\ulx})$, so that
    \[\Tor_1^{\CS^-}(\CS,\bk e^{\ulx})\neq 0,\]
    as claimed.
\end{EX*}

As a remark, let us note that the above proposition actually gives us a projective (in fact, free) resolution of the Verma module $\Delta_w$, in the shape of a (hyper)cube. Indeed, it says that
\begin{COR}
We have a resolution
\[\CS\otimes_\BK\CS^{-,!,\vee,\blt} e^w\simeq \Delta_w;\] 
as before, the maps are
\begin{align*}
        \d\colon \CS\otimes_\BK\CS^{-,!,\vee}_{k} e^w&\lto \CS\otimes_\BK\CS^{-,!,\vee}_{k-1} e^w\\
        x\otimes P(\lol_1 e^w,\cdotsc, \lol_n e^w)&\lmto \sum_{\te{ways to write }P=\lol_i\cdot P'} x\cdot\lol_i\otimes P',
    \end{align*}
where $P\in\CS^{-,!,\vee}_{k}$ is a degree $k$ anti-commutative polynomial in the lollipops. 
\end{COR}
This is not in conflict with the Kazhdan-Lusztig problem because in general it will be hard to write each of these projective terms in terms of indecomposable projectives; indeed, the decomposition coefficients are
\[(\CS e^{\ulx}:P_w)=\rnk_q\wan{\sq;\sq}_w{}^\ulx=\dim_q L_w{}^\ulx=\on{coeff}_{b_w}(b_{x_1}\cdots b_{x_k}),\] 
and computing the rank of the intersection form is not an easy task in general.

Together, Propositions \ref{prop:regseq} and \ref{prop:lollipopkoszul} tell us that
\begin{THM}\label{thm:SoergelnilKoszul}
    With $\CS_\lbd^-\subset\CS_\lbd$ as the nilalgebra, $\CS_\lbd$ is nil-Koszul.

    In fact, more generally, for an appropriate truncation $\CS$ (defined above) of any cell ideal of a cyclotomic Soergel calculus of any type and any rank, we have that $\CS$ is nil-Koszul with $\CS^-$ as the nilalgebra.
\end{THM}

% Lastly let us give a brief remark for why we chose to call this nilalgebra $\CS^-$ rather than $\CS^-$. Usually our philosophy would be to let $\tY$-diagrams in a triangular-based algebra generate a subalgebra. In particular, a subalgebra called $A^-$ should satisfy something like $A=A^+\otimes A^0\otimes A^-$. In this Soergel case we sort of had to guess $\CS^-$ manually -- in particular, it is \textit{not} the subalgebra generated by $\tY$-diagrams, namely (single) light leaves. For this reason we have avoided calling this subalgebra $\CS^-$. 

% \subsection{(Cyclotomic) Soergel is nil-Koszul}

\subsection{Extracting Kostant's theorem}
Proposition \ref{prop:regseq} justifies the following:
\[\rhom_\CS(\Delta_w,\sq)=\rhom_\CS(\CS\lotimes_{\CS^-} \bk e^w,\sq)=\rhom_{\CS^-}(\bk e^w,\sq).\]
We are interested in the particular case $\sq=L_1$. Since in the cyclotomic quotient every diagram in $e^1 \CS e^1$ is zero except for the empty diagram, we know $L_1=\bk e^1$. Hence
\[\Ext_\CS(\Delta_w,L_1)=\Ext_{\CS^-}(\bk e^w, \bk e^1)=e^w \CS^{-,!} e^1.\] 
Therefore, to compute the LHS, it suffices to compute the Koszul dual $\CS^{-,!}$. This is easy to do directly via the quadratic dual; since lollipops commute past each other vertically in $\CS^-$, we know that they anti-commute in $\CS^{-,!}$, so that visibly
% \[\dim e^w \CS^{-,!}e^1=q^{-\ell(w)}\bk,\]
% \[\dim_q e^w \CS^{-,!}e^1=q^{\ell(w)}\bk,\]
\[e^w\CS^{-,!}e^1=q^{-\ell(w)}\bk[-\ell(w)],\]
%%% what is the correct convention here?
where up to sign the homological degree matches with the Koszul degree since $\CS^-$ is Koszul. 
\begin{THM}[Kostant]\label{thm:kostant}
    For any $w\in W_\lbd$, we have
    \[\Ext_{\CS_\lbd}^\blt(\Delta_w,L_1)=\Ext_{\CS^-_\lbd}^\blt(\bk e^w, \bk e^1)=e^w \CS_\lbd^{-,!} e^1=q^{-\ell(w)}\bk[-\ell(w)].\]
    This space is spanned inside $\CS_\lbd^{-,!}$ by the following diagram:
%     \[\begin{diagram}
%     \draw[thick,orange] (0,-2.5)--(5.5,-2.5);
%     \draw[thick,orange] (0,4)--(5.5,4);
%     \draw (0.5,4)--(0.5,3);
%     \fill (0.5,3) circle (5pt);
%     \draw (1,4)--(1,2);
%     \fill (1,2) circle (5pt);
%     \draw (1.5,4)--(1.5,1);
%     \fill (1.5,1) circle (5pt);
%     \node at (3,2) { $\ddots$};
%     \draw (4.5,4)--(4.5,-0.5);
%     \fill (4.5,-0.5) circle (5pt);
%     \draw (5,4)--(5,-1.5);
%     \fill (5,-1.5) circle (5pt);
%     \node at (2.75,4.5) {\tiny $w$};
%     % \node at (5.7,4.1) {$!$};
%     \node at (5.7,4) {$\scriptstyle!$};
% \end{diagram}\]
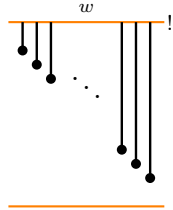
\[\begin{diagram}
    \draw[thick,orange] (0,-2.5)--(5.5,-2.5);
    \draw[thick,orange] (0,4)--(5.5,4);
    \draw (0.5,4)--(0.5,3);
    \fill (0.5,3) circle (5pt);
    \draw (1,4)--(1,2.5);
    \fill (1,2.5) circle (5pt);
    \draw (1.5,4)--(1.5,2);
    \fill (1.5,2) circle (5pt);
    \node at (2.75,2) { $\ddots$};
    \draw (4,4)--(4,-0.5);
    \fill (4,-0.5) circle (5pt);
    \draw (4.5,4)--(4.5,-1);
    \fill (4.5,-1) circle (5pt);
    \draw (5,4)--(5,-1.5);
    \fill (5,-1.5) circle (5pt);
    \node at (2.75,4.5) {\tiny $w$};
    % \node at (5.7,4.1) {$!$};
    \node at (5.7,4) {$\scriptstyle!$};
\end{diagram}\]
where the colors at the top boundary give the preferred redex for $w$.

More generally this is true for an appropriate truncation $\CS$ of any cell ideal of any cyclotomic Soergel calculus: if $W_\CS\subseteq W$ is the ideal corresponding to $\CS$, we have for any $w\in W_\CS$ that
\[\Ext^\blt_\CS(\Delta_w,L_1)=\Ext_{\CS^-}^\blt(\bk e^w,\bk e^1)=e^w\CS^{-,!}e^1=q^{-\ell(w)}\bk[-\ell(w)],\] 
being spanned by the same diagram as above. In the language of nilcohomology,
\[H^\blt(\CS^-:L_1)=\bigoplus_{w\in W_\CS} q^{-\ell(w)}\bk[-\ell(w)].\]

By Morita equivalence this guarantees 
\[\Ext^\blt_{\DCR_\lbd}(\Delta_w,L_1)=q^{-\ell(w)}\bk[-\ell(w)].\]
\end{THM}

% \vspace{-3em}
\section{BGG resolutions}\label{sect:BGG}
Finally let us invoke Theorem \ref{thm:triangularfiltration} to show the following. 
\begin{THM}\label{thm:JTBGG}
    There is a (finite) BGG resolution in $\Mod\DCR_\lbd$,
    \[0\lto \Delta_{w_\lbd}\lto  \cdots\lto \bigoplus_{\substack{\ell(w)=k\\w\in W_\lbd}}\Delta_w\lto\cdots\lto \Delta_1\lto L_1\lto 0.\]
    % \[0\lto \Delta_{w_\lbd}\lto  \cdots\lto \bigoplus_{w\in W_\lbd:\ell(w)=k}\Delta_w\lto\cdots\lto \Delta_1\lto L_1\lto 0.\]
    The differentials can be described explicitly, see Theorem \ref{thm:JTBGGmaps}. 
    
    When restricted to the action of $S_n$ via the map $\vphi_\lbd^\te{BK}\colon \BC S_n\linj \wh\CH_n\lsurj \wh\CH^\omega_\alpha\simlto \CR^\omega_\alpha\lsurj \DCR_\lbd$, this resolution becomes
    \[0\lto E_{w_\lbd\circ\lbd}\lto\cdots\lto \bigoplus_{\substack{\ell(w)=k\\w\in W_\lbd}}E_{w\circ\lbd} \lto\cdots\lto E_{\lbd}\lto \Sigma_\lbd\lto 0,\] 
    % \[0\lto E_{w_\lbd\circ\lbd}\lto\cdots\lto \bigoplus_{w\in W_\lbd:\ell(w)=k}E_{w\circ\lbd} \lto\cdots\lto E_{\lbd}\lto \Sigma_\lbd\lto 0,\] 
    where $E_\alpha$ is the permutation module of $S_n$ associated to the composition $\alpha$, and $\Sigma_\lbd$ is the usual simple Specht module of $S_n$ labeled by $\lbd$. Decategorifying this resolution via alternating sum of Frobenius character/cycle index series recovers the Jacobi-Trudi determinant formula,
    $$\tsl s_\lbd=\det(\tsl h_{\lbd_i-i+j})_{i,j}.$$
\end{THM}
\vspace{-0.5em}
\begin{PRF}
In the previous section it was shown that $\Ext_{\DCR_\lbd}^\blt(\Delta_w,L_1)=q^{-\ell(w)}\bk[-\ell(w)].$
Directly applying \ref{thm:triangularfiltration}, we can plug $L_1$ into the input in 
\[E_1^{p,q}=\bigoplus_{\ell(w)=-p}\Delta_w\otimes_{\bk e^w}\Ext^{-(p+q)}_{\DCR_\lbd}(\Delta_w,\sq^\dag)^\dag \specseqimplies E_\infty^{p,q}=\gr^{-p}H^{p+q}(\sq)\]
to obtain
\[E_1^{p,q}=\bigoplus_{\ell(w)=-p}\Delta_w\otimes_{\bk e^w}\Ext^{-(p+q)}_{\DCR_\lbd}(\Delta_w,L_1)^\dag \specseqimplies E_\infty^{p,q}=\gr^{-p}H^{p+q}(L_1).\]
Since $L_1$ is simple, the filtration on $L_1$ must be quite trivial. The concentration of the Ext groups as above guarantee us that the spectral sequence $E^{p,q}_1$ lies on the horizontal axis $q=0$. Hence the spectral sequence must converge on the second page, and this is precisely the BGG resolution. The terms of the resolution are the modules $\Delta_w$; we have already seen that $\Delta_w$ restricts to the permutation module $E_{w\circ\lbd}$ over $S_n$, so by a character argument, the module they resolve must restrict to the simple Specht module $\Sigma_\lbd$ over $S_n$. Hence $\Res^{\DCR_\lbd}_{\BC S_n} L_1=\Sigma_\lbd$.
% Hence already we have the following theorem.

But perhaps using character theory could be considered cheating. Another way to argue $\Res^{\DCR_\lbd}_{\BC S_n} L_1=\Sigma_\lbd$ is to note that since we are over characteristic zero, we know that the level-1 cyclotomic KLR $\CR_{\cont\lbd}^{\varpi_\delta}$ is isomorphic to the $\lbd$-block of $\BC S_n$; it is cellular with a single cell of size $\abs{\Std\lbd}$ and a single simple $L$, and $\Res^{\CR_{\cont\lbd}^{\varpi_\delta}}_{\BC S_n} L=\Sigma_\lbd$. On the other hand, we have the diagram
    \begin{center}
        \begin{tikzcd}[ampersand replacement = \&]
            \&\BC S_n\arrow{dr}{\te{BK}}\arrow[swap]{dl}{\te{BK}}\&\\
            \CR_{\cont\lbd}^{\varpi_\delta} \& \&\CR_{\cont\lbd}^{\varpi_\delta+\cdots+\varpi_{\delta-\ell+1}}\arrow[twoheadrightarrow]{ll}
        \end{tikzcd}\ ,
    \end{center}
    where the quotient $\CR_{\cont\lbd}^{\varpi_\delta+\cdots+\varpi_{\delta-\ell+1}}\lsurj \CR_{\cont\lbd}^{\varpi_\delta}$ is requiring that 
    \[\raisebox{-0.5em}{$\begin{diagram}
    \draw (0,0)--(0,2);
    % \fill (0,1) circle (5pt);
    % \node at (-1.2,1.2) {\tiny $\alpha_i^*(\omega)$};
    \node at (0,-0.5) {\tiny $i$};
    \draw (1,0)--(1,2);
    \node at (2.1,1) {$\cdots$};
    \draw (3,0)--(3,2);
\end{diagram}$}=0\]
    for $\delta-\ell+1\le i\le \delta-1$. Moreover the diagrams in $\CR_{\cont\lbd}^{\varpi_\delta}$ can also be considered as diagrams in $\CR_{\cont\lbd}^{\varpi_\delta+\cdots+\varpi_{\delta-\ell+1}}$; in particular the module $\CR_{\cont\lbd}^{\varpi_\delta}\actson L$ contains the vector $e_{\cont\tlie^\lbd}$. Also 
    \[\Res^{\CR_{\cont\lbd}^{\varpi_\delta}}_{\CR_{\cont\lbd}^{\varpi_\delta+\cdots+\varpi_{\delta-\ell+1}}} \Res^{\CR_{\cont\lbd}^{\varpi_\delta+\cdots+\varpi_{\delta-\ell+1}}}_{\BC S_n} L=\Res^{\CR_{\cont\lbd}^{\varpi_\delta}}_{\BC S_n} L=\Sigma_\lbd.\]
    Hence we have for free that $\Res^{\CR_{\cont\lbd}^{\varpi_\delta}}_{\CR_{\cont\lbd}^{\varpi_\delta+\cdots+\varpi_{\delta-\ell+1}}} L=L_1$ is simple, so that $\Res^{\DCR_\lbd}_{\BC S_n}L_1=\Sigma_\lbd$ is simple.
\end{PRF}
Hence the algebra $\DCR_\lbd$ is a quasi-hereditary algebra categorifying the Jacobi-Trudi determinant, as desired.

\subsection{BGG for Soergel}
Though our goal was to prove $\DCR_\lbd$ has a BGG resolution, we could have proven the same thing about the Soergel calculus $\CS_\lbd$. The argument is the same -- we stratify the module category over $\CS_\lbd$ using the idempotents $\{e^w\in\CS_\lbd:w\in W_\lbd\}$, get a filtration of the identity functor on $\D\Mod\CS_\lbd$, produce a functorial spectral sequence, and apply it to the simple $L_1$ of $\CS_\lbd$; due to Theorem \ref{thm:kostant},
this spectral sequence is a resolution, written with the same symbols as the resolution for $\DCR_\lbd$, except that these symbols now stand for modules over $\CS_\lbd$.
\begin{THM}\label{thm:SoergelBGG}
    There is a (finite) BGG resolution in $\Mod\CS_\lbd$, 
    %namely a resolution of the simple module $L_1$ by Vermas,
    \[0\lto \Delta_{w_\lbd}\lto  \cdots\lto \bigoplus_{\substack{\ell(w)=k\\ w\in W_\lbd}}\Delta_w\lto\cdots\lto \Delta_1\lto L_1\lto 0.\]
    This is coming from the spectral sequence in Theorem \ref{thm:triangularfiltration}, which in this case reads
    \[E_1^{p,0}=\bigoplus_{\substack{\ell(w)=-p\\w\in W_\lbd}}\Delta_w\otimes \pr*{e^w\CS_\lbd^{-,!} e^1[\ell(w)]}^\dag\specseqimplies E_2^{0,0}=L_1.\]
    Moreover the differentials can be described explicitly, see Theorem \ref{thm:SoergelBGGmaps}.

In fact, this more generally holds for an appropriate truncation $\CS$ of any cell ideal (corresponding to $W_\CS\subseteq W$) inside a cyclotomic Soergel calculus of any type and any rank -- for such an algebra we have the BGG resolution
\[\cdots\lto \bigoplus_{\substack{\ell(w)=k\\ w\in W_\CS}} \Delta_w\lto\cdots\lto \Delta_1\lto L_1\lto 0\]
with differentials the same as in Theorem \ref{thm:SoergelBGGmaps}. 
\end{THM}
% One can deduce that this must be the differential by unwinding the construction in Theorem \ref{thm:folklorefiltration}. 
% Alternatively, in this case, maps between Vermas are restricted enough (the space of homs is 1-dimensional) that it's easy to see this must be the differential. Let us also remark that the space $e^w\CS^{-,!}e^1[\ell(w)]$ had a $q$-shift, namely $q^{\ell(w)}$; with this shift, the differentials are then homogeneous maps. 
Let us remark that the 1-dimensional space $e^w\CS^{-,!}e^1[\ell(w)]$ had a $q$-shift, namely $q^{-\ell(w)}$; after taking the contragredient dual $\pr*{e^w\CS^{-,!}e^1[\ell(w)]}^\dag$, this $q$-shift becomes $q^{\ell(w)}$ -- with this shift, the differentials are then homogeneous maps. 

In this case, maps between Vermas are rather restricted. 
\begin{LEM}
    If $w\gtrdot u$, then any morphism $\Delta_w\lto\Delta_u$ must be given by right multiplication by a lollipop:
    \[\sq\cdot c_i\begin{diagram}
        \draw[orange](0,0)--(7,0);
        \draw[orange](0,2)--(7,2);
        \draw(0.5,0)--(0.5,2);
        \node at (1.6,1) {\tiny$\cdots$};
        \draw (2.5,0)--(2.5,2);
        \draw (3.5,2)--(3.5,1);
        \fill (3.5,1) circle (5pt);
        \draw(4.5,0)--(4.5,2);
        \node at (5.6,1) {\tiny$\cdots$};
        \draw (6.5,0)--(6.5,2);
        \node at (3.5,2.5) {\tiny $w$};
        \node at (3.5,-0.5) {\tiny $u$};
        % \node at (7.2,2) {$\scriptstyle!$};
    \end{diagram}\colon q\Delta_w\lto \Delta_{u}.\]
\end{LEM}
\vspace{-0.5em}
\begin{PRF}
    This is essentially a classical fact in category $\CO$. One way to see it diagrammatically is to note that $\hom_\CS(\Delta_w,\Delta_u)=\hom_\CS(\CS\otimes_{\CS^-}\bk e^w,\Delta_u)=\hom_{\CS^-}(\bk e^w,\Delta_u)=(\Delta_u^{\CS^-})^w$ is the $w$-weight space of the space of vectors in $\Delta_u$ killed by the positive-degree ideal in $\CS^-$. Vectors in $\Delta_u{}^w$ are spanned by single light leaves $\olLL_u^w$ which begin at $u$ and end at $w$. However, as $w\gtrdot u$, light leaf theory tells us there can only be one such diagram, namely the one above. It is a straightforward check that this vector of $\Delta_u$ is killed by $\CS^-$.
\end{PRF}
It is therefore not difficult to guess what must be the differential. However, there is a more conceptually satisfying way to deduce the differential, namely that of Koszul duality, which we take up next. 

\subsection{The Koszul perspective on BGG}\label{subsect:koszulperspective}
It was noted earlier that the Ext group information, i.e. the $\jmath^w\ol\imath_w^*$ piece, in the reconstruction spectral sequence looked suspiciously like the Koszul duality functor. Another hint is the strange signs appearing in the usual BGG resolution. In this subsection we elucidate this a little further -- it turns out that the BGG resolution can also be phrased in terms of Koszul duality. 

Roughly speaking, the spectral sequence from Theorem \ref{thm:triangularfiltration} in this case is sort of saying
    \[\te{``}\Delta_W\otimes_\BK \CK_{\CS^+}(\sq) \specseqimplies \Id\te{''},\] 
    where $\Delta_W=\bigoplus_w \Delta_w\ractson\BK$ and $\CK_{\CS^+}$ is the Koszul duality functor with respect to the subalgebra generated by upside-down lollipops, defined by
    % \[\CK_{\CS^+}=\sh(\BK\lotimes_{\CS^+}\refl\sq)=\sh\rhom_{\CS^-}(\BK,\refl\sq^\dag)^\dag,\]
    \[\CK_{\CS^+}=\sh\refl(\BK\lotimes_{\CS^+}\sq)=\sh\refl\rhom_{\CS^-}(\BK,\sq^\dag)^\dag.\]
    Here 
    \[\te{$\sh M=M[n]$ if $M$ is concentrated in Koszul degree $n$},\]
    and $\refl=\refl_\tK$ is the reflection of Koszul degrees, namely $$\refl(M)_j=M_{-j}.$$ 
    % Another reflection we will use occasionally is reflection of homological degree,
    % \[\refl^\tH(M)^i=M^{-i}.\]
    % by the way you cannot reflect in homological degree, it is not well behaved because the directions of the differentials get all fucked up 
    The discussion in this subsection will culminate in elucidating this. 

We had in Subsection \ref{subsect:nilkoszul} briefly described the Koszul dual algebra, but there is also a notion of the Koszul dual \textit{coalgebra}, defined by
\[A^\shrek\coloneqq A^{!,\vee},\]  
with
\[A_n^\shrek=\bigcap_i A_1^{\otimes i}\otimes\qidl\otimes A_1^{\otimes n-i-2}\]  
and comultiplication defined by
\[x_1\otimes\cdots\otimes x_n\lmto \sum_i (x_1\otimes\cdots\otimes x_i)\otimes(x_{i+1}\otimes\cdots\otimes x_n).\] 

\begin{EX*}
    For $\CS=\CS_\yd{1,1,1}$, there are four cells in the Soergel calculus corresponding to $1,s_1,s_2,s_2s_1$. The subalgebra $\CS^+$ is spanned by
    \[\begin{diagram}
        \draw[orange](0,0)--(2,0);
          \draw[orange](0,2)--(2,2);
    \end{diagram}\ ,\ 
    \begin{diagram}
     \draw[orange](0,0)--(2,0);
      \draw[orange](0,2)--(2,2);
      \draw(1,0)--(1,2);
 \end{diagram}\ , \ 
 \begin{diagram}
     \draw[orange](0,0)--(2,0);
      \draw[orange](0,2)--(2,2);
      \draw[red](1,0)--(1,2);
 \end{diagram}\ , \ 
    \begin{diagram}
     \draw[orange](0,0)--(2,0);
      \draw[orange](0,2)--(2,2);
      \draw[red](0.5,0)--(0.5,2);
      \draw(1.5,0)--(1.5,2);
 \end{diagram}\ ,\ 
    \begin{diagram}
          \draw[orange](0,0)--(2,0);
          \draw[orange](0,2)--(2,2);
          \draw(1,2)--(1,2-0.67);
          \fill(1,2-0.67) circle (5pt); 
      \end{diagram}\ , \ 
      \begin{diagram}
          \draw[orange](0,0)--(2,0);
          \draw[orange](0,2)--(2,2);
          \draw[red](1,2)--(1,2-0.67);
          \fill[red](1,2-0.67) circle (5pt); 
      \end{diagram}\ ,\ 
      \begin{diagram}
     \draw[orange](0,0)--(2,0);
      \draw[orange](0,2)--(2,2);
      \draw(1.5,0)--(1.5,2);
      \draw[red](0.5,2)--(0.5,2-0.67);
      \fill[red] (0.5,2-0.67) circle (5pt); 
 \end{diagram}\ , \ 
 \begin{diagram}
     \draw[orange](0,0)--(2,0);
      \draw[orange](0,2)--(2,2);
      \draw[red](0.5,0)--(0.5,2);
      \draw(1.5,2)--(1.5,2-0.67);
      \fill (1.5,2-0.67) circle (5pt); 
 \end{diagram}\ , \ 
      \begin{diagram}
          \draw[orange](0,0)--(2,0);
          \draw[orange](0,2)--(2,2);
          \draw[red](0.5,2-0)--(0.5,2-0.67);
          \fill[red] (0.5,2-0.67) circle (5pt);
          \draw(1.5,2-0)--(1.5,2-0.67);
          \fill (1.5,2-0.67) circle (5pt);
      \end{diagram}
 \ . \] 
 The comultiplication on the Koszul dual coalgebra $\CS^{+,\shrek}$ sends e.g. (again the dashed orange lines stand in for tensor symbols)
 \begin{align*}
     \Delta\colon \begin{diagram}
        \draw[orange] (0/2,0/2)--(4/2,0/2);
        \draw[orange] (0/2,4/2)--(4/2,4/2);
        \draw[orange,dashed] (0/2,2/2)--(4/2,2/2);
        \draw[red](1/2,2-0/2)--(1/2,1/2);
        \fill[red] (1/2,1/2) circle (5pt);
        \draw(3/2,2-0/2)--(3/2,3/2);
        \fill (3/2,3/2) circle (5pt);
    \end{diagram}^\shrek
    \,-\;
    \begin{diagram}
        \draw[orange] (0/2,0/2)--(4/2,0/2);
        \draw[orange] (0/2,4/2)--(4/2,4/2);
        \draw[orange,dashed] (0/2,2/2)--(4/2,2/2);
        \draw[red](1/2,2-0/2)--(1/2,3/2);
        \fill[red] (1/2,3/2) circle (5pt);
        \draw(3/2,2-0/2)--(3/2,1/2);
        \fill (3/2,1/2) circle (5pt);
    \end{diagram}^\shrek
    &\lmto
    \bigg(\begin{diagram}
        \draw[orange] (0/2,0/2)--(4/2,0/2);
        \draw[orange] (0/2,4/2)--(4/2,4/2);
        \draw[orange,dashed] (0/2,2/2)--(4/2,2/2);
        \draw[red](1/2,2-0/2)--(1/2,1/2);
        \fill[red] (1/2,1/2) circle (5pt);
        \draw(3/2,2-0/2)--(3/2,3/2);
        \fill (3/2,3/2) circle (5pt);
    \end{diagram}^\shrek
    \,-\;
    \begin{diagram}
        \draw[orange] (0/2,0/2)--(4/2,0/2);
        \draw[orange] (0/2,4/2)--(4/2,4/2);
        \draw[orange,dashed] (0/2,2/2)--(4/2,2/2);
        \draw[red](1/2,2-0/2)--(1/2,3/2);
        \fill[red] (1/2,3/2) circle (5pt);
        \draw(3/2,2-0/2)--(3/2,1/2);
        \fill (3/2,1/2) circle (5pt);
    \end{diagram}^\shrek\bigg)
    \otimes \begin{diagram}
        \draw[orange] (0/2,0/2)--(4/2,0/2);
        \draw[orange] (0/2,4/2)--(4/2,4/2);
        \end{diagram}^\shrek \\
        &\qquad\qquad 
        +\;\begin{diagram}
            \draw[orange](0,0)--(2,0);
            \draw[orange](0,2)--(2,2);
            \draw[red](0.5,2)--(0.5,0);
            \draw(1.5,2)--(1.5,1);
            \fill(1.5,1)circle(5pt);
        \end{diagram}^\shrek \otimes
        \begin{diagram}
            \draw[orange](0,0)--(2,0);
            \draw[orange](0,2)--(2,2);
            \draw[red](0.5,2)--(0.5,1);
            \fill[red](0.5,1)circle(5pt);
        \end{diagram}^\shrek
        \,-\;
        \begin{diagram}
            \draw[orange](0,0)--(2,0);
            \draw[orange](0,2)--(2,2);
            \draw[red](0.5,2)--(0.5,1);
            \fill[red](0.5,1)circle(5pt);
            \draw(1.5,2)--(1.5,0);
        \end{diagram}^\shrek \otimes
        \begin{diagram}
            \draw[orange](0,0)--(2,0);
            \draw[orange](0,2)--(2,2);
            \draw(1.5,2)--(1.5,1);
            \fill(1.5,1)circle(5pt);
        \end{diagram}^\shrek\\
        &\qquad\qquad\quad
        +\; \begin{diagram}
            \draw[orange](0,0)--(2,0);
            \draw[orange](0,2)--(2,2);
            \draw[red](0.5,2)--(0.5,0);
            \draw(1.5,2)--(1.5,0);
        \end{diagram}^\shrek\otimes \bigg(\begin{diagram}
        \draw[orange] (0/2,0/2)--(4/2,0/2);
        \draw[orange] (0/2,4/2)--(4/2,4/2);
        \draw[orange,dashed] (0/2,2/2)--(4/2,2/2);
        \draw[red](1/2,2-0/2)--(1/2,1/2);
        \fill[red] (1/2,1/2) circle (5pt);
        \draw(3/2,2-0/2)--(3/2,3/2);
        \fill (3/2,3/2) circle (5pt);
    \end{diagram}^\shrek
    \,-\;
    \begin{diagram}
        \draw[orange] (0/2,0/2)--(4/2,0/2);
        \draw[orange] (0/2,4/2)--(4/2,4/2);
        \draw[orange,dashed] (0/2,2/2)--(4/2,2/2);
        \draw[red](1/2,2-0/2)--(1/2,3/2);
        \fill[red] (1/2,3/2) circle (5pt);
        \draw(3/2,2-0/2)--(3/2,1/2);
        \fill (3/2,1/2) circle (5pt);
    \end{diagram}^\shrek\bigg)
 \end{align*} 
\end{EX*}

If $A$ were a suitable triangular-based algebra such as $\CS$, in particular with an appropriate anti-involution such as the one flipping diagrams upside-down in $\CS$ which gave an (anti-)isomorphism $\tau(\CS^{+})\cong\CS^-$, then we have
\[A^{\pm,\vee}=\refl A^{\mp}.\] 
% Moreover, since $A^{\pm,!,\vee}=\refl^\tH A^{\pm,\vee,!}$, we have 
% Moreover, since $A^{\pm,!,\vee}=\sh\sh A^{\pm,\vee,!}$, we have 
% % \[A^{\pm,\shrek}\cong \refl^\tH\refl_\tK A^{\mp,!},\] 
% \[A^{\pm,\shrek}\cong \sh\sh\refl A^{\mp,!},\] 
Moreover, since $A^{\pm,!,\vee}=\refl A^{\pm,\vee,!,\dag}=\sh\sh A^{\pm,\vee,!}$, we have 
% \[A^{\pm,\shrek}\cong \refl^\tH\refl_\tK A^{\mp,!},\] 
\[A^{\pm,\shrek}\cong A^{\mp,!,\dag}\cong \sh\sh\refl A^{\mp,!},\] 
so that $A^{\pm,!}$ and $A^{\mp,\shrek}$ are both algebras and coalgebras. 

One perspective on Koszul duality, for example in \cite{positselski2023survey}, is that Koszul duality is a functor (equivalence) between derived dg-modules over an algebra and coderived dg-comodules over its Koszul dual coalgebra:
\begin{center}
    \begin{tikzcd}
        \D\Mod^\te{dg} A \arrow[rr, "A^{\shrek }\otimes_\BK^\tau \sq",bend left] & \ \ \ \sim & \coD\on\coMod^\te{dg} A^\shrek \arrow[ll,"A\otimes_\BK^\tau\sq", bend left]
\end{tikzcd}
\end{center}
Koszul duality then tells us that
\[\Id_{\D\Mod^\te{dg} A}\simeq A\otimes^\tau_\BK A^\shrek\otimes^\tau_\BK \sq.\] 

%\[\shrek ! \vrot{!}{180} A^! A^{\shrek !} \otimes_{A^{\shrek !}}\]

Let us briefly explain the sign rules for the twisted tensor product $\otimes^\tau$. Given $A\actson M$, and letting $C=A^\shrek$ be the Koszul dual coalgebra, we have that $C\otimes^\tau M$ is a dg $C$-comodule where the coaction is on the first entry and the differential is:
\[\d^\tau (c\otimes v)=\d(c)\otimes v+(-1)^{|c|}c\otimes\d(v)+(-1)^{|c_{(1)}|}c_{(1)}\otimes\tau(c_{(2)})v.\] 
Similarly, given a comodule $C\coactson N$, we obtain $A\otimes^\tau N$ a dg $A$-module, where the action is on the first entry and the differential is
\[\d^\tau(a\otimes u)=\d(a)\otimes u+(-1)^{|a|}a\otimes \d(u)+(-1)^{|a|+1}a\tau(u_{(-1)})\otimes u_{(0)}.\] 

Let us also briefly explain what the coderived category is. The short thing to say is that it is the localization of the category of (cocomplete, meaning that $N$ is the union of the kernels of $N\lto \ol C^{\otimes n}\otimes N$) dg $C$-comodules at the class of morphisms which become quasi-isomorphisms under the functor $A\otimes^\tau\sq$. The longer thing to say is that the coderived category is the quotient category of the homotopy category of dg comodules by the minimal triangulated subcategory closed under infinite direct sums which contains the totalization comodules of all exact triples of $C$-comodules. 

% \warn{...more exposition on koszul here?}
%\warn{i dont think you need to because it seems that in the end you do not actually appeal to koszul duality}

The idea is that Koszul duality with respect to $\CS^+$,
\[\Res_{\CS^+}\simeq \CS^+\otimes^\tau_\BK \CS^{+,\shrek}\otimes^\tau_\BK \sq,\] 
should recover the BGG resolution, but this cannot be quite correct as for example $\CS^+$ does not contain the trivalent vertex. To fix this we will need to introduce a variant of the ``Jucys-Murphy'' algebra, which was introduced in \cite{ryom2020jucys}. Let 
\[\CS^{+\otimes -}\coloneqq \CS^+\otimes_\BK \CS^-.\]
Let $\CS^{+\otimes -}\actson \CS^+$ by having $\CS^+\otimes 1$ act via multiplication in $\CS^+$ and $1\otimes \CS^-$ act by zero. 
\begin{LEM}
    The algebra $\CS^{+\otimes -}$ has the following property:
    \[\CS\lotimes_{\CS^{+\otimes -}}\CS^+\otimes_\BK^\tau \bk e^w\simeq\Delta_w.\] 
\end{LEM}
\vspace{-0.5em}
\begin{PRF}
    Indeed, 
    \begin{align*}
    \CS\lotimes_{\CS^{+\otimes -}}\CS^+\otimes_\BK^\tau \bk e^w&=\CS\lotimes_{\CS^{+\otimes -}}\CS^+\otimes_\BK \bk e^w\\
    &\simeq \CS\lotimes_{\CS^{+\otimes -}}\CS^+\otimes_\BK \CS^-\otimes_\BK \CS^{-,!,\vee,\blt}e^w\\
    &=\CS\otimes_\BK \CS^{-,!,\vee,\blt}e^w\\
    &\simeq \Delta_w,
\end{align*}
where $\CS^+\otimes_\BK^\tau \bk e^w$ could be replaced with $\CS^+\otimes_\BK\bk e^w$ since there is no differential on either factor and the coproduct on $\bk e^w$ is trivial.  
\end{PRF}
This motivates us to define the ``universal Verma'' as 
\[\Delta_W\coloneqq \bigoplus_{w\in W}\Delta_w=\CS\lotimes_{\CS^{+\otimes -}}\CS^+\otimes_\BK^\tau \bk e^W,\] 
where $e^W=\sum_{w\in W}e^w$.

Note also that, since the differential in the Koszul resolution $\BK\simeq A^{+,\shrek}\otimes_\BK A^+$ is precisely the one prescribed by the twisted tensor product $\otimes^\tau$, we have 
% \begin{align*}
%     \CK_{A^+}&=\sh(\BK\lotimes_{A^+}\refl \sq)\\
%     &=\sh(A^{+,\shrek}\otimes_\BK^\tau A^+\lotimes_{A^+}\refl \sq)\\
%     &=\sh(A^{+,\shrek}\otimes_\BK^\tau \refl \sq),
% \end{align*}
\begin{align*}
    \CK_{A^+}&=\sh\refl(\BK\lotimes_{A^+} \sq)\\
    &=\sh\refl(A^{+,\shrek}\otimes_\BK^\tau A^+\lotimes_{A^+} \sq)\\
    &=\sh\refl(A^{+,\shrek}\otimes_\BK^\tau \sq),
\end{align*}
so that the BGS definition of Koszul duality agrees with Positselski, up to the shear and reflection. 
\begin{remark}
    The reader familiar with the conventions of \cite{beilinson1996koszul} will note that their construction of the Koszul duality functor has the reflection appearing in a different place; in particular, it is not next to the shear there. This difference is because in their convention, the Koszul dual $A^!$ is positively graded in the Koszul grading, whereas we have taken the opposite convention in Section \ref{subsect:nilkoszul}. 
\end{remark}

In relation to the homological data we were computing previously, let us note that
\begin{align*}
    \jmath^w \ol\imath_w^*&=\rhom_A(\Delta_w,\sq^\dag)^\dag\\
    &=e^w A^{\ge w}\lotimes_A \sq\\
    &=\bk e^w\lotimes_{A^+} A\lotimes_A \sq\\
    &=\bk e^w\lotimes_{A^+} \sq\\
    &=e^w A^{+,\shrek}\otimes_\BK^\tau A^+ \lotimes_{A^+}\sq\\
    &=e^w A^{+,\shrek}\otimes_\BK^\tau \sq.
\end{align*}
If one ignores the shear and the reflection, this is almost $e^w \CK_{A^+}$. To get the terms of the first page $E_1$ of the spectral sequence from Theorem \ref{thm:triangularfiltration}, we also need to tensor this by the appropriate Verma. Using our description of $\Delta_w$ above, this would look like 
\[\CS\lotimes_{\CS^{+\otimes -}} \CS^+\otimes_\BK^\tau e^W \CS^{+,\shrek}\otimes^\tau_\BK\sq =\Delta_W\otimes_\BK (\CS^{+,\shrek}\otimes_\BK^\tau \sq)\approx \Delta_W\otimes_\BK \CK_{\CS^+}\sq,\] 
% Since this functor preserves Vermas, and since the category $\D\Mod\CS$ is generated by Vermas, we know this functor must be the identity:
% \[\Id_{\D\Mod\CS}\simeq \CS\lotimes_{\CS^{+\otimes -}} \CS^+\otimes_\BK^\tau e^W \CS^{+,\shrek}\otimes^\tau_\BK\sq.\] 
where we use the $\approx$ symbol due to the shear and the reflection, which are mostly cosmetic in order to match the spectral sequence convention. 

The above observations motivate us to study the functor $\Delta_W\otimes_\BK \CK_{\CS^+}$. Given a complex $M^\blt\in \D^-\Mod\CS$, place it on a 2-dimensional grid with the horizontal axis labeled by the Koszul degree and the vertical axis labeled by the homological degree by placing $e^{\ulx} M^q$ in position $(\ell(\ulx),q)$, where $\ell(\ulx)$ is the length of the word $\ulx$, which might not be reduced. Let us further assume that $M^\blt$ lies in a cone\footnote{This is, up to convention shift, the same cone as \cite{beilinson1996koszul}.}, namely 
\[M_p^q=0\qquad\te{if }q\gg 0\te{ or }p+q\ll 0.\]
Here is a schematic picture:
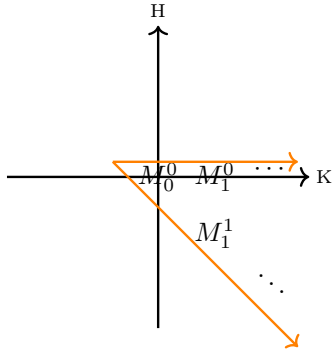
\begin{center}
        \begin{tikzpicture}
            \draw[->] (-2,0)--(2,0);
            \draw[->] (0,-2)--(0,2);
            \node at (2.2,0) {\tiny$\tK$};
            \node at (0,2.2) {\tiny$\tH$};
            % \node at (0.25,0.2) {\small $M_0^0$};
            % \node at (0.75,0.2) {\small $M_1^0$};
            % \node at (0.75,-0.5) {\small $M_1^0$};
            \node at (0,0) {\small $M_0^0$};
            \node at (0.75,0) {\small $M_1^0$};
            \node at (0.75,-0.75) {\small $M_1^1$};
            \node at (1.5,0.1) {\small $\cdots$};
            \node at (1.5,-1.3) {\small $\ddots$};
            \draw[orange,->] (-0.6,0.2)--(1.85,0.2);
            \draw[orange,->] (-0.6,0.2)--(1.85,0.2-2.45);
        \end{tikzpicture}
\end{center}
One can then apply $\CK_{\CS^+}$ to this. The shear and reflection is done with respect to this 2-dimensional grid, so that e.g. reflection reflects across the vertical axis. 

Now let $E_1=E_1(\calid)$ be the spectral sequence from Theorem \ref{thm:triangularfiltration} applied to $\D^-\Mod\CS$. Let $e^k\coloneqq \sum_{\ell(w)=k}e^w$, and $\Delta_k=\bigoplus_{\ell(w)=k} \Delta_w$. Then $E_1^{p,q}=\Delta_{-p}\otimes (E_1')^{p,q}$, where 
\begin{align*}
    (E_1')^{p,q}&=\Ext_\CS^{-(p+q)}(\Delta_{-p},M^\dag)^\dag\\
    &=H^{p+q}(\rhom_\CS(\Delta_{-p},M^\dag)^\dag)\\
    &=H^{p+q}(e^{-p}\CS^{+,\shrek}\otimes^\tau M).
\end{align*}
On the other hand, we can check what the $(p,q)$-term of $e^W\CK_{\CS^+}(M)$ is:
\begin{align*}
    (e^W\CK_{\CS^+}(M))^{p,q}&=H^q(e^{-p}\sh\refl \CS^{+,\shrek}\otimes^\tau M)\\
    &=H^{p+q}(e^{-p}\refl \CS^{+,\shrek}\otimes^\tau M)\\
    &=H^{p+q}(e^{-p}\CS^{+,\shrek}\otimes^\tau M). 
\end{align*}
Hence one sees that, at least on the level of objects on the first page,
\begin{LEM}
    \[E_1^{p,q}(\calid(M))=(\Delta_W\otimes_\BK \CK_{\CS^+}(M))^{p,q}.\] 
\end{LEM}

In fact, we can even say something about what this does to maps between objects. 
\begin{LEM}\label{lem:preserveverma}
    The functor $\Delta_W\otimes_\BK \CS^{+,\shrek}\otimes^\tau\sq$ preserves Vermas.
    % \[\Delta_W\otimes_\BK \CK_{\CS^+}(\Delta_w)=\Delta_w.\]

    % Moreover, if $w\gtrdot u$ and we have a map $\Delta_w\lto\Delta_u$, the functor will preserve this map.
    Moreover, if $\ell(w_j)=k+1$ and $\ell(u_i)=k$ and we have a map $\bigoplus_j\Delta_{w_j}\lto\bigoplus_i \Delta_{u_i}$, the functor will preserve this map.
\end{LEM}
\begin{adjustwidth}{2em}{0pt}
\begin{proof}
    We have
\begin{align*}
    e^W \CS^{+,\shrek}\otimes^\tau_\BK \Delta_w&=\bigoplus_{u\in W} \rhom_\CS(\Delta_u,\Delta_w^\dag)^\dag\\
    &=\bk e^w,
\end{align*}
so that together we know 
\[\CS\lotimes_{\CS^{+\otimes -}} \CS^+\otimes_\BK^\tau e^W \CS^{+,\shrek}\otimes^\tau_\BK(\Delta_w)\simeq \Delta_w.\]

For the claim about maps, it suffices to prove the claim in the case that there is only one $w=w_j$. Let $0\lto q\Delta_{w}\lto \bigoplus_i\Delta_{u_i}\lto M\lto 0$ (note we have included the $q$-shift so that the map is homogeneous), where the first map is given by some scalars multiple of right multiplications by the appropriate lollipops:
% \footnote{The fact that the maps must be scalar multiples of such lollipops is a classical fact in category $\CO$; one way to see it diagrammatically is to note that $\hom_\CS(\Delta_w,\Delta_u)=\hom_\CS(\CS\otimes_{\CS^-}\bk e^w,\Delta_u)=\hom_{\CS^-}(\bk e^w,\Delta_u)=(\Delta_u^{\CS^-})^w$ is the $w$-weight space of the space of vectors in $\Delta_u$ killed by the positive-degree ideal in $\CS^-$. Vectors in $\Delta_u{}^w$ are spanned by single light leaves $\olLL_u^w$ which begin at $u$ and end at $w$. However, as $w\gtrdot u$, light leaf theory tells us there can only be one such diagram, namely the one above. It is a straightforward check that this vector of $\Delta_u$ is killed by $\CS^-$.}:
\[\sq\cdot c_i\begin{diagram}
        \draw[orange](0,0)--(7,0);
        \draw[orange](0,2)--(7,2);
        \draw(0.5,0)--(0.5,2);
        \node at (1.6,1) {\tiny$\cdots$};
        \draw (2.5,0)--(2.5,2);
        \draw (3.5,2)--(3.5,1);
        \fill (3.5,1) circle (5pt);
        \draw(4.5,0)--(4.5,2);
        \node at (5.6,1) {\tiny$\cdots$};
        \draw (6.5,0)--(6.5,2);
        \node at (3.5,2.5) {\tiny $w$};
        \node at (3.5,-0.5) {\tiny $u_i$};
        % \node at (7.2,2) {$\scriptstyle!$};
    \end{diagram}\colon q\Delta_w\lto \Delta_{u_i}.\]
Recall that the differential maps in $\Delta_W\otimes \CK_{\CS^+}(M)$ arise from the comultiplication on $e^x\CS^{+,\shrek}\otimes^\tau M=\rhom(\Delta_x,M^\dag)^\dag$. However, the SES above gives the LES
\begin{center}
    \begin{tikzcd}
       0\arrow{r}& \hom(\Delta_x,M^\dag) \arrow{r} & \hom(\Delta_x,\bigoplus_i\!\Nabla_{u_i}) \snakeanchor\arrow{r}& \hom(\Delta_x,q^{-1}\Nabla_w)\snakearrow{}  \\
        &   \Ext^1(\Delta_x,M^\dag) \arrow{r} &\cdots &\
    \end{tikzcd}
\end{center}
which tells us that 
\[ \Ext(\Delta_x,M^\dag)=\begin{cases} \bk[0] & x=u_i\\ q^{-1}\bk[-1] & x=w \\ 0 &\te{else}\end{cases},\] 
so that
\[H^\blt(e^x\CS^{+,\shrek}\otimes^\tau M)=\begin{cases} \bk[0] & x=u_i\\ q\bk[1] & x=w \\ 0 &\te{else}\end{cases}.\]
Degree consideration tells us that in the first case these must be spanned by the propagating strings idempotents. In the second case, let us consider the element
\[a^\shrek\otimes v\coloneqq \pr*{\sum_i c_i\begin{diagram}
        \draw[orange](0,0)--(7,0);
        \draw[orange](0,2)--(7,2);
        \draw(0.5,0)--(0.5,2);
        \node at (1.6,1) {\tiny$\cdots$};
        \draw (2.5,0)--(2.5,2);
        \draw (3.5,2)--(3.5,1);
        \fill (3.5,1) circle (5pt);
        \draw(4.5,0)--(4.5,2);
        \node at (5.6,1) {\tiny$\cdots$};
        \draw (6.5,0)--(6.5,2);
        \node at (3.5,2.5) {\tiny $w$};
        \node at (3.5,-0.5) {\tiny $u_i$};
        \node at (7.2,2) {$\scriptstyle\shrek$};
    \end{diagram} }\otimes \pr*{ \begin{diagram}
        \draw[orange](0,0)--(7,0);
        \draw[orange](0,2)--(7,2);
        \draw(0.5,0)--(0.5,2);
        \node at (1.6,1) {\tiny$\cdots$};
        \draw (2.5,0)--(2.5,2);
        % \draw (3.5,2)--(3.5,0);
        % \fill (3.5,1) circle (5pt);
        \draw(4.5,0)--(4.5,2);
        \node at (5.6,1) {\tiny$\cdots$};
        \draw (6.5,0)--(6.5,2);
        \node at (3.5,2.5) {\tiny $u_i$};
        \node at (3.5,-0.5) {\tiny $u_i$};
        % \node at (7.2,2) {$\scriptstyle!$};
        \node at (7.5,0) {$\scriptstyle M$};
    \end{diagram}  }_i\in e^w \CS^{+,\shrek}\otimes^\tau M,\]
where the subscript $M$ indicates that the diagram is to be thought of as a vector in $M$ (which is a tuple over $i$); as there is no differential in $M$ or $\CS^{+,\shrek}$, it is clear that $a^\shrek\otimes v\in \ker \d^\tau$, as $\d^\tau(a^\shrek\otimes v)\in \img(\Delta_w\to \bigoplus_i\Delta_{u_i})$. Moreover $a^\shrek\otimes v\not\in \img \d^\tau$ because $v$ is not in the image of $\CS^+_1\cdot M$. The nonzero element $a^\shrek\otimes v\in H^{-1}(e^w\CS^{+,\shrek}\otimes^\tau M)$ must then span.

So in conclusion $H^\blt(e^x\CS^{+,\shrek}\otimes^\tau M)$ must be spanned by\footnote{We are implicitly abusing notation here, suppressing the second tensor factor.}
\begin{align*}
\begin{diagram}
        \draw[orange](0,0)--(7,0);
        \draw[orange](0,2)--(7,2);
        \draw(0.5,0)--(0.5,2);
        \node at (1.6,1) {\tiny$\cdots$};
        \draw (2.5,0)--(2.5,2);
        % \draw (3.5,2)--(3.5,0);
        % \fill (3.5,1) circle (5pt);
        \draw(4.5,0)--(4.5,2);
        \node at (5.6,1) {\tiny$\cdots$};
        \draw (6.5,0)--(6.5,2);
        \node at (3.5,2.5) {\tiny $u_i$};
        \node at (3.5,-0.5) {\tiny $u_i$};
        \node at (7.2,2) {$\scriptstyle\shrek$};
    \end{diagram}&\qquad\te{if }x=u_i\te{ for some }i,\\
    \sum_i c_i\begin{diagram}
        \draw[orange](0,0)--(7,0);
        \draw[orange](0,2)--(7,2);
        \draw(0.5,0)--(0.5,2);
        \node at (1.6,1) {\tiny$\cdots$};
        \draw (2.5,0)--(2.5,2);
        \draw (3.5,2)--(3.5,1);
        \fill (3.5,1) circle (5pt);
        \draw(4.5,0)--(4.5,2);
        \node at (5.6,1) {\tiny$\cdots$};
        \draw (6.5,0)--(6.5,2);
        \node at (3.5,2.5) {\tiny $w$};
        \node at (3.5,-0.5) {\tiny $u_i$};
        \node at (7.2,2) {$\scriptstyle\shrek$};
    \end{diagram}&\qquad\te{if }x=w.
\end{align*}
It is then clear that the differential $(\tau\otimes 1)\circ\Delta$ coming from comultiplication, sending 
\[(\tau\otimes 1)\circ\Delta\colon \sum_i c_i\begin{diagram}
        \draw[orange](0,0)--(7,0);
        \draw[orange](0,2)--(7,2);
        \draw(0.5,0)--(0.5,2);
        \node at (1.6,1) {\tiny$\cdots$};
        \draw (2.5,0)--(2.5,2);
        \draw (3.5,2)--(3.5,1);
        \fill (3.5,1) circle (5pt);
        \draw(4.5,0)--(4.5,2);
        \node at (5.6,1) {\tiny$\cdots$};
        \draw (6.5,0)--(6.5,2);
        \node at (3.5,2.5) {\tiny $w$};
        \node at (3.5,-0.5) {\tiny $u_i$};
        \node at (7.2,2) {$\scriptstyle\shrek$};
    \end{diagram}\lmto \sum_i c_i\begin{diagram}
        \draw[orange](0,0)--(7,0);
        \draw[orange](0,2)--(7,2);
        \draw(0.5,0)--(0.5,2);
        \node at (1.6,1) {\tiny$\cdots$};
        \draw (2.5,0)--(2.5,2);
        \draw (3.5,2)--(3.5,1);
        \fill (3.5,1) circle (5pt);
        \draw(4.5,0)--(4.5,2);
        \node at (5.6,1) {\tiny$\cdots$};
        \draw (6.5,0)--(6.5,2);
        \node at (3.5,2.5) {\tiny $w$};
        \node at (3.5,-0.5) {\tiny $u_i$};
        % \node at (7.2,2) {$\scriptstyle!$};
    \end{diagram}\otimes \begin{diagram}
        \draw[orange](0,0)--(7,0);
        \draw[orange](0,2)--(7,2);
        \draw(0.5,0)--(0.5,2);
        \node at (1.6,1) {\tiny$\cdots$};
        \draw (2.5,0)--(2.5,2);
        % \draw (3.5,2)--(3.5,0);
        % \fill (3.5,1) circle (5pt);
        \draw(4.5,0)--(4.5,2);
        \node at (5.6,1) {\tiny$\cdots$};
        \draw (6.5,0)--(6.5,2);
        \node at (3.5,2.5) {\tiny $u_i$};
        \node at (3.5,-0.5) {\tiny $u_i$};
        \node at (7.2,2) {$\scriptstyle\shrek$};
    \end{diagram},\] 
    agrees with the map we started with. 
\end{proof}
\end{adjustwidth}

Since this functor preserves Vermas as well as maps between Vermas labeled by elements of adjacent length, and since from Theorem \ref{thm:triangularfiltration} any object of $\D^-\Mod \CS$ has a spectral sequence of Vermas converging to (a filtration of) it, we know that this functor must act identically on the first page $E_1=E_1(\calid)$:
% \[\Delta_W\otimes_\BK \CK_{\CS^+}\specseqimplies \Id_{\D^-\Mod\CS},\] 
\[E_1(\calid_{\D^-\Mod\CS})=\Delta_W\otimes_\BK \CK_{\CS^+},\] 
where the differentials $\d^\tau$ on the right are coming from the comultiplication -- that is, $\CS^{+,\shrek}\coactson \CK_{\CS^+}(\sq)$, and $\tau(v_{(-1)})$ acts on $\Delta_W$ from the right.
% Since this functor preserves Vermas, and since from Theorem \ref{thm:triangularfiltration} any object of $\D^-\Mod \CS$ has a spectral sequence of Vermas converging to (a filtration of) it, we know that
% \[\Delta_W\otimes_\BK \CK_{\CS^+}\specseqimplies \Id_{\D^-\Mod\CS},\] 
% where the differentials $\d^\tau$ on the left are coming from the comultiplication -- that is, $\CS^{+,\shrek}\coactson \CK_{\CS^+}(\sq)$, and $\tau(v_{(-1)})$ acts on $\Delta_W$ from the right.
% \begin{THM}\label{thm:koszulBGG}
%     We have the following convergence (in the same sense as spectral sequences)
%     % \[\CS\lotimes_{\CS^{+\otimes -}}\CS^+\otimes^\tau 
%     \[\Delta_W\otimes_\BK \CK_{\CS^+}\specseqimplies \Id_{\D^-\Mod\CS},\]
%     where the differentials $\d^\tau$ on the left are coming from the comultiplication -- that is, $\CS^{+,\shrek}\coactson \CK_{\CS^+}(\sq)$, and $\tau(v_{(-1)})$ acts on $\Delta_W$ from the right.
% \end{THM}
In particular, we know that 
\[\Delta_W\otimes_\BK \CK_{\CS^+}(L_1)\simeq L_1.\] 
% where the differentials are coming from the comultiplication $\CS^{+,\shrek}\coactson \CK_{\CS^+}(L_1)$. 
% We claim that this differential $\d^\tau$ is the same as the one coming from Theorem \ref{thm:triangularfiltration}. 
By Lemma \ref{lem:preserveverma}, this differential $\d^\tau$ must agree with the one coming from Theorem \ref{thm:triangularfiltration}. 

\begin{remark}
    Here is an alternative argument for why the differentials from the Koszul perspective must agree with those on $E_1(\calid)$. 

    The following claim tells us what the filtration $\fil^w\Delta_u$ is.
\begin{LEM}
    Recalling $\fil^w\sq=\ol\imath_w\ol\imath_w^*=\CS^{\ge w}\lotimes_\CS \sq$, we have
    \[\fil^w\Delta_u=\CS^{\ge w}\lotimes_\CS \Delta_u=\begin{cases}
        \Delta_u & \te{if }u\ge w\\
        0 & \te{if }u\not\ge w
    \end{cases}.\] 
\end{LEM}
\vspace{-0.5em}
\begin{PRF}
    Recall that $\CS^{\ge w}$ is an extension of the cells $\CS^{=v}=\CS^{\ge v}e^v\otimes_\BK e^v\CS^{\ge v}$ for $v\ge w$. Hence $\CS^{\ge w}\lotimes_\CS \Delta_u$ is an extension of all of the $\CS^{=v}\lotimes_\CS \Delta_u$ for $v\ge w$. However, 
    \begin{align*}
        \CS^{=v}\lotimes_\CS\Delta_u&=A^{\ge v}e^v\otimes_\BK e^v A^{\ge v}\lotimes_\CS \Delta_u\\
        &=A^{\ge v}e^v\otimes_\BK \rhom_A(\Delta_v,\Nabla_u)\\
        &=\Delta_v\delta_{v,u};
    \end{align*}
    so if $u\not\ge w$ then $\fil^w\Delta_u$ is an extension of modules which are all zero, so $\fil^w\Delta_u=0$, and if $u\ge w$ then only the $v=u$ layer (which is $\Delta_u$) is nonzero, so that $\fil^w\Delta_u=\Delta_u$, as advertised. 
\end{PRF}
% \[\Id\simeq \CS\mathop{\lotimes}\limits_{\CS^{+\otimes-}} \CS^+\otimes^\tau_\BK \CS^{+,\shrek}\otimes^\tau_\BK \sq.\] 
% By duality, we have that 
% \[\Id\simeq \CS\lotimes_{\CS^{+\otimes-}} \CS^+\otimes^\tau_\BK \CS^{+,\shrek}\otimes^\tau_\BK \sq.\] 
Then we know
\[\fil^w L_1=\Delta_{\ge w}\otimes_\BK \CK_{\CS^+}L_1;\] 
letting $\fil^k=\bigoplus_{\ell(w)=k} \fil^w$, we have
\[\fil^k L_1=\Delta_{\ge k}\otimes_\BK \CK_{\CS^+}L_1.\] 
We have
\[\gr^k L_1\lto \fil^k L_1\lto \fil^{k+1}L_1\lto [+1],\] 
i.e. 
\begin{center}
    \begin{tikzcd}
        \Delta_k\arrow{r}{\sim} & \Delta_k &\  &\ \\
        &\Delta_{k+1}\arrow{u}{\d^\tau}\arrow{r}{\sim} & \Delta_{k+1}\arrow{r}{\d^\tau} & \Delta_k\\
        &\Delta_{k+2}\arrow{u}\arrow{r}{\sim}&\Delta_{k+2}\arrow{u} & \\
        &\vdots\arrow{u} &\vdots\arrow{u} &\\
    \end{tikzcd}
\end{center}
Hence, up to automorphisms of the terms, the map $\d^\tau\colon\Delta_k\lto\Delta_{k-1}$ agrees with the map $\gr^k L_1\lto \fil^k L_1\lto \gr^{k-1}L_1[+1]$, which is the differential predicted by Theorem \ref{thm:folklorefiltration}. In fact, replacing $L_1$ above with any complex and $\Delta_k$ with $E_1^{k,\blt}$, one can see that the differentials agree more generally. 
\end{remark}

Hence in conclusion
% \begin{THM}\label{thm:differentialsagree}
%     The differentials
% \end{THM}
% original theorem below
% \begin{THM}\label{thm:koszulBGG}
%     We have the following convergence (in the same sense as spectral sequences)
%     % \[\CS\lotimes_{\CS^{+\otimes -}}\CS^+\otimes^\tau 
%     \[\Delta_W\otimes_\BK \CK_{\CS^+}\specseqimplies \Id_{\D^-\Mod\CS},\]
%     where the differentials $\d^\tau$ on the left are coming from the comultiplication -- that is, $\CS^{+,\shrek}\coactson \CK_{\CS^+}(\sq)$, and $\tau(v_{(-1)})$ acts on $\Delta_W$ from the right. Moreover, this differential $\d^\tau$ agrees with the one coming from Theorem \ref{thm:folklorefiltration}.
% \end{THM}
\begin{THM}\label{thm:koszulBGG}
    The first page $E_1$ of the spectral sequence in Theorems \ref{thm:triangularfiltration} and \ref{thm:SoergelBGG} coincides with $\Delta_W\otimes_\BK \CK_{\CS^+}\sq$:
    \[E_1(\calid_{\D^-\Mod\CS})=\Delta_W\otimes_\BK \CK_{\CS^+}.\]
    That is, the objects coincide, and the differentials of $E_1$ also coincide with the differential $\d^\tau$ of $\Delta_W\otimes_\BK \CK_{\CS^+}$. 

    For projective modules and modules which are Koszul with respect to $\CS^+$, in particular the dominant simple, this is enough to recover the entire original spectral sequence:
    % We have the following convergence (in the same sense as spectral sequences)
    % % \[\CS\lotimes_{\CS^{+\otimes -}}\CS^+\otimes^\tau 
    % \[\Delta_W\otimes_\BK \CK_{\CS^+}\specseqimplies \Id_{\D^-\Mod\CS},\]
    % where the differentials $\d^\tau$ on the left are coming from the comultiplication -- that is, $\CS^{+,\shrek}\coactson \CK_{\CS^+}(\sq)$, and $\tau(v_{(-1)})$ acts on $\Delta_W$ from the right. Moreover, this differential $\d^\tau$ agrees with the one coming from Theorem \ref{thm:folklorefiltration}.
    \[\Delta_W\otimes_{\BK}\CK_{\CS^+}(L_1)\simeq L_1.\]
\end{THM}
What we would really like to say is that we have a convergence
\[\te{``}\Delta_W\otimes_\BK\CK_{\CS^+}\specseqimplies \Id_{\D^-\Mod\CS} \te{''},\] 
but it is unclear to us at the moment what to do about the differentials from $E_2$ onwards. But at least for all modules which are Koszul with respect to $\CS^+$, and also for projective modules, simple arithmetic with the indices shows that the only nontrivial differentials are on the first page; so for projective modules and modules Koszul with respect to $\CS^+$ (such as Kostant simple modules, and in particular $L_1$), this $\Delta\otimes\CK$ perspective sees everything. 

\begin{remark}
Here we have chosen to only have the shearing and the reflection on half of the ``double Koszul duality'' (i.e. only on $\CK_{\CS^+}$ and not $\Delta_W$) in order for the conventions to match up with those of the spectral sequence in Theorem \ref{thm:triangularfiltration}. The benefit of having the shearing and reflection on both halves is that the BGG resolution then genuinely occurs in the homological (vertical) direction; the downside is that the indexing convention no longer agrees with the typical one by algebraic geometers. 
\end{remark}

\subsection{The differentials over Soergel}
% Incidentally, from this point of view, the differentials are predicted directly by the Koszul dual $\CS^{-,!}$ of $\CS^-$. 
Since $\CS^{+,\shrek}\cong\CS^{-,!}$, Theorem \ref{thm:koszulBGG} tells us that the differentials are coming from the $\CS^{-,!}$ coaction on nilcohomology. Let us demonstrate with an example.
\begin{EX*}
    Let $\CS=\CS_\yd{1,1,1}$, which has 4 cells, corresponding to $1,s_1,s_2,s_2s_1$. The fact that there are signs sprinkled into the BGG resolution is due to the fact that we have
    % \[\begin{diagram}
    %     \draw[orange] (0,0)--(4,0);
    %     \draw[orange] (0,4)--(4,4);
    %     \draw[orange,dashed] (0,2)--(4,2);
    %     \draw[red](1,0)--(1,1);
    %     \fill[red](1,1) circle (5pt);
    %     \draw(3,0)--(3,3);
    %     \fill (3,3) circle (5pt);
    % \end{diagram}
    % \ =\ 
    % \begin{diagram}
    %     \draw[orange] (0,0)--(4,0);
    %     \draw[orange] (0,4)--(4,4);
    %     \draw[orange,dashed] (0,2)--(4,2);
    %     \draw[red](1,0)--(1,3);
    %     \fill[red] (1,3) circle (5pt);
    %     \draw(3,0)--(3,1);
    %     \fill (3,1) circle (5pt);
    % \end{diagram}\] 
    \[\begin{diagram}
        \draw[orange] (0,0)--(4/1.5,0);
        \draw[orange] (0,4/1.5)--(4/1.5,4/1.5);
        \draw[orange,dashed] (0,2/1.5)--(4/1.5,2/1.5);
        \draw[red](1/1.5,0)--(1/1.5,1/1.5);
        \fill[red](1/1.5,1/1.5) circle (5pt);
        \draw(3/1.5,0)--(3/1.5,3/1.5);
        \fill (3/1.5,3/1.5) circle (5pt);
    \end{diagram}
    \ =\ 
    \begin{diagram}
        \draw[orange] (0,0)--(4/1.5,0);
        \draw[orange] (0,4/1.5)--(4/1.5,4/1.5);
        \draw[orange,dashed] (0,2/1.5)--(4/1.5,2/1.5);
        \draw[red](1/1.5,0)--(1/1.5,3/1.5);
        \fill[red] (1/1.5,3/1.5) circle (5pt);
        \draw(3/1.5,0)--(3/1.5,1/1.5);
        \fill (3/1.5,1/1.5) circle (5pt);
    \end{diagram}\] 
    in $\CS^-$, so that in $\CS^{-,!}$ we have
    % \[\begin{raisediagram}[4pt]
    %     \draw[orange] (0,0)--(4,0);
    %     \draw[orange] (0,4)--(4,4);
    %     \draw[orange,dashed] (0,2)--(4,2);
    %     \draw[red](1,4)--(1,3);
    %     \fill[red] (1,3) circle (5pt);
    %     \draw(3,4)--(3,1);
    %     \fill (3,1) circle (5pt);
    %     % \node at (4.2,4.1) {$!$};
    %     \node at (4.2,4) {$\scriptstyle!$};
    % \end{raisediagram}
    % \ =\ -\ 
    % \begin{raisediagram}[4pt]
    %     \draw[orange] (0,0)--(4,0);
    %     \draw[orange] (0,4)--(4,4);
    %     \draw[orange,dashed] (0,2)--(4,2);
    %     \draw[red](1,4)--(1,1);
    %     \fill[red] (1,1) circle (5pt);
    %     \draw(3,4)--(3,3);
    %     \fill (3,3) circle (5pt);
    %     \node at (4.2,4) {$\scriptstyle!$};
    % \end{raisediagram};\]
    \[\begin{raisediagram}[3pt]
        \draw[orange] (0,0)--(4/1.5,0);
        \draw[orange] (0,4/1.5)--(4/1.5,4/1.5);
        \draw[orange,dashed] (0,2/1.5)--(4/1.5,2/1.5);
        \draw[red](1/1.5,4/1.5)--(1/1.5,3/1.5);
        \fill[red] (1/1.5,3/1.5) circle (5pt);
        \draw(3/1.5,4/1.5)--(3/1.5,1/1.5);
        \fill (3/1.5,1/1.5) circle (5pt);
        % \node at (4.2,4.1) {$!$};
        \node at (4.2/1.5,4/1.5) {$\scriptstyle!$};
    \end{raisediagram}
    \ =\ -\ 
    \begin{raisediagram}[3pt]
        \draw[orange] (0,0)--(4/1.5,0);
        \draw[orange] (0,4/1.5)--(4/1.5,4/1.5);
        \draw[orange,dashed] (0,2/1.5)--(4/1.5,2/1.5);
        \draw[red](1/1.5,4/1.5)--(1/1.5,1/1.5);
        \fill[red] (1/1.5,1/1.5) circle (5pt);
        \draw(3/1.5,4/1.5)--(3/1.5,3/1.5);
        \fill (3/1.5,3/1.5) circle (5pt);
        \node at (4.2/1.5,4/1.5) {$\scriptstyle!$};
    \end{raisediagram};\]
    since $\CK_{\CS^+}(L_1)$ is spanned by the diagram $\begin{diagram}
        \draw[orange] (0/2,0/2)--(4/2,0/2);
        \draw[orange] (0/2,4/2)--(4/2,4/2);
        \draw[orange,dashed] (0/2,2/2)--(4/2,2/2);
        \draw[red](1/2,2-0/2)--(1/2,3/2);
        \fill[red] (1/2,3/2) circle (5pt);
        \draw(3/2,2-0/2)--(3/2,1/2);
        \fill (3/2,1/2) circle (5pt);
    \end{diagram}^!$, we then have that the differential sends
    \[(\tau\otimes 1)\circ \Delta\colon  \;\begin{diagram}
        \draw[orange] (0/2,0/2)--(4/2,0/2);
        \draw[orange] (0/2,4/2)--(4/2,4/2);
        \draw[orange,dashed] (0/2,2/2)--(4/2,2/2);
        \draw[red](1/2,2-0/2)--(1/2,3/2);
        \fill[red] (1/2,3/2) circle (5pt);
        \draw(3/2,2-0/2)--(3/2,1/2);
        \fill (3/2,1/2) circle (5pt);
    \end{diagram}^!\;\lmto \;\begin{diagram}
            \draw[orange](0,0)--(2,0);
            \draw[orange](0,2)--(2,2);
            \draw[red](0.5,2)--(0.5,1);
            \fill[red](0.5,1)circle(5pt);
            \draw(1.5,2)--(1.5,0);
        \end{diagram} \otimes
        \begin{diagram}
            \draw[orange](0,0)--(2,0);
            \draw[orange](0,2)--(2,2);
            \draw(1.5,2)--(1.5,1);
            \fill(1.5,1)circle(5pt);
        \end{diagram}^!
        \,-\;
        \begin{diagram}
            \draw[orange](0,0)--(2,0);
            \draw[orange](0,2)--(2,2);
            \draw[red](0.5,2)--(0.5,0);
            \draw(1.5,2)--(1.5,1);
            \fill(1.5,1)circle(5pt);
        \end{diagram} \otimes
        \begin{diagram}
            \draw[orange](0,0)--(2,0);
            \draw[orange](0,2)--(2,2);
            \draw[red](0.5,2)--(0.5,1);
            \fill[red](0.5,1)circle(5pt);
        \end{diagram}^!.\]
    Hence
    \begin{align*}
        \d_{s_2s_1\to s_1}&=\sq\cdot \;\begin{diagram}
            \draw[orange](0,0)--(2,0);
            \draw[orange](0,2)--(2,2);
            \draw[red](0.5,2)--(0.5,1);
            \fill[red](0.5,1)circle(5pt);
            \draw(1.5,2)--(1.5,0);
        \end{diagram}\;,\\
        \d_{s_2s_1\to s_2}&=-\sq\cdot\;\begin{diagram}
            \draw[orange](0,0)--(2,0);
            \draw[orange](0,2)--(2,2);
            \draw[red](0.5,2)--(0.5,0);
            \draw(1.5,2)--(1.5,1);
            \fill(1.5,1)circle(5pt);
        \end{diagram}\;.
    \end{align*}
    The sign on $\d_{s_2s_1\to s_2}$ ensures that
    \[\d_{s_2\to 1}\circ \d_{s_2s_1\to s_2}=-\d_{s_1\to 1}\circ\d_{s_2s_1\to s_1},\] 
    so that
    \[\d_1\circ\d_2=0,\] 
    where $\d_1=\d_{s_1\to 1}+\d_{s_2\to 1}$ and $\d_2=\d_{s_2s_1\to s_1}\oplus \d_{s_2s_1\to s_2}$.     
\end{EX*}

    More generally, each differential $\d_k$ is the sum of all possible $\d_{w\to u}$ for $\ell(w)=k$, $\ell(u)=k-1$, and $u<w$. In terms of Soergel diagrams, the morphism $\d_{w\to u}\colon \Delta_w\lto \Delta_u$ of modules over $\CS$ is given by right multiplication by a diagram, namely
    \[\d_{w\to u}=\sq\cdot \sgn_{w\to u}\begin{diagram}
        \draw[orange](0,0)--(7,0);
        \draw[orange](0,2)--(7,2);
        \draw(0.5,0)--(0.5,2);
        \node at (1.6,1) {\tiny$\cdots$};
        \draw (2.5,0)--(2.5,2);
        \draw (3.5,2)--(3.5,1);
        \fill (3.5,1) circle (5pt);
        \draw(4.5,0)--(4.5,2);
        \node at (5.6,1) {\tiny$\cdots$};
        \draw (6.5,0)--(6.5,2);
        \node at (3.5,2.5) {\tiny $w$};
        \node at (3.5,-0.5) {\tiny $u$};
    \end{diagram},\]
where the signs $\sgn_{w\to u}$ % can be assigned by using the multiplication in the Koszul dual $\CS^{-,!}$.
appear from anti-commuting the lollipop to the correct position in $\CS^{-,!}$. 
More precisely, if $u$ can be obtained from $w$ by removing the $i$-th letter, then $\sgn_{w\to u}=(-1)^{i-1}$ (because the lollipop must be pulled past $i-1$ other lollipops). Hence 
\begin{THM}\label{thm:SoergelBGGmaps}
    In the BGG resolution over Soergel calculus of Theorem \ref{thm:SoergelBGG}, the differentials are 
% \[\d_{w\to u}\colon \Delta_w\otimes e^w \CS_\lbd^{-,!}e^1\lto \Delta_u\otimes e^u\CS_\lbd^{-,!}e^1\] 
given by pulling one lollipop (the one changing $w$ to $u$) through the tensor sign; namely, if removing the $i$-th (colored red below) transposition from the redex for $w$ gives $u$ (so $u\lessdot w$), then  
\begin{align*}
    \d_{w\to u}\colon \Delta_w\otimes e^w \CS_\lbd^{-,!}e^1[\ell(w)]&\lto \Delta_u\otimes e^u\CS_\lbd^{-,!}e^1[\ell(u)]\\
    v\otimes \begin{raisediagram}[0.5em]
            \draw[thick,orange] (0,-2.5)--(5.5,-2.5);
            \draw[thick,orange] (0,4)--(5.5,4);
            \draw (0.5,4)--(0.5,3);
            \fill (0.5,3) circle (5pt);
            % \draw (1,4)--(1,2.5);
            % \fill (1,2.5) circle (5pt);
            % \draw (1.5,4)--(1.5,2);
            % \fill (1.5,2) circle (5pt);
            % \node at (2.75,2) { $\ddots$};
            \node at (1.35,2.75) {\tiny $\ddots$};
            % \node at (1.45,2.75) {\tiny $\ddots$};
            \draw (2.25,4)--(2.25,1.25);
            \fill (2.25,1.25) circle (5pt);
            \draw[red] (2.75,4)--(2.75,0.75);
            \fill[red] (2.75,0.75) circle (5pt);
            \draw (3.25,4)--(3.25,0.25);
            \fill (3.25,0.25) circle (5pt);
            \node at (4.15,-0.1) {\tiny $\ddots$};
            % \node at (4.2,0) {\tiny $\ddots$};
            % \draw (4,4)--(4,-0.5);
            % \fill (4,-0.5) circle (5pt);
            % \draw (4.5,4)--(4.5,-1);
            % \fill (4.5,-1) circle (5pt);
            \draw (5,4)--(5,-1.5);
            \fill (5,-1.5) circle (5pt);
            \node at (2.75,4.5) {\tiny $w$};
            % \node at (5.7,4.1) {$!$};
            \node at (5.7,4) {$\scriptstyle!$};
\end{raisediagram} &\lmto 
v\cdot \pr*{(-1)^{i-1} \begin{diagram}
        \draw[thick,orange] (0,2)--(5.5,2);
            \draw[thick,orange] (0,4)--(5.5,4);
            \draw (0.5,4)--(0.5,2);
            % \fill (0.5,3) circle (5pt);
            % \draw (1,4)--(1,2.5);
            % \fill (1,2.5) circle (5pt);
            % \draw (1.5,4)--(1.5,2);
            % \fill (1.5,2) circle (5pt);
            % \node at (2.75,2) { $\ddots$};
            \node at (1.45,3) {\tiny $\cdots$};
            \draw (2.25,4)--(2.25,2);
            % \fill (2.25,1.25) circle (5pt);
            \draw[red] (2.75,4)--(2.75,3);
            \fill[red] (2.75,3) circle (5pt);
            \draw (3.25,4)--(3.25,2);
            % \fill (3.25,0.25) circle (5pt);
            \node at (4.2,3) {\tiny $\cdots$};
            % \draw (4,4)--(4,-0.5);
            % \fill (4,-0.5) circle (5pt);
            % \draw (4.5,4)--(4.5,-1);
            % \fill (4.5,-1) circle (5pt);
            \draw (5,4)--(5,2);
            % \fill (5,-1.5) circle (5pt);
            \node at (2.75,4.5) {\tiny $w$};
            % \node at (5.7,4.1) {$!$};
            \node at (2.75,1.5) {\tiny $u$};
    \end{diagram}}\otimes \begin{raisediagram}[0.5em]
            \draw[thick,orange] (0,-2.5)--(5.5,-2.5);
            \draw[thick,orange] (0,4)--(5.5,4);
            \draw (0.5,4)--(0.5,3);
            \fill (0.5,3) circle (5pt);
            % \draw (1,4)--(1,2.5);
            % \fill (1,2.5) circle (5pt);
            % \draw (1.5,4)--(1.5,2);
            % \fill (1.5,2) circle (5pt);
            % \node at (2.75,2) { $\ddots$};
            \node at (1.35,2.75) {\tiny $\ddots$};
            % \node at (1.45,2.75) {\tiny $\ddots$};
            \draw (2.25,4)--(2.25,1.25);
            \fill (2.25,1.25) circle (5pt);
            % \draw[red] (2.75,4)--(2.75,0.75);
            % \fill[red] (2.75,0.75) circle (5pt);
            \draw (3.25,4)--(3.25,0.25);
            \fill (3.25,0.25) circle (5pt);
            \node at (4.15,-0.1) {\tiny $\ddots$};
            % \node at (4.2,0) {\tiny $\ddots$};
            % \draw (4,4)--(4,-0.5);
            % \fill (4,-0.5) circle (5pt);
            % \draw (4.5,4)--(4.5,-1);
            % \fill (4.5,-1) circle (5pt);
            \draw (5,4)--(5,-1.5);
            \fill (5,-1.5) circle (5pt);
            \node at (2.75,4.5) {\tiny $u$};
            % \node at (5.7,4.1) {$!$};
            \node at (5.7,4) {$\scriptstyle!$};
\end{raisediagram}
\end{align*}
where we are multiplying the diagram $v\in\Delta_w$ on the right/bottom by the $i$-th lollipop. The differential is then $$\d_k=\bigoplus_{\ell(u)=k-1} \sum_{w\gtrdot u}\d_{w\to u}.$$
\end{THM}

\subsection{The differentials over KLR}
In terms of KLR diagrams, the morphism $\d_{w\to u}\colon\Delta_w\lto \Delta_u$ of modules over $\DCR_\lbd$ is given by right multiplication by the KLR diagram $\psi_u^w$, defined as follows. Let us define for $u\lessdot w$ the standard tableau $\tlie_u^w$ of shape $\mu(u)$ which is given by taking the standard tableau $\tlie^w$ and applying an inverse diagonal shift, while remembering the labels of the boxes, to obtain a tableau of shape $\mu(u)$. The KLR element $\psi_u^w$ is defined by
% \[\psi_{u,w}\coloneqq\te{the diagram from }1^w\te{ to }1^u\te{ with the minimum number of crossings and no dots};\]
\[\psi_u^w\coloneqq \sgn_u^w\psi(\tlie^{w}_{u},\tlie^{u})=(-1)^{\#\te{ of degree }-2\te{ crossings}}\psi_{\tlie^{w}_{u}},\]
%\sgn_{\tlie^{\mu(u)}}^{\tlie^{\mu(w)}_{\mu(u)}}
where $\sgn_{u}^{w}$ is $(-1)^{\#\te{ of degree }-2\te{ crossings}}$, a sign counting the number of degree $-2$ crossings in the diagram $\psi(\tlie^{w}_{u},\tlie^{u})=\psi_{\tlie^{w}_{u}}$. 
\begin{EX*}
    Perhaps an example is in order to illustrate how $\tlie^{w}_{u}$ is obtained. Consider the shape $\lbd=\ydia{2,2,2}$ and the elements $u=s_2s_1$ and $w=s_1s_2s_1$. Then 
    \[\tlie^{s_2s_1}=\thbo{\begin{ytableau}
        1
    \end{ytableau}}{\begin{ytableau}
        2
    \end{ytableau}}{\begin{ytableau}
        3&4&5&6
    \end{ytableau}},\qquad\tlie^{s_1s_2s_1}=\thbo{\emptyset}{\begin{ytableau}
        1&2
    \end{ytableau}}{\begin{ytableau}
        3&4&5&6
    \end{ytableau}};\] 
    to obtain $\tlie^{s_1s_2s_1}_{s_2s_1}$, we take $\tlie^{s_1s_2s_1}$ and apply the inverse diagonal shifting move $s_1$, which gives
    \[\tlie^{s_1s_2s_1}_{s_2s_1}=\thbo{\begin{ytableau}
        2
    \end{ytableau}}{\begin{ytableau}
        1
    \end{ytableau}}{\begin{ytableau}
        3&4&5&6
    \end{ytableau}}.\] 
    Then $\psi_{s_1s_2s_1,s_2s_1}$ is
    \[\psi_{s_1s_2s_1,s_2s_1}=
    \begin{diagram}
        \draw[orange](0,0)--(3.5,0);
        \draw[orange](0,2)--(3.5,2);
        \node at (0.5,-0.5) {\tiny $3$};
        \node at (1,-0.5) {\tiny $2$};
        \node at (1.5,-0.5) {\tiny $1$};
        \node at (2,-0.5) {\tiny $2$};
        \node at (2.5,-0.5) {\tiny $3$};
        \node at (3,-0.5) {\tiny $4$};
        \node at (0.5,0.5+2) {\tiny $2$};
        \node at (1,0.5+2) {\tiny $3$};
        \node at (1.5,0.5+2) {\tiny $1$};
        \node at (2,0.5+2) {\tiny $2$};
        \node at (2.5,0.5+2) {\tiny $3$};
        \node at (3,0.5+2) {\tiny $4$};
        \draw(0.5,0)--(1,2);
        \draw(1,0)--(0.5,2);
        \draw(1.5,0)--(1.5,2);
        \draw(2,0)--(2,2);
        \draw(2.5,0)--(2.5,2);
        \draw(3,0)--(3,2);
    \end{diagram}\ .\]
\end{EX*}
The below tells us that, in the language of \cite[Section 2]{bowman2022lightleaves}, $\tlie_u^w$ has a reduced path vector of length 1. 
\begin{LEM}
    Let $w\gtrdot u$ with $(a,b)w=u$, where $a<b$. One has
    \[\deg \tlie^w_u=1\] 
    and by extension $\psi(\tlie^w_u,\tlie^u)$ and $\psi^w_u$ also have
    \[\deg\psi_u^w=1.\] 
    In fact moreover 
    \[\deg\tlie_u^w\rv_{\le k}=\begin{cases}
        0 & k\le \mu(w)^{(a)}+b-a+\sum_{i=1}^{b-1} \mu(w)^{(i)}\\
        1 & k> \mu(w)^{(a)}+b-a+\sum_{i=1}^{b-1} \mu(w)^{(i)}
    \end{cases}.\]
\end{LEM}
\vspace{-0.5em}
\begin{PRF}
    Let $\rho(1)=(\rho_1,\cdotsc,\rho_\ell)$ denote the residues of the last boxes of each row of $\mu(1)$, so that $\rho_i$ is the residue of the last box of $\mu(1)^{(i)}$, or equivalently the last box of the $i$-th row of $\lbd$. The condition that $\lbd_1\ge\lbd_2\cdots\ge\lbd_\ell$ implies $\rho_1>\cdots>\rho_\ell$. It is clear that applying $w$ to $\lbd$ via $w\circ\lbd$ permutes the residues of the last boxes of each row, i.e. the list of residues of the last boxes of each row of $\mu(w)$ is $\rho(w)=(\rho_{w(1)},\cdotsc,\rho_{w(\ell)})$. The reason we are interested in this $\rho(w)$ statistic is that in any $i$-th row of $\mu(w)$, only the box of content $\rho_{w(i)}$ is removable, and only the node of content $\rho_{w(i)}+1$ is addable. 

    Since $(a,b)w=u$, we know that $w(a)>w(b)$ and there is no $a<c<b$ has $w(a)>w(c)>w(b)$. This implies there is no $a<c<b$ with $\rho_{w(a)}>\rho_{w(c)}>\rho_{w(b)}$.

    Note that to obtain $\tlie_u^w$, we begin with $\tlie^w$ and move all boxes strictly to the right of the $(\mu(w)^{(a)}+b-a)$-th box on the $b$-th row into the $a$-th row. 

    To compute $\deg\tlie_u^w\rv_{\le k}$, we scan the boxes $B$ of $\tlie_u^w\rv_{\le k}$ in the order of their labels and add up the $d_B(\shape\tlie_u^w\rv_{\le k})$. Up until (and including) the label $\mu(w)^{(a)}+b-a+\sum_{i=1}^{b-1}\mu(w)^{(i)}$ (which is the $(\mu(w)^{(a)}+b-a)$-th box of the $b$-th row) the computation is the same as that of $\tlie^w$, namely $\deg\tlie_u^w\rv_{\le k}=0$ for all $k$ in that range. Because the remainder of the $b$-th row is diagonally shifted into the $a$-th row to obtain $\tlie_u^w$, for the label $\mu(w)^{(a)}+b-a+1+\sum_{i=1}^{b-1}\mu(w)^{(i)}$ and onwards until $\sum_{i=1}^b \mu(w)^{(i)}$, we must check the addable/removable nodes in rows $a+1$ through $b$. It is a straightforward check that the boxes with these labels have contents $\rho_{w(a)}+1,\cdotsc,\rho_{w(b)}$. Since we know there is no $a<c<b$ with  $\rho_{w(a)}-1\ge \rho_{w(c)}\ge \rho_{w(b)}+1$ or $\rho_{w(a)}\ge\rho_{w(c)}+1\ge\rho_{w(b)}+2$, we know that there is no addable or removable node of content in the set $\{\rho_{w(a)}+1,\cdotsc,\rho_{w(b)}\}$ in rows $a+1$ through $b-1$. In row $b+1$ there is a single node addable with respect to the box of label $\mu(w)^{(a)}+b-a+1+\sum_{i=1}^{b-1}\mu(w)^{(i)}$, which is of content $\rho_{w(a)}+1$. 

    The rest of the boxes of $\tlie_u^w$ have the same computation as that of $\tlie^w$, as they occur below row $b$.
    
    Hence for $k\ge \mu(w)^{(a)}+b-a+1+\sum_{i=1}^{b-1}\mu(w)^{(i)}$ we have $\deg\tlie_u^w\rv_{\le k=1}$, as desired.
\end{PRF}
% \warn{insert discussion here regarding lollipops to psis}
By \cite[Theorem 2.13]{bowman2022lightleaves} %Theorem 3.12 as well?
and \cite[Section 7.2]{bowman2023klrvssoergel}, we then have that $\psi_u^w$ corresponds to the element of $\CS^+_1$ giving the differential in Theorem \ref{thm:SoergelBGGmaps}. In other words, in terms of KLR,
\[\d_{w\to u}=\sq\cdot \sgn_{w\to u}\psi_u^w=\sq\cdot \sgn_{w\to u}\sgn_{u}^{w}\psi(\tlie^{w}_{u},\tlie^{u}).\]

% \begin{EX*}
%     Perhaps an example is in order to illustrate how $\tlie^{\mu(w)}_{\mu(u)}$ is obtained. Consider the shape $\lbd=\ydia{2,2,2}$ and the elements $u=1$ and $w=s_1s_2$. Then 
%     \[\tlie^{\mu(1)}=\thbo{\begin{ytableau}
%         1&2
%     \end{ytableau}}{\begin{ytableau}
%         3&4
%     \end{ytableau}}{\begin{ytableau}
%         5&6
%     \end{ytableau}},\qquad\tlie^{\mu(s_1s_2)}=\thbo{\emptyset}{\begin{ytableau}
%         1&2&3
%     \end{ytableau}}{\begin{ytableau}
%         4&5&6
%     \end{ytableau}};\] 
%     to obtain $\tlie^{\mu(s_1s_2)}_{\mu(1)}$, we take $\tlie^{\mu(s_1s_2)}$ and apply the inverse diagonal shifting move $s_1$, which gives
%     \[\thbo{\begin{ytableau}
%         2&3
%     \end{ytableau}}{\begin{ytableau}
%         1
%     \end{ytableau}}{\begin{ytableau}
%         4&5&6
%     \end{ytableau}},\] 
%     and then applying the inverse diagonal shifting move $s_2$, which gives
%     \[\tlie^{\mu(s_1s_2)}_{\mu(1)}=\thbo{\begin{ytableau}
%         2&3
%     \end{ytableau}}{\begin{ytableau}
%         1&6
%     \end{ytableau}}{\begin{ytableau}
%         4&5
%     \end{ytableau}}.\] 
% \end{EX*}

In terms of Section \ref{sect:JTnilkoszul}, we can also write this differential in terms of KLR by saying
\[\d_{w\to u}=\sq\cdot \sgn_{w\to u}\psi_{\sh_\lbd\Plie_{\pad(w)}}^{\tlie^w}\vphi_\lbd^\te{BCH}\pr*{\begin{diagram}
        \draw[orange](0,0)--(7,0);
        \draw[orange](0,2)--(7,2);
        \draw(0.5,0)--(0.5,2);
        \node at (1.6,1) {\tiny$\cdots$};
        \draw (2.5,0)--(2.5,2);
        \draw (3.5,2)--(3.5,1);
        \fill (3.5,1) circle (5pt);
        \draw(4.5,0)--(4.5,2);
        \node at (5.6,1) {\tiny$\cdots$};
        \draw (6.5,0)--(6.5,2);
        \node at (3.5,2.5) {\tiny $w$};
        \node at (3.5,-0.5) {\tiny $u$};
    \end{diagram}}\psi_{\tlie^u}^{\sh_\lbd\Plie_{\pad(u)}}\]
for $u\lessdot w$, where $\pad(w)=\ul{w \id^{b-\ell(w)}}$ is the `padded' version of $w$, i.e. $w$ with many $\id$'s to the right. By unwinding the machinery of \cite{bowman2023klrvssoergel}, it is not hard to see that this agrees with the $\d_{w\to u}$ constructed above; indeed, as $\deg \tlie^{\mu(w)}=\deg\tlie^{\mu(u)}=0$ and $\deg\tlie^{\mu(w)}_{\mu(u)}=1$, we can see that there is only one path morphism of degree 1 between the two. 
\begin{remark}
    In the above, the diagrams $\psi_{\sh_\lbd\Plie_{\pad(w)}}^{\tlie^w}$ and $\psi_{\tlie^u}^{\sh_\lbd\Plie_{\pad(w)}}$ are only there because we need to match the morphism up with our choice of $e^w$ (which was defined using $\tlie^w\coloneqq \tlie^{\mu(w)}$) in how we defined $\Delta_w=\DCR_\lbd^{\ge w} e^w$. In the language of \cite{bowman2023klrvssoergel}, these diagrams are ``adjustment'' diagrams. Also, the inverse diagonal move is corresponding to a single reflection across a $\vrho$-shifted hyperplane in \cite{bowman2023klrvssoergel}. 

    If one did not want to use these adjustment relations, one could instead (re-)define the cell modules (or indeed the cellular basis) by using the idempotents $\vphi_\lbd^\te{BCH}(e^w)\in\DCR_\lbd$ (that defining the cellular basis this way is okay is due to \cite[Theorem 2.13, Corollary 2.14]{bowman2022lightleaves}); then, the top and bottom boundary idempotents of $\vphi_\lbd^\te{BCH}(\usdlol)$ having agreed with the conventions for cell modules, no adjustment is necessary.
\end{remark}

\subsection{The differentials over $S_n$}
In general the Brundan-Kleshchev isomorphism, though very explicit, can famously be difficult to handle sometimes. The differentials of the BGG resolution over $\DCR_\lbd$ in Theorem \ref{thm:JTBGG} can in principle be written as maps between $S_n$-modules, though in practice the unwind quickly becomes difficult. But at least for the first differential (namely $\d_{s_i\to 1}$) we can say something explicit. 

Let us define a map of $S_n$-modules
\[\on{shift}_i\colon E_{s_i\circ\lbd}\lto E_\lbd\]
by sending the standard tableaux $\tlie^{s_i\circ\lbd}$ (or any tableaux, by $S_n$-invariance) to the sum over all $\binom{\lbd_i+1}{\lbd_{i+1}}$ ways of choosing $(\lbd_i-\lbd_{i+1}+1)$ boxes (with their labels) from the $(i+1)$-th row of $s_i\circ\lbd$ and moving them upwards to the $i$-th row to obtain a tableaux of shape $\lbd$. We illustrate this map with an example. 
\begin{EX*}
    For the partition $\lbd=\ydia{4,2}$, we have that
    \begin{align*}
        \on{shift}_1\colon E_{\yd{1,5}}&\lto E_\yd{4,2}\\
        \begin{ytableau}
            1\\2&3&4&5&6
        \end{ytableau}&\lmto \begin{ytableau}
            1&4&5&6\\2&3
        \end{ytableau}+
        \begin{ytableau}
            1&3&5&6\\2&4
        \end{ytableau}+
        \begin{ytableau}
            1&3&4&6\\2&5
        \end{ytableau}+
        \begin{ytableau}
            1&3&4&5\\2&6
        \end{ytableau}+\begin{ytableau}
            1&2&5&6\\3&4
        \end{ytableau}\\
        &\qquad\qquad\qquad+\begin{ytableau}
            1&2&4&6\\3&5
        \end{ytableau}+
        \begin{ytableau}
            1&2&4&5\\3&6
        \end{ytableau}+
        \begin{ytableau}
            1&2&3&6\\4&5
        \end{ytableau}+
        \begin{ytableau}
            1&2&3&5\\4&6
        \end{ytableau}+
        \begin{ytableau}
            1&2&3&4\\5&6
        \end{ytableau}.
    \end{align*}
\end{EX*}
Then the theorem is that
\begin{THM}\label{thm:JTBGGmaps}
    The differentials of the BGG resolution over $\DCR_\lbd$ in Theorem \ref{thm:JTBGG} are the sums $\d_k=\bigoplus_{\ell(u)=k-1} \sum_{w\gtrdot u}\d_{w\to u}$ of the maps
    \[\d_{w\to u}=\sq\cdot \sgn_{w\to u}\psi_u^w=\sq\cdot \sgn_{w\to u}\sgn_{u}^{w}\psi(\tlie^{w}_{u},\tlie^{u}).\]

    In particular, when restricted to $\BC S_n$, the maps $\d_{s_i\to 1}$ can be written as
    \begin{align*}
        \d_{s_i\to 1}\colon E_{s_i\circ\lbd}&\lto E_\lbd \\
        \tlie^{s_i\circ\lbd}&\lmto \lbd_i\binom{\lbd_i-1}{\lbd_{i+1}-1} \cdot \on{shift}_i(\tlie^{s_i\circ\lbd}).
    \end{align*}
\end{THM}
Note that since we are working over $\BC$ here, the constants in the map above don't really matter.
\begin{PRF}
    That the differentials over $\DCR_\lbd$ are as claimed follows from the discussion above.

    It suffices to prove the claim for $\d_{s_i\to 1}$ for $i=1$ and $\ell(\lbd)=2$. In this case, the tableau $\tlie_1^s$ is obtained by by taking the tableau $\tlie^s$ and shifting the last $\lbd_1-\lbd_2+1$ boxes of the second row into the first row. We indicate below an example with $\lbd=(8,4)$.
    
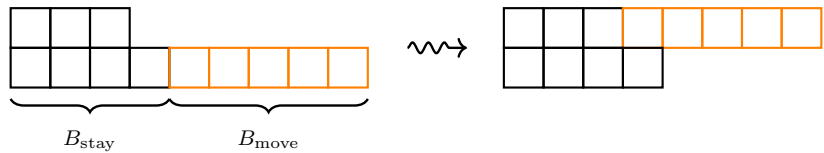
\begin{center}
        \begin{tikzpicture}[scale=0.75]
\def\s{0.7}

% ---- Left figure ----
% Upper 2x3 block (black), rows y=1 and y=2
\foreach \col in {0,1,2} {
  \foreach \row in {0,1} {
    \draw[thick] (\col*\s, \row*\s) rectangle ++(\s,\s);
  }
}
% Bottom row: cols 0-2 black, cols 3-6 orange
\foreach \col in {3} {
  \draw[thick] (\col*\s, 0) rectangle ++(\s,\s);
}
\foreach \col in {4,5,6,7,8} {
  \draw[orange, thick] (\col*\s, 0) rectangle ++(\s,\s);
}
% Braces
\draw[decorate, decoration={brace, amplitude=5pt, mirror}]
  (0,-0.2) -- (4*\s,-0.2)
  node[midway, below=8pt] {\scriptsize$B_{\mathrm{stay}}$};
\draw[decorate, decoration={brace, amplitude=5pt, mirror}]
  (4*\s,-0.2) -- (9*\s,-0.2)
  node[midway, below=8pt] {\scriptsize$B_{\mathrm{move}}$};

% ---- Squiggly arrow ----
\draw[->, decorate,
  decoration={snake, amplitude=2pt, segment length=6pt, post length=4pt}]
  (10*\s, 1*\s) -- (11.5*\s, 1*\s);

% ---- Right figure ----
\begin{scope}[xshift=8.7cm]
  % Top row: cols 0-2 black, cols 3-6 orange (at y=2)
  \foreach \col in {0,1,2} {
    \draw[thick] (\col*\s, 1*\s) rectangle ++(\s,\s);
  }
  \foreach \col in {3,4,5,6,7} {
    \draw[orange, thick] (\col*\s, 1*\s) rectangle ++(\s,\s);
  }
  % Lower 2x3 block (black), rows y=0 and y=1
  \foreach \col in {0,1,2,3} {
    \foreach \row in {0} {
      \draw[thick] (\col*\s, \row*\s) rectangle ++(\s,\s);
    }
  }
\end{scope}

\end{tikzpicture}
    \end{center}
    The corresponding diagram $\psi(\tlie_1^s,\tlie^1)$ looks like the following:
    
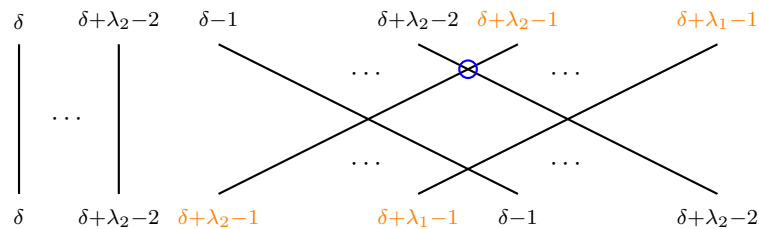
\begin{center}
        \begin{tikzpicture}[
  every node/.style={font=\small},
  thick,scale=0.66
]

%% ── Left: two straight vertical wires ────────────────────────────────────────
\draw (0,3)  -- (0,0);
\draw (2,3)  -- (2,0);

\node[above] at (0, 3.1)  {\scriptsize$\delta$};
\node[below] at (0, -0.1)  {\scriptsize$\delta$};
\node[above] at (2, 3.1)  {\scriptsize$\delta{+}\lambda_2{-}2$};
\node[below] at (2, -0.1)  {\scriptsize$\delta{+}\lambda_2{-}2$};

% dots between the two straight wires
\node at (1, 1.5) {$\cdots$};

%% ── Top labels for crossing section ──────────────────────────────────────────
\node[above] at (4,  3.1) {\scriptsize$\delta{-}1$};
\node[above] at (8,  3.1) {\scriptsize$\delta{+}\lambda_2{-}2$};
\node[above] at (10, 3.1) {\scriptsize$\textcolor{orange}{\delta{+}\lambda_2{-}1}$};
\node[above] at (14, 3.1) {\scriptsize$\textcolor{orange}{\delta{+}\lambda_1{-}1}$};

% dots between top labels
\node at (7,  2.4) {$\cdots$};
\node at (11, 2.4) {$\cdots$};

%% ── Bottom labels for crossing section ───────────────────────────────────────
\node[below] at (4,  -0.1) {\scriptsize$\textcolor{orange}{\delta{+}\lambda_2{-}1}$};
\node[below] at (8,  -0.1) {\scriptsize$\textcolor{orange}{\delta{+}\lambda_1{-}1}$};
\node[below] at (10, -0.1) {\scriptsize$\delta{-}1$};
\node[below] at (14, -0.1) {\scriptsize$\delta{+}\lambda_2{-}2$};

% dots between bottom labels
\node at (7,  0.75-0.15) {$\cdots$};
\node at (11, 0.75-0.15) {$\cdots$};

%% ── First X crossing: wires between positions 1 & 3 ─────────────────────────
%  top (4,3) -> bottom (10,0)  and  top (10,3) -> bottom (4,0)
\draw (4,  3) -- (10, 0);
\draw (10, 3) -- (4,  0);

% dots in the interior of the first crossing fan
% \node at (6.2, 1.5) {$\cdots$};

%% ── Second X crossing: wires between positions 2 & 4 ────────────────────────
%  top (8,3) -> bottom (14,0)  and  top (14,3) -> bottom (8,0)
\draw (8,  3) -- (14, 0);
\draw (14, 3) -- (8,  0);

% dots in the interior of the second crossing fan
% \node at (11.8, 1.5) {$\cdots$};

\draw[blue] (9,2.5) circle (5pt);

    \end{tikzpicture}
    \end{center}
    Here $\delta+\lbd_2-1,\cdotsc,\delta+\lbd_1-1$ are the contents of the boxes which were moved from the second row into the first (colored orange again), while $\delta-1,\cdots \delta+\lbd_2-2$ are the contents of the boxes remaining in the second row. We denote the former set of boxes by $B_\te{move}$ and the latter set by $B_\te{stay}$. Note well that there is no crossing of degree $-2$, so in particular if there were any dots anywhere on this diagram, they would commute with all crossings. Also note that there is exactly one crossing of degree $+1$, namely the one circled in blue in the diagram above.

    Recall (for instance \cite[Definition 5.9]{kleshchev2012universal} or \cite[Corollary 3.6.3]{mathas2015cyclotomic}) that, as representations over $S_n$, we have $\Res_{S_n}\Delta_1=\Ind_{S_\lbd}^{S_n}\on{triv}$. Hence to determine what $\d_{w\to u}$ is as a map of $S_n$-modules, it suffices to know what it does to the generating vector from $\on{triv}$. Since $\d_{s\to 1}$ sends the generating vector of $\Delta_{s}$ to $\psi_1^s\in\Delta_1$, it suffices to compute what element of $\BC S_n$ sends the generating vector $e^1\in\Delta_1$ to $\psi_1^s$. 

    By \cite[Definition 5.9]{kleshchev2012universal}, we know that the dots act on the generating vector $e^1$ of $\Delta_1$ by zero. We noted above that all dots commute with the crossings in $\psi_1^s$. Hence in unwinding the Brundan-Kleshchev isomorphism we can set all dots to zero. The BK isomorphism tells us that $p_ie_\beta=\frac{1}{(\beta_i-\beta_{i+1})-(y_{i+1}-y_i)}e_\beta$ for each of the crossings in $\psi_1^s$, but since we can set all dots to zero we know actually 
    \[p_ie_\beta=\frac{1}{\beta_i-\beta_{i+1}}e_\beta.\] 
    The $q_ie_\beta$'s are prescribed as
    \[q_ie_\beta=\begin{cases}
        1 &\te{for the degree }+1\te{ crossing}\\
        1-p_i & \te{for degree }0\te{ crossings}
    \end{cases}.\] 
    Hence $\psi_1^s$ corresponds under BK to
    \[\vphi_\te{BK}^{-1}(\psi_1^s)= \prod_{\te{all }\psi_i\te{ in }\psi_1^s}(s_i+p_i)\cdot\prod_{\deg\psi_i=0} \frac{1}{1-p_i}.\] 
    Note that as we range over all crossings $\psi_i$ in $\psi_1^s$, the set of $p_i$'s appearing are precisely 
    \[p_i=\frac{1}{\res b_i-\res b_i'},\] 
    where $b_i\in B_\te{move}$ and $b_i'\in B_\te{stay}$. In particular all $p_i>0$, so no cancellation can happen.
    % For the degree $+1$ crossing, we have that $q_ie_\beta=1$, while for all other crossings we have $q_ie_\beta$ the crossing as a $S_n$ element is $\psi_i=s_i+1$, for in that case $q_i=1$ in that case.
    
    It is a straightforward computation that 
    \begin{align*}
        \prod_{\psi_i} (1+p_i) &=2\cdot \prod_{\substack{b\in B_\te{move} \\ b'\in B_\te{stay} \\ \res b-\res b'>1}} \frac{\res b-\res b'+1}{\res b-\res b'}\\
        &=2\cdot \prod_{b\in B_\te{move}} \frac{\res b-(\delta-1)+1}{\res b-(\delta-1)}\\
        &\hspace{25pt}\cdot \cdots \cdot \prod_{b\in B_\te{move}} \frac{\res b-(\delta+\lbd_2-3)+1}{\res b-(\delta+\lbd_2-3)}\\
        &\hspace{75pt}\cdot \prod_{\substack{b\in B_\te{move}\\ \res b\ge\delta+\lbd_2}} \frac{\res b-(\delta+\lbd_2-2)+1}{\res b-(\delta+\lbd_2-2)}\\
        &=2\cdot \frac{\lbd_1+1}{\lbd_1}\cdot\frac{\lbd_1}{\lbd_1-1}\cdots\frac{\lbd_1-\lbd_2+2}{\lbd_1-\lbd_2+1}\\
        &\hspace{25pt} \cdots  \frac{3}{2}\cdot \frac{4}{3}\cdots \frac{\lbd_1-\lbd_2+3}{\lbd_1-\lbd_2+2}\\
        &\hspace{75pt}\cdot \frac{3}{2}\cdot\frac{4}{3}\cdots\frac{\lbd_1-\lbd_2+2}{\lbd_1-\lbd_2+1}\\
        &=2\cdot \frac{\lbd_1+1}{\lbd_2}\cdot\frac{\lbd_1}{\lbd_2-1}\cdots \frac{\lbd_1-\lbd_2+3}{2}\cdot\frac{\lbd_1-\lbd_2+2}{2}\\
        &=\binom{\lbd_1+1}{\lbd_2}.
    \end{align*}
    It is a similarly straightforward computation to check that 
    \begin{align*}
        \prod_{\deg\psi_i=0} \frac{1}{1-p_i}&=\prod_{\substack{b\in B_\te{move} \\ b'\in B_\te{stay} \\ \res b-\res b'>1}} \frac{\res b-\res b'}{\res b-\res b'-1}\\
        &=\frac{\lbd_1}{\lbd_2-1}\cdot\frac{\lbd_1-1}{\lbd_2-2}\cdots\frac{\lbd_1-\lbd_2+2}{1}\cdot\frac{\lbd_1-\lbd_2+1}{1}\\
        &=\lbd_1\binom{\lbd_1-1}{\lbd_2-1}.
    \end{align*}

    Hence the leading term in $\vphi_\te{BK}^{-1}(\psi_1^s)$ is some constant times $\prod_{\psi_i}s_i$, namely
    \[\vphi_\te{BK}^{-1}(\psi_1^s)=\lbd_1\binom{\lbd_1-1}{\lbd_2-1}\cdot w_{\tlie_1^s}+\cdots,\]
    where $w_{\tlie_1^s}=w(\tlie_1^s,\tlie^1)$ is the permutation sending $\tlie^1$ to $\tlie_1^s$. Since $w_{\tlie_1^s}$ appears as a term and all signs are positive, for this map to be a $S_n$-intertwining map all the summands of $\on{shift}_1$ must appear as well. There are $\binom{\lbd_1+1}{\lbd_2}$ many such summands. However, when we set each $s_i\in S_n$ to $1$ formally, we already have
    \[\vphi_\te{BK}^{-1}(\psi_1^s)\rv_{s_i=1}= \lbd_1\binom{\lbd_1-1}{\lbd_2-1}\binom{\lbd_1+1}{\lbd_2}.\] 
    Hence there can be no further terms beyond those in $\on{shift}_1$, as desired.
\end{PRF}
The above in particular allows us to construct $\Sigma_\lbd$ as a quotient of $E_\lbd$ by the images of the maps $\d_{s_i\to 1}$, i.e. the images of the maps $\on{shift}_i$; this was likely already known classically. %\warn{is this really what Garnir says?}

However, there is much that the above theorem lacks. We have only unwound over $S_n$ the first differential, $\d_{s_i\to 1}$; what of the other maps? The principal difficulty lies in dealing with the degree $-2$ crossings which appear in the general case -- the dots $y_i$ no longer commute across them to the bottom, and under the BK isomorphism these crossings contribute to negative signs in the element in $\BC S_n$. 
% The general case seems more difficult but still tractable -- considering the present length of the paper, we state here what we believe is true but have not yet proven. 
% \begin{CONJ}\label{conj:maps}
%     \warn{state here the general conjecture with the alternating signs}
% \end{CONJ}

\section{Nil-Koszulity of the Jacobi-Trudi algebra}\label{sect:JTnilkoszul}
In order to complete the BGG argument in Theorem \ref{thm:JTBGG}, we had to compute $\Ext^\blt_{\DCR_\lbd}(\DCR_\lbd^{\ge w}e^w, L_1)$, which we did by passing to $\CS_\lbd$ via Morita equivalence and then showing that $\CS_\lbd$ was nil-Koszul. One cannot help but wonder if it is possible to show $\DCR_\lbd$ is nil-Koszul directly. This is not so obvious a priori -- indeed, even identifying the nilalgebra is not obvious because KLR diagrams are in some sense inscrutable, but the work of \cite{bowman2023klrvssoergel} allows us to give what the nilalgebra ought to be, once we know what the nilalgebra is for Soergel. Then the previous arguments could have taken place entirely within the world of KLR.

\subsection{Recollections from \cite{bowman2023klrvssoergel}}

Let $b\ge\binom{\ell}{2}$ be an integer such that the shape $w\circ\lbd$ fits inside $w\circ(\ell\times b)$ for all $w\in W_\lbd$. Let $\wh\CS_{\ell\times b}$ be the breadth-enhanced cyclotomic Soergel calculus of BCH \cite[Section 3.2]{bowman2023klrvssoergel} in $\ell$ colors with $b$ strings (including the string corresponding to $1$). Let $\bf f$ be the idempotent from \cite[Section 3.3.2]{bowman2023klrvssoergel} corresponding to the tableaux which are obtainable as contextualized concatenations of paths of form $\Plie_\alpha,\Plie_\alpha^\flat,\Plie_\emptyset$, so that the BCH isomorphism of \cite[Theorem A]{bowman2023klrvssoergel} reads
% \begin{THM}[Bowman-Cox-Hazi]
% \[\vphi_\te{BCH}\colon \wh\CS_{\ell\times b}\simlto \bff \DCR_{\ell\times b} \bf f.\] 
% \end{THM}
\[\vphi_\te{BCH}\colon \wh\CS_{\ell\times b}\simlto \bff \DCR_{\ell\times b} \bf f.\] 
For the sake of completeness, let us now briefly recall how this isomorphism works in our case. 

For this present paragraph, the far left of each equation refers to the notation of BCH. For us $e=\infty$ and $\sigma=(-\kappa_1,\cdots,-\kappa_\ell)=-\kappa=(-(\delta-\ell+1),\cdotsc,-\delta)$. The partitions of BCH are drawn in a way transpose to the usual conventions; we stick to the usual conventions here. However this would flip the signs on all contents of boxes; for this reason the sign on $\sigma=-\kappa$ is introduced, to account for this flip. We also have $\ul h=(1,\cdotsc,1)$, so that $h=\ell$, and $b_\alpha=1$. 

The key philosophy of \cite{bowman2023klrvssoergel} is to interpret (multi-)tableaux as alcove paths and to interpret KLR diagrams and morphisms between such paths. We explain this now for the special case of our setting. A ``path'' in $\hlie^*(\gl_\ell)$ of length $n$ is an ordered sequence $(\eps_{i_1},\cdotsc,\eps_{i_n})$, where $\eps_i$ is the functional outputting the $i$-th diagonal entry. Such a path ends at $\lbd=\eps_{i_1}+\cdots+\eps_{i_n}$ (interpreted as a multipartition by putting all $\eps_i$ in the $i$-th level) and is equivalent to the information of a standard multitableaux of shape $\lbd$, where the labels of the boxes determine the order in which the steps are taken.

There are three special types of paths/tableaux, from which we build all others. They are\footnote{In \cite{bowman2023klrvssoergel}, the notation is $P_\alpha$. We use mathfrak here to make the notation uniform with symbols such as $\tlie$ for tableaux.}:
\begin{align*}
    \Plie_{s_i} &\coloneqq \Plie_i\coloneqq (\eps_1,\cdotsc,\eps_{i-1},\wh\eps_i,\eps_{i+1},\cdotsc,\eps_\ell,\eps_{i+1}),\\
    \Plie_{s_i}^\flat&\coloneqq \Plie_i^\flat\coloneqq (\eps_1,\cdotsc,\eps_{i-1},\wh\eps_i,\eps_{i+1},\cdotsc,\eps_\ell,\eps_{i}),\\
    \Plie_\id&\coloneqq \Plie_\emptyset \coloneqq  (\eps_1,\cdotsc,\eps_\ell),
\end{align*}
where the hat on $\wh\eps_i$ indicates that this step is skipped. $\Plie_{s_i}$ should be thought of as corresponding to 1 in a Deodhar word, while $\Plie_{s_i}^\flat$ should be thought of as 0. 
\begin{EX*}
    One can think of such paths as tableaux. For example, if $\ell=3$, we have
    \begin{align*}
        \Plie_{s_1}&=\thbo{\emptyset}{\begin{ytableau}
            1 & 3
        \end{ytableau}}{\begin{ytableau}
            2
        \end{ytableau}},
        &&\Plie_{s_1}^\flat=\thbo{\begin{ytableau}
            3
        \end{ytableau}}{\begin{ytableau}
            1
        \end{ytableau}}{\begin{ytableau}
            2
        \end{ytableau}},\\
        \Plie_{s_2}&=\thbo{\begin{ytableau}
            1 
        \end{ytableau}}{\emptyset}{\begin{ytableau}
            2 & 3
        \end{ytableau}},
        &&\Plie_{s_2}^\flat=\thbo{\begin{ytableau}
            1
        \end{ytableau}}{\begin{ytableau}
            3
        \end{ytableau}}{\begin{ytableau}
            2
        \end{ytableau}},\\
\end{align*}
\vspace{-3em}
    \begin{align*}    \Plie_\id&=\thbo{\begin{ytableau}
            1
        \end{ytableau}}{\begin{ytableau}
            2
        \end{ytableau}}{\begin{ytableau}
            3
        \end{ytableau}}.
    \end{align*}
\end{EX*}

The most naive way to put paths together is to concatenate two paths directly, which is called ``naive concatenation'' and denoted $\boxtimes$ in \cite{bowman2023klrvssoergel}. This is not the correct thing in this context -- instead, \cite{bowman2023klrvssoergel} constructs a ``contextualized concatenation'', denoted $\otimes_w$, which is defined by saying that if $\Plie=(\eps_{i_1},\cdotsc,\eps_{i_p})$ ends in the closure of the alcove corresponding to $w$, then its contextualized concatenation with $\Qlie=(\eps_{j_1},\cdotsc,\eps_{j_q})$ is
\[\Plie\otimes_w \frak \Qlie \coloneqq (\eps_{i_1},\cdotsc,\eps_{i_\ell},\eps_{w(j_1)},\cdotsc,\eps_{w(j_q)}).\] 
We may sometimes drop the subscript on the tensor. In terms of tableaux, this is swapping the rows of $Q$ according to $w$ (which we denote by $wQ$) and then `concatenating' (whose definition will be clear from the example below) the two tableaux together. For a word $\ulx$ in the alphabet $S\cup\{\id\}$, we can then recursively define the path $\Plie_\ulx$ as follows. If $\ulx=s\ulx'$ for $s\in S\cup\{\id\}$, then 
\[\Plie_\ulx\coloneqq \Plie_s\otimes_s \Plie_{\ulx'}.\] 
Note that we have chosen to bracket from right to left implicitly. Also note that $\Plie_{s_i}$ always ends in the alcove corresponding to $s_i$, while $\Plie_\id,\Plie_{s_i}^\flat$ always end in the alcove corresponding to $\id$. %This definition is so chosen because it corresponds to 
\begin{EX*}
    Again for $\ell=3$, we have
    \[\Plie_{s_1s_2}=\Plie_{s_1}\otimes \Plie_{s_2}=\Plie_{s_1}\boxtimes s_1(\Plie_{s_2})=\thbo{\emptyset}{\begin{ytableau}
            1 & 3
        \end{ytableau}}{\begin{ytableau}
            2
        \end{ytableau}}\boxtimes s_1{\thbo{\begin{ytableau}
            1 
        \end{ytableau}}{\emptyset}{\begin{ytableau}
            2 & 3
        \end{ytableau}} }=\thbo{\emptyset}{\begin{ytableau}
            1 & 3
        \end{ytableau}}{\begin{ytableau}
            2
        \end{ytableau}}\boxtimes \thbo{\emptyset}{\begin{ytableau}
            1 
        \end{ytableau}}{\begin{ytableau}
            2 & 3
        \end{ytableau}}=\thbo{\emptyset}{\begin{ytableau}
            1 & 3 & 4
        \end{ytableau}}{\begin{ytableau}
            2 & 5 & 6
        \end{ytableau}}\]
        and 
        \[\Plie_{s_1}^\flat\otimes \Plie_{s_2}=\Plie_{s_1}^\flat\boxtimes \id(\Plie_{s_2})=\thbo{\begin{ytableau}
            3
        \end{ytableau}}{\begin{ytableau}
            1 
        \end{ytableau}}{\begin{ytableau}
            2
        \end{ytableau}}\boxtimes {\thbo{\begin{ytableau}
            1 
        \end{ytableau}}{\emptyset}{\begin{ytableau}
            2 & 3
        \end{ytableau}} }=\thbo{\begin{ytableau}
            3 & 4
        \end{ytableau}}{\begin{ytableau}
            1 
        \end{ytableau}}{\begin{ytableau}
            2 & 5 & 6
        \end{ytableau}}\]
\end{EX*}
The point of this is that \cite{bowman2023klrvssoergel} uses contextualized concatenation to give a ``contextualized monoidal product'' on the KLR side corresponding to the usual monoidal product on the Soergel side. Let $w_\Plie^\Qlie$ be the permutation such that $w_\Plie^\Qlie(\Plie)=\Qlie$, where both sides are thought of as tableaux. Let 
\[\sgn_\Plie^\Qlie=(-1)^{\#\{1\le i<j\le n: w_\Plie^\Qlie(i)>w_\Plie^\Qlie(j),\ \res(\Plie)_i=\res(\Qlie)_j\}}\]
be a sign counting the number of degree $-2$ crossings in the KLR diagram $e_\Qlie\psi_{w_\Plie^\Qlie}e_\Plie$, and let 
\[\psi_\Plie^\Qlie\coloneqq \sgn_\Plie^\Qlie e_\Qlie \psi_{w_\Plie^\Qlie}e_{\Plie}\]
be the corresponding KLR element\footnote{After fixing a preferred reduced expression for each $w_\Plie^\Qlie$; part of the point of \cite{bowman2023klrvssoergel} is that this choice does not matter.}. Then the ``contextualized monoidal product'' for KLR is given by
\[\psi_{\Plie_1}^{\Qlie_1}\otimes \psi_{\Plie_2}^{\Qlie_2}\coloneqq \psi_{\Plie_1\otimes \Plie_2}^{\Qlie_1\otimes \Qlie_2}.\] 

Let $\wh\CS_{\ell\times b}$ be the breadth-enhanced cyclotomic Soergel calculus from \cite[Section 3.2]{bowman2023klrvssoergel}, where the boundaries of the diagrams are now words in the alphabet $S\cup\{\id\}$ and the diagrams are the same as before; we can keep track of the information of $\id$ by drawing dotted lines between them, and such dotted lines never cross each other and behave distantly to all other solid colors. Let $\bff$ be the idempotent of KLR corresponding to all tableaux which are obtainable as contextualized concatenations of the paths $\Plie_{s_i},\Plie_{s_i}^\flat,\Plie_\id$. 
We may finally state \cite[Theorem A, Section 5.7]{bowman2023klrvssoergel}.
\begin{THM}[Bowman-Cox-Hazi]
There is an isomorphism
\begin{align*}
    \vphi^\te{BCH}\colon \wh\CS_{\ell\times b}&\simlto \bff \DCR_{\ell\times b} \bf f\\
    \begin{raisediagram}[0em]
        \draw[orange](0,0)--(2,0);
        \draw[orange](0,2)--(2,2);
        \draw[gray](1,0)--(1,0.33);
        \draw[gray](1,2-0)--(1,2-0.33);
    \end{raisediagram} &\lmto e_{\Plie_\emptyset},\\
    \begin{raisediagram}[-0.5em]
        \draw[orange](0,0)--(2,0);
        \draw[orange](0,2)--(2,2);
        \draw(1,0)--(1,2);
        \node at (1,-0.5) {\tiny $i$};
    \end{raisediagram} &\lmto e_{\Plie_i},\\
    \begin{raisediagram}[-0.5em]
        \draw[orange](0,0)--(2,0);
        \draw[orange](0,2)--(2,2);
        % \draw[dotted](0.5,0)--(1.5,2);
        \draw[gray](0.5,0)--(0.5,0.33);
        \draw[gray](1.5,2)--(1.5,2-0.33);
        \draw(1.5,0)--(0.5,2);
        \node at (1.5,-0.5) {\tiny $i$};
    \end{raisediagram} &\lmto %\sgn(\Plie_{\emptyset, i},\Plie_{i,\emptyset}) 
    \psi^{\Plie_{i,\emptyset}}_{\Plie_{\emptyset, i}},\\
    \begin{raisediagram}[-0.5em]
        \draw[orange](0,0)--(2,0);
        \draw[orange](0,2)--(2,2);
        \draw(1,0)--(1,1);
        \fill (1,1) circle (5pt);
        \draw[gray](1,2)--(1,2-0.33);
        \node at (1,-0.5) {\tiny $i$};
    \end{raisediagram} &\lmto %\sgn(\Plie_i^\flat,\Plie_\emptyset)
    \psi_{\Plie_i^\flat}^{\Plie_\emptyset},\\
    \begin{raisediagram}[-0.5em]
        \draw[orange](0,0)--(2,0);
        \draw[orange](0,2)--(2,2);
        \draw[gray](0.5,2)--(0.5,2-0.33);
        \draw(1.5,0)--(1.5,2);
        \draw(0.5,0)..controls(0.5,0.75)..(1.5,1);
        \node at (0.5,-0.5) {\tiny$i$};
        \node at (1.5,-0.5) {\tiny$i$};
        % \node at (1.5,2.5) {\tiny$i$};
    \end{raisediagram} &\lmto %\sgn(\Plie_i\otimes\Plie_i^\flat,\Plie_\emptyset\otimes\Plie_i)
    \psi_{\Plie_i\otimes\Plie_i^\flat}^{\Plie_\emptyset\otimes\Plie_i},\\ 
    \begin{raisediagram}[-0.5em]
    \begin{scope}[shift={(2,0)},xscale=-1]
        \draw[orange](0,0)--(2,0);
        \draw[orange](0,2)--(2,2);
        \draw[gray](0.5,2)--(0.5,2-0.33);
        \draw(1.5,0)--(1.5,2);
        \draw(0.5,0)..controls(0.5,0.75)..(1.5,1);
        \end{scope}
        \node at (0.5,-0.5) {\tiny$i$};
        \node at (1.5,-0.5) {\tiny$i$};
        % \node at (0.5,2.5) {\tiny$i$};
    \end{raisediagram} &\lmto %\sgn(\Plie_i^\flat\otimes\Plie_i,\Plie_i\otimes \Plie_\emptyset)
    \psi_{\Plie_i^\flat\otimes\Plie_i}^{\Plie_i\otimes \Plie_\emptyset},\\ 
    \begin{raisediagram}[-0.5em]
        \draw[orange](-0.4,0)--(2.4,0);
        \draw[orange](-0.4,2)--(2.4,2);
        \draw(0,0)--(1,1)--(1,2);
        \draw(2,0)--(1,1);
        \draw[red](1,0)--(1,1)--(0,2);
        \draw[red](1,1)--(2,2);
        \node at (0,-0.5) {\tiny $i$};
        \node at (1,-0.54) {\tiny $i\!+\!1$};
        \node at (2,-0.5) {\tiny $i$};
    \end{raisediagram} &\lmto %\sgn(\Plie_{i,i+1,i},\Plie_{i+1,i,i+1}) 
    \psi_{\Plie_{i,i+1,i}}^{\Plie_{i+1,i,i+1}},\\ 
    \begin{raisediagram}[-0.5em]
        \draw[orange](0,0)--(2,0);
        \draw[orange](0,2)--(2,2);
        \draw(0.5,0)--(1.5,2);
        \draw[cyan](1.5,0)--(0.5,2);
        \node at (0.5,-0.5) {\tiny $i$};
        \node at (1.5,-0.5) {\tiny $j$};
    \end{raisediagram} &\lmto %\sgn(\Plie_{i,j},\Plie_{j,i})
    \psi_{\Plie_{i,j}}^{\Plie_{j,i}},
\end{align*}
respecting anti-involution and turning the monoidal product on the Soergel side into the contextualized monoidal product on the KLR side. 
\end{THM}
It may be worth remarking that in this isomorphism, $x\in\wh\CS_{\ell\times b}$ is sent to $\psi_{\Plie_{\ulx,\alpha}}^{\Plie_{\uly,\beta}}$, where 
\[\Plie_{\ulx,\alpha}\coloneqq \bigotimes_i \Plie_{x_i}^{\alpha_i},\]
where $x_i\in S\cup\{\id\}$ and $\alpha_i\in\{0,1\}$ if $x_i\in S$ and $\alpha_i=1$ if $x_i=\id$ and $\Plie_{x_i}^1\coloneqq \Plie_{x_i}$ and $\Plie_{x_i}^0\coloneqq \Plie_{x_i}^\flat$. When $\alpha=1^{\ell(\ulx)}$ is all $1$'s, we will simply denote
\[\Plie_\ulx\coloneqq \Plie_{\ulx,1^{\ell(\ulx)}}.\]

% \subsubsection{Another basis theorem}
% As before, let us fix a preferred redex per $w\in S_{\ell(\lbd)}$, for instance the lexicographically minimal one. Then let us define the KLR idempotent
% \[e^{\Plie(w)}\coloneqq e_{\cont\Plie_w}.\] 
% Then 

\subsection{Identifying the nilalgebra in Jacobi-Trudi}
In \cite[Proposition 7.6]{bowman2023klrvssoergel}, the authors construct an isomorphism 
\[\DCR_\lbd\simlto e_re_{W_\lbd\circ\lbd'}\DCR_{\lbd'} e_{W_\lbd\circ\lbd'}e_r,\] 
where $\lbd'$ is $\lbd+\sq$ for some addable box of residue $r$ and $e_r$ is the sum of all propagating idempotents where the last color is $r$. Using this, one can add more and more boxes until we obtain a rectangle $\lbd'=\ell\times b$ of sufficiently large width $b$. So one obtains a map
\[\DCR_\lbd\simlto e_\te{right} e_{W_\lbd\circ(\ell\times b)}\DCR_{\ell\times b}e_{W_\lbd\circ(\ell\times b)} e_\te{right},\] 
where the $e_\te{right}$ indicates that we are requiring the new colors (coming from the boxes you need to add to obtain $\ell\times b$ from $\lbd$) to appear on the right. One can then apply (a modification of) the reverse BCH isomorphism to this truncation, which we describe next.

If $\tlie\in\Std(w\circ\lbd)$, denote by $\pad(\tlie)\in\Std(w\circ(\ell\times b))$ the tableau obtained by adding new boxes in some order (for instance, for $\lbd$ we could add boxes top to bottom first, then left to right, and we could then require that boxes are added to $w\circ\lbd$ in such a way that the order of the residues of the added boxes remains fixed). One can, following \cite{bowman2023klrvssoergel}, interpret this padded tableau as a path in $\hlie^*$, and to this path we may associate the information of a word $\ulx=\ulx_{\pad\tlie}$, given by recording every wall that the path touches\footnote{In such a way that if that path travels for several steps along the wall, this is counted as touching the wall once, but if the path leaves the wall and comes back to it then this is counted as touching the wall twice.}, as well as the information of a Deodhar word $\alpha=\alpha_{\pad\tlie}$, given by keeping track of whether the path crosses the wall fully (giving 1) or simply touches it and turns back (giving 0). Then we will send %the cellular basis element $\psi_{\pad\tlie} e^{w\circ(\ell\times b)}$ of the truncation of 
\begin{align*}
    e_\te{right} e_{W_\lbd\circ(\ell\times b)}\DCR_{\ell\times b}e_{W_\lbd\circ(\ell\times b)} e_\te{right}&\lto \CS_{\ell\times b}\\
    \psi_{\pad(\tlie)} e^{\Plie\pr{w\circ(\ell\times b)}}&\lmto \olLL_{\ulx,\alpha} e^{\ulw},
\end{align*}
where $\olLL_{\ulx,\alpha}$ is the light leaf associated to the word $\ulx$ with the Deodhar word $\alpha$ such that $\ulx^\alpha=w$, $\ulx,\alpha$ are determined by the path $\pad(\tlie)$ and its interaction with the ($\vrho$-shifted) hyperplanes, and $\ulw=\ulw_{\pad(\tlie)}=\ulx^\alpha$ is the reduced expression for $w$ obtained by keeping track of the order in which ($\vrho$-shifted) hyperplanes are crossed in the tableau $\pad(\tlie)$. By construction this moreover will land inside of $\CS_\lbd$, because the inputs have already been truncated by $e_{W_\lbd\circ(\ell\times b)}$, so that the output is automatically only going to have $e^{\ulw}$ for reduced words for elements in $W_\lbd$. 
\begin{FACT}[Bowman-Cox-Hazi]
Hence we have a composition inducing Morita equivalence,
\begin{align*}
    \vphi_\lbd^\te{HCB}\colon \DCR_\lbd&\lto \CS_\lbd\\
    \psi_\tlie e^{\Plie\pr{w}}&\lmto \olLL_{\ulx,\alpha} e^{\ulw},
\end{align*}
where $\ulx=\ulx_{\pad(\tlie)}$, $\alpha=\alpha_{\pad(\tlie)}$, and $\ulw=\ulx^\alpha$ is the reduced expression for $w$ corresponding to the path $\pad(\tlie)$. We only described where half of the cellular basis goes, but the other half is dictated by noting that this map respects the anti-involution.

% We can of course compose this with the previous map $\BC S_n\lto \DCR_\lbd$ to obtain
% \[\BC S_n\lto \CS_\lbd,\] 
% whereby modules over $\CS_\lbd$ pick up a $S_n$-action. 

% Moreover, we can define a nilalgebra of $\DCR_\lbd$ to be the preimage of the lollipop subalgebra,
% \[\DCR_\lbd^-\coloneqq \vphi_\lbd^{\te{HCB},-1}(\CS_\lbd^-).\] 
\end{FACT}
% (We choose the notation $\vphi_+$ here because one needs to add boxes in this construction.)
We choose the notation $\vphi_\lbd^\te{HCB}$ here because the map goes in the opposite direction to the main theorem of \cite{bowman2023klrvssoergel}, so that the names of the authors are listed in reverse alphabetical order.
\begin{remark}
    In the above map of light leaves bases, we could have given an isomorphism \newline$e_\te{right} e_{W_\lbd\circ(\ell\times b)}\DCR_{\ell\times b}e_{W_\lbd\circ(\ell\times b)} e_\te{right}\simlto \wh\CS_{\ell\times b}$, which is what \cite{bowman2023klrvssoergel} does. The above map is this map composed with the map forgetting breadth $\wh\CS_{\ell\times b}\lto \CS_{\ell\times b}$. Alternatively, we can take a map $\CS_{\ell\times b}\lto \wh\CS_{\ell\times b}$ which is given by some type of summing or averaging over ways to insert the letter $\id$, for example the map described later in this section; then we can compose with the isomorphisms $\wh\CS_{\ell\times b}\lto e_\te{right} e_{W_\lbd\circ(\ell\times b)}\DCR_{\ell\times b}e_{W_\lbd\circ(\ell\times b)} e_\te{right}\lto \DCR_\lbd$. 
\end{remark}

Another way to construct a nilalgebra for $\DCR_\lbd$ is to map 
\[\CS_\lbd\lto \wh\CS_{\ell\times b}\] 
by sending a diagram with $b_1$ strings at the bottom and $b_2$ strings at the top to the average over all ways to insert blank strings at the top and bottom of the same diagram in such a way that in the end there are $b$ strings in total at the top and bottom, and the original strings appear in the first $\ell(w_\lbd)$ positions. There are $\binom{\ell(w_\lbd)}{b_1}\binom{\ell(w_\lbd)}{b_2}$ terms in this sum. 
% We can then follow up with the BCH isomorphism,
% \[\CS_\lbd\lto \wh\CS_{\ell\times b}\simlto \bff\DCR_{\ell\times b}\bff.\] 
% In this isomorphism, a Soergel picture is sent to a KLR morphism between paths, which are tableaux. By thinking of $w\circ\lbd$ as a left-justified shape inside $w\circ(\ell\times b)$, we can delete or `shave off' boxes from each tableau $\Plie$ to obtain a tableau $\Plie'$ of the correct shape $w\circ\lbd$; by taking the KLR morphism $\sgn(\Plie',\Qlie')\psi_{\Plie'}^{\Qlie'}$ corresponding to these two shaved tableaux of shape $w\circ\lbd$, we obtain a map
Suppose this map sends $a$ to $\wh a$; we can then follow up with the BCH isomorphism, which will send $\wh a$ to some $\sum \psi_{\Plie_{\ulx,\alpha}}^{\Plie_{\uly,\beta}}$. Given a $\Plie_\ulx$, we can define the ``shaved path'' $\sh_\lbd \Plie_\ulx$ to be the path obtained by shaving off excess boxes from the right to obtain the shape $\ulx\circ\lbd$ inside $\ulx\circ(\ell\times b)$ (appropriately relabeling in a way preserving the relative order). We can then reflect this path across the $\vrho$-shifted hyperplanes in a way according to $\alpha$ to obtain the path $\refl_\alpha\sh_\lbd\Plie_\ulx$ -- to do this, we start from the right end of the word $\ulx=x_1\cdotsc x_k$; if $\alpha_k=0$ and $H_k$ is the $k$-th hyperplane the path $\Plie_\ulx$ crosses, then we reflect the portion of the path after $H_k$, across $H_k$. We keep iterating this to the left end of the word to obtain $\refl_\alpha\sh_\lbd \Plie_\ulx$. Then we can let 
\begin{align*}
    \vphi_\lbd^\te{BCH}\colon \CS_\lbd&\lto \DCR_\lbd\\
    a&\lmto \sum \psi_{\refl_\alpha\sh_\lbd\Plie_\ulx}^{\refl_\beta\sh_\lbd\Plie_\uly}
\end{align*}
Note that $\vphi_\lbd^\te{BCH}$ is injective. We can then define %We can alternatively define
\[\DCR_\lbd^-\coloneqq \vphi_\lbd^\te{BCH}(\CS_\lbd^-).\] 
% Whichever point of view we take, w
We then have that $\DCR_\lbd^-$ is a nilalgebra of $\DCR_\lbd$ because $\DCR_\lbd\lotimes_{\DCR_\lbd^-}\bk e^w$ is concentrated in degree 0, which is so because $\CS_\lbd\lotimes_{\CS_\lbd^-}\bk e^w$ is concentrated in degree 0. Moreover, as $\DCR_\lbd^-$ is the image of a Koszul algebra under an injection, it is also Koszul. Hence
\begin{THM}\label{thm:JTnilKoszul}
    $\DCR_\lbd$ is nil-Koszul with $\DCR_\lbd^-=\vphi_\lbd^\te{BCH}(\CS_\lbd^-)$ as the nilalgebra. 
\end{THM}
\begin{remark}
    If one were so inclined, by using the Morita equivalences of \cite{bowman2023klrvssoergel} and \cite{bowman2022lightleaves}, we can bypass the arguments passing through category $\CO$ in Section \ref{sect:JTstratandrecon} -- there we had argued that $\DCR_\lbd$ was Morita equivalent to $\CS_\lbd$ by passing through category $\CO$, but BCH \cite{bowman2023klrvssoergel} shows us that this can be done directly. For instance, Theorem \ref{thm:kostant} can be phrased as saying
    \[H^\blt(\DCR_\lbd^-:L_1)=\bigoplus_{w\in W_\lbd} q^{-\ell(w)}\bk[-\ell(w)].\]
\end{remark}

\subsection{Koszul nil-duality for Jacobi-Trudi}
Now that we know how to find a nilalgebra $\DCR_\lbd^-\subset \DCR_\lbd$ (or equivalently the image $\DCR_\lbd^+$ under the anti-involution), we can redo the discussion of Section \ref{subsect:koszulperspective} directly on the KLR side:
% \[\Delta_{W_\lbd}\otimes_\BK \CK_{\DCR_\lbd^+}(\sq) \specseqimplies \Id,\]
\begin{THM}\label{thm:JTkoszulBGG}
    The first page $E_1$ of the spectral sequence in Theorems \ref{thm:triangularfiltration} applied to $\DCR_\lbd$ coincides with $\Delta_{W_\lbd}\otimes_\BK \CK_{\DCR_\lbd^+}\sq$:
    \[E_1(\calid_{\D^-\Mod\DCR_\lbd})=\Delta_{W_\lbd}\otimes_\BK \CK_{\DCR_\lbd^+}.\]
    That is, the objects coincide, and the differentials of $E_1$ also coincide with the differential $\d^\tau$ of $\Delta_{W_\lbd}\otimes_\BK \CK_{\DCR_\lbd^+}$. 

    For projective modules and modules which are Koszul with respect to $\DCR_\lbd^+$, in particular the dominant simple, this is enough to recover the entire original spectral sequence:
    \[\Delta_{W_\lbd}\otimes_\BK \CK_{\DCR_\lbd^+}(L_1)\simeq L_1.\]
    % We have the following convergence (in the same sense as spectral sequences)
    % % \[\CS\lotimes_{\CS^{+\otimes -}}\CS^+\otimes^\tau 
    % \[\Delta_W\otimes_\BK \CK_{\CS^+}\specseqimplies \Id_{\D^-\Mod\CS},\]
    % where the differentials $\d^\tau$ on the left are coming from the comultiplication -- that is, $\CS^{+,\shrek}\coactson \CK_{\CS^+}(\sq)$, and $\tau(v_{(-1)})$ acts on $\Delta_W$ from the right. Moreover, this differential $\d^\tau$ agrees with the one coming from Theorem \ref{thm:folklorefiltration}.
\end{THM}
As before, the Koszul duality functor is 
% \[\CK_{\DCR_\lbd^+}=\sh(\BK\lotimes_{\DCR_\lbd^+}\refl\sq)=\sh\rhom_{\DCR_\lbd^-}(\BK,\refl\sq^\dag)^\dag.\]
\[\CK_{\DCR_\lbd^+}=\sh\refl\bigpr{\BK\lotimes_{\DCR_\lbd^+}\sq}=\sh\refl\bigpr{\rhom_{\DCR_\lbd^-}(\BK,\sq^\dag)^\dag}=\sh\refl\bigpr{\DCR_\lbd^{+,\shrek}\otimes^\tau \sq}.\]

\section{Monoidal structure}\label{sect:monoidal}

In this section we must include the $\delta$ as part of the information of $\DCR_\lbd=\DCR_\lbd(\delta)$, because the monoidal product will depend upon this information. The analogue is going from a principal block of category $\CO$ to all blocks of category $\CO$. 

As before, we identify compositions of length $\ell$ with 1-row multicompositions with $\ell$ levels. Let 
% \[\Theta=\{\te{multipartitions }\theta=(\lbd,\delta):\nexistss\mu\domg \lbd\te{ with }\res_\delta(\mu)=\res_\delta(\lbd)\}.\] 
% \[\Theta=\{\te{multicompositions }\theta=(\mu,\delta)\},\]
\[\Theta=\{\te{compositions }\theta=(\mu,\delta)\},\]
where the contents of boxes of $\mu$ are determined by declaring that the first box of $\mu_1$ has content $\delta$. If $\theta=(\mu,\delta)$ and $\phi=(\nu,\eps)$ satisfy that $\delta-\ell(\mu)\ge \eps$, then let us define
\[\theta\sqcup\phi\coloneqq \big((\mu,0^{\delta-\ell(\mu)-\eps},\nu),\delta \big),\]
where $(\mu,0^{\delta-\ell(\mu)-\eps},\nu)=(\mu_1,\cdotsc,\mu_{\ell(\mu)},0,\cdotsc,0,\nu_1,\cdotsc,\nu_{\ell(\nu)})$ is a composition of length $\delta-\eps+\ell(\nu)$. The corresponding multicomposition is obtained by inserting $\delta-\ell(\mu)-\eps$ levels between $\mu$ and $\nu$, all of which are the zero shape. If $\delta-\ell(\mu)-\eps<0$, then we say $\theta\sqcup\phi$ is undefined.

Given $\theta=(\mu,\delta)$, let $\mu_\te{dom}$ be the multicomposition such that $\mu_\te{dom}$ is the maximal (under dominance) 1-row multicomposition with $\cont_\delta(\mu_\te{dom})=\cont_\delta(\mu)$. In Section \ref{sect:JT} we constructed the algebra $\DCR_{\mu_\te{dom}}(\delta)$ by quotienting out all multi-row multicompositions (the construction did not depend on the beginning shape being a partition), which we called `imposing a 1-row condition'. Let
\[\DCR_\theta\coloneqq \DCR_\mu(\delta)\coloneqq \mfrac{\DCR_{\mu_\te{dom}}(\delta)}{\DCR_{\mu_\te{dom}}(\delta)^{\not\le\mu}}.\]
If $\theta\sqcup\phi$ is undefined, we define $\DCR_{\theta\sqcup\phi}$ to be the zero algebra. Let
\[\DCR\coloneqq \bigoplus_{\theta\in\Theta}\DCR_\theta.\] 
\begin{PROP}\label{prop:monoidal}
    There is a monoidal product on $\DCR$ assembled from, for $\theta=(\mu,\delta),\ \phi=(\nu,\eps)$ with $\delta-\ell(\mu)-\eps\ge 0$, 
    \[\DCR_\theta\otimes\DCR_\phi\lto\DCR_{\theta\sqcup\phi},\] 
    which is given simply by the naive monoidal product, namely placing diagrams next to each other horizontally. For $\theta,\phi$ with $\delta-\ell(\mu)-\eps<0$ the monoidal product is defined to be zero. 

    Hence one can construct an induction tensor product of modules,
    \[\Ind_{\DCR_\theta\otimes\DCR_\phi}^{\DCR_{\theta\sqcup\phi}} \colon\Mod\DCR_\theta \times \Mod\DCR_\phi\lto \Mod\DCR_{\theta\sqcup\phi}.\] 
\end{PROP}
\vspace{-0.5em}
\begin{PRF}
    The naive monoidal product is the one coming from $\CR^\omega_\alpha\otimes\CR_\beta\lto \CR_{\alpha+\beta}^{\omega}$. Hence the content in this statement is that the cyclotomy conditions on $\DCR_\phi$ are respected by the naive monoidal product, as well as the 1-row condition. It turns out this is true due to the 1-row condition on the target.

    Indeed, if $\delta-\ell(\mu)-\eps\ge 1$, and $y\in\DCR_\phi$ has a dot on the first string colored by $i\in\{\eps,\cdotsc,\eps-\ell(\nu)+1\}$, then this dot can be pulled past all strings of $1\in\DCR_\theta$ so that $1\otimes y=0\in\DCR_{\theta\sqcup\phi}$; this is because the colors $\eps,\cdotsc,\eps-\ell(\nu)+1$ are all distant from the possible colors in $\DCR_\theta$, which are bounded below by $\delta-\ell(\mu)+1$. 

    If $\delta-\ell(\mu)-\eps=0$, and if $y\in\DCR_\phi$ has a dot on the first string colored by $i\in\{\eps-1,\cdotsc,\eps-\ell(\nu)+1\}$, then again the dot can be pulled to the far left to show $1\otimes y=0$. If the dot on the first string is colored $i=\eps$, then $1\otimes y=y^{\big (\mu_1,\cdotsc,\mu_{\ell(\mu)-1},(\mu_{\ell(\mu)},\nu_1),0,\nu_{2},\cdotsc,\nu_{\ell(\nu)}\big)}$, which is zero in $\DCR_{\theta\sqcup\phi}$ due to the 1-row condition, as the $\ell(\mu)$-th level of this multicomposition has 2 rows.
    % $y=y^{\big (\mu^{(1)},\cdotsc,\mu^{(\ell(\mu)-1},(\mu^{\ell(\mu)},\nu_1),0,\nu^{(2)},\cdotsc,\nu^{(\ell(\nu)}\big)}$

    Lastly, the 1-row condition on $\DCR_\phi$ is respected simply because the same conditions hold for $\DCR_{\theta\sqcup\phi}$. 
\end{PRF}
This monoidal product is perhaps surprising at first, because the cyclotomic KLR algebras famously do not have a monoidal structure (because they categorify modules, not algebras). However from the point of view of Soergel calculus this is ``obvious'' -- the first factor uses Soergel colors from $S_{\ell(\mu)}$ while the second uses Soergel colors from $S_{\ell(\nu)}$, and the two parabolically include into $S_{\ell(\mu)+\ell(\nu)}$ in such a way that their images commute with each other, so that the cyclotomic condition on the second factor is respected because barbells in $y$ will commute all the way to the far left of $x\otimes y$. 

It is natural to ask what this induction product categorifies. We hope there is an interpretation of Littlewood-Richardson from this perspective, perhaps relating to \cite[Theorem 3.5]{leclerc2000littlewood}, and we plan to study this in a future paper.

\section{Further questions}
    % We showed that $\DCR_\lbd$ is Morita equivalent to $\CS_\lbd$, which is nil-Koszul. A natural question is whether $\DCR_\lbd$ itself is nil-Koszul. In other words, we were able to find $\CS^-_\lbd\subset\CS_\lbd$ witnessing that $\CS_\lbd$ is nil-Koszul; can we find an analogous subalgebra $\DCR_\lbd^-\subset \DCR_\lbd$ showing $\DCR_\lbd$ is nil-Koszul? 
% 
    % Another natural question is how to construct a monoidal product on $\DCR=\bigoplus_\lbd \DCR_\lbd$. This gives maps $\DCR\otimes\DCR\lto \DCR$, so that we can construct ``induction products''. This should categorify Littlewood-Richardson in a non-semisimple way. In particular $\Ind L_1\otimes\cdots\otimes L_1$ should be a Verma module $\Delta_1$ for simples of $\DCR_\lbd$ for $\lbd=\lbd_1$.

    % \warn{remove at your own discretion, depending on how much of this ends up being answered above}

    % A natural question is how to construct a monoidal product on $\DCR=\bigoplus_\lbd \DCR_\lbd$. This gives maps $\DCR\otimes\DCR\lto \DCR$, so that we can construct ``induction products''. This should categorify Littlewood-Richardson in a non-semisimple way. In particular $\Ind L_1\otimes\cdots\otimes L_1$ should be a Verma module $\Delta_1$ for simples of $\DCR_\lbd$ for $\lbd=\lbd_1$.

    The monoidal product constructed in Section \ref{sect:monoidal} should categorify Littlewood-Richardson somehow, in a non-semisimple way. For certain shapes $\theta$ we expect $\DCR_\theta$ to categorify the skew Schur functions and the skew Jacobi-Trudi formulae. We plan to explore this direction in a future paper, as well as extending the results of this paper to the modular setting.

    Another possible direction is searching for geometric versions of this story. The BGG resolution for category $\CO$ was able to be witnessed geometrically via a stratification of the flag variety; can the Jacobi-Trudi story also be constructed geometrically? Also, is there some geometric explanation for the nil-Koszulity of $\CS$ and $\DCR$?

    In Section \ref{subsect:koszulperspective}, we showed that the first page of $E_1$ from Theorem \ref{thm:triangularfiltration} was the same as $\Delta_W\otimes\CK_{\CS^+}$. We were only able to match the differential on the first page, essentially because the twisting chain $\tau$ only sees degree 1 elements. It is natural to wonder if the Koszul perspective is secretly somehow capable of seeing all pages, not just the first. To do so one would likely need to define higher twisting chains which see higher degree elements. 

    In Theorem \ref{thm:JTBGGmaps}, we described the first maps of the BGG resolution of Theorem \ref{thm:JTBGG} restricted to $S_n$. A natural next step is to explicitly describe the other maps.

\bibliographystyle{alpha}
\bibliography{references}

\end{document}